\documentclass{amsart}
\usepackage{hyperref}
\numberwithin{equation}{section}

\usepackage{amssymb}
\usepackage{amsmath, amsfonts,amsthm,amssymb,amscd, verbatim,graphicx,color,multirow,tikz,tikz-cd,booktabs, caption, mathdots,bm,chngcntr,adjustbox,longtable%,enumitem
}
\usepackage{tikz-cd}
\usetikzlibrary{positioning}
\newtheorem{theorem}{Theorem}[section]
\newtheorem{corollary}[theorem]{Corollary}
\newtheorem{conjecture}[theorem]{Conjecture}
\newtheorem{lemma}[theorem]{Lemma}

\newtheorem{remark}[theorem]{Remark}
\newtheorem{proposition}[theorem]{Proposition}
\newtheorem{notation}[theorem]{Notation}
\newtheorem{problem}[theorem]{Problem}
\newtheorem{definition}[theorem]{Definition}

\def\nor#1#2{{\bf N}_{{#1}}({{#2}})}
\def\arrvline{\hfil\kern\arraycolsep\vline\kern-\arraycolsep\hfilneg}
\def\cent#1#2{{\bf C}_{{#1}}({{#2}})}
\def\Z#1{{\bf Z}({{#1}})}
\def\order#1{{\bf o}({{#1}})}

\begin{document}

\title[Normal $2$-coverings  and their generalizations]{Normal $2$-coverings of the finite simple groups and their generalizations}

\author[D.~Bubboloni]{Daniela Bubboloni}

\address{Daniela Bubboloni,
DIMAI, Universit\`a degli Studi di Firenze, Viale Morgagni 67/a,
50134 Firenze, Italy}
\email{daniela.bubboloni@unifi.it}

\author[P. Spiga]{Pablo Spiga}
\address{Pablo Spiga, Dipartimento di Matematica Pura e Applicata,
 University of Milano-Bicocca, Via Cozzi 55, 20126 Milano, Italy}
\email{pablo.spiga@unimib.it}

\author[Th. Weigel]{Thomas Weigel}
\address{Thomas Weigel, Dipartimento di Matematica Pura e Applicata,
 University of Milano-Bicocca, Via Cozzi 55, 20126 Milano, Italy}
\email{thomas.weigel@unimib.it}


  \begin{abstract}
Given a finite group $G$, we say that $G$ has weak normal covering number $\gamma_w(G)$ if  $\gamma_w(G)$ is the smallest integer with $G$ admitting proper subgroups $H_1,\ldots,H_{\gamma_w(G)}$ such that each element of $G$ has a conjugate in $H_i$, for some $i\in \{1,\ldots,\gamma_w(G)\}$, via an element in the automorphism group of $G$.

We prove that the weak normal covering number  of every non-abelian simple group is at least $2$ and we classify the non-abelian simple groups attaining $2$. As an application, we classify the non-abelian simple groups having normal covering number $2$. We also show that the weak normal covering number  of an almost simple group is at least two up to one exception.

We determine the weak normal covering number and the normal covering number of the almost simple groups having socle a sporadic simple group. Using similar methods we find the clique number of the invariably generating graph  of the almost simple groups having socle a sporadic simple group.
\end{abstract}

\maketitle
\section{Introduction}\label{s:intro}

It is well known that a finite group  cannot be the union of conjugates of a proper subgroup.
However, there are examples of finite groups which are the union of conjugates of two proper subgroups. There are rather intriguing examples of this phenomenon. For instance, it was already shown by Dye~\cite{Dye} in 1979 that, in even characteristic, the symplectic group $\mathrm{Sp}_n(q)$ is the union of conjugates of its subgroups $\mathrm{SO}^-_{n}(q)$ and
$\mathrm{SO}^{+}_{n}(q)$.

\begin{definition}\label{def:1}{\rm
Let $k$ be a positive integer and let $G$ be a finite non-cyclic group. A \textbf{\textit{normal $k$-covering}} of $G$ is a set $\mu=\{H_1,\ldots, H_k\}$ of $k$  proper subgroups of $G$ with the property that every element of $G$ belongs to the conjugate $H_i^g$,
for some $i \in \{1, \ldots , k\}$ and for some $g \in G$, that is,
$$G=\bigcup_{i=1}^k\bigcup_{g\in G}H_i^g.$$We refer to $H_1,\ldots,H_k$ as the \textbf{\textit{components}} of $\mu$.  If $H_1,\ldots,H_k$ are maximal subgroups of $G$, we refer to them as \textbf{\textit{maximal components}}.
Clearly, if $G$ is a cyclic group, then $G$ admits no normal $k$-covering, because the generators of $G$ lie in no proper subgroup.

The \textbf{\textit{normal covering number}} of the group $G$, denoted by $\gamma(G)$, is the smallest integer
$k$ such that $G$ admits a normal $k$-covering. Note that in a normal $k$-covering $\{H_1,\ldots,H_k\}$ with $k=\gamma(G)$, the proper subgroups $H_1,\ldots,H_k$ are in distinct $G$-conjugacy classes.}
\end{definition}

Finite groups having normal covering number $2$ are often  algebraically and combinatorially very interesting; for instance, the examples of Dye have been used in~\cite{guest} to give new  solutions to Perlis' equation~\cite{perlis} in algebraic number fields. One of the first motivations for investigating finite groups having normal covering number $2$ goes back to a problem in Galois theory (for more details see~\cite[Section 1]{bbh}) and is linked to the study of intersective polynomials, that is, integer polynomials having a root modulo $p$, for every prime number $p$ (see ~\cite{BS} and ~\cite{RS}).
Recently, simple groups having normal covering number $2$ have been used to construct sparsely connected invariably generating graphs~\cite{garzoni}. There is also a unexpected  connection between the theory of transitive permutation groups $G$ in which  every derangement is a $p$-element for some prime $p$ and the normal $2$-coverings of $G$. Indeed, as observed in \cite{BT},  $G$ has such property if and only if $G$ admits a normal $2$-covering with components given by a point stabilizer  and a $p$-Sylow subgroup of $G$. As an application of our work, we give a brief proof of one of the main results in~\cite{BT} in Section~\ref{BTproof}. We also recall that the normal covering number $\gamma(G)$ has connections with some questions about the generation of the group $G$. Consider the number $\kappa(G)$,  introduced
by Britnell and Mar\'oti~\cite{BM} as the maximum size of a set $X$
of conjugacy classes of $G$ such that any pair of elements from distinct classes in $X$
generates $G$. As already observed in \cite{JA}, from this definition, it is clear  that
\begin{equation}\label{kappa-gamma}
\kappa(G)\leq \gamma(G).
\end{equation}
 Now, Guralnick and Malle \cite[Theorem 1.3]{GM} and Kantor, Lubotzky and Shalev~\cite[Theorem~1.3]{KLS} have shown that for non-abelian finite simple groups we have
\begin{equation}\label{kappa-simple}
\kappa(G)\geq 2.
\end{equation}

In this paper we determine, among other things, the finite simple groups $G$ having normal covering number as small as possible, that is, with $\gamma(G)=2$. Thus, in view of  \eqref{kappa-gamma} and \eqref{kappa-simple}, our results also allow to  describe some of the finite simple groups $G$ attaining the minimum value for $\kappa$, that is  $\kappa(G)=2.$

Actually, in applications, it is often important to consider a more general situation. This more general situation arises for instance investigating small cliques in derangement graphs of permutation groups~\cite{meagher} and also in the reduction theorem in~\cite{Lucchini} for investigating arbitrary finite groups having small normal covering number, a similar reduction theorem appears also in~\cite{Praeger1KC}. Another typical example where this more general situation is important is the study of Kronecker classes in number fields, see~\cite{Preager3KC,Octic,Praeger2KC,Praeger1KC,Sa88}. 

\begin{definition}\label{def:12}{\rm
Let $k$ be a positive integer and let $G$ be a finite non-cyclic group. A \textbf{\textit{weak normal $k$-covering}} of $G$ is a family $\mu=\{H_1,\ldots, H_k\}$ of $k$ distinct proper subgroups of $G$ with the property that every element of $G$ belongs to $H_i^g$,
for some $i \in \{1, \ldots , k\}$ and for some $g \in \mathrm{Aut}(G)$, that is,
$$G=\bigcup_{i=1}^k\bigcup_{g\in \mathrm{Aut}(G) }H_i^g.$$

As for normal coverings, we refer to $H_1,\ldots,H_k$ as the \textbf{\textit{components}} of the weak normal $k$-covering and if $H_1,\ldots,H_k$ are maximal subgroups of $G$, we refer to them as \textbf{\textit{maximal components}}.
The \textbf{\textit{weak normal covering number}} of $G$, denoted by $\gamma_w(G)$, is the smallest integer
$k$ such that $G$ admits a weak normal $k$-covering. Since $\mathrm{Aut}(G)$ contains all the inner automorphisms of $G$, we have $\gamma_w(G)\le \gamma(G)$.

Note that in a weak normal $k$-covering $\{H_1,\ldots,H_k\}$ with $k=\gamma_w(G)$, the proper subgroups $H_1,\ldots,H_k$ are in distinct $\mathrm{Aut}(G)$-conjugacy classes.}
\end{definition}
Observe that, if $G$ is a group having (weak) normal covering number $k$, then $G$ has a (weak) normal $k$-covering whose components are maximal subgroups of $G$. In this paper we are mainly concerned on maximal components.

\subsection{The weak normal covering number and the normal covering number of the finite non-abelian simple groups}

The following is the main result of the paper.

\begin{theorem}\label{main theorem}
Let $G$ be a finite non-abelian simple group. Then $\gamma_w(G)\geq2$. Moreover,
\begin{itemize}
\item [(i)] $\gamma_w(G)=2$ if and only if $G$ appears in the first column of {\rm Tables~\ref{00}--\ref{000====}};
\item [(ii)] $\gamma (G)=2$ if and only if $G$ appears in the first column of {\rm Tables~\ref{00}--\ref{000====}} and in the fifth column of those tables appears at least one time a number different from $0$.
\end{itemize}
\end{theorem}

Saxl~\cite{Sa88} in his investigation on Kronecker classes has proved that $\gamma_w(G)\ge 2$, for every finite non-abelian simple group. However, since his main concern was showing $\gamma_w(G)\ne 1$, his work does not attempt to classify the finite non-abelian simple groups attaining the minimum $\gamma_w(G)=2$. We do this in our main result Theorem~\ref{main theorem}. 
We emphasize that the inequality $\gamma_w(G)\geq2$ can also be deduced from the results in \cite{KLS} and \cite{GM},  as commented in Section \ref{asso-graph}.   The hard part of Theorem \ref{main theorem} relies on the classification of the finite non-abelian simple groups for which the functions $\gamma_w$ and $\gamma$ attain their minimum value $2$.

Weak normal $1$-coverings of finite groups have already appeared in the literature a few times in the study of Kronecker classes~\cite{Preager3KC,Octic,Praeger2KC,Praeger1KC}. It is still open an interesting question of Neumann and Praeger~\cite[Problem~11.71]{Kouvorka}. We report here, to give further evidence, their question. Let $A$ be a finite group with a normal subgroup $G$. A subgroup $H$ of $G$ is called an $A$-\textit{\textbf{covering subgroup}} of $G$ if
$$G=\bigcup_{a\in A}H^a.$$
Clearly, if $H$ is an $A$-covering subgroup of $G$, then $\gamma_w(G)=1$ and $\{H\}$ is a weak normal $1$-covering of $G$.

\smallskip

\noindent\textbf{Question: }
Is there a function $f : \mathbb{N} \to \mathbb{N}$
such that whenever $H < G < A$, where $A$ is a finite group, $G$ is a normal subgroup
of $A$ of index $n$, and $H$ is an $A$-covering subgroup of $G$, such that $|G : H | \le f (n)$?

\begin{notation}\label{notation} {\rm We explain the notation in Tables~\ref{00}--\ref{000====}. There are six columns in total. In the first column, we denote the simple group $G$. In the sixth column, we have reported the number of weak normal $2$-coverings of $G$ with maximal components up to $\mathrm{Aut}(G)$-\textbf{\textit{conjugacy}}, where we say that two weak normal $2$-coverings having components $\{H_1,K_1\}$ and $\{H_2,K_2\}$ are $\mathrm{Aut}(G)$-\textit{\textbf{conjugate}} if there exist $\varphi,\psi\in \mathrm{Aut}(G)$ with either $$H_1^\varphi=H_2 \hbox{ and }K_1^\psi=K_2,\,\, \hbox{or } H_1^\varphi=K_2 \hbox{ and }K_1^\psi=H_2.$$

In the second and in the third column, we present the maximal components $H$ and $K$ for the weak normal $2$-coverings of $G$. We present a row for each $\mathrm{Aut}(G)$-conjugacy class. In particular, for a fixed $G$, the number of rows appearing in Tables~\ref{00}--\ref{000====} for $G$ is exactly the number appearing in the sixth column. For instance, when $G:=A_6$, there are two $\mathrm{Aut}(G)$-conjugacy classes of weak normal $2$-coverings and in the two rows concerning $A_6$ we have reported
the two weak normal $2$-coverings with maximal components $\{A_6\cap(S_2\times S_4),A_5\}$ and $\{A_6\cap(S_3\mathrm{wr}S_2),A_5\}$.

In the fourth column, we collect some very basic comments (typically reminding  some isomorphisms between the groups in the list).

We finally explain the number appearing in the fifth column. We say that two weak normal $2$-coverings of $G$ having components $\{H_1,K_1\}$ and $\{H_2,K_2\}$  are $G$-\textit{\textbf{conjugate}} if there exist $\varphi,\psi\in G$ with either $$H_1^\varphi=H_2 \hbox{ and }K_1^\psi=K_2,\,\, \hbox{or } H_1^\varphi=K_2 \hbox{ and }K_1^\psi=H_2.$$
Clearly, this defines an equivalence relation on the set of weak normal $2$-coverings, which refines the equivalence relation given by the $\mathrm{Aut}(G)$-conjugacy classes. Now, let $\{H,K\}$ be a weak normal $2$-covering of $G$ with maximal components appearing  in some row of Tables~\ref{00}--\ref{000====}. The $\mathrm{Aut}(G)$-conjugacy class $\{\{H^\varphi,K^\psi\}\mid \varphi,\psi\in\mathrm{Aut}(G)\}$ is a union  of $G$-conjugacy classes. Observe that only some of these $G$-conjugacy classes (possibly none) give rise  in fact to normal $2$-coverings and, in the fifth column we report their number.
Coming back to our example of $G:=A_6$,  the $\mathrm{Aut}(G)$-conjugacy class  represented by the weak normal $2$-covering with components $\{A_6\cap(S_2\times S_4),A_5\}$ splits into two $G$-conjugacy classes of normal $2$-coverings.
A similar comment applies for the components $\{A_6\cap(S_3\mathrm{wr}S_2),A_5\}$. Analogously, from Table~\ref{00}, $G:=G_2(3)$ has a unique $\mathrm{Aut}(G)$-conjugacy class of weak normal $2$-coverings, but none of the corresponding $G$-classes gives rise to a normal $2$-covering.
When an $\mathrm{Aut}(G)$-conjugacy class gives rise to at least two normal $2$-coverings with maximal components, information about the embeddings of the corresponding components in $G$ can be found directly in the statements  related to $G$. For instance, the two normal $2$-coverings of $\mathrm{Sp}_4(q)$, for $q\geq 8$ even, arising from the unique $\mathrm{Aut}(\mathrm{Sp}_4(q))$-conjugacy class of weak normal $2$-coverings are described in Lemma \ref{dimension4evensymplectic}.

Note that, by adding the numbers appearing on the fifth column for a given group $G$, we obtain the number of normal $2$-coverings of $G$, up to $G$-conjugacy.

}
\end{notation}

The reader might find some similarities between the maximal components appearing in Tables~~\ref{00}--\ref{000====} and the non-trivial maximal factorizations of simple groups classified by Liebeck, Praeger and Saxl in~\cite{LPS3,LPS4}. There are simple groups (like the symplectic groups $\mathrm{PSp}_6(3^f)$ with  $f\ge 2$) admitting normal $2$-coverings which do not admit non-trivial factorizations; vice versa,  there are simple groups (like alternating groups of prime degree $p\ge 11$)  admitting non-trivial factorizations which do not admit normal $2$-coverings. Therefore, there is no direct connection between normal $2$-coverings and factorizations. For the same reason, the reader might find some similarities between the groups appearing in Tables~~\ref{00}--\ref{000====} and the factorizations of the form $G={\bf N}_G(\langle x\rangle){\bf N}_G(\langle y\rangle)$ of the almost  simple groups $G$ obtained in~\cite{GGS}.

 In our proof of Theorem~\ref{main theorem}, we keep track of the maximal components involved in our weak normal $2$-coverings and hence we obtain the following refinement.

\begin{theorem}\label{main theorem1}
Let $G$ be a finite non-abelian simple group and let $\mu=\{H,K\}$ be  a weak normal $2$-covering of $G$ with maximal components. Then  the pair $(H,K)$ appears in {\rm Tables~\ref{00}--\ref{000====}}, up to  $\mathrm{Aut}(G)$-conjugacy. Moreover, $\mu$ gives rise to at least a normal $2$-covering of $G$ if and only if in the corresponding row of the fifth column of those tables appears a number greater than $0.$
\end{theorem}
The notation we use to describe the components $H$ and $K$ is standard. In particular, for the simple classical groups we use the notation in~\cite{bhr,kl}. We try to give as much information as possible on $H$ and $K$ in Tables~\ref{00}--\ref{000====}, however when this is not possible we just give a rough description.

In particular, Theorem~\ref{main theorem1} gives a complete classification of the finite simple groups $G$ with $\gamma(G)=2$ or with $\gamma_w(G)=2$ and of the (weak) normal $2$-coverings with maximal components. An immediate consequence of Theorem~\ref{main theorem1} is the following corollary.

\begin{corollary}\label{corollary}
Let $G$ be a finite non-abelian simple group. Then $\gamma_w(G)=2$ if and only if one of the following holds
\begin{enumerate}
\item $\gamma(G)=2$ and $G$ is isomorphic to one of the groups listed in Table~$\ref{tablegamma2}$,
\item $G$ is the alternating group $A_9$,
\item $G$ is the sporadic group $M_{12}$,
\item $G$ is the exceptional group of Lie type $G_2(3)$,
\item $G$ is the unitary group $\mathrm{PSU}_6(2)$,
\item $G$ is the orthogonal group $\mathrm{P}\Omega_8^+(2)$ or $\mathrm{P}\Omega_8^+(3)$.
\end{enumerate}
\end{corollary}

\begin{table}[ht]
\begin{tabular}{c|c}
\toprule[1.5pt]
%\multicolumn{4}{c||}{Simple groups}&\multicolumn{3}{|c}{Non simple groups}\\
%\midrule[1.5pt]
Group&details\\
\midrule[1.5pt]
alternating&$A_5$, $A_6$, $A_7$, $A_8$\\
sporadic&$M_{11}$\\
exceptional& $G_2(2^a)$ with $a\ge 2$, $G_2(2)'$, ${}^2G_2(3)'$\\
& ${}^2F_4(2)'$, $F_4(3^a)$ with $a\ge1$\\
linear &$\mathrm{PSL}_2(q)$ with $q\ge 4$\\
& $\mathrm{PSL}_3(q)$\\
& $\mathrm{PSL}_4(q)$\\
unitary&$\mathrm{PSU}_2(q)\cong\mathrm{PSL}_2(q)$ with $q\ge 4$\\
& $\mathrm{PSU}_3(3^a)$, $\mathrm{PSU}_3(5)$\\
& $\mathrm{PSU}_4(q)$\\
symplectic& $\mathrm{PSp}_2(q)\cong\mathrm{PSL}_2(q)$ with $q\ge 4$\\
& $\mathrm{PSp}_4(3)\cong\mathrm{PSU}_4(2)$, $\mathrm{PSp}_4(2)'\cong A_6$\\
& $\mathrm{PSp}_n(2^a)$ with $(n,a)\ne (4,2)$\\
& $\mathrm{PSp}_6(3^a)$\\
orthogonal&$\Omega_{3}(q)\cong\mathrm{PSL}_2(q)$, $\Omega_5(3)\cong\mathrm{PSp}_4(3)$,\\
& $\mathrm{P}\Omega_4^-(q)\cong\mathrm{PSL}_2(q^2)$, $\mathrm{P}\Omega_6^-(q)\cong\mathrm{PSU}_4(q)$\\
& $\mathrm{P}\Omega_6^+(q)\cong\mathrm{PSL}_4(q)$\\
\bottomrule[1.5pt]
\end{tabular}
\caption{Classification of finite simple groups $G$ with $\gamma(G)=2$. }
\label{tablegamma2}
\end{table}
In Table~\ref{tablegamma2}, we have highlighted some isomorphisms among non-abelian simple groups. Recall also that $\mathrm{P}\Omega_4^+(q)\cong\mathrm{PSL}_2(q)\times\mathrm{PSL}_2(q)$ is not  simple.

It follows immediately from Corollary~\ref{corollary} that, apart from six exceptional cases, for each non-abelian simple group $G$, we have $\gamma(G)=2$ if and only if $\gamma_w(G)=2$.  We actually conjecture that a similar pattern holds in general and hence, in this case, nature is not as diverse as it can possibly be: $\gamma(G)$ always coincides with $\gamma_w(G)$, except for a ``low level noise''.
\begin{conjecture}\label{conjecture1}
{\rm There exists a function $f:\mathbb{N}\to\mathbb{N}$ such that, if $G$ is a finite non-abelian simple group with $\gamma_w(G)=m$, then either $\gamma(G)=m$, or $|G|\le f(m)$, or $G\cong\mathrm{P}\Omega_n^+(q)$.}
\end{conjecture}
Clearly, for each $m$, it is of independent interest classifying the finite non-abelian simple groups with $\gamma_w(G)=m\ne \gamma(G)$. Note that Corollary~\ref{corollary} gives that classification when $m=2$. In Conjecture~\ref{conjecture1}, we believe that $\gamma_w(\mathrm{P}\Omega_n^+(q))$ and $\gamma(\mathrm{P}\Omega_n^+(q))$ differ for infinitely many values of $n$ and $q$, because of the role of graph automorphisms.

Inspired by the work in this paper, we now also make a conjecture concerning normal coverings of finite non-abelian simple groups of bounded cardinality. Broadly speaking, we conjecture that, for every positive integer $c$, the number of finite simple groups having normal covering at most $c$ is ``small''. However, we need to clarify what we mean with small.
\begin{conjecture}\label{conjecture}
{\rm There exists a function $f:\mathbb{N}\to\mathbb{N}$ such that if $G$ is a finite non-abelian simple group with $\gamma(G)\le c$, then one of the following holds:
\begin{itemize}
\item $|G|\le f(c)$,
\item $G$ is a finite simple group of Lie type having Lie rank at most $f(c)$,
\item $G=\mathrm{Sp}_n(q)$ with $q$ even.
\end{itemize}}
\end{conjecture}
We have little evidence towards Conjecture~\ref{conjecture}. For linear groups, it follows from the work of Britnell and Mar\'oti~\cite{BM}. For alternating groups, it follows from the linear bounds on the normal covering number of the symmetric and alternating groups in~\cite{JA,BPS2}.

Clearly, for a fixed $c$, the classification of the finite non-abelian simple groups having normal covering number at most $c$ can be rather hard and, in this context, our paper deals with the first meaningful case $c=2$.

Our Tables~\ref{00}--\ref{000====} look rather complex and difficult and, in a sense, they are, because we are pinning down one by one all weak normal $2$-coverings. However, at a careful analysis, some pattern arises for classical groups. Indeed, for the classical groups, there are weak normal $2$-coverings which hold for each value of $q$: these might be regarded as generic weak normal $2$-coverings, as defined in~\cite{blw}. All other weak normal $2$-coverings that arise are when $q$ is small.

\subsection{The normal $2$-coverings of the almost simple groups and the weak normal covering number of the almost simple groups}\label{sec:introintro}
Weak normal $2$-coverings of non-abelian simple groups allow to shed light on the normal $2$-coverings of almost simple groups. Let $A$ be an almost simple group with socle the non-abelian simple group $G$. Suppose $A$ admits a normal $2$-covering with components $H\leq A$ and $K\leq A$. There are only three kinds of choice for those components:
\begin{itemize}
\item[ND:]  $H\ngeq G$ and $K\ngeq G$,
\item[DI:] $G\le H\cap K$,
\item[DII:]  replacing $H$ with $K$ if necessary, $H\ge G$ and $K\ngeq G$.
\end{itemize}
The label ND stands for non-degenerate, DI for degenerate of first type and DII for degenerate of second type.

 In the first case, that is ND, from $A=\bigcup_{a\in A}(H^a\cup K^a)$, we deduce $$G=\bigcup_{a\in A}((H\cap G)^a\cup (K\cap G)^a)$$ with $H\cap G$ and $K\cap G$ proper subgroups of $G$. Therefore $\gamma_w(G)=2$ and hence $G$ is one of the groups in the first column of Tables~\ref{00}--\ref{000====}. Moreover, building on that idea, we prove the following result on the weak normal covering number of almost simple groups.
\begin{theorem}\label{main theorem2}
Let $X$ be a finite almost simple group. Then one of the following holds
\begin{enumerate}
\item\label{main theorem2_1} $\gamma_w(X)\geq 2$,
\item\label{main theorem2_2} the socle $G$ of $X$ is $\mathrm{P}\Omega_8^+(q)$, $q$ is odd, $\mathrm{PO}_8^+(q)\le X$, $\mathrm{Aut}(X)$ contains a triality automorphism of $G$ and $X$ does not contain a triality automorphism of $G$. Moreover, each weak normal $1$-covering $\{H\}$ of $X$ satisfies $G\le H$.
\end{enumerate}
\end{theorem}
We observe that there exist almost simple groups $X$ with $\gamma_w(X)=1$ and hence instances in part~\eqref{main theorem2_2}  of Theorem~\ref{main theorem2} do arise.  For instance, let $G:=\mathrm{P}\Omega_8^+(q)$ with $q$ an odd prime number and let $A:=\mathrm{Aut}(G)$. Observe
that the group of outer automorphisms $\mathrm{Out}(G)$ of $G$ is isomorphic to the symmetric group $S_4$. Therefore, $A$ contains a normal subgroup $X$ with $G\le X$ and with $X/G$ equals to the Klein group of $\mathrm{Out}(G)$. Let $H:=\mathrm{PO}_8^+(q)$ and observe that
$$G\le H\le X$$
and $|H:G|=|X:H|=2$.
Now, the conjugates of $H$ under $A$ cover the whole of $X$, that is, $X=\bigcup_{a\in A}H^a$, and hence $\{H\}$ is a weak normal $2$-covering of $X$. The main contribution in Theorem~\ref{main theorem2} is that all almost simple groups $X$ with $\gamma_w(X)=1$ arise from a similar construction.
For more details on the possible groups $X$ arising in part~\eqref{main theorem2_2}  of Theorem~\ref{main theorem2} see the proof of Theorem~\ref{main theorem2} and Remark~\ref{remark:final} in Section~\ref{sec:final}.

 Degenerate normal $2$-coverings of $A$ of the first type  correspond to normal $2$-coverings of the quotient group $A/G$. Note that, when $\mathrm{Aut}(G)/G$ is cyclic (like for sporadic simple groups and alternating groups of degree different from $6$), these normal  $2$-coverings do not arise. In Theorem~\ref{theorem:silly}, we do give a complete classification of the non-abelian simple groups admitting a $2$-covering of this type.  
 
 The degenerate normal $2$-coverings of $A$ of the second type are studied in Theorem~\ref{theorem:silly2}. The analysis of these degenerate normal $2$-coverings is intimately related to the Memoir of Guralnick, M\"{u}ller and Saxl~\cite{GMS}. For more details see Section~\ref{sec:degenerateasdf}. The statement of Theorem~\ref{theorem:silly2} is too technical to be reported in this introductory section. Here we simply give  an  elementary example of these rather peculiar normal $2$-coverings. The group $A:=\mathrm{P}\Gamma\mathrm{L}_2 (8)$ admits a normal 2-covering using $H := \mathrm{PSL}_2(8)$ and $K := {\bf N}_A (P )$, where $P$ is a $3$-Sylow of $A$. This example is part of an infinite family of examples found in Theorem~\ref{theorem:silly2} part~\eqref{parTT3}.

\subsection{The invariably generating graph and the Aut-invariably generating graph}\label{asso-graph}
The connection between $\gamma(G)$ and $\kappa(G)$ runs deeper than what we have mentioned in~\eqref{kappa-gamma}. The right context to talk about this is the \textbf{\textit{invariably generating graph}}. Let $G$ be a group and let $g_1,\ldots,g_\kappa\in G$. Then, $\{g_1,\ldots,g_\kappa\}$ is said to \textbf{\textit{invariably generate}} $G$ if, for every $x_1,\ldots,x_\kappa\in G$, we have
$$G=\langle g_1^{x_1},\ldots,g_\kappa^{x_\kappa}\rangle.$$
The reference work on invariable generation is the paper of Kantor, Lubotzky and Shalev~\cite{KLS}, where among other things it is proved~\cite[Theorem~$1.3$]{KLS} that every non-abelian finite simple group is invariably generated by two elements.
Recall that there is a combinatorial gadget $\Lambda(G)$, the invariably generating graph of $G$, that can be efficiently used to investigate invariably generating pairs. The vertices of $\Lambda(G)$ are the non-trivial conjugacy classes of $G$ and two conjugacy classes $g_1^G$ and $g_2^G$ are declared to be adjacent if  the set  $\{g_1,g_2\}$ invariably generates $G$. Using this terminology, $\kappa(G)$ is the clique number of $\Lambda(G)$, that is, the largest cardinality of a set of pair-wise adjacent vertices of $\Lambda(G)$. The invariably generation of groups has tight connection with groups admitting Beauville structures and with the spread of a group. It is an interesting invariant associated to a group that has recently received some attention~\cite{Garzoni1,Garzoni2,Garzoni3} and its investigation requires some deep results, like the Fulman-Guralnick solution~\cite{G1,G2,G22,G3} of the Boston-Shalev~\cite{Boston} conjecture on derangements in finite simple groups.

Here, we simply observe that, at the moment, for non-abelian simple groups, we only have  a handful of examples where $\gamma(G)\ne \kappa(G)$. For instance, $\kappa(G_2(3))=2=\kappa(J_2)$, but $\gamma(G_2(3))\ne 2\ne \gamma(J_2)$. We are therefore wondering whether there exist only  finitely many non-abelian simple groups $G$ with $\kappa(G)\ne \gamma(G)$. We doubt very much that this is the case and hence we propose the following more general question.
\begin{problem}\label{proble1}
{\rm Classify the finite non-abelian simple groups $G$ with $\kappa(G)\ne \gamma(G)$.}
\end{problem}
In the paper we completely classify the sporadic simple groups $G$ with $\kappa(G)\ne \gamma(G)$ (see Table \ref{tableSporadic}).
Moreover, motivated by Theorem~\ref{main theorem}, by \eqref{kappa-gamma} and by \eqref{kappa-simple}, we propose the following.
\begin{problem}\label{proble2}
{\rm Classify the finite non-abelian simple groups $G$ with $\kappa(G)=2$.}
\end{problem}

Inspired by our definition of weak normal covering and by some results in the recent literature  like Theorem 5.1 in \cite{KLS} and Corollary 7.2 in \cite{GM}, we propose a new graph to be associated with a group. Let $G$ be a group and let $g_1,\ldots,g_\kappa\in G$. Then, $\{g_1,\ldots,g_\kappa\}$ is said to \textbf{\textit{Aut-invariably generate}} $G$ if, for every $x_1,\ldots,x_\kappa\in \mathrm{Aut}(G)$, we have
$$G=\langle g_1^{x_1},\ldots,g_\kappa^{x_\kappa}\rangle.$$
Of course, if  $\{g_1,\ldots,g_\kappa\}$ Aut-invariably generates $G$, then $\{g_1,\ldots,g_\kappa\}$ invariably generates $G$ too.
We define the \textbf{\textit {Aut-invariably generating graph}}  $\Upsilon(G)$,  considering as vertices the non-trivial $\mathrm{Aut}(G)$-conjugacy classes of $G$ and two classes $g_1^{\mathrm{Aut}(G)}$ and $g_2^{\mathrm{Aut}(G)}$ to be adjacent if  the set $\{g_1,g_2\}$ Aut-invariably generates $G$. We define next, $\kappa_w(G)$ as the clique number of $\Upsilon(G)$.
It is easily checked that the following inequalities hold
\begin{equation}\label{new-kappa}
\kappa_w(G)\leq \gamma_w(G),\qquad \kappa_w(G)\leq \kappa(G). \qquad
\end{equation}
By \cite[Theorem 5.1]{KLS}  or by \cite[Corollary 7.2]{GM},  for every non-abelian simple group $G$ with the exception of $G=\mathrm{P}\Omega_8^+(q)$ for $q\leq 3$, we have 
\begin{equation}\label{kappa-weak-simple}
\kappa_w(G)\geq 2.
\end{equation}
Using  the computer algebra systems \textrm{GAP}~\cite{GAP} and \texttt{magma}~\cite{magma} we computed the values of $ \kappa_w$ for the two missing groups above:
\begin{equation}\label{kappa-weak-ort}
\kappa_w(\mathrm{P}\Omega_8^+(2))=1\qquad \hbox{and}\qquad  \kappa_w(\mathrm{P}\Omega_8^+(3))=2.
\end{equation}

In particular, by \eqref{new-kappa} and \eqref{kappa-weak-simple}, we deduce that, for every non-abelian simple group $G$ with the exception of $G=\mathrm{P}\Omega_8^+(q)$ for $q\leq 3$, we have 
$\gamma_w(G)\geq 2$. On the other hand, in  Lemma \ref{dimension8orthogonal+}, we show that $\gamma_w(\mathrm{P}\Omega_8^+(2))=\gamma_w(\mathrm{P}\Omega_8^+(3))=2$ (see also Table~\ref{000====}). Thus, with no exceptions, for every non-abelian simple group $G$, we have
\begin{equation}\label{kappa-weak-ineq}
\gamma_w(G)\geq 2.
\end{equation}
In Table~\ref{tableSporadic} we have reported  the values of $\gamma$, $\gamma_w$, $\kappa$ and $\kappa_w$ for the almost simple group having socle a sporadic group.
These values are computed by the invaluable help of computer algebra systems \textrm{GAP}~\cite{GAP} and \texttt{magma}~\cite{magma}. Details on these computations are in Section~\ref{section:new}.

Looking at Table~\ref{tableSporadic} we see examples of sporadic simple groups $G$ with  $\gamma_w(G)\ne \kappa_w(G)$. For instance, $\kappa_w(J_2)=2\neq \gamma_w(J_2)=3$. On the other hand there are no examples of sporadic simple groups $G$ with $\kappa_w(G)\neq \kappa(G)$. 
We then propose  the following question.
\begin{problem}\label{proble1-w}
{\rm Classify the finite non-abelian simple groups $G$ with $\kappa_w(G)\ne \gamma_w(G)$ and those with $\kappa_w(G)\ne \kappa(G)$.}
\end{problem}
Of course Table \ref{tableSporadic} completely classifies the sporadic simple groups $G$ with $\kappa_w(G)\ne \gamma_w(G)$.

Moreover, motivated by Theorem~\ref{main theorem}, by  \eqref{new-kappa}, \eqref{kappa-weak-simple} and \eqref{kappa-weak-ort}, we propose the following last problem.
\begin{problem}\label{proble2-w}
{\rm Classify the finite non-abelian simple groups $G$ with $\kappa_w(G)=2$.}
\end{problem}
Certainly, by\eqref{new-kappa}, some of the finite non-abelian simple groups $G$ such that $\kappa_w(G)=2$ come as those satisfying $\gamma_w(G)=2$ and hence  are described in Theorem~\ref{main theorem}.
\begin{table}[h!]
\begin{tabular}{c|c|c|c||c|c|c}
\toprule[1.5pt]
%\multicolumn{4}{c||}{Simple groups}&\multicolumn{3}{|c}{Non simple groups}\\
%\midrule[1.5pt]
Group & $\gamma(-)$&$\gamma_w(-)$&$\kappa(-)=\kappa_w(-)$&Group &$\gamma(-)=\gamma_w(-)$&$\kappa(-)=\kappa_w(-)$\\
\midrule[1.5pt]
$M_{11}$&2&2&2&$M_{12}.2$&3&2\\
$M_{12}$&3&2&2&$M_{22}.2$&3&2\\
$M_{22}$&3&3&3&$J_2.2$&3&3\\
$M_{23}$&3&3&3&$Suz.2$&4&4\\
$M_{24}$&3&3&3&$HS.2$&3&3\\
$J_{1}$&4&4&4&$McL.2$&3&2\\
$J_{2}$&3&3&2&$He.2$&4&3\\
$J_{3}$&3&3&3&$Fi_{22}.2$&4&3\\
$J_{4}$&7&7&7&$Fi_{24}$&5&5\\
$Co_{1}$&5&5&4&$HN.2$&4&4\\
$Co_{2}$&4&4&3&$O'N.2$&3&3\\
$Co_{3}$&3&3&3&$J_3.2$&3&3\\
$Fi_{22}$&4&4&3&&&\\
$Fi_{23}$&5&5&5&&&\\
$Fi_{24}'$&5&5&5&&&\\
$Suz$&4&4&4&&&\\
$Ru$&3&3&3&&&\\
$Ly$&5&5&5&&&\\
$O'N$&3&3&3&&&\\
$McL$&3&3&3&&&\\
$HS$&3&3&2&&&\\
$HN$&4&4&4&&&\\
$He$&3&3&3&&&\\
$Th$&5&5&4&&&\\
$B$&7&7&7&&&\\
$M$&9&9&9&&&\\
\bottomrule[1.5pt]
\end{tabular}
\caption{Values for $\gamma$, $\kappa$, $\gamma_w$, $\kappa_w$ for the almost simple groups with sporadic socle.}
\label{tableSporadic}
\end{table}

\begin{table}[ht]
\begin{tabular}{c|c|c|c|c|c}
\toprule[1.5pt]
Group& Comp. $H$&Comp. $K$&Comments&Normal&Nr. \\
\midrule[1.5pt]
$A_5$&$A_5\cap(S_2\times S_3)$&$D_{10}$&&1&2\\
                 &$A_4$&$D_{10}$&&1&\\
$A_6$&$A_6\cap(S_2\times S_4)$ &$A_5$ &&2&2\\
                 &$A_6\cap (S_3\mathrm{wr} S_2)$ &$A_5$ &&2&\\

$A_7$&$A_7\cap(S_2\times S_5)$&$\mathrm{SL}_3(2)$&&2&1\\

$A_8$&$A_8\cap(S_3\times S_5)$&$2^3:\mathrm{SL}_3(2)$&$A_8\cong \mathrm{PSL}_4(2)$&2&1\\

$A_9$&$A_9\cap(S_4\times S_5)$&$\mathrm{P}\Gamma\mathrm{L}_2(8)$&&0&1\\

\midrule[1.5pt]
$G_2(q)$&$\mathrm{SL}_3(q).2$&$\mathrm{SU}_3(q).2$ &$q\ge 4$, $q$ even&1&1\\
$G_2(2)'$&$\mathrm{PSL}_2(7)$&$4\cdot S_4$&$G_2(2)'\cong \mathrm{PSU}_3(3)$&1&2\\
&$\mathrm{PSL}_2(7)$&$3_+^{1+2}:8$&&1&\\
$G_2(3)$&$\mathrm{PSL}_2(13)$&$[q^5]:\mathrm{GL}_2(3)$&&0&2\\
&$\mathrm{PSL}_3(3):2$&$\mathrm{PSL}_2(8):3$&&0&\\
$^{2}G_2(3)'$&$D_{18}$&$D_{14}$&$^{2}G_2(3)'\cong \mathrm{PSL}_2(8)$&1&2\\
&$D_{18}$&$2^3:7$&&1&\\
$^{2}F_4(2)'$&$2.[2^8]:5:4$&$\mathrm{PSL}_3(3):2$&&1&2\\
&$2.[2^8]:5:4$&$\mathrm{PSL}_2(25)$&&1&\\
$F_4(q)$&${}^3D_4(q).3$&$\mathrm{Spin}_9(q)$&$q=3^a$&1&1\\
\midrule[1.5pt]
$M_{11}$&$M_8:S_3\cong 2^\cdot S_4$&$\mathrm{PSL}_2(11)$&notation from~\cite{atlas}&1&3\\
&$M_9:S_2\cong 3^2:Q_8.2$&$\mathrm{PSL}_2(11)$&&1&\\
&$M_{10}\cong A_6.2$&$\mathrm{PSL}_2(11)$&&1&\\
$M_{12}$&$M_{10}:2\cong A_6.2^2$&$\mathrm{PSL}_2(11)$&notation from~\cite{atlas}&0&2\\
&$M_{11}$&$2\times S_5$&&0&\\
\bottomrule[1.5pt]
\end{tabular}
\caption{Weak normal $2$-coverings of non-abelian simple groups: alternating, exceptional and sporadic simple groups}\label{00}
\end{table}

\begin{table}[ht]
\begin{tabular}{c|c|c|c|c|c}
\toprule[1.5pt]
Group& Comp. $H$&Comp. $K$&Comments&Normal&Nr.\\
\midrule[1.5pt]
$\mathrm{PSL}_2(7)$&$S_4 $ &parabolic &&$2$ &1\\
 $\mathrm{PSL}_2(q)$&$D_{q+1}$ &parabolic &$q$ odd, $q>9$ &$1$&1\\
$\mathrm{PSL}_2(q)$                  &$D_{2(q+1)}$&parabolic&$q>4$, $q$ even &1&2\\
                   &$D_{2(q+1)}$&$D_{2(q-1)}$&$q>4$, $q$ even&1&\\

$\mathrm{PSL}_3(q)$&$\left(\frac{q^2+q+1}{\gcd(3,q-1)}\right):3$&parabolic&$q\ne 4$  &2&$1$\\

$\mathrm{PSL}_3(4)$&$\mathrm{SL}_3(2)$&$A_6$&&0&$2$\\
$\mathrm{PSL}_3(4)$&$\mathrm{SL}_3(2)$&parabolic&&6&\\

$\mathrm{PSL}_4(q)$&$\frac{1}{d}\mathrm{SL}_2(q^2).(q+1).2$&$\frac{1}{d}E_q^3:\mathrm{GL}_3(q)$&$d:=\gcd(4,q-1)$& $2$ & $1$\\
\bottomrule[1.5pt]
\end{tabular}
\caption{Weak normal $2$-coverings of non-abelian simple groups: linear groups
(For $\mathrm{PSL}_2(4)$,
$\mathrm{PSL}_2(5)$,
$\mathrm{PSL}_2(9)$, see Table~\ref{00})}\label{0Linear}
\end{table}

\begin{table}[ht]
\begin{tabular}{c|c|c|c|c|c}
\toprule[1.5pt]
Group& Comp. $H$&Comp. $K$&Comments&Normal&Nr.\\
\midrule[1.5pt]
$\mathrm{PSU}_3(q)$&$(q^2-q+1):3$&$\mathrm{GU}_2(q)$&$q=3^a$,\ $a>1$&1&$1$\\
     
$\mathrm{PSU}_3(3)$&$\mathrm{PSL}_2(7)$&$\mathrm{GU}_2(3)$&&1&$2$\\
&$\mathrm{PSL}_2(7)$&$E_3^{1+2}:8$& &1&\\

$\mathrm{PSU}_3(5)$&$A_7$&$\frac{1}{3}\mathrm{GU}_2(5)$&&0&2\\
&$A_7$&$\frac{1}{3}E_5^{1+2}:24$& &3&\\

$\mathrm{PSU}_4(q)$&$\frac{1}{d}\mathrm{GU}_3(q)$&$\frac{1}{d}E_q^4:\mathrm{SL}_2(q^2):(q-1)$&$d:=\gcd(4,q+1)$&$1$&$1$\\
                   &             &                  &$q\ge 4$        & &   \\         

$\mathrm{PSU}_4(2)$&$\mathrm{GU}_3(2)$&$\mathrm{Sp}_4(2)$&&1&$2$\\
&$\mathrm{GU}_3(2)$&$E_2^4:\mathrm{SL}_2(4)$&&1&\\

$\mathrm{PSU}_4(3)$&$A_7$&$\frac{1}{4}E_3^{1+4}:\mathrm{SU}_2(3):8$&&4&4\\
&$\frac{1}{4}\mathrm{GU}_3(3)$&$\frac{1}{4}E_3^4:\mathrm{SL}_2(9):2$&&1&\\
&$\mathrm{PSL}_3(4)$&$\frac{1}{4}E_3^{1+4}:\mathrm{SU}_2(3):8$&&2&\\

&$\frac{1}{4}\mathrm{GU}_3(3)$&$\mathrm{PSU}_4(2)$&&0&\\

$\mathrm{PSU}_6(2)$&$\mathrm{Sp}_6(2)$&$\mathrm{PGU}_5(2)$&&0&2\\
&$\mathrm{PSU}_4(3)$&$\mathrm{PGU}_5(2)$&&0&\\

\bottomrule[1.5pt]
\end{tabular}
\caption{Weak normal $2$-coverings of non-abelian simple groups: unitary groups}\label{00000}
\end{table}

\begin{table}[ht]
\begin{tabular}{c|c|c|c|c|c}
\toprule[1.5pt]
Group& Comp. $H$&Comp. $K$&Comments&Normal&Nr.\\
\midrule[1.5pt]
$\mathrm{PSp}_4(3)$&$\frac{1}{2}E_3^{1+2}:(2\times\mathrm{Sp}_2(3))$&$\mathrm{PSp}_2(9):2$&&1&2\\
&$\frac{1}{2}E_3^{1+2}:(2\times \mathrm{Sp}_2(3))$&$2^4.A_5$&&1&\\
$\mathrm{Sp}_n(q)$& $\mathrm{SO}_n^-(q)$&$\mathrm{SO}_n^+(q)$&$n	\ge 6$&1&$1$\\
                   &             &                  &$q$ even        & &   \\         

$\mathrm{Sp}_4(q)$& $\mathrm{SO}_4^-(q)$&$\mathrm{SO}_4^+(q)$& $q$ even &2&$1$\\
                   &             &                  &$q>4$        & &   \\         

$\mathrm{Sp}_4(4)$& $\mathrm{SO}_4^-(4)$&$\mathrm{SO}_4^+(4)$&&2&$2$\\
&$\mathrm{Sp}_2(16):2$&$\mathrm{Sp}_4(2)$&&0&\\
$\mathrm{PSp}_6(3^f)$&$\frac{1}{2}\left(\mathrm{Sp}_2(3^f)\perp\mathrm{Sp}_4(3^f)\right)$&$\frac{1}{2}\mathrm{Sp}_2(3^{3f}):3$&&1&1\\
\bottomrule[1.5pt]
\end{tabular}
\caption{Weak normal $2$-coverings of non-abelian simple groups: symplectic groups
(For
$\mathrm{PSp}_4(2)'\cong A_6$, see Table~\ref{00})}\label{000=0=0}
\end{table}

\begin{table}[ht]
\begin{tabular}{c|c|c|c|c|c}
\toprule[1.5pt]
Group& Comp. $H$&Comp. $K$&Comments&Normal&Nr.\\
\midrule[1.5pt]
$\mathrm{P}\Omega_8^+(2)$&$A_9$&$E_2^{1+8}:(\mathrm{GL}_2(2)\times\Omega_4^+(2))$&$K$ stab. isotropic line&0&4\\
$\mathrm{P}\Omega_8^+(2)$&$\mathrm{Sp}_6(2)$&$E_2^{1+8}:(\mathrm{GL}_2(2)\times\Omega_4^+(2))$&$K$ stab. isotropic line&0&\\
$\mathrm{P}\Omega_8^+(2)$&$\mathrm{Sp}_6(2)$&$E_2^{6}:\Omega_6^+(2)$& $K$ stab. isotropic point&0&\\

$\mathrm{P}\Omega_8^+(2)$&$(\Omega_2^-(2)\times\Omega_6^-(2)).2$&$E_2^{6}:\Omega_6^+(2)$&$K$ stab. isotropic point&0&\\

$\mathrm{P}\Omega_8^+(3)$&$\Omega_7(3)$&$\frac{1}{2}(\Omega_3(3)\times\Omega_5(3)).[4]$&&0&4\\
$\mathrm{P}\Omega_8^+(3)$&$\Omega_7(3)$&parabolic&$K$ stab.  isotropic line&0&\\
$\mathrm{P}\Omega_8^+(3)$&$\Omega_7(3)$&parabolic &$K$ stab. isotropic point&0&\\
$\mathrm{P}\Omega_8^+(3)$&parabolic &$\frac{1}{2}(\Omega_2^-(3)\times\Omega_6^-(3)).[4]$&$H$ stab. isotropic point&0&\\

\bottomrule[1.5pt]
\end{tabular}
\caption{Weak normal $2$-coverings of non-abelian simple groups: orthogonal groups}\label{000====}
\end{table}

\subsection{Structure of the paper and  background comments}\label{sub:preliminaries}
In our proof of Theorems~\ref{main theorem} and~\ref{main theorem1}, we use the Classification of the Finite Simple Groups. Thus in Table~\ref{00} we have collected the alternating groups, the exceptional groups of Lie type and the sporadic simple groups, in Tables~\ref{0Linear}--\ref{000====} we have collected the classical groups. We take into account the various isomorphisms among finite simple groups. For instance, we do not list the weak normal $2$-coverings of $\mathrm{PSL}_2(4)$, $\mathrm{PSL}_2(5)$, $\mathrm{PSL}_2(9)$ and $\mathrm{PSp}_4(2)'$, because these already appear among the alternating groups.

The classification of the alternating groups admitting normal covering number $2$ appears in~\cite{b}. There it is shown that the alternating group $A_n$ has normal covering number $2$ if and only if $4\le n\le 8$. In this paper, we do not repeat the arguments in~\cite{b} to prove our  more general result for weak normal coverings. Indeed, it is easily checked that the whole reasoning in~\cite{b} works for our more general classification apart for the case  $A_9$, which is one of the groups appearing in Corollary~\ref{corollary}. The heart of the matter is that the $9$-cycles   in $A_9$ split into two conjugacy classes while the $9$-cycles in $\mathrm{P}\Gamma\mathrm{L}_2(8)$ are all $S_9$-conjugate. Thus  $\bigcup_{g\in A_9 }\mathrm{P}\Gamma\mathrm{L}_2(8)^g$  does not contain all the $9$-cycles of $A_9$, while $\bigcup_{g\in S_9 }\mathrm{P}\Gamma\mathrm{L}_2(8)^g$ does (for details see \cite[pages $17$-$18$]{b}).

The classification of the sporadic simple groups having normal covering number $2$ is in~\cite{pellegrini} and it is carried on in two parts. In the first general part, the arguments are via group orders and this part applies verbatim to our situation. In the second part the author analyses, often with the use of a computer, the exceptional cases not covered by the previous investigation. For these exceptional cases, reported in ~\cite[Table 2] {pellegrini} we have used a computer for finding weak normal 2-coverings. The only exceptional case that appears is the Mathieu group $M_{12}.$
% Again, the argument in~\cite[Section~$3$]{pellegrini} applies also in the proof of the more general Theorem~\ref{main theorem} with no modification, the only exceptional example that arises is the Mathieu group $M_{12}$.

In~\cite{blw,lucido,pellegrini}, the authors classify the finite simple exceptional groups of Lie type having covering number $2$. Again, we do not replicate the argument in these papers. We simply observe that, with minor modifications, the reasoning in~\cite{lucido,pellegrini} give a proof to Theorem~\ref{main theorem} in the case of finite simple exceptional groups of Lie type. The only exceptional case arising here is $G_2(3)$. The study of  $G_2(3)$ in~\cite{pellegrini} was carried on via a computer computation and ours too.

Therefore, the bulk of our work is the proof of Theorems~\ref{main theorem} and~\ref{main theorem1} for simple classical groups. In~\cite{bl}, it has been proved that, if the projective special linear group
$\mathrm{PSL}_{n}(q)$ admits a normal  $2$-covering then $2 \le n \le 4$. The same proof yields the same result for weak normal $2$-coverings.
Since in ~\cite{bl} it is also shown that  for $2 \le n \le 4$ there  exists a normal $2$-covering of $\mathrm{PSL}_{n}(q)$, we deduce that $\gamma_w(\mathrm{PSL}_{n}(q))\leq 2$ for those $n$.
Thus, in dealing with the simple groups $\mathrm{PSL}_n(q)$, we may suppose that $2\le n\le 4$ and concentrate on the explicit description of all the possible maximal components of a weak normal $2$-covering and on excluding $\gamma_w(\mathrm{PSL}_{n}(q))=1.$

The structure of the paper is straightforward. In Section~\ref{preliminaries}, we give some basic preliminaries which will also serve useful for setting some notation. Then, in Section~\ref{sec:linear}, we consider the linear groups $\mathrm{PSL}_n(q)$; in Section~\ref{sec:unitary}, we consider the unitary groups $\mathrm{PSU}_n(q)$; in Section~\ref{sec:symplectic}, we consider the symplectic groups $\mathrm{PSp}_n(q)$; in Section~\ref{sec:oddorthogonal}, we consider the odd dimensional orthogonal groups $\mathrm{P}\Omega_n(q)$;  in Sections~\ref{sec:orthogonal-} and~\ref{sec:orthogonal+}, we deal with even dimensional orthogonal groups $\mathrm{P}\Omega_n^-(q)$ and $\mathrm{P}\Omega_n^+(q)$; in Section \ref{sec:final}, we prove Theorem \ref{main theorem2};  in Section~\ref{section:new}, we investigate the almost simple groups with sporadic socle
giving the details for obtaining Table~\ref{tableSporadic}.

In Section~\ref{sec:appendix}, we use Theorem~\ref{main theorem1} to classify the weak normal $2$-coverings and the normal $2$-coverings of the non-abelian simple groups. This means that we drop the hypothesis about the maximality of the components and give the description of all components in a weak normal $2$-covering of $G$, when $G$ is a non-abelian simple group with $\gamma_w(G)=2$. Our main result is Theorem~\ref{main theoremgeneral} and to our surprise most weak normal $2$-coverings do require at least one of the components to be maximal, see Corollary~\ref{cor:21}.

In Section~\ref{sec:degenerateasdf}, we study the degenerate normal $2$-coverings of almost simple groups.

\section{Preliminaries}\label{preliminaries}

Given a group $G$, we denote as customary its derived group by $G'.$ A group is called perfect when $G'=G.$
\subsection{The players} \label{players}
Throughout the paper let $n, f$ be  positive integers, $p$ be a prime and $q=p^f$.
We consider the
following \textbf{\em classical groups}  $\tilde G$ of dimension $n$:
\begin{itemize}
\item$\mathrm{SL}_n(q)$ with $n\ge 1$,
\item$\mathrm{SU}_n(q)$  with $n\ge 1$,
\item$\mathrm{Sp}_{n}(q)$  with $n$ even and $n\ge 2$,
\item$\Omega_{n}(q)$  with $qn$ odd and $n\ge 1$, and
\item$\Omega_{n}^{\pm}(q)$ with $n$ even and $n\ge 2$.
\end{itemize}

The corresponding
\textbf{\em simple classical groups} $G:=\tilde G/ Z (\tilde G)$ are
$$
\mathrm{PSL}_n(q),\,
\mathrm{PSU}_n(q),\,
\mathrm{PSp}_{n}(q),\,
\mathrm{P}\Omega_{n}(q),\hbox{ and }
\mathrm{P}\Omega_{n}^{\pm}(q),$$
where we consider only $n\ge 3$ for unitary groups, $n\ge 4$ for symplectic groups, $n\ge 7$ for odd dimensional orthogonal groups and $n\ge 8$ for even dimensional orthogonal groups.
With the restrictions on $n$ as above, these are indeed non-abelian simple groups, except for $\mathrm{PSL}_2(2)$, $\mathrm{PSL}_2(3)$, $\mathrm{PSU}_3(2)$ and $\mathrm{PSp}_4(2)$. Moreover, taking into account the various isomorphisms among  simple groups (see~\cite[Section~2.9]{kl}), the choices for $n$ guarantee that every finite non-abelian simple group is considered just one time. %For instance, $\mathrm{SL}_2(q)\cong\mathrm{SU}_2(q)\cong \mathrm{Sp}_2(q)$.

For some of our preliminary results or some arguments, we need to deal with arbitrary classical groups as defined above and hence with no restrictions on $n$. 
However, recall that our main results are only concerned with non-abelian simple groups. 
 
We also need the \textbf{\em general classical groups}, which we denote by $\hat G$, defined by
$$
\mathrm{GL}_n(q),\,
\mathrm{GU}_n(q),\,
\mathrm{Sp}_{n}(q),\,
\mathrm{O}_{n}(q),\hbox{ and }
\mathrm{O}_{n}^{\pm}(q).$$
For the groups $\hat G$ we adopt the same limitations for $n$ adopted for the corresponding $\tilde G.$ As usual, for the notation we follow~\cite[Section~2.1]{kl}. 
Recall that $\hat {G}'=\tilde G$.

\subsection{Normal and weak normal coverings of classical and simple classical groups}\label{cove-classic-simple}
Let $\tilde G$ be a classical group such that $G$ is a non-abelian simple group. Then $\tilde G$ is a perfect group, see for instance~\cite[Proposition~$2.9.2$~(ii)]{kl}. As a consequence if $H$ is a proper subgroup of $\tilde G$, then $HZ(\tilde G)$ is also a proper subgroup of $\tilde G$. In particular, every  maximal subgroup of $\tilde G$ contains $Z(\tilde G)$. We illustrate the natural link between coverings of $\tilde G$ and coverings of $G$. That link is extensively used throughout the paper. To begin with, if $\pi:\tilde G\rightarrow G$ is the natural projection of $\tilde G$ onto $G$ and $X\leq G$, we define $\tilde X:=\pi^{-1}(X).$ Of course, $\tilde X$ is a subgroup of $\tilde G$ containing $Z(\tilde G).$ Note also that $\pi(\tilde X)=X$ and that  $X$ is  maximal in $G$ if and only if $\tilde X$ is maximal in $\tilde G$.

Let $k$ be a positive integer and  $ \mu=\{H_1,\ldots,H_k\}$ be a set of distinct proper subgroups of $G$. 
Then, $\tilde \mu:=\{\tilde H_1,\ldots,\tilde H_k\}$ is a set of distinct proper subgroups of $\tilde G$ containing $Z(\tilde G)$ and  the map $ \mu\mapsto \tilde \mu$ establishes a bijection between  sets of $k$ distinct proper subgroups of $G$ and sets of $k$ distinct proper subgroups  of $\tilde G$ containing $Z(\tilde G)$. It is immediate to check that $\mu$ is a normal covering of $G$ if and only if $\tilde \mu$ is a normal covering of $\tilde G$ with components containing $Z(\tilde G)$.
Hence,  if $\tilde \mu$ is a normal covering of $\tilde G$ with maximal components, then $\mu$ is a normal covering of $G$ (with maximal components). 
It follows that
$$\gamma(\tilde G)=\gamma(G).$$
Moreover our main result can be used to determine  the non-solvable classical groups having covering number $2$ and to characterize the maximal components of their normal $2$-coverings.

We now explore the link between the weak normal coverings of $\tilde G$ and $G$.  Since $Z(\tilde G)$ is a characteristic subgroup of $\tilde G$, we have a natural homomorphism of $\mathrm{Aut}(\tilde G)$ in $\mathrm{Aut}(G)$. It follows that if $\tilde \mu$ is a weak normal covering of $\tilde G$ with components containing $Z(\tilde G)$, then $\mu$ is a weak normal covering of $G$. In particular, we have
$$\gamma_w(\tilde G)\ge \gamma_w(G).$$
From~\cite[Chapter~2]{kl}, we see that, when $G$ is simple, the homomorphism of $\mathrm{Aut}(\tilde G)$ in $\mathrm{Aut}(G)$ is surjective, with the only exception of $G=\mathrm{P}\Omega_8^+(q)$ and $q$ odd. Using that fact it easily follows that, besides that exceptional case,  if $\mu$ is a weak normal covering of $G$, then $\tilde \mu$ is a weak normal covering of $\tilde G$.
In particular, besides that exceptional case, we have
\begin{equation}\label{weak-G-Gonda}
\gamma_w(\tilde G)=\gamma_w(G)
\end{equation}
and our main result can be used to determine the non-solvable classical groups having weak normal covering number $2$ and to characterize the maximal components of their weak normal $2$-coverings. 

In other words, besides the exceptional case, we may freely work between weak normal coverings of $G$ and $\tilde G$. Extra care has to be taken when $G=\mathrm{P}\Omega_8^+(q)$ with $q$ odd, because the triality automorphism of $G$ does not lift to an automorphism of $\tilde G$. From Table~\ref{000====}, we deduce that, among simple groups $G$ having weak normal covering number $2$, the only possible exception to~\eqref{weak-G-Gonda} is for $G:=\mathrm{P}\Omega_8^+(3)$. As we all know ``the devil is in the details'' and in fact, it can be checked  with a computer that $\gamma_w(\Omega_8^+(3))\geq 3>\gamma_w(\mathrm{P}\Omega_8^+(3))=2.$
\begin{theorem}Let $\tilde{G}$ be a classical group such that $G$ is a non-abelian simple group. Then the following hold
\begin{itemize}
\item $\gamma(\tilde{G})=\gamma(G)$,
\item$\gamma_w(\tilde{G})=\gamma_w(G)$ if and only if $G\not\cong\mathrm{P}\Omega_8^+(3)$.
\end{itemize}
\end{theorem}

\subsection{From weak normal coverings to normal coverings}\label{sec:newwen}
In this section, we describe a situation that we often face in our work.

Let $G$ be a non-abelian simple group and let $\{H,K\}$ be a weak normal $2$-covering of $G$. Observe that from the theorem of Saxl~\cite{Sa88}, $H$ and $K$ are in distinct $\mathrm{Aut}(G)$-classes. Clearly,  $\mathcal{C}:=\{\{H^\varphi,K^\psi\}\mid \varphi,\psi\in\mathrm{Aut}(G)\}$ is the $\mathrm{Aut}(G)$-class of weak normal $2$-coverings with representative $\{H,K\}$. Now, $\mathcal{C}$ is a union of, $C$ say, $G$-classes of normal $2$-coverings. The case $C=0$ is special and consists in the case where no normal $2$-covering can be extracted from this $\mathrm{Aut}(G)$-class.  Suppose then $C\ge 1$, that is, there exist $\varphi,\psi\in\mathrm{Aut}(G)$ such that $\{H^\varphi,K^\psi\}$ is a normal $2$-covering. For simplicity, we may suppose that $\{H,K\}$ itself is a normal $2$-covering. Here, we give a lower and an upper bound for $C$.

Let $h:=|\mathrm{Aut}(G):G{\bf N}_{\mathrm{Aut}(G)}(H)|$ and $k:=|\mathrm{Aut}(G):G{\bf N}_{\mathrm{Aut}(G)}(K)|$. In particular, $\{H^\varphi\mid \varphi\in\mathrm{Aut}(G)\}$ is a union of $h$ distinct $G$-conjugacy classes and similarly $\{K^\varphi\mid \varphi\in\mathrm{Aut}(G)\}$ is a union of $k$ distinct $G$-conjugacy classes. It is not hard to verify that
$$\max\{h,k\}\hbox{ divides } C\hbox{ and }C\le hk.$$
In particular, when $h=1$ (respectively $k=1$), we have $C=k$ (respectively $C=h$). In most of our arguments, we use this basic observation for extracting information about normal $2$-coverings from weak normal $2$-coverings.

For classical groups and for their maximal subgroups, the values of $h$ and $k$ is in  the literature in the ``c'' column in the tables in~\cite{bhr,kl}.

\subsection{Action of $\hat G$ on its natural module}

Set $\delta:=1$ when $\hat{G}\ne\mathrm{GU}_n(q)$ and set $\delta:=2$ when $\hat{G}=\mathrm{GU}_n(q)$. Let $V=(\mathbb{F}_{q^\delta})^n$ be the natural $n$-dimensional $\mathbb{F}_{q^\delta}\hat{G}$-module consisting of row vectors $x=(x_1,\dots,x_n)$ in $n$ components $x_i\in \mathbb{F}_{q^\delta}$ and endowed with the
suitable $\hat{G}$-invariant form. This form is non-degenerate, with the only exception of $\hat G=\mathrm{GL}_n(q)$. If $x\in \mathrm{GL}_n(q)$  admits the eigenvalue $\lambda\in \mathbb{F}_{q}$, we denote the corresponding  eigenspace by $V_{\lambda}(x)$. When $\lambda=1$, we also use ${\bf C}_V(x)$ to denote $V_1(x)$, because $V_1(x)$ is indeed the centralizer in $V$ of $x$ in the semidirect product $V\rtimes \mathrm{GL}_n(q)$.

In our proofs, for discussing weak normal coverings of $\tilde G$, we use the action of $\tilde G\leq \hat{G}$ on the natural module $V$. Observe, that except when $G=\mathrm{PSL}_n(q)$ with $n\ge 3$, $G=\mathrm{Sp}_4(2^f)$ and $G=\mathrm{P}\Omega_8^+(q)$, the group $\mathrm{Aut}(\tilde G)$ acts on the natural module as a semilinear group and hence we may also discuss the action of the elements of $\mathrm{Aut}(\tilde G)$ on $V$. Extra care has to be taken when $G=\mathrm{PSL}_n(q)$ with $n\ge 3$, $G=\mathrm{Sp}_4(2^f)$ and $G=\mathrm{P}\Omega_8^+(q)$.
\smallskip

 Given $g\in \hat{G}$, when $V$ is a completely reducible $\mathbb{F}_{q^\delta}\langle g\rangle$-module, that is, $V$ decomposes into the direct sum of
irreducible $\mathbb{F}_{q^\delta}\langle g\rangle$-submodules $V_i$ of dimensions $d_i$, for $i\in \{1,\dots,k\}$,
we say that \textbf{\em the action of $g$ is of
type $d_1\oplus\cdots\oplus d_k.$} Note that $g$ acts irreducibly (that is, $V$ is an irreducible $\mathbb{F}_{q^\delta}\langle g\rangle$-module)
 if and only if its characteristic polynomial is irreducible.

Since we usually work with semisimple elements, we shall use some
facts about the maximal tori of the classical groups,
 in particular their orders and their action on
the natural module $V$, which can be found in~\cite{GLS3}. We also make use of the work of Wall~\cite{wall0,wall}, which broadly speaking describes the Jordan forms of the matrices belonging to a given classical group. Incidentally, we denote by $\order g$ the order of a group element $g$. Another good source of detailed information on the conjugacy classes in classical groups is~\cite{burness}.

Since we want to determine whether there exist maximal subgroups $H,\ K$
of $G$ such that any element of $G$ is $\mathrm{Aut}(G)$-conjugate to an element of $H$
or $K$, we adopt the systematic description of
the maximal subgroups of the classical groups given by Aschbacher
in~\cite{a}. Actually, since Aschbacher does not consider almost simple groups containing a graph automorphism, we use the extended, and slightly different, notation in~\cite{kl}. The maximal subgroups of $G$, or of $\hat{G}$, or of any almost simple group with socle $G$ are divided into nine families. The families $\mathcal{C}_i$, with $i\in \{1,\dots,8\}$, are defined in terms of the geometric properties of
their action on  $V$ and the main result of Aschbacher
states that any maximal subgroup belongs to
$\bigcup_{i=1}^{8}\mathcal{C}_i$ or to an additional family
$\mathcal{S}$, consisting of groups satisfying certain irreducibility conditions. For notation and structure theorems on these
maximal subgroups and other details of our investigation, we refer to
the book of Kleidman and Liebeck~\cite{kl} or to the book of Bray, Holt and Roney-Dougal~\cite{bhr} when the rank is small. As we mentioned above,~\cite{bhr,kl} take also into account almost simple groups having socle a simple classical group and containing a graph automorphism.

\subsection{Huppert's Theorem and Singer cycles}\label{Huppert}
In this section the vector space $V$ admits a non-degenerate form or quadratic
form of classical type which is preserved by the group $G$.  We frequently make use of a theorem of
Huppert~\cite[Satz 2]{hu}, which we apply to semisimple elements $s\in G$.
Such elements generate a subgroup acting completely reducibly on $V$, and by Huppert's Theorem, $V$
admits an orthogonal decomposition of the following form:
\begin{eqnarray}\label{eq:decomp}
V&=&V_+\perp V_-\perp ((V_{1,1}\oplus V_{1,1}')\perp \cdots \perp (V_{1,m_1}\oplus V_{1,m_1}'))\perp \cdots \\
&&\perp ((V_{r,1}\oplus V_{r,1}')\perp \cdots \perp (V_{r,m_r}\oplus V_{r,m_r}'))\nonumber\\
&&\perp (V_{r+1,1}\perp \cdots \perp V_{r+1,m_{r+1}})
\perp\cdots \perp (V_{t',1}\perp \cdots \perp V_{t',m_{t'}}), \nonumber
\end{eqnarray}
where $V_{+}$ and $V_{-}$ are the  eigenspaces of $s$ for the eigenvalues
$1$ and $-1$, of dimensions $d_+$ and $d_-$, respectively (note $V_\pm$ is non-degenerate if
$d_\pm>0$ and we set $d_-=0$ if $q$ is even),
and each  $V_{i,j}$ is an irreducible $\mathbb{F}_{q^\delta}\langle s\rangle$-submodule.
Moreover for $i=r+1,\ldots,t'$, $V_{i,j}$ is non-degenerate of dimension $2d_i/\delta$  and $s$
induces an element $y_{i,j}$ of order dividing $q^{d_i}+1$ on $V_{i,j}$ (in the unitary case $\delta=2$ and
the dimension $d_i$  is odd).

For $i=1,\ldots,r$, $V_{i,j}$ and $V_{i,j}'$ are totally isotropic  of dimension $d_i/\delta$ (here $d_i$ is even if $\delta=2$),
$V_{i,j}\oplus V_{i,j}'$ is non-degenerate, and $s$
induces an element $y_{i,j}$ of order dividing $q^{d_i}-1$ on $V_{i,j}$ while
inducing the adjoint representation $(y_{i,j}^{-1})^{T}$ on  $V_{i,j}'$, where
$x^{T}$ denotes the transpose of the matrix $x$.

 For our claims about the
orders of the $y_{ij}$ and for some standard facts on the structure of the maximal tori of the finite classical groups, we also refer to~\cite{BC,KS}.

Some special semisimple elements help our task of identifying the components of a weak normal $2$-covering of a classical group. Among them a main role is played by the  \textbf{\em Singer cycles}, that is, those elements in $\hat {G}$ or $\tilde G$ acting irreducibly  on $V$ and having maximum order. These elements were intensively studied  by Huppert in~\cite{hu} and by Hestens in~\cite{hest}.

We recall now some main facts about Singer cycles. Other information needed for the paper is given in Section \ref{further-Singer}. As usual, let $\mathbb{F}_{q^{\delta n}}$ be the field with $q^{\delta n}$ elements. Let  $a\in \mathbb{F}_{q^n}$ and consider the multiplication $\pi_a:\mathbb{F}_{q^n}\rightarrow \mathbb{F}_{q^n}$ defined by $\pi_a(x)=ax$ for all $x\in \mathbb{F}_{q^n}$. The maps $\pi_a$ with $a\neq 0$ form a group isomorphic to $\mathbb{F}_{q^n}^*$.
For every  $d\mid n$, the set $V:=\mathbb{F}_{q^n}$ can be interpreted as a vector space of dimension $n/d$ over the field $\mathbb{F}_{q^d}$ and the map $\pi_a$ is a $\mathbb{F}_{q^d}$-linear transformation of $V$. Thus, once a base is fixed, $\pi_a$ induces a matrix belonging to
$ \mathrm{GL}_{n/d}(q^d)$. If $\langle a\rangle=\mathbb{F}_{q^n}^*$, by \cite[II. Satz 7.3]{hu-book}, we have that
\begin{equation}\label{piadet}
\mathrm{det}_{\mathbb{F}_{q^d}}(\pi_a)=a^{\frac{q^n-1}{q^d-1} }.
\end{equation}
Only when necessary, or useful for certain purposes, we keep track of the field $\mathbb{F}_{q^d}$ under consideration in the notation of the determinant of $\pi_a$.

It is well-known that the Singer cycles  in $\mathrm{GL}_n(q)$ are, up to conjugacy, exactly the matrices representing $\pi_a$ as a $\mathbb{F}_{q}$-linear transformation of $V$, for $a$  a generator of $\mathbb{F}_{q^n}^*$.  Hence if $s\in \mathrm{GL}_n(q)$ is any Singer cycle, by \eqref{piadet}, we get $\mathrm{det}(s)=\mathrm{det}_{\mathbb{F}_{q}}(\pi_a)=a^{\frac{q^n-1}{q-1} }$, for some  $a\in \mathbb{F}_{q^n}^*$ such that $\langle a\rangle=\mathbb{F}_{q^n}^*$. Thus  $\mathrm{det}(s)$ has order $q-1$ and generates $\mathbb{F}_{q}^*.$

There is a useful construction of Singer cycles by suitable multiplications also for the other general classical groups.

For the unitary case, with $n$ odd,  let  $\langle a\rangle=\mathbb{F}_{q^{2n}}^*$. Then, by \cite[Theorem 5.2]{hest},  the multiplication $\pi_{a^{q^n-1}}$ seen as a $\mathbb{F}_{q^2}$-linear transformation of $\mathbb{F}_{q^{2n}}$ induces a Singer cycle $s$ for  $\mathrm{GU}_n(q)$. Note that, by \eqref{piadet},
$$\mathrm{det}(s)=\left(\mathrm{det}_{\mathbb{F}_{q^{2}}}(\pi_a)\right)^{q^n-1}=\left(a^{\frac{q^{2n}-1}{q^2-1} }\right)^{{q^n-1}}=a^{\frac{(q^{2n}-1)(q^n-1)}{q^2-1} }\in  \mathbb{F}_{q^{2}}^*$$
  belongs to the subgroup of order $q+1$ of  $\mathbb{F}_{q^{2}}^*$.

  For the symplectic and orthogonal case, with $n$ even, let $\langle a\rangle=\mathbb{F}_{q^n}^*$. Then, by \cite[Theorem 5.6]{hest}, the multiplication $\pi_{a^{q^{n/2}-1}}$ seen as a $\mathbb{F}_{q}$-linear transformation of $V$ is a Singer cycle $s$ both for  $\mathrm{Sp}_n(q)$ and for  $\mathrm{O}_{n}^-(q)$ with $$\mathrm{det}(s)=\left(a^{\frac{q^{n}-1}{q-1} }\right)^{{q^{n/2}-1}}=1.$$ In particular $s\in \mathrm{SO}_{n}^-(q)$.

In Table~\ref{singer-order} we report the general classical groups $\hat{G}$ admitting Singer cycles, the order of the Singer cycles in $\hat G$ and the order of the Singer cycles in $\tilde G$ for every $n\geq 1$  with $n$ odd in the unitary case and $n$ even in the symplectic and orthogonal case.
While the information in the second column comes directly from ~\cite{hu} and no exception for the existence arises,  the third column needs a little explanation and two truly exceptions arise as clarified in  Lemma \ref{action} and Proposition \ref{sing-ord}.

 \begin{table}[ht]
\begin{tabular}{c|c|c|c}
\toprule[1.5pt]
$\hat{G}$ & order of  Singer cycle in $\hat G$
   &  order of Singer cycle in $\tilde G$& Comments \\
\midrule[1.5pt]
 $\mathrm{GL}_n(q)$ &  $q^n-1$ & $\frac{q^n-1}{q-1}$&  \\
$\mathrm{GU}_{n}(q)$ &  $q^n+1$  & $\frac{q^n+1}{q+1}$ & $n$ odd\\
                    &            &                     & $\tilde{G}\ne \mathrm{SU}_3(2)$\\
 $\mathrm{Sp}_{n}(q)$ &  $q^{\frac{n}{2}}+1$ & $q^{\frac{n}{2}}+1$ & \\
$\mathrm{O}_{n}^-(q)$ &  $q^{\frac{n}{2}}+1$  & $\frac{q^{\frac{n}{2}}+1}{\gcd(2,q-1)}$& $\tilde G\neq  \Omega^-_2(3)$ \\
\bottomrule[1.5pt]
\end{tabular}
\caption{Singer cycles}\label{singer-order}
\end{table}

\subsection{Some arithmetics}
The following lemma presents some arithmetical facts  which will be very useful throughout the paper. This is possibly part of folklore but we could not find a complete reference.
\begin{lemma}\label{aritme}
Let $q\geq 2$ be an integer. For every $a$ and $b$  positive integers we have
 \begin{enumerate}
\item\label{eq:arithme1} $\gcd(q^{a}-1, q^{b}-1)=q^{\gcd(a,b)} -1$,
 \item\label{eq:arithme2}
 \[
 \gcd(q^a+1,q^b+1)=
 \begin{cases}
 q^{\gcd(a,b)}+1&\mathrm{if }\,\frac{a}{\gcd(a,b)}\,\, \mathrm{ and }\,\,\frac{b}{\gcd(a,b)}\,\, \mathrm{are\, odd},\\
 \gcd(2,q-1)&\mathrm{otherwise},
 \end{cases}
 \]
 \item\label{eq:arithme3}
 \[
 \gcd(q^a+1,q^b-1)=
 \begin{cases}
 q^{\gcd(a,b)}+1&\mathrm{if }\,\,\frac{b}{\gcd(a,b)}\,\, \mathrm{is\, even},\\
 \gcd(2,q-1)&\mathrm{otherwise}.
 \end{cases}
 \]
\end{enumerate}
\end{lemma}
\begin{proof} The proof of all the formulas (1)-(3) is by induction on $a+b$.  If $a+b=2$, then we have $a=b=1$ and all the formulas are clear. Now, let $n$ be a positive integer with $n\ge 2$. Assume that (1)-(3) hold for every $a$ and $b$ with $a+b\leq n$ and we show that they hold also for every $a$ and $b$ with $a+b=n+1.$

(1) The formula is symmetric in $a$ and $b$ and trivially holds for $a=b$. Thus we assume that $a>b.$ Since
$$q^a-1=q^{a-b}(q^b-1)+(q^{a-b}-1),$$
by the inductive hypothesis  for (1) we deduce that
 $$\gcd(q^a-1,q^b-1)=\gcd(q^b-1,q^{a-b}-1)=q^{\gcd(a,b)} -1,$$
because $a+(b-a)=b<a+b$.

(2) The formula is symmetric in $a$ and $b$ and trivially holds for $a=b$. Thus we assume that $a>b.$ Since
$$q^a+1=q^{a-b}(q^b+1)-(q^{a-b}-1),$$
by the inductive hypothesis  for (3) we deduce that
\begin{align*}
\gcd(q^a+1,q^b+1)=\gcd(q^b+1,q^{a-b}-1)&=\begin{cases}
q^{\gcd(a,b-a)}+1&\textrm{when } \frac{a-b}{\gcd(a,a-b)} \textrm{ is even},\\
\gcd(2,q-1)&\textrm{otherwise}.
\end{cases}
 \end{align*}
Now it is enough to observe that $\gcd(a,a-b)=\gcd(a,b)$ and $(a-b)/\gcd(a,b-a)$ is even if and only if both $a/\gcd(a,b)$ and $b/\gcd(a,b)$ are odd.

(3) In this case, the formula is not symmetric in $a$ and $b$ and hence we need to distinguish three cases: $a=b$, $a<b$ and $a>b$. When $a=b$, the formula can be easily checked. Assume next that $a< b.$
Since
$$q^b-1=q^{b-a}(q^a+1)-(q^{b-a}+1),$$
by the inductive hypothesis  for (2) we deduce that
\begin{align*}
\gcd(q^a+1,q^b-1)=\gcd(q^a+1,q^{b-a}+1)&=\begin{cases}
q^{\gcd(a,b-a)}+1&\textrm{when } \frac{a}{\gcd(a,b-a)} \textrm{ and}\\
&  \frac{b-a}{\gcd(a,b-a)} \textrm{ are odd},\\
\gcd(2,q-1)&\textrm{otherwise}.
\end{cases}
 \end{align*}
Now it is enough to observe that $\gcd(a,a-b)=\gcd(a,b)$ and that $\frac{a}{\gcd(a,b)}$, $\frac{b-a}{\gcd(a,b)}$  are odd if and only if $\frac{b}{\gcd(a,b)}$ is even.

Assume finally that $a>b.$ Since
$$q^a+1=q^{a-b}(q^b-1)+(q^{a-b}+1),$$
by the inductive hypothesis  for (3) we deduce that
\begin{align*}
\gcd(q^a+1,q^b-1)=\gcd(q^b-1,q^{a-b}+1)&=\begin{cases}
q^{\gcd(a,b
)}+1&\textrm{when } \frac{b}{\gcd(a,b)} \textrm{ is even},\\
\gcd(2,q-1)&\textrm{otherwise}.\qedhere
\end{cases}
 \end{align*}
\end{proof}
Given a positive integer $n$, we denote by
$\pi(n)$
the set of prime divisors of $n$.

\subsection{Primitive prime divisors}
Given  a prime power $q$ and an integer $t\geq 2$, a prime $r$ is called a \textit{\textbf{primitive prime divisor}} of $q^t-1$ if $r$ divides $q^t-1$ and $r$  does not divide $q^i-1$, for each $i\in \{1,\ldots,t-1\}$. We let $P_t(q)$ denote \textbf{\textit{the set of primitive prime divisors}} of $q^t-1$.
From a celebrated theorem of Zsigmondy~\cite{zs}, the following hold
\begin{itemize}
 \item for $t\geq 3$, we have $P_t(q)\neq \varnothing$ with the only exception of $(t,q)=(6,2)$,
\item $P_2(q)\neq \varnothing$ with the only exception of $q$ being a Mersenne prime.
\end{itemize}
Note that if $r$ is a primitive prime divisor of $q^t-1$, then $q$
has order $t$ modulo $r$ and thus $t$ divides $r-1.$ Hence there exists a positive integer $k$ such that $r=kt+1.$
In
particular,
\begin{equation}\label{boundppd}
r \ge t+1,
\end{equation}
and when $t$ is odd
\begin{equation}\label{boundppd-odd}
r\ge 2t+1.
\end{equation}
We collect in the following lemma some elementary facts about $P_t(q)$.
\begin{lemma}\label{primitivi} Let $s,t,k,b$ be positive integers with $s,t\geq 2$ and let $q$ be a prime power. Then the following hold
\begin{enumerate}
\item\label{primitivi1}  $P_{kt}(q)\subseteq P_t(q^k),$
\item\label{primitivi2} if $t\neq s$, then
$P_t(q)\cap P_s(q)=\varnothing$,
\item\label{primitivi3} given $r\in P_t(q)$, $r$ divides
$q^b-1$ if and only if $t$ divides $b$. In particular, if $r$ divides $q^b+1,$
then $t$ divides $2b.$
\end{enumerate}
\end{lemma}

\subsection{Further facts on Singer cycles}\label{further-Singer}
The concept of primitive prime divisor allows, among other things, to obtain an easy test to recognize the Singer cycles by checking only the  element orders.
We first treat the action of $\Omega_2^-(q)$ on its natural module $V$.
\begin{lemma}\label{action}The cyclic group
$\Omega_2^-(q)\cong C_{\frac{q+1}{(2,q-1)}}$
acts irreducibly on $V$ if $q\neq
3$  and scalarly if $q=3$. In particular, $\Omega_2^-(3)$ admits no Singer cycle.
\end{lemma}
\begin{proof}
This follows from~\cite[Section~$2.10$]{kl}, but here we give a direct proof. If $V$ is not an irreducible $\mathbb{F}_q\Omega_2^-(q)$-module, then $V$ decomposes as the direct sum of two $1$-dimensional $\mathbb{F}_q\Omega_2^-(q)$-submodules. Hence $\Omega_2^-(q)$ is conjugate to  a cyclic subgroup of the group of $2\times 2$ diagonal matrices. Therefore $(q+1)/\gcd(2,q-1)=|\Omega_2^-(q)|$ divides $q-1$. By Lemma \ref{aritme}, this yields $q=3$. On the other hand, $\Omega_2^-(3)\cong C_{2}$ is the group of scalar matrices of $\mathrm{GL}_{2}(3)$ so that obviously $\Omega_2^-(3)$ cannot admit cyclic subgroups acting irreducibly on $V.$
\end{proof}
In particular, Lemma~\ref{action} explains the exception of $\Omega_2^-(3)$ in Table~\ref{singer-order}.

\begin{proposition}\label{sing-ord} Let $\hat G$ be a general classical group and $\tilde G$ be the corresponding  classical group. Then the following hold:
\begin{enumerate}
\item Assume that $\hat G$ or $\tilde G$ are as in {\rm Table \ref{singer-order}}. If $x$ belongs to $\hat G$ or to $\tilde G$ and has the order of a Singer cycle in $\hat G$ or $\tilde G$ respectively, then $x$ is a Singer cycle.
\item $\tilde G$ admits a Singer cycle if and only if it appears in {\rm Table \ref{singer-order}}.
\item Let $\hat G=\mathrm{GU}_{n}(q),$ with $n$ odd and $\lambda\in \mathbb{F}_{q^2}^*$ be an element of order $q+1$. Then there exists a Singer cycle $s$ of $\hat G$ such that $\mathrm{det}(s)=\lambda.$
\item Let $\hat G=\mathrm{GL}_{n}(q),$  and $\lambda\in \mathbb{F}_{q}^*$ be an element of order $q-1$. Then there exists a Singer cycle $s$ of $\hat G$ such that $\mathrm{det}(s)=\lambda.$
\end{enumerate}
\end{proposition}
\begin{proof}  Let $V$ be the natural $n$-dimensional $\mathbb{F}_{q^\delta}\hat{G}$-module. If $n=1$ all the statements are trivial. So assume that  $n\geq 2$.

$(1)$ Regard $V$ as a $\mathbb{F}_{q^\delta}\langle x\rangle$-module and assume, by contradiction, that $V$ is not irreducible. Then there exists an irreducible $\mathbb{F}_{q^\delta}\langle x\rangle$-submodule $W$ of $V$ and  $m=\mathrm{dim}_{\mathbb{F}_{q^\delta}}W$ is a positive integer less than $n$. By Schur Lemma, we have that $\langle x\rangle\leq C_{q^{\delta m}-1}$ and hence
\begin{equation}\label{div}
\order s=\order x\mid q^{\delta m}-1.
\end{equation}

If $\hat G=\mathrm{GL}_n(q)$, then \eqref{div} becomes $q^n-1\mid q^{m}-1$, against $q^{m}-1<q^{n}-1.$

 If $\tilde G=\mathrm{SL}_n(q)$, then \eqref{div} means $\frac{q^n-1}{q-1}\mid q^{m}-1$. Assume first that $(n,q)\neq(6,2)$ and $(n,q)\neq (2,p)$ for all the Mersenne primes $p.$ Then there exists $r\in P_n(q)$. Since $r\nmid q-1$, we have that $r\mid \frac{q^n-1}{q-1}$ and $\frac{q^n-1}{q-1}\mid q^{m}-1$; however, putting together these divisibility conditions we obtain a contradiction. If $(n,q)=(6,2)$, then
$\mathrm{SL}_n(q)=\mathrm{SL}_6(2)=\mathrm{GL}_6(2)$ an we have seen before that the result holds. Let next $(n,q)= (2,p)$ for some Mersenne prime $p.$ Then $m=1$ and  \eqref{div} becomes $p+1\mid p-1$, which is obviously impossible.

If $\hat G=\mathrm{GU}_n(q)$, for some $n\geq 3$ odd, then \eqref{div} becomes $q^n+1\mid q^{2m}-1.$ Assume first that $(n,q)\neq(3,2)$. Then $(2n,q)\neq (6,2)$ and $2n\neq 2$ so that there exists $r\in P_{2n}(q)$. Since $r\mid q^n+1$ and $q^n+1\mid q^{2m}-1$ with $2m<2n$, we get a contradiction.
When $(n,q)=(3,2)$ we have  $\hat G=\mathrm{GU}_3(2)\leq \mathrm{GL}_3(4)$, $\order s=9$ and \eqref{div} becomes $9\mid 4^{m}-1,$ which is impossible for all $1\leq m\leq 2.$

If $\tilde G=\mathrm{SU}_n(q)$, for some $n\geq 3$ odd, then  \eqref{div} becomes $\frac{q^n+1}{q+1}\mid q^{2m}-1.$ Assume first $(n,q)\neq (3,2)$. Pick $r\in P_{2n}(q)\neq\varnothing$ and note that $r\nmid q^2-1$ so that $r\nmid q+1$. Since $r\mid q^n+1$, we also have that $r\mid \frac{q^n+1}{q+1}\mid q^{2m}-1,$ a contradiction.
When $(n,q)=(3,2)$ we have  $\tilde G=\mathrm{SU}_3(2)$ and $\order s=3$. 
This is a truly exception because $\mathrm{SU}_3(2)$ does not admit Singer cycles. Assume, on the contrary, that $y$ is a Singer cycle of $\mathrm{SU}_3(2)$ and let $v$ be an arbitrary non-zero vector in $V=\mathbb{F}_{4}^3$. Set $w:=v+vy+vy^2$. Observe that $wy=vy+vy^2+vy^3=w$. If $w\ne 0$, this shows that $y$ stabilizes the $1$-dimensional subspace of $V$ spanned by $w$, contradicting the fact that $y$ acts irreducibly. If $w=0$, then $vy^2=v+vy\in \langle v,vy\rangle$ and hence $y$ stabilizes the non-zero proper subspace $\langle v ,vy\rangle$ of $V$, contradicting the fact that $y$ acts irreducibly. Therefore, $\mathrm{SU}_3(2)$ has no Singer cycles and hence it is one of the exceptions in Table~\ref{singer-order}. \footnote{Singer elements in $\mathrm{SU}_3(q)$ are discussed in detail in ~\cite[Section $2$]{be}. However, there is a slight ambiguity in the discussion on $\mathrm{SU}_3(2)$ there and hence, rather than referencing directly to~\cite{be}, we have included our own argument.}

Let next $\hat G=\tilde G=\mathrm{Sp}_{n}(q)$. If $n=2$, then  \eqref{div} becomes $q+1\mid q-1$, which is absurd. Let then
 $n\geq 4$ be even.  Then \eqref{div} becomes $q^{n/2}+1\mid q^{m}-1.$ Assume first that $(n,q)\neq(6,2)$. Then, since $n\neq 2$, there exists $r\in P_{n}(q)$ and $r\mid q^{n/2}+1\mid q^{m}-1,$ a contradiction. Let now $(n,q)=(6,2)$, so that $\tilde G=\mathrm{Sp}_{6}(2)$. Then \eqref{div} becomes $9\mid 2^{m}-1,$ which is impossible for all $1\leq m\leq 5.$

Let now $\hat G=\mathrm{O}_{n}^-(q)$ with $n$ even. Then \eqref{div} becomes $q^{n/2}+1\mid q^{m}-1$ and the same arithmetic arguments used for the symplectic case apply.

Let finally consider $\tilde G=\mathrm{\Omega}_{n}^-(q)$ with $n$ even and $(n,q)\neq (2,3)$ as required in {\rm Table \ref{singer-order}}.
The possibility $n=2$ cannot arise because when $n=2$ we have $\order x=\frac{q+1}{\gcd(2,q-1)}$ so that $\langle x\rangle=\mathrm{\Omega}_{2}^-(q)$ and, by Lemma \ref{action}, $V$ is irreducible.
Assume then $n\geq 4$. Since we have observed that $9\mid 2^{m}-1$ is impossible for all $1\leq m\leq 5,$ we are left with $(n,q)\neq (6,2)$. Thus $P_{n}(q)\neq\varnothing$ and we argue as in the symplectic case because, by \eqref{boundppd}, any $r\in P_{n}(q)$ is odd and hence divides $\frac{q^{n/2}+1}{\gcd(2,q-1)}$.
\vspace{2mm}

$(2)$ Assume that $\tilde G$ appears in {\rm Table \ref{singer-order}}. By $(1)$ it is enough to exhibit an element of $\tilde G$ having the order of a Singer cycle of $\tilde G$. This is immediately done by suitable multiplications. Indeed, let $\langle a\rangle=\mathbb{F}_{q^{\delta n}}^*$. By what explained in Section \ref{Huppert} we have that $\pi_{a^{q-1}}\in \mathrm{SL}_n(q)$ and has order $\frac{q^n-1}{q-1}$; for $n$ odd, $\pi_{a^{(q^n-1)(q+1)}}\in \mathrm{SU}_n(q)$ and has order $\frac{q^n+1}{q+1}$; $\pi_{a^{q^{n/2}-1}}\in \mathrm{Sp}_n(q)$ and has order $q^{n/2}+1$;  $\pi_{a^{2(q^{n/2}-1)}}\in \mathrm{\Omega}^-_n(q)$ and has order $\frac{q^{\frac{n}{2}}+1}{\gcd(2,q-1)}$.

Conversely, assume that $\tilde G$ admits a Singer cycle $s$. Then $\tilde G\neq  \mathrm{\Omega}^-_2(3)$ because $ \mathrm{\Omega}^-_2(3)$ does not admit Singer cycles, by Lemma \ref{action}. Moreover, $\tilde G\neq  \mathrm{SU}_3(2)$, because above we have proved that $ \mathrm{SU}_3(2)$ does not admit Singer cycles. 
Observe that $\langle s\rangle$ is a cyclic subgroup of $\hat G$ acting irreducibly on $V$ and thus, by  the results in \cite{hu},  the possible groups $\hat G$ are those appearing in {\rm Table \ref{singer-order}}. But then also $\tilde G$ appear in {\rm Table \ref{singer-order}}.
\vspace{2mm}

$(3)$ Let $a\in\mathbb{F}_{q^{2n}}^*$ be  such that $\langle a\rangle=\mathbb{F}_{q^{2n}}^*$ and look at $\pi_a$ as a $\mathbb{F}_{q^2}$-linear transformation of $\mathbb{F}_{q^{2n}}$. By \eqref{piadet},  we have $\mathrm{det}_{\mathbb{F}_{q^2}}(\pi_a)=a^{\frac{q^{2n}-1}{q^2-1} }\in \mathbb{F}_{q^{2}}^* $ so that  $\order {\mathrm{det}_{\mathbb{F}_{q^2}}(\pi_a)}=q^2-1.$ By Section \ref{Huppert}, we know that $s:=\pi_{a^{q^n-1}}$ is a Singer cycle for $\mathrm{GU}_n(q)$. Let $\mu:=\mathrm{det}(s)=(\mathrm{det}_{\mathbb{F}_{q^2}}(\pi_a))^{q^n-1}\in \mathbb{F}_{q^2}^*$ and observe that, by Lemma \ref{aritme}, we have
 $$\order{\mu}=\frac{q^2-1}{\gcd(q^2-1, q^n-1)}=\frac{q^2-1}{q-1}=q+1,$$
so that $\mu$ generates the subgroup $U$ of $\mathbb{F}_{q^2}^*$ of size $q+1$. Since $\lambda$ also generates $U$, there exists an integer $k$ with $1\leq k\leq q+1$ and $\gcd(k, q+1)=1$ such that $\lambda=\mu^k.$ Consider now the arithmetic progression $k_{\ell}:= k+\ell(q+1)$. By Dirichlet's Theorem there exists $\ell'$ such that
 $k':=k+\ell'(q+1)$ is a prime number and $k'>q^n+1.$ Define then $s':=s^{k'}$. Since $k'$ is coprime with $\order s=q^n+1$, we have
  $\langle s'\rangle=\langle s\rangle$ and hence $s'$ is a Singer cycle for $\mathrm{GU}_n(q)$. Moreover $\mathrm{det}(s')=\mu^{k'}=\mu^{k}=\lambda.$
\vspace{2mm}

  $(4)$ The proof is similar to that of $(3)$  and thus omitted.
\vspace{2mm}
\end{proof}

\subsection{The $ppd$-elements} We introduce now  the fundamental facts about \textbf{\textit{primitive prime divisor elements}} (a.k.a. $ppd$-elements) developed by
Guralnick, Penttila, Praeger and Saxl in~\cite{gpps}.
 Throughout this paper, that theory is the main tool in finding the maximal
  subgroups containing elements with order divisible by certain ``large''
  primes in classical groups.

An element $x\in \mathrm{GL}_n(q)$ is said to be a $ppd(n,q;e)$\textit{\textbf{-element}}, for $e$ integer such that $n/2 < e \le n$,
if $\order x$ is divisible by some prime $r\in P_e(q)$. Note that this notion implicitly requires that $P_e(q)\neq \varnothing.$

 A subgroup $M$ of $\mathrm{GL}_n(q)$ containing a $ppd(n,q;e)$-element is said
a \textbf{\textit{$ppd(n,q;e)$-group}}. The main result in~\cite{gpps} gives
a satisfactory description of the $ppd(n,q;e)$-groups. There the $ppd(n,q;e)$-groups are divided into nine classes, described through the Examples~$2.1$--$2.9$, and it is shown that every $ppd(n,q;e)$-group belongs to one of these classes. We are particularly interested in maximal $ppd(n,q;e)$-groups. In particular, with a careful analysis of the  Examples~$2.1$--$2.9$ from~\cite{gpps}, we may deduce the following result.

\begin{theorem}[{{\em\cite[Main Theorem]{gpps}}}] \label{main}
Let $q=p^f$ be a prime power and let $n$ be an integer with $n \ge2$. Then  a subgroup $M$ of $\mathrm{GL}_n(q)$ is a $ppd(n,q;e)$-group if and only if $M$ is one of the groups in Examples~$2.1$--$2.9$ in~\cite{gpps}. Moreover,
\begin{enumerate}
\item\label{main1} $M\not \in \mathcal{C}_4\cup
\mathcal{C}_7;$
\item\label{main2}if $M\in\mathcal{S}$, then $M$ is one of
the groups described in Examples~$2.6$--$2.9$;
\item\label{main3} if $e \le n-3$ and $M$ is  one of
the groups described in Examples~$2.5$,~$2.6$~$b)$, $2.6$~$c)$,~$2.7$,~
$2.8$,~$2.9$, then using the notation in~\cite{gpps}
one of the followings holds:
\begin{itemize}
\item $n=7$, $e=4$, $M'\cong\mathrm{Sp}_6(2)$ and $p>2$,
\item $n=9$, $e=6$, $M'\cong\mathrm{SL}_3(q)$, $f$ is even and $\sqrt{q}\equiv 1\pmod 3$,
\item $n=9$, $e=6$, $M'\cong\mathrm{PSL}_3(q)$, $f$ is even and $\sqrt{q}\not\equiv 1\pmod 3$,
\item $n=s+1$, $e=s-2$, $M'\cong\mathrm{PSL}_2(s)$, $s\ge 7$, $s=2^c$ with $c$ a prime.
\end{itemize}
\end{enumerate}
Moreover, when $e\le n-4$, $M$ does not lie in $\mathcal{C}_6$ and if $M$ does lie in $\mathcal{S}$, then $M$ appears in Example~$2.6$~$a)$ of~\cite{gpps}.
\end{theorem}
Observe that, when $e\le n-3$ and $M$ is one of the groups described in Examples~$2.6 a)$, Theorem~\ref{main} does not give any additional information. This lack of additional information will play little role in our proofs later.

In order to exclude the Aschbacher class $\mathcal{C}_5$, it is useful to introduce a further notion. A $ppd(n,q;e)$-element is called a
\textbf{\textit{strong $ppd(n,q;e)$-element}} if its order is divisible by every
$r\in P_e(q)$. Of course also this notion implicitly requires that $P_e(q)\neq \varnothing.$

\begin{lemma}\label{no-c5}
Let $\tilde G\leq \mathrm{GL}_n( q^{\delta})$ be a classical group and  let $n/2 < e \le n$. Then the following hold
\begin{enumerate}
\item If $M\in \mathcal{C}_5$ is a maximal subgroup of $\tilde G$, then there exists a prime $k\mid \delta f$ such that $M\leq \mathrm{GL}_n( q^{\delta/k})=\mathrm{GL}_n( p^{(\delta f)/k}).$
\item Let $M\in \mathcal{C}_5$ be a maximal subgroup of $\tilde G$ with $M\leq \mathrm{GL}_n( q^{\delta/k}),$ where $k\mid \delta f$ is a prime. Then $M$ contains no element of $\tilde G$ having order divisible by some prime in $ P_{ke}(q^{\delta/k})$.
Moreover, if $P_{ke}(q^{\delta/k})\neq \varnothing$, then $M$ contains no strong $ppd(n,q;e)$-element.
\item If $n\geq 6$, then a maximal subgroup of $\tilde G$ in the class $\mathcal{C}_5$ never contains a strong
$ppd(n,q;e)$-element.
\end{enumerate}
\end{lemma}
\begin{proof} $(1)$ This follows immediately by the definition of the maximal subgroups $M$ of $\tilde G$ belonging to class $\mathcal{C}_5$ as given in~\cite[Chapters~3 and~4]{kl}. Note that in the unitary case $\delta=2$ and there are three cases to consider: one with $k$ odd and two with $k=2.$

  $(2)$  Assume, by contradiction, that  there exist $y\in \tilde G$ and  $r\in P_{ke}(q^{\delta/k})$ such that $r\mid \order y$ and $y\in M.$ Then we have
  \begin{equation*}\label{divisibility}
  p\neq r\mid |M|\mid |\mathrm{GL}_n( q^{\delta/k})|.
  \end{equation*}
 Hence $r\mid (q^{\delta/k})^i-1$ for some $i\in \{1,\dots,n\}$. By the definition of primitive prime divisor, this implies $ke\leq n$. Then we get $n\ge ke\geq 2e>n$, a contradiction.

  Assume next that $M$ contains a strong $ppd(n,q;e)$-element $y\in \tilde G$. By Lemma~\ref{primitivi}~\eqref{primitivi1}, we
have $P_{ke}(q^{\delta/k})\subseteq P_{e}(q^{\delta})$. Since every prime in  $P_{e}(q^{\delta})$ divides $\order y$ and since $P_{ke}(q^{\delta/k})\neq \varnothing$, there exists $r\in P_{ke}(q^{\delta/k})$ such that $r\mid \order y,$ in contrast to what previously shown.

 $(3)$ When $n\geq 6$, we have $ke\geq 2e\geq n+1\geq 7>6$ and thus, by Zsigmondy theorem, $P_{ke}(q^{\delta/k})\neq \varnothing$. Now apply $(2)$.
\end{proof}

\subsection{Bertrand elements}\label{bernum}
We use $ppd(d,q;e)$-elements only for some
special values $t$ of $e$, which we now describe.
Recall that, for every $n\geq 7,$  Bertrand's postulate guarantees the existence of a prime number $t$ such that $n/2< t \leq n-2.$
When  $n \ge 8,$ it follows from Bertrand's postulate that there exists a prime $t$ such that $n/2< t \leq n-3$, see~\cite[page~273]{HW}. The relevance of $n-3$  here is related to $n-3$ appearing in  Theorem~\ref{main} part~\eqref{main3}, but this will be more clear when we will embark in the proof of Theorems~\ref{main theorem} and~\ref{main theorem1}.
We call such a prime $t$ a  \textbf{\textit{Bertrand number}} for $n$. We extend the notion of Bertrand numbers also when $n=5,6,7$ by setting $t:=3, 4, 5$, respectively. We do not define Bertrand numbers when $n<5$. Note that if $t$ is a Bertrand number for $n\geq 5$, then $t\nmid n$ and $n/2< t \leq n-2$.

In Table~\ref{2}, we define  an element  $z$ in some classical groups depending on a certain Bertrand number $t$. We refer to such a $z$ as a \textbf{\textit{Bertrand element}}. In this preliminary section, we only give a very rough description of $z$. More details are given inside each section dealing with a certain family of classical groups.
In particular, the existence of the Bertrand elements and their action on $V$ will be specified in due course.
Strictly speaking, in the second column of Table~\ref{2} we only define the order of $z$.  In the third column of Table~\ref{2} we define the natural number with respect to which $t$ is a Bertrand number.  Apart from a few degenerate cases, Bertrand elements are $ppd(n,q;e)$-elements and we define in the fourth column of Table~\ref{2} the value of $e$. The degenerate cases will be discussed later.
 \begin{table}[ht]
\begin{tabular}{c|c|c|c}
\toprule[1.5pt]
$G$ & order of $z$ & $t$ & $ e$ \\
\midrule[1.5pt]
 $\mathrm{SU}_n(q),\ n\neq 6$ & $\frac{( q^t +1)(
q^{n-t} +(-1)^n)}{ \gcd(q^t+1,q^{n-t}+(-1)^n)}$ & $t$ Bertrand number for $n$ & $t$\\[2mm]
$\mathrm{Sp}_{n}(q)$  & $\frac{(q^t + 1)(q^{n/2-t} + 1)}{\gcd(q^t +
1,q^{n/2-t} + 1)}$ & $t$ Bertrand number for $\frac{n}{2}$ & $2t$ \\[2mm]
$\Omega_{n}( q)$ & $\frac{(q^t + 1)(q^{(n-1)/2-t}+  1)}{\gcd(q^t +
1,q^{(n-1)/2-t} + 1)}$ & $t$ Bertrand number for $\frac{n-1}{2}$ & $2t$ \\
\bottomrule[1.5pt]
\end{tabular}
\caption{Bertrand elements $z$} \label{2}\end{table}

\subsection{The spinor norm and the Bertrand elements}\label{spinor-ber-sec}
We recall the definition and properties of the spinor norm  following the description in~\cite[pages~29,~30]{kl}. Let $\varepsilon\in \{+,-\}$, let $\ell$ be an even positive integer  and $q$ be odd. Consider the special orthogonal group $\mathrm{SO}_\ell^\varepsilon(q)$ with respect to a non-degenerate quadratic form $Q:\mathbb{F}_q^\ell\to\mathbb{F}_q$ on $\mathbb{F}_q^\ell$. We let $$\langle\cdot,\cdot\rangle:\mathbb{F}_q^\ell\times\mathbb{F}_q^\ell\to\mathbb{F}_q$$ be the non-degenerate symmetric form polarizing to $Q$. This is clearly well-defined and unique because $q$ is odd. Since $q$ is odd, we also infer that the multiplicative group $\mathbb{F}_q^\ast$ of $\mathbb{F}_q$ has even order $q-1$ and hence the subgroup $$(\mathbb{F}_q^\ast)^2:=\{x^2\mid x\in \mathbb{F}_q^\ast\}$$ of $\mathbb{F}_q^\ast$ consisting of the square elements of the field $\mathbb{F}_q$  different from $0$ has index $2$ in $\mathbb{F}_q^\ast$. In other words,
$\mathbb{F}_q^\ast/(\mathbb{F}_q^\ast)^2$ is a cyclic group of order $2$. Let $\lambda\in\mathbb{F}_q^\ast\setminus(\mathbb{F}_q^\ast)^2$. Then $\langle \lambda (\mathbb{F}_q^\ast)^2\rangle=\mathbb{F}_q^\ast/(\mathbb{F}_q^\ast)^2$ and $\lambda^2\in (\mathbb{F}_q^\ast)^2$.

From~\cite[Proposition~2.5.6]{kl}, every element of $g\in \mathrm{SO}_\ell^\varepsilon(q)$ can be written as the product of an even number of reflections, that is, $g=r_{v_1}\cdots r_{v_{2\kappa}}$ for some reflections $r_{v_1},\ldots,r_{v_{2\kappa}}$ and for some non-negative integer $\kappa$. Here, following the notation in~\cite{kl}, we are denoting with $r_v$ the reflection with respect to the axis $v\in \mathbb{F}_q^\ell$. Observe that the reflection $$x\mapsto xr_v=x-\frac{\langle x,v\rangle}{Q(v)}v$$ is well-defined only when $v$ is a non-degenerate vector with respect to $Q$. The spinor norm is the mapping
$$\theta:\mathrm{SO}_\ell^\varepsilon(q)\to \mathbb{F}_q^\ast/(\mathbb{F}_q^\ast)^2$$
defined by
$$g=r_{v_1}\cdots r_{v_{2\kappa}}\mapsto\prod_{i=1}^{2\kappa}\langle v_i,v_i\rangle\pmod{(\mathbb{F}_q^\ast)^2}.$$
Observe that since  $v_i$ is non-degenerate for each $i$, we have $\prod_{i=1}^{2\kappa}\langle v_i,v_i\rangle\ne 0$ and hence  $\prod_{i=1}^{2\kappa}\langle v_i,v_i\rangle\in\mathbb{F}_q^\ast$. There are various things remarkable about this mapping. First, $\theta$ is well defined, that is, the value of $\theta(g)$ does not depend on the way we express $g$ as a product of reflections.
Second, $\theta$ is a group homomorphism from the special orthogonal group $\mathrm{SO}_\ell^\varepsilon(q)$ to the cyclic group $\mathbb{F}_q^\ast/(\mathbb{F}_q^\ast)^2$ of order $2$. Third, $\theta$ is surjective. Fourth, the kernel of $\theta$ is a subgroup of $\mathrm{SO}_\ell^\varepsilon(q)$ having index $2$ and it is indeed our player $\Omega_\ell^\varepsilon(q)$. All of these facts can be found in~\cite[pages~29,~30]{kl} and in the bibliography therein.

For every $q$, since $|\mathrm{SO}_\ell^\varepsilon(q):\Omega_\ell^\varepsilon(q)|$ is a power of $2$, we deduce that $\Omega_\ell^\varepsilon(q)$ contains all the odd order elements of $\mathrm{SO}_\ell^\varepsilon(q)$.

The situation when $q$ is even is much simpler. Indeed, except for $\mathrm{O}_4^+(2)$, when when $q$ is even,  $\Omega_\ell^\varepsilon(q)$ is the subgroup of $\mathrm{SO}_\ell^\varepsilon(q)$ consisting of products of an even number of reflections. For more details see~\cite[page~30]{kl} and~\cite[Proposition~2.5.9]{kl}.

We now present a result which, among other things, will help us in proving the existence of Bertrand elements for the orthogonal groups throughout the paper.

\begin{proposition}\label{spinor-ber-prop} Let $n,m$ be even positive integers with $2\leq m\leq n/2$ and consider the embedding of $\mathrm{SO}^-_{m}(q)\perp \mathrm{SO}^-_{n-m}(q)$ in $\mathrm{SO}^{+}_{n}(q)$. 
Let $s_m\in \mathrm{SO}^-_{m}(q)$ and $s_{n-m}\in \mathrm{SO}^-_{n-m}(q)$ be Singer cycles and define $x:=s_m\oplus s_{n-m}\in\mathrm{SO}_n^+(q).$ Then $x$ has action type $m\oplus(n-m)$, 
$${\bf o}(x)=\frac{(q^{\frac{m}{2}}+1)(q^{\frac{n}{2}-\frac{m}{2}}+1)}{\gcd(q^{\frac{m}{2}}+1,q^{\frac{n}{2}-\frac{m}{2}}+1)}$$
and $x\in\Omega_n^+(q)$. 
\end{proposition}
\begin{proof}
We have observed in Section~\ref{Huppert} that $\mathrm{SO}^-_{m}(q)$ and $\mathrm{SO}^-_{n-m}(q)$ contain Singer cycles $s_m$ and $s_{n-m}$ having order $q^{m/2}+1$ and $q^{n/2-m/2}+1$, respectively. Thus ${\bf o}(x)$ and its action immediately follows and we only need to show that $x\in\Omega_n^+(q)$.
From~\cite[Table~2.1C]{kl}, we have $|\mathrm{SO}_n^+(q):\Omega_n^+(q)|=2$.  If $q$ is even, we immediately get $x\in \Omega_n^+(q)$, because ${\bf o}(x)$ is odd.  Suppose next that  $q$ is odd.
Let $\theta:\mathrm{SO}_n^+(q)\to \mathbb{F}_q^\ast/(\mathbb{F}_q^\ast)^2$  be the spinor norm of $\mathrm{SO}_n^+(q)$. Observe that the spinor norms of $\mathrm{SO}_m^-(q)$ and $\mathrm{SO}_{n-m}^-(q)$ are simply the restrictions of the spinor norm $\theta$ to $\mathrm{SO}_m^-(q)$ and $\mathrm{SO}_{n-m}^-(q)$, respectively, because the quadratic forms defining $\mathrm{SO}_m^-(q)$ and $\mathrm{SO}_{n-m}^-(q)$ are simply the restrictions of the quadratic form defining $\mathrm{SO}_n^+(q)$.

From Table~\ref{singer-order},
we see that, for every even positive integer $\ell$, when $q$ is odd,  a Singer cycle in $\mathrm{O}_\ell^-(q)$ is never contained in $\Omega_\ell^-(q)$. Hence
$\Omega_m^-(q)$ and $\Omega_{n-m}^-(q)$ do not contain $s_m$ and $s_{n-m}$ and hence the spinor norms of $s_m$ and $s_{n-m}$ are not the identity, that is,
$$\theta(s_m)=\lambda(\mathbb{F}_q^\ast)^2\quad \hbox{and}\quad \theta(s_{n-m})=\lambda(\mathbb{F}_q^\ast)^2,$$
where $\lambda \in \mathbb{F}_q^\ast$ and  $\langle\lambda (\mathbb{F}_q^\ast)^2\rangle=\mathbb{F}_q^\ast/(\mathbb{F}_q^\ast)^2.$

It follows that $$\theta(x)=\theta(s_m\oplus s_{n-m})=\theta(s_m)\theta(s_{n-m})=\lambda^2(\mathbb{F}_q^\ast)^2=(\mathbb{F}_q^\ast)^2.$$
This shows that $\theta(x)$ is the identity element of  $\mathbb{F}_q^\ast/(\mathbb{F}_q^\ast)^2$ and hence $x\in \Omega_n^+(q)$.
\end{proof}
\smallskip

Sometimes, throughout the paper, we will freely use some of the facts in Section~\ref{preliminaries} with no further reference.

\section{Linear groups}\label{sec:linear}
In Section~\ref{sub:preliminaries}, we have explained  that, by \cite{bl}, we know that $\gamma_w(\mathrm{PSL}_n(q))\geq3$ for $n\geq 5$ and $\gamma_w(\mathrm{PSL}_n(q))\leq2$ for $2\leq n\le 4$.
Thus we only need to deal with projective special linear groups $\mathrm{PSL}_n(q)$ with $2\leq n\le 4$, where  $q\notin\{2,3\}$ when $n=2,$ with the purpose to exclude $\gamma_w(\mathrm{PSL}_n(q))=1$ and to describe the maximal components of a weak normal $2$-covering.

Since $\mathrm{PSL}_2(4)\cong\mathrm{PSL}_2(5)\cong A_5$ and $\mathrm{PSL}_2(9)\cong A_6$ and since we have already discussed alternating groups in Section~\ref{sub:preliminaries}, we may suppose that, when $n=2$, we have $q\notin\{2,3,4,5,9\}$. 

Recall that, given a subgroup $X$ of $\mathrm{PSL}_n(q)$, we denote by $\tilde{X}$ its preimage under the natural projection $\mathrm{SL}_n(q)\to \mathrm{PSL}_n(q)$. Moreover, since the automorphism group of $\mathrm{Aut}(\mathrm{SL}_n(q))$ projects onto the automorphism group of $\mathrm{PSL}_n(q)$, in the proofs of the lemmas in this section we may work with the linear group $\mathrm{SL}_n(q)$ (see Section \ref{cove-classic-simple}).

For some of the facts that we use in the proof of Lemmas~\ref{dimension2linear},~\ref{dimension3linear} and~\ref{dimension4linear}, we could simply refer to some results already known in the literature. However, we have decided to include full arguments of these lemmas because this will serve as a warm up for later.

\begin{lemma}\label{dimension2linear}
Let $q$ be a prime power with $q\notin\{2,3,4,5,9\}$. Then  $\gamma_w(\mathrm{PSL}_2(q))=2$. Moreover, if $H$ and $K$ are the two maximal components  of a weak normal $2$-covering of $\mathrm{PSL}_2(q)$, then up to $\mathrm{Aut}(\mathrm{PSL}_2(q))$-conjugacy  one of the following holds
\begin{enumerate}
\item\label{eq:sl1} $q\ne 7$, $H\cong D_{2(q+1)/\gcd(2,q-1)}$ is in class $\mathcal{C}_3$ and $K$ is a parabolic subgroup in class $\mathcal{C}_1$,
\item\label{eq:sl2} $q$ is even, $H\cong D_{2(q+1)}$ is in class $\mathcal{C}_3$ and $K\cong D_{2(q-1)}$ is in class $\mathcal{C}_2$, 
\item\label{eq:sl4}$q=7$, $H\cong S_4$ is in  class $\mathcal{C}_6$ and $K$ is a parabolic subgroup in  class $\mathcal{C}_1$.
\end{enumerate}
Moreover, each of these weak normal $2$-coverings gives rise to a single normal covering. % Un diedrale $D_{2m}$ ha un normale di ordine $2$ solo se $m$ è pari. Tale sottogruppo è centrale e incluso nel ciclico di ordine m. Se fai il quoziente trovi quindi $D_{m}$, giusto?

\end{lemma}
\begin{proof}
We use the list of the maximal subgroups of $\mathrm{PSL}_2(q)$ in~\cite[Tables~8.1,~8.2]{bhr} and the notation therein.

When $q\in \{ 7, 11\}$, the proof follows with a computer computation. For the rest of the proof, we suppose that $q\ge 13$. Let $\mu$  be a weak normal covering of $\mathrm{PSL}_2(q)$  of minimum size and maximal components and $\tilde \mu$ be the corresponding weak normal covering of $\mathrm{SL}_2(q)$. Let $x$ be an element of $\mathrm{SL}_2(q)$ having order $q+1$, that is, a Singer cycle for $\mathrm{SL}_2(q)$. 
From~\cite[Tables~8.1,~8.2]{bhr} and from the fact that $q\notin\{5,7,9\}$, we see that the only maximal subgroups of $\mathrm{SL}_2(q)$ containing a Singer cycle form  a unique $\mathrm{SL}_2(q)$-conjugacy class of subgroups  given by the maximal subgroups in class $\mathcal{C}_3$. Therefore, without loss of generality, we may suppose that $x\in \tilde H$ for some $\tilde H\in \tilde \mu$ belonging to class $\mathcal{C}_3$.

From~\cite[Table~$8.1$,~$8.2$]{bhr}, we deduce $\tilde H\cong Q_{2(q+1)}$ when $q$ is odd and  $\tilde H\cong D_{2(q+1)}$ when $q$ is even. In both cases, $\tilde H\cong (q+1).2$.  Let $y$ be an element of $\mathrm{SL}_2(q)$ having order $q-1$. Using the fact  that $q\notin\{5,7,9,11\}$, another case-by-case analysis on the maximal subgroups of $\mathrm{SL}_2(q)$ in~\cite[Tables~8.1,~8.2]{bhr} reveals that there are two $\mathrm{Aut}(\mathrm{SL}_2(q))$-conjugacy classes of maximal subgroups containing an  $\mathrm{Aut}(\mathrm{SL}_2(q))$-conjugate of $y$. Namely,
\begin{itemize}
\item[i)] $E_q:(q-1)\in \mathcal{C}_1$,
\item[ii)]  $Q_{2(q-1)}\in \mathcal{C}_2$, when $q$ is odd, 
\item[iii)]  $D_{2(q-1)}\in \mathcal{C}_2$, when $q$ is even.
\end{itemize}
As $\tilde H$ is not in this list, we deduce that there exists a further component $\tilde K\in \tilde \mu.$ Thus $\gamma_w(\mathrm{PSL}_2(q))=2$ and $\tilde K$ must be in one of the above two possibilities.
Suppose first that $\tilde K=Q_{2(q-1)}$ when $q$ is odd. Now, $\tilde H$ and $\tilde K$ do not contain unipotent elements and hence this case does not give rise to a weak normal $2$-covering of $\mathrm{SL}_2(q)$. Therefore, when $q$ is odd, the only possibility is $\tilde K=E_q:(q-1).$ Hence we obtain as candidates for the weak normal $2$-coverings  of  $\mathrm{PSL}_2(q)$ those in~\eqref{eq:sl1} and~\eqref{eq:sl2}. 

We show that actually, in both cases, $\tilde{\mu}=\{\tilde H, \tilde K\}$ is a normal $2$-covering of $\mathrm{SL}_2(q)$. Assume first that case ~\eqref{eq:sl1} holds. Then $\tilde H$ contains a Singer cycle and $\tilde K$ is parabolic. Let $z\in \mathrm{SL}_2(q)$. If $z$ has no eigenvalues then $\langle z\rangle$ acts irreducibly on the natural module  $V=\mathbb{F}_q^2$ so that $z$ is the power of a Singer cycle (\cite[ Lemma 2.4]{bl}). Since Singer subgroups of $\mathrm{SL}_2(q)$ are $\mathrm{SL}_2(q)$-conjugate, we deduce that $z$ belongs to a $\mathrm{SL}_2(q)$-conjugate of $\tilde H$. If $z$ admits an eigenvalue $\lambda\in\mathbb{F}_q^*$, then choosing a suitable eigenvector for $\lambda$ and completing to a basis for $V$ we construct a matrix $c\in \mathrm{SL}_2(q)$ such that $c^{-1}zc$ is an upper triangular matrix with diagonal entries $\lambda,\lambda^{-1}$ which belongs to $\tilde K$.

Assume next that case~\eqref{eq:sl2} holds. Thus $q$ is even. Let $z\in \mathrm{SL}_2(q)$. If $z$ has no eigenvalues, arguing as in case~\eqref{eq:sl1}, we deduce that $z$ belongs to an $\mathrm{SL}_2(q)$-conjugate of $\tilde H$. By~\cite[Lemma 3.1.3]{bhr}, we know the standard copy of $D_{2(q-1)}$ inside $\mathrm{SL}_2(q)$ and thus we can assume that $\tilde K$ is this standard copy. In other words, we  can assume  $\tilde K=\langle \mathrm{diag}(\omega,\omega^{-1}),  \mathrm{antidiag}(1,1)\rangle $, where $\omega$ generates the multiplicative group $\mathbb{F}_q^*$. Suppose now that $z$ has an eigenvalue $\lambda\in\mathbb{F}_q$. If $\lambda=1$, then $z$ is unipotent and, since involutions in $\mathrm{SL}_2(q)$ form a conjugacy class, we obtain that $z$ has a conjugate in $\tilde H$ and $\tilde K$. If $\lambda\ne 1$, then $z$ has two distinct eigenvalues $\lambda,\lambda^{-1}$. From this it follows that $z$ is $\mathrm{SL}_2(q)$-conjugate to the diagonal matrix $\mathrm{diag}(\lambda,\lambda^{-1})\in\tilde{K}$. 
\end{proof}

\begin{lemma}\label{dimension3linear} Let $q$ be a prime power. Then  $\gamma_w(\mathrm{PSL}_3(q))=2$. Moreover, if $H$ and $K$ are the two maximal  components  of a weak normal $2$-covering of $\mathrm{PSL}_3(q)$, then up to $\mathrm{Aut}(\mathrm{PSL}_3(q))$-conjugacy  one of the following holds
\begin{enumerate}
\item\label{eq:sl31} $q\ne 4$, $H\cong\left(\frac{q^2+q+1}{\gcd(3,q-1)}\right):3$ is in class $\mathcal{C}_3$ and $K$ is a parabolic subgroup  in class $\mathcal{C}_1$,
\item\label{eq:sl32}$q=4$, $H\cong\mathrm{SL}_3(2)$ is in class $\mathcal{C}_5$ and $K$ is a parabolic subgroup in  class $\mathcal{C}_1$,
\item\label{eq:sl33}$q=4$, $H\cong\mathrm{SL}_3(2)$ is in  class $\mathcal{C}_5$ and $K\cong A_6$ is in class $\mathcal{S}$.
\end{enumerate}
In ~\eqref{eq:sl31} the weak normal $2$-covering gives rise to two normal coverings; in ~\eqref{eq:sl32} the weak normal $2$-covering gives rise to six normal coverings; in ~\eqref{eq:sl33} the weak normal $2$-covering does not produce a normal covering.

In particular, up to $\mathrm{Aut}(\mathrm{PSL}_3(q))$-conjugacy,  there exists a unique weak normal $2$-covering of $\mathrm{PSL}_3(q)$ by maximal components when $q\neq 4$; whereas, there are exactly two weak normal $2$-coverings of $\mathrm{PSL}_3(q)$ by maximal components when $q=4$.
\end{lemma}
\begin{proof}
We use the list of the maximal subgroups of $\mathrm{PSL}_3(q)$ in~\cite[Tables~8.3,~8.4]{bhr} and the notation therein.

When $q\in\{2, 4\}$, the proof follows with a computation using the computer algebra system \texttt{magma}~\cite{magma}. We give some details of our computation when $q=4$. The weak normal $2$-covering in~\eqref{eq:sl32} splits into six normal $2$-coverings of $\mathrm{PSL}_3(q)$. There are three conjugacy classes of subgroups $\mathrm{SL}_3(2)$  and two conjugacy classes of parabolic subgroups and choosing a representative for $\mathrm{SL}_3(2)$ and one for a parabolic gives always rise to a normal covering. Those six normal coverings are distinct up to $\mathrm{PSL}_3(4)$-conjugacy. 

For the rest of the proof, we suppose that $q\notin\{2, 4\}$.
Let $\mu$  be a weak normal covering of $\mathrm{PSL}_3(q)$  of minimum size and maximal components and $\tilde \mu$ be the corresponding weak normal covering of $\mathrm{SL}_3(q)$.

Let $x$ be an element of $\mathrm{SL}_3(q)$ having order $q^2+q+1$, that is, a Singer cycle for $\mathrm{SL}_3(q)$.  From~\cite[Tables~8.3,~8.4]{bhr} and from the fact that $q\ne 4$, we see that 
the only maximal subgroups of $\mathrm{SL}_3(q)$ containing a Singer cycle form  a unique $\mathrm{SL}_3(q)$-conjugacy class of subgroups  given by the maximal subgroups in class $\mathcal{C}_3$. Therefore, without loss of generality, we may suppose that $x\in \tilde H\cong (q^2+q+1):3$ for some $\tilde H\in \tilde \mu$ belonging to class $\mathcal{C}_3$.

 Let $y$ be an element of $\mathrm{SL}_3(q)$ having order $q^2-1$. Using the fact  that $q\notin \{2,4\}$, another case-by-case analysis on the maximal subgroups of $\mathrm{SL}_3(q)$ in~\cite[Tables~8.2,~8.3]{bhr} reveals that there are two $\mathrm{SL}_3(q)$-conjugacy classes of maximal subgroups $M$ of $\mathrm{SL}_3(q)$ containing a conjugate of $y$: namely, the two parabolic subgroups of $\mathrm{SL}_3(q)$, both having structure $E_q^2:\mathrm{GL}_2(q).$ These two conjugacy classes are fused by a graph automorphism. As $\tilde H\not\cong E_q^2:\mathrm{GL}_2(q)$, the weak normal covering number of $\mathrm{SL}_3(q)$ is $2$ and  there is a unique possibility for the further component $\tilde K\in \tilde \mu$. 
 
It is finally easy to see that for every $q$, $\tilde H\cong(q^2+q+1):3$ and $\tilde K\cong E_q^2:\mathrm{GL}_2(q)$ are components of a normal $2$-covering for  $\mathrm{SL}_3(q)$ (see \cite[Proposition 4.2, Corollary 4.3]{bl}) and hence,  $H$ and $K$ as in the statement $(1)$ are maximal components of a weak normal $2$-covering for  $\mathrm{PSL}_3(q)$. That covering is, for $q\neq 4$, the unique weak normal $2$-covering for  $\mathrm{PSL}_3(q)$ by maximal components.
\end{proof}

Note that $\tilde G:=\mathrm{SL}_3(3^f)$ admits also the normal covering with components given by
 $\tilde H\cong q^2+q+1:3$ and the Levi subgroup $L$ of the parabolic subgroup $K$
\cite[Thm.~A, Table 1.2]{blw}. However, this example is not considered in Table~\ref{00} because $L$ is not
maximal. Of course, this example does appear in our classification of normal $2$-coverings in Section~\ref{sec:appendix}, see Table~\ref{00:linear}.

\begin{lemma}\label{dimension4linear} Let $q$ be a prime power. Then  $\gamma_w(\mathrm{PSL}_4(q))=2$.
 Moreover, if $H$ and $K$ are the two maximal components of a weak normal $2$-covering of $\mathrm{PSL}_4(q)$, then  up to $\mathrm{Aut}(\mathrm{PSL}_4(q))$-conjugacy  $\tilde{H}\cong \mathrm{SL}_2(q^2).(q+1).2$ is in class $\mathcal{C}_3$ and $\tilde{K}\cong E_q^3:\mathrm{GL}_3(q)$ is in class $\mathcal{C}_1$.
This weak normal covering gives rise to two normal coverings.
\end{lemma}
\begin{proof}
We use the list of the maximal subgroups of $\mathrm{PSL}_4(q)$ in~\cite[Tables~$8$,~$9$]{bhr} and the notation therein. Let $\mu$  be a weak normal covering of $\mathrm{PSL}_4(q)$  of minimum size and maximal components and $\tilde \mu$ be the corresponding weak normal covering of $\mathrm{SL}_4(q)$.
%Let $H$ and $K$ be the maximal components of a weak normal $2$-covering of $\mathrm{PSL}_4(q)$.  Then $\tilde H$ and $\tilde K$ are maximal components of a weak normal $2$-covering of $\mathrm{SL}_4(q)$.

Let $x$ be a Singer cycle of $\mathrm{SL}_4(q)$. From~\cite[Tables~8.8,~8.9]{bhr}, we see that there exists a unique $\mathrm{SL}_4(q)$-conjugacy class of maximal subgroups containing a Singer cycle, that is, the maximal subgroups in class $\mathcal{C}_3$. Hence there exists $\tilde H\in \tilde \mu$ belonging to class $\mathcal{C}_3$ and, without loss of generality, we may suppose that $$x\in \tilde H\cong\mathrm{SL}_2(q^2).(q+1).2.$$

Let $y$ be an element of $\mathrm{SL}_4(q)$ having order $q^3-1$. Another case-by-case analysis on the maximal subgroups of $\mathrm{SL}_4(q)$ in~\cite[Tables~8.8,~8.9]{bhr} reveals that there are two $\mathrm{SL}_4(q)$-conjugacy classes of maximal subgroups containing a conjugate of $y$: namely, the parabolic subgroups of $\mathrm{SL}_4(q)$ isomorphic to $E_q^3:\mathrm{GL}_3(q)$. These two conjugacy classes are fused by a graph automorphism. 
As $\tilde H\not\cong E_q^3:\mathrm{GL}_3(q)$, the weak normal covering number of $\mathrm{SL}_4(q)$ is $2$ and  the unique possibility for the further component  $\tilde K\in \tilde \mu$ is given by $\tilde K\cong E_q^3:\mathrm{GL}_3(q).$

The fact that $\tilde H$ and $\tilde K$ are indeed components of a normal $2$-covering for $\mathrm{PSL}_4(q)$ follows by \cite[Proposition~4.2]{bl}. To deduce the number of $\tilde G$-classes of normal $2$-coverings, we use the observations in Section~\ref{sec:newwen}. Indeed, by~\cite{bhr}, the $\mathrm{Aut}(\tilde G)$-class  $\{\tilde{H}^\varphi\mid \varphi\in \mathrm{Aut}(\tilde{G})\}$ is a $\tilde{G}$-conjugacy class, whereas  the $\mathrm{Aut}(\tilde G)$-class  $\{\tilde{K}^\varphi\mid \varphi\in \mathrm{Aut}(\tilde{G})\}$ is a union of two distinct $\tilde{G}$-conjugacy classes. Therefore, the weak normal covering $\{\tilde{H},\tilde{K}\}$ gives rise to two distinct $\tilde{G}$-classes of normal $2$-coverings.
\end{proof}

Now, the veracity of Table~\ref{0Linear} follows from the results in this section.

\section{Unitary  groups}\label{sec:unitary}

Recall again that given a subgroup $X$ of $\mathrm{PSU}_n(q)$, we denote by $\tilde{X}$ its preimage under the natural projection $\mathrm{SU}_n(q)\to \mathrm{PSU}_n(q)$. Moreover, since the automorphism group of $\mathrm{SU}_n(q)$ projects onto the automorphism group of $\mathrm{PSU}_n(q)$, we may work with the special unitary group $\mathrm{SU}_n(q)$ (see Section \ref{cove-classic-simple}).

We start with a lemma that, except for a handful of  few small cases, will help us to identify one of the components of a weak normal covering of the unitary group.
\begin{lemma}[{\cite[Theorem 1.1]{msw}}]\label{malleu}
Let $n$ be an integer with $n\geq 3$ and $(n,q)\neq (3,2)$. Let $M$ be a maximal subgroup
of $\mathrm{SU}_n(q)$ containing a Singer cycle when $n$ is odd, and  a semisimple element of order $q^{n-1}+1$ and action type $1\oplus (n-1)$ when $n$ is even.
Then one of the following holds
\begin{enumerate}
\item\label{malleu:1} $n$ is odd,  $$M
\cong
\mathrm{SU}_{n/k}(q^k).\,\frac{q^k+1}{q+1}\,.\,k,$$
for some prime number $k$  with $k\mid n$, and $M$ is in  class $\mathcal{C}_3$,
\item\label{malleu:2} $n$ is  even and  $M \cong \mathrm{GU}_{n-1}(q)$ is in class $\mathcal{C}_1$,
\item\label{malleu:33}$(n,q)=(3,3)$ and $M\cong\mathrm{PSL}_2(7)$ is in  class $\mathcal{S}$,
\item\label{malleu:3}$(n,q)=(3,5)$ and $M\cong 3.A_7$ is in class $\mathcal{S}$,
\item\label{malleu:333}$(n,q)=(4,2)$ and $M\cong 3^3.S_4$ is in  class $\mathcal{C}_2$,
\item\label{malleu:3333}$(n,q)=(4,3)$ and $M\cong 4. A_7$ is in  class $\mathcal{S}$,
\item\label{malleu:33333}$(n,q)=(4,3)$ and $M\cong 4.\mathrm{PSL}_3(4)$ is in  class $\mathcal{S}$,
\item\label{malleu:4}$(n,q)=(5,2)$ and $M\cong\mathrm{PSL}_2(11)$ is in class $\mathcal{S}$,
\item\label{malleu:5}$(n,q)=(6,2)$ and $M\cong 3.\mathrm{M}_{22}$ is in class $\mathcal{S}$.
 \end{enumerate}
 \end{lemma}

\subsection{Small dimensional unitary groups}\label{sec:smallunitary}
We start our analysis with small dimensional unitary groups $\mathrm{PSU}_n(q)$ with $3\le n\le 6$. Observe that $\mathrm{SU}_n(q)$ is solvable when $(n,q)=(3,2)$ and thus we do not consider that case. For the subgroup structure of $\mathrm{SU}_3(q)$, $\mathrm{SU}_4(q)$, $\mathrm{SU}_5(q)$ and $\mathrm{SU}_6(q)$ we use~\cite{bhr}.

\begin{lemma}\label{dimension3unitary}
The weak normal covering number of $\mathrm{PSU}_3(q)$ for $q>2$ is at least $2$. Moreover, if $H$ and $K$ are maximal components  of a weak normal $2$-covering of $\mathrm{PSU}_3(q)$, then up to $\mathrm{Aut}(\mathrm{PSU}_3(q))$-conjugacy one of the following holds
\begin{enumerate}
\item\label{3unitary:1} $q=3$,  $\tilde H\cong\mathrm{PSL}_2(7)$ is in  class $\mathcal{S}$ and $\tilde K\cong \mathrm{GU}_2(q)$ is in  class $\mathcal{C}_1$,
\item\label{3unitary:2}  $q=3$,  $\tilde H\cong\mathrm{PSL}_2(7)$ is in  class $\mathcal{S}$ and $\tilde K\cong E_{q}^{1+2}:(q^2-1)$ is in class $\mathcal{C}_1$,
\item\label{3unitary:2bis}  $q$ is a power of $3$, $q>3$,  $\tilde H\cong(q^2-q+1):3$ is in class $\mathcal{C}_3$ and $\tilde K\cong\mathrm{GU}_2(q)$ is in class $\mathcal{C}_1$,
\item\label{3unitary:3}  $q=5$,  $\tilde H\cong3.A_7$ is in  class $\mathcal{S}$ and $\tilde K\cong E_{q}^{1+2}:(q^2-1)$ is in  class $\mathcal{C}_1$,
\item\label{3unitary:4}  $q=5$,  $\tilde H\cong3.A_7$ is in class $\mathcal{S}$ and $\tilde K \cong\mathrm{GU}_2(q)$ is in  class $\mathcal{C}_1$.
\end{enumerate}
In cases~\eqref{3unitary:1}-\eqref{3unitary:2bis} each weak normal $2$-covering of $\mathrm{PSU}_3(q)$ gives rise to a single normal $2$-covering, in case \eqref{3unitary:3} the weak normal $2$-covering of $\mathrm{PSU}_3(q)$ gives rise to three normal $2$-coverings and in case ~\eqref{3unitary:4} no normal $2$-covering arises.
\end{lemma}
\begin{proof}
With the help of the computer algebra system \texttt{magma}~\cite{magma} we find that  for $q\in \{3,5\}$, the weak normal $2$-coverings of $\mathrm{SU}_3(q)$ are the ones reported above. We have also verified that the weak normal coverings in~\eqref{3unitary:1},~\eqref{3unitary:2} gives rise to a single normal $2$-covering and that in ~\eqref{3unitary:3}  to three normal $2$-coverings, whereas the weak normal covering in~\eqref{3unitary:4} does not give rise to a  normal covering. See also Section~\ref{sec:newwen}.
 
 Therefore, for the rest of the proof, we suppose $q\notin \{2,3,5\}$.

Let $\mu$  be a weak normal covering of $\mathrm{PSU}_3(q)$  of minimum size and maximal components. Let
$H\in\mu$ containing a Singer cycle of $\mathrm{PSU}_3(q)$.
Then $\tilde H$ is a maximal component of a weak normal covering $\tilde\mu$ of $\mathrm{SU}_3(q)$ containing a Singer cycle of $\mathrm{SU}_3(q)$.
From Lemma~\ref{malleu}, we have $$\tilde H\cong\left(\frac{q^3+1}{q+1}\right):3=(q^2-q+1):3.$$

Consider now the elements of $\mathrm{SU}_3(q)$ having order $q^2-1$. Using the fact  that $q\notin\{2, 3,5\}$, a case-by-case analysis on the maximal subgroups of $\mathrm{SU}_3(q)$ in~\cite[Tables~8.5,~8.6]{bhr} reveals that there are two $\mathrm{SU}_3(q)$-conjugacy classes of maximal subgroups containing elements of order $q^2-1$.
Namely,
\begin{itemize}
\item[i)] $E_{q}^{1+2}:(q^2-1)\in \mathcal{C}_1$,
\item[ii)] $\mathrm{GU}_2(q)\in \mathcal{C}_1$.
\end{itemize}
As $\tilde H$ is in not in this list, there  must exist a further maximal component  $\tilde K$ of $\tilde\mu$ that contains an element of order $q^2-1$. Thus
$\gamma_w(\mathrm{SU}_3(q))=\gamma_w(\mathrm{PSU}_3(q))\geq 2$ and $\tilde K$ is in one of the above possibilities i),\,ii).

Assume now that $\gamma_w(\mathrm{SU}_3(q))=2$. Then $\tilde\mu=\{\tilde H, \tilde K\}.$
Suppose first that
$$\tilde K\cong E_q^{1+2}:(q^2-1).$$
We investigate elements having order $q+1$. Assume, by contradiction, that $\tilde H$ contains an element of order $q+1$. Then,
an easy calculation gives
$$q+1=\gcd(|\tilde H|,q+1)=\gcd(3(q^2-q+1),q+1)\mid \gcd(9,q+1)\mid 9.$$
It follows that $q=8$ and $\tilde H\cong 57:3$ contains an element of order $9$. However the $3$-Sylow subgroup of a semidirect product $57:3$ cannot be cyclic of order $9$. Therefore, we deduce that every element of  $\mathrm{SU}_3(q)$ of order $q+1$  has a conjugate in  $\tilde K$ via an element in the automorphism group of $\mathrm{SU}_3(q)$. The group $\mathrm{SU}_3(q)$ contains elements $g$ with $\order g=q+1$ and with $g$ having three distinct eigenvalues in $\mathbb{F}_{q^2}$. Indeed, suppose that the Hermitian matrix preserved by $\mathrm{SU}_3(q)$ is
\[
J=\begin{pmatrix}
1&0&0\\
0&1&0\\
0&0&1
\end{pmatrix}.
\]
Let $\lambda\in \mathbb{F}_{q^2}$ be an element of order $q+1$. Then it is easy to verify that
\[
\begin{pmatrix}
\lambda&0&0\\
0&\lambda^{-1}&0\\
0&0&1
\end{pmatrix}\in\mathrm{SU}_3(q)
\] 
and its three  eigenvalues are distinct.  However, we now show that all elements of $\tilde K\cong E_{q}^{1+2}:(q^2-1)$ having order $q+1$ admit an eigenvalue with multiplicity $\geq 2$. Since the outer automorphisms of $\mathrm{SU}_3(q)$ consist of field automorphisms,  this property is preserved by conjugation under $\mathrm{Aut}(\mathrm{SU}_3(q))$. Therefore, this case will not arise. Suppose that the Hermitian form preserved by $\mathrm{SU}_3(q)$ has matrix
\[J=\begin{pmatrix}
0&0&1\\
0&1&0\\
1&0&0
\end{pmatrix}.\]
As $\tilde K$ is the stabilizer of a totally singular $1$-dimensional subspace of $\mathbb{F}_{q^2}^3$, we may assume that $\tilde{K}$ is the stabilizer of $\langle (1,0,0)\rangle_{\mathbb{F}_{q^2}}$. Now, a computation with the matrix $J$ shows that  $\tilde K$ consists of the matrices 
\[\begin{pmatrix}
a&0&0\\
b&c&0\\
e&f&g
\end{pmatrix},
\]
where $ a,b,c,e,f,g\in \mathbb{F}_{q^2}$, with $acg=1$, $g=a^{-q}$,  $c^{q+1}=1$, $cf^q+bg^q=ge^q+f^{q+1}+eg^q=0$.  
A Levi complement of $\tilde K$ is
\[
\left\{
\begin{pmatrix}
a&0&0\\
0&a^{q-1}&0\\
0&0&a^{-q}
\end{pmatrix}\mid a\in \mathbb{F}_{q^2}^*
\right\}.\]
The elements having order $q+1$ in this Levi complement are of the form
\[\begin{pmatrix}
a^{q-1}&0&0\\
0&a^{(q-1)^2}&0\\
0&0&a^{-q(q-1)}
\end{pmatrix},\]
where $a$ is a generator of the multiplicative group $\mathbb{F}_{q^2}^*$.
As $$a^{-q(q-1)}=a^{-q^2+q}=a^{-1+q}=a^{q-1},$$ we see that the elements of order $q+1$ in the Levi complement have an eigenvalue with multiplicity $\ge 2$. Finally, observe that every element of order $q+1$ in $\tilde{K}$ is $\tilde{K}$-conjugate to an element of the Levi complement.

\smallskip

Suppose next that
$$\tilde K\cong \mathrm{GU}_2(q).$$ We now consider unipotent elements. The group $\mathrm{SU}_3(q)$ contains unipotent elements $u_1$ and $u_2$ with $\dim_{\mathbb{F}_{q^2}} \cent V {u_1}=1$ and $\dim_{\mathbb{F}_{q^2}} \cent V{u_2}=2$.  The former are precisely the regular unipotent elements.
Since every non-identity unipotent element $u$ of $\mathrm{GU}_2(q)$ has the property that $\dim_{\mathbb{F}_{q^2}}\cent V u=2$, we deduce $\tilde H$ contains non-identity unipotent elements. Thus,
$$ 1\neq\gcd(|\tilde H|,q)=\gcd(3,q)$$
and hence $q$ is a power of $3$. From~\cite[Table~$1.2$]{blw}, we see that when $q>3$ is a power of $3$ we have that  $\tilde H\cong(q^2-q+1):3$ and $\tilde K\cong\mathrm{GU}_2(q)$ are indeed the components of a normal $2$-covering of $\mathrm{SU}_3(q)$ and hence also of a weak normal $2$-covering of $\mathrm{SU}_3(q)$. Section~\ref{sec:newwen} justifies the fact that we only have one $G$-conjugacy class of such coverings.
\end{proof}
\begin{lemma}\label{dimension4unitary}
The weak normal covering number of $\mathrm{PSU}_4(q)$ is  $2$. Moreover, if $H$ and $K$ are maximal components  of a weak normal $2$-covering of $\mathrm{PSU}_4(q)$, then up to $\mathrm{Aut}(\mathrm{PSU}_4(q))$-conjugacy  one of the following holds
\begin{enumerate}
\item\label{4unitary:1} $q=2$,  $\tilde H\cong\mathrm{GU}_3(q)\in \mathcal{C}_1$ and $\tilde K\cong \mathrm{Sp}_4(q)\in\mathcal{C}_5$,
\item\label{4unitary:2} $q=3$,  $\tilde H\cong 4.A_7\in\mathcal{S}$ and $\tilde K\cong E_{q}^{1+4}:\mathrm{SU}_2(q):(q^2-1)\in\mathcal{C}_1$,
\item\label{4unitary:3} $q=3$,  $\tilde H\cong\mathrm{GU}_3(q)\in\mathcal{C}_1$ and $\tilde K\cong4.\mathrm{PSU}_4(2)\in\mathcal{S}$ and, 
\item\label{4unitary:4} $q=3$,  $\tilde H\cong4.\mathrm{PSL}_3(4)\in\mathcal{S}$ and $\tilde K\cong E_{q}^{1+4}:\mathrm{SU}_2(q):(q^2-1)\in\mathcal{C}_1$,
\item\label{4unitary:5} $\tilde H\cong\mathrm{GU}_3(q)$ and $\tilde K\cong E_{q}^{4}:\mathrm{SL}_2(q^2):(q-1)\in\mathcal{C}_1.$ 
\end{enumerate}
Except for~\eqref{4unitary:3}, each weak normal $2$-covering of $\mathrm{PSU}_4(q)$ gives rise to a single normal $2$-covering. In case ~\eqref{4unitary:3} no normal $2$-covering of $\mathrm{PSU}_4(q)$  arises. 
\end{lemma}
\begin{proof}
When $q\in \{2,3\}$, the result follows with a computation with the computer algebra system \texttt{magma}~\cite{magma}. Therefore, for the rest of the argument, we suppose $q\ge 4$.

Let $\mu$  be a weak normal covering of $\mathrm{PSU}_4(q)$  of minimum size and maximal components. Let
$\tilde H\in\tilde \mu$ containing some element of order $q^3+1$ and action type $1\oplus 3$. From Lemma~\ref{malleu},  we deduce that $$\tilde
H\cong\mathrm{GU}_3(q).$$

Using the fact  that $q\notin\{2, 3\}$, a case-by-case analysis on the maximal subgroups of $\mathrm{SU}_4(q)$ in~\cite[Tables~8.10,~8.11]{bhr} reveals that there are precisely  two $\mathrm{SU}_4(q)$-conjugacy classes of maximal subgroups  containing  elements of order $(q^4-1)/(q+1)$. Namely,
\begin{itemize}
\item[i)] $E_{q}^{4}:\mathrm{SL}_2(q^2):(q-1)\in\mathcal{C}_1$,
\item[ii)]  $\mathrm{SL}_2(q^2).(q-1).2\in \mathcal{C}_2$.
\end{itemize}
As $\tilde H$ is in not in this list, there  must exist a further maximal
component  $\tilde K$ of $\tilde\mu$ that contains an element of order $(q^4-1)/(q+1)$. Thus $\gamma_w(\mathrm{SU}_4(q))=\gamma_w(\mathrm{PSU}_4(q))\geq 2$.
Assume now that $\gamma_w(\mathrm{SU}_4(q))=2$. Then $\tilde\mu=\{\tilde H, \tilde K\}$ and $\tilde K$ is in one of the above possibilities i),\,ii). We show that necessarily $\tilde{K}\cong E_{q}^{4}:\mathrm{SL}_2(q^2):(q-1)\in\mathcal{C}_1$.

Suppose first that $q$ is odd. As neither $\mathrm{GU}_3(q)$ nor $\mathrm{SL}_2(q^2).(q-1).2\in \mathcal{C}_2$
contain regular unipotent elements, $\tilde{K}$ is isomorphic  to $E_{q}^{4}:\mathrm{SL}_2(q^2):(q-1)\in\mathcal{C}_1.$
By~\cite[Proposition~5.1]{blw},
$$\tilde H\cong\mathrm{GU}_3(q) \hbox{  and  }\tilde K\cong E_{q}^4:\mathrm{SL}_2(q^2):(q-1)$$
are indeed the components of a normal $2$-covering for $\mathrm{SU}_4(q)$ and hence also of a weak normal $2$-covering of $\mathrm{SU}_3(q)$. As usual Section~\ref{sec:newwen} yields that this weak normal $2$-covering gives rise to one $G$-class of normal $2$-coverings. For the rest of our argument we deal with the case $q$ even.

\smallskip

We claim that $\mathrm{SU}_4(q)$ contains an element $g$ with $\order g=2(q^2-1)$. To see this, we suppose that the Hermitian form preserved by $\mathrm{SU}_4(q)$ is given by the matrix
\[
J:=\begin{pmatrix}
0&0&0&1\\
0&0&1&0\\
0&1&0&0\\
1&0&0&0
\end{pmatrix}.
\]
Let $\mathbb{F}_{q^2}^*=\langle \lambda\rangle$ and set $\mu:=\lambda^{q}$. Consider now the matrix
\[
g:=\begin{pmatrix}
0&0&0&\lambda\\
0&\mu&0&0\\
0&0&\mu^{-q}&0\\
\lambda^{-q}&0&0&0
\end{pmatrix}.
\]
An easy computation shows that $g$ preserves the Hermitian form given by $J$ and hence $g\in\mathrm{GU}_4(q)$. Moreover, $$\det (g)=-\mu^{1-q}\lambda^{1-q}=-\lambda^{q-q^2}\lambda^{1-q}=-\lambda^{1-q^2}=-1=1$$
and thus $g\in\mathrm{SU}_4(q)$. Moreover, as
\begin{equation*}g^2=\begin{pmatrix}
\lambda^{1-q}&0&0&0\\
0&\mu^2&0&0\\
0&0&\mu^{-2q}&0\\
0&0&0&\lambda^{1-q}
\end{pmatrix},
\end{equation*}
 one sees easily that $g$ has order $2(q^2-1)$.

 Assume, by contradiction, that $g$ belongs to a stabilizer $\tilde H\cong\mathrm{GU}_3(q)$ of some non-degenerate $1$-dimensional $\mathbb{F}_{q^2}$-subspace $V_1$ of $V$. Then we have $y:=g^2\in \tilde H$ and $g\in \cent {\tilde H}{ y}$. In order to reach a contradiction, it is enough to see that $\cent {\tilde H}{ y}$ is formed by diagonal matrices. Note that we have $\order {\lambda^{1-q}}=q+1$, $\order \mu=q^2-1$ and $$\order {\mu^2}=\order {\mu^{-2q}}=\frac{q^2-1}{\gcd(2, q^2-1)}=q^2-1.$$ 
Since $q\geq 4$, this implies that $\lambda^{1-q}, \mu^2$ and $ \mu^{-2q}$ are distinct, and that the only one among them with order divisible by $q+1$ is $\lambda^{1-q}$. Clearly we have $V_{\lambda^{1-q}}(y)=\langle e_1,e_4\rangle$, $V_{\mu^{2}}(y)=\langle e_2\rangle$ and $V_{\mu^{-2q}}(y)=\langle e_3\rangle$. Moreover, taking into account our Hermitian form, we see  that  $V=V_{\lambda^{1-q}}(y)\perp (V_{\mu^{2}}(y)\oplus V_{\mu^{-2q}}(y)).$ Since $y$ stabilizes $V_1$, we have that $V_1$ is an eigenspace for $y$ corresponding to an eigenvalue whose order must be a divisor of $q+1$, because $V_1$ is non-degenerate. Thus that eigenvalue necessarily  equals $\lambda^{1-q}$ and then  $V_1\subset V_{\lambda^{1-q}}(y)$.
Pick then $w_1\in V_{\lambda^{1-q}}(y)$ orthogonal to $V_1$ and note that the set of vectors $\{w_1,e_2,e_3\}$  is
a basis of $V_1^\perp$  consisting of eigenvectors of $y$ with respect to the pairwise distinct eigenvalues $\lambda^{1-q}, \mu^2$ and $ \mu^{-2q}$. Then if
$$d:=\begin{pmatrix}
\mu^2&0&0\\
0&\mu^{-2q}&0\\
0&0&\lambda^{1-q}\\
\end{pmatrix},$$
we have that $\cent {\tilde H}{ y}\cong \cent {\mathrm{GU}(V_1^\perp)}{ d}$ is isomorphic to a subgroup of $\cent {\mathrm{GL}_3(q^2)}{d}$ and
since the diagonal entries of $d$ are distinct, $\cent {\mathrm{GL}_3(q^2)}{d}$ is formed only by diagonal matrices. Summing up, we have shown that $g$ does not belong to the stabilizer of a $1$-dimensional non-degenerate subspace.

Thus $g$ must be $\mathrm{Aut}(\mathrm{SU}_4(q))$-conjugate to an element of $\tilde K$. We now show that $\mathrm{SL}_2(q^2).(q-1).2$ has no element of order $2(q^2-1)$. Suppose that the Hermitian form preserved by $\mathrm{SU}_4(q)$ has matrix \[
J:=\begin{pmatrix}
0&0&1&0\\
0&0&0&1\\
1&0&0&0\\
0&1&0&0
\end{pmatrix}.
\]
As $\mathrm{SL}_2(q^2).(q-1).2\in\mathcal{C}_2$, we see that $\mathrm{SL}_2(q^2).(q-1).2$ is the stabilizer of a direct sum decomposition of totally singular subspaces of dimension $2$, which we may assume to be $\langle e_1,e_2\rangle\oplus\langle e_3,e_4\rangle$, where $e_1,e_2,e_3,e_4$ is the canonical basis for $V=\mathbb{F}_{q^2}^4$. A direct computation with $J$ shows that the subgroup of $\mathrm{SU}_4(q)$ preserving this direct sum decomposition is
$$\tilde K=\left\langle
\begin{pmatrix}
A&0\\
0&\bar{A}^{-T}
\end{pmatrix},
\begin{pmatrix}
0&I\\
I&0
\end{pmatrix}
\mid A\in \mathrm{GL}_2(q^2), \det (A)\in \mathbb{F}_q
\right\rangle,$$
here we are denoting by $\bar A$ the image of $A$ under the involutory field automorphism.
We argue by contradiction and we let $x=us=su\in \tilde K$ having order $2(q^2-1)$, with $u$ unipotent and $s$ semisimple. If $u\in \mathrm{SL}_2(q^2)$, then $\cent {\mathrm{SL}_2(q^2)}u$ is the Sylow $2$-subgroup of order $q^2$ and hence $|\cent {\tilde K}u|$ divides $q^2(q-1)2$, contradicting the fact that $s\in \cent {\tilde K}u$ has order $q^2-1$. Therefore, $u$ is an involution in $\tilde{K}\setminus\mathrm{SL}_2(q^2)$.

By~\cite[p.~46, Exercise 1]{jps}, the $\mathrm{SL}_2(q^2)$-conjugacy classes of involutions in $\tilde{K}\setminus\mathrm{SL}_2(q^2)$
are in $1$-to-$1$ correspondence to elements in
 the non-abelian cohomology set $H^1(\mathrm{Gal}(\mathbb{F}_{q^2}/\mathbb{F}_q),\mathrm{SL}_2(q^2))$. By
the Lang-Steinberg theorem $H^1(\mathrm{Gal}(\mathbb{F}_{q^2}/\mathbb{F}_q),\mathrm{SL}_2(q^2))$  consists of a single element (see
\cite[\S 2.3, Theorem 1']{jps}).
Hence we may suppose that
$$u=\begin{pmatrix}
0&I\\
I&0
\end{pmatrix}.
$$Now, from a straightforward computation one concludes that
$$\cent{\tilde K} u=\left\langle
\begin{pmatrix}
A&0\\
0&A
\end{pmatrix}
\mid A\in \mathrm{GL}_2(q^2), \det (A)\in \mathbb{F}_q,A=(\bar{A}^{-1})^T
\right\rangle.
$$Now the condition $A=\bar{A}^{-T}$ implies that $A\in \mathrm{GU}_2(q)$. As in dimension $2$ we have $\mathrm{GU}_2(q)=(q+1)\times \mathrm{SU}_2(q)$ and as $\det (A)\in \mathbb{F}_q$, we have $A\in \mathrm{SU}_2(q)$. However, there are no elements of order $q^2-1$ in $\mathrm{SU}_2(q)\cong\mathrm{SL}_2(q)$. From this argument,
one concludes that $\mathrm{SL}_2(q^2).(q-1).2$ does not contain elements of order $2(q^2-1)$. Thus we deduce that
$$\tilde K\cong E_{q}^{4}:\mathrm{SL}_2(q^2):(q-1).$$
Finally, by \cite[Proposition~5.1]{blw},
$$\tilde H\cong\mathrm{GU}_3(q) \hbox{  and  }\tilde K\cong E_{q}^4:\mathrm{SL}_2(q^2):(q-1)$$
are indeed the components of a normal $2$-covering for $\mathrm{SU}_4(q)$ and hence also of a weak normal $2$-covering of $\mathrm{SU}_3(q)$. As usual Section~\ref{sec:newwen} yields that this weak normal $2$-covering gives rise to one $G$-class of normal $2$-coverings.
\end{proof}

\begin{lemma}\label{dimension5unitary}
The weak normal covering number of $\mathrm{PSU}_5(q)$ is at least $3$.
\end{lemma}
\begin{proof}
When $q\in \{2,3\}$, the result follows with a computation with the computer algebra system \texttt{magma}~\cite{magma}. Therefore, for the rest of the argument, we suppose $q\ge 4$.

We argue by contradiction and suppose that  $H$ and $K$ are the maximal components of a weak normal $2$-covering $\mu$ of $\mathrm{PSU}_5(q)$. Then $\tilde \mu=\{\tilde H, \tilde K\}$ is a weak normal $2$-covering with maximal components of $\mathrm{SU}_5(q)$.
Since $\mathrm{SU}_5(q)$ contains Singer cycles, from Lemma~\ref{malleu}, replacing $H$ with $K$ if necessary, we have $$\tilde H\cong\left(\frac{q^5+1}{q+1}\right):5.$$

Consider the elements of $\mathrm{SU}_5(q)$ having order $q^4-1$. Using the fact  that $q\neq 2$, another case-by-case analysis on the maximal subgroups of $\mathrm{SU}_5(q)$ in~\cite[Tables~8.20,~8.21]{bhr} reveals that there are two $\mathrm{SU}_5(q)$-conjugacy classes of maximal subgroups containing elements of  order $q^4-1$. Namely,
\begin{itemize}
\item[i)] $E_{q}^{4+4}:\mathrm{GL}_2(q^2)\in\mathcal{C}_1$,
\item[ii)]  $\mathrm{GU}_4(q)\in\mathcal{C}_1$.
\end{itemize}
In particular, $\tilde K$ must be in one of these two possibilities.

The group $\mathrm{SU}_5(q)$ contains semisimple elements $g$ having order $(q^2-q+1)\cdot(q-1)$. To obtain such elements it suffices to start with an orthogonal decomposition $V=W\perp W'$, with $\dim_{\mathbb{F}_{q^2}}W=3$, $\dim_{\mathbb{F}_{q^2}}W'=2$, and then take a Singer cycle of order $q^2-q+1$ for $\mathrm{SU}_3(q)$ on $W$ (which exists because $q\ne 2$) and a diagonal matrix for $\mathrm{SU}_2(q)$ on $W'$. Actually, if the Hermitian form induced on $W'$ is
\[
\begin{pmatrix}
0&1\\
1&0
\end{pmatrix},
\]
then we may suppose that the matrix induced on $W'$ is
\[
\begin{pmatrix}
a&0\\
0&a^{-1}
\end{pmatrix},
\]
for $a\in\mathbb{F}_q^\ast$ with $\langle a\rangle=\mathbb{F}_q^\ast$. Observe that,  since $q\notin\{2,3\}$, we have  $a\ne a^{-1}$ and thus the only irreducible subspaces of $V$ left invariant by $g$ are $W$, and two totally singular subspaces of $W'$. In particular, no $\mathrm{Aut}(\mathrm{SU}_5(q))$-conjugate of $g$ lies in $\mathrm{GU}_4(q)$, because $\mathrm{GU}_4(q)$ is a stabilizer of a non-degenerate $1$-dimensional subspace of $V$.

Finally, since $E_{q}^{4+4}:\mathrm{GL}_2(q^2)$ contains no elements having order divisible by $q^2-q+1$, we see that no $\mathrm{Aut}(\mathrm{SU}_5(q))$-conjugate of $g$ lies in $E_{q}^{4+4}:\mathrm{GL}_2(q^2)$.
\end{proof}

\begin{lemma}\label{dimension6unitary}
The weak normal covering number of $\mathrm{PSU}_6(q)$ is at least $2$. Moreover, if $H$ and $K$ are maximal components  of a weak normal $2$-covering of $\mathrm{PSU}_6(q)$, then up to $\mathrm{Aut}(\mathrm{PSU}_6(q))$-conjugacy one of the following holds
\begin{enumerate}
\item\label{6unitary:1} $q=2$,  $\tilde H\cong 3\times \mathrm{Sp}_6(q)\in \mathcal{C}_5$ and $\tilde K\cong\mathrm{GU}_5(q)\in\mathcal{C}_1$,
\item\label{6unitary:2} $q=2$,  $\tilde H\cong 3.\mathrm{PSU}_4(3)\in\mathcal{S}$ and $\tilde K\cong \mathrm{GU}_5(q)\in\mathcal{C}_1$.
\end{enumerate}
The weak normal $2$-coverings in~\eqref{6unitary:1} and~\eqref{6unitary:2} do not give rise to normal coverings. In particular $\gamma(\mathrm{PSU}_6(q))\geq 3.$
\end{lemma}
\begin{proof}
When $q=2$, the result follows with a computation with the computer algebra system \texttt{magma}~\cite{magma}. Therefore, for the rest of the argument, we suppose $q\ge 3$.

Let $\mu$  be a weak normal covering of $\mathrm{PSU}_6(q)$  of minimum size and maximal components. Let
$\tilde H\in\tilde \mu$ containing some element of order $q^5+1$ and action of type $1\oplus 5$. From Lemma~\ref{malleu},  we then have
$$\tilde H\cong\mathrm{GU}_5(q).$$
Consider the elements of $\mathrm{SU}_6(q)$ having order $(q^6-1)/(q+1)$. A case-by-case analysis on the maximal subgroups of $\mathrm{SU}_6(q)$ in~\cite[Tables~8.26,~8.27]{bhr} reveals that there are three $\mathrm{SU}_6(q)$-conjugacy classes of maximal subgroups containing elements of order $(q^6-1)/(q+1)$. Namely,
\begin{itemize}
\item[i)] $E_{q}^{9}:\mathrm{SL}_3(q^2):(q-1)\in \mathcal{C}_1$,
\item[ii)] $\mathrm{SL}_3(q^2).(q-1).2\in\mathcal{C}_2$,
\item[iii)] $\mathrm{SU}_2(q^3).(q^2-q+1).3\in \mathcal{C}_3$.
\end{itemize}
Since none of the above groups is isomorphic to $\tilde H$, we deduce that there exists a further component
$\tilde K\in \tilde \mu$.  Hence $\gamma_w(\mathrm{PSU}_6(q))\geq 2$, which confirms the result in~\cite{Sa88}. Assume now that there exists a weak normal $2$-covering of $\mathrm{PSU}_6(q)$ with maximal components. By the above arguments, we have $\tilde \mu=\{\tilde H,\tilde K\}$ and $\tilde K$  must be in one of the three possibilities i)-iii).

Let $W_1$ be an arbitrary $3$-dimensional non-degenerate subspace of $V$ and let $W_2:=W_1^\perp$. Now, let $g\in \mathrm{SU}_6(q)$ be a semisimple element such that $W_1,W_2$ are the only $g$-invariant  irreducible subspaces of $V$. This can be arranged by inducing on $W_1$ and on $W_2$ two suitable distinct Singer cycles. Since all elements of $\tilde H\cong\mathrm{GU}_5(q)$ fix a $1$-dimensional non-degenerate subspace of $V$ and since this property is preserved by $\mathrm{Aut}(\mathrm{SU}_6(q))$-conjugation, we deduce that $g$ is $\mathrm{Aut}(\mathrm{SU}_6(q))$-conjugate to an element in $\tilde K$.

As $E_{q}^{9}:\mathrm{SL}_3(q^2):(q-1)$ is the stabilizer of a $3$-dimensional totally singular subspace of $V$, $g$ cannot be $\mathrm{Aut}(\mathrm{SU}_6(q))$-conjugate to an element in $E_{q}^{9}:\mathrm{SL}_3(q^2):(q-1)$. Thus $\tilde K\ncong E_{q}^{9}:\mathrm{SL}_3(q^2):(q-1)$.

As $\mathrm{SL}_3(q^2).(q-1).2$ is the stabilizer of a direct decomposition $V=V_1\oplus V_2$ where $V_1$ and $V_2$ are $3$-dimensional totally singular subspaces of $V$, $g$ cannot be $\mathrm{Aut}(\mathrm{SU}_6(q))$-conjugate to an element in $\mathrm{SL}_3(q^2).(q-1).2\in\mathcal{C}_2$. Thus $\tilde K\ncong \mathrm{SL}_3(q^2).(q-1).2$.

Therefore $\tilde K$ is isomorphic to the extension field subgroup $\mathrm{SU}_2(q^3).(q^2-q+1).3$.
Now, for every $i\in \{1,2,3,4,5\}$, $\mathrm{SU}_6(q)$ contains a unipotent element  $u$ with $\dim_{\mathbb{F}_{q^2}}\cent V{u}=i$. Clearly, every unipotent element $u$ of $\tilde H\cong \mathrm{GU}_5(q)$ satisfies $\dim_{\mathbb{F}_{q^2}}\cent V u\ge 2$. Therefore, all unipotent elements $u$ of $\mathrm{SU}_6(q)$ satisfying $\dim_{\mathbb{F}_{q^2}}\cent Vu=1$ are conjugate, via an element in $\mathrm{Aut}(\mathrm{SU}_6(q))$, to an element of $\tilde{K}$. As  $\tilde K\cong \mathrm{SU}_2(q^3).(q^2-q+1).3$ and as every unipotent element $u\in\mathrm{SU}_2(q^3)$ satisfies $\dim_{\mathbb{F}_{q^2}}\cent Vu\ge 2$, this case can arise only when $\gcd (q,(q^2-q+1)3)\ne 1$, that is, when $q$ is a power of $3$. Summing up, $q$ is a power of $3$ and $\tilde K\cong \mathrm{SU}_2(q^3).(q^2-q+1).3$.
Without loss of generality, we may suppose that the Hermitian form left invariant by $\mathrm{SU}_6(q)$ is given via the matrix
\[
\begin{pmatrix}
0&1&0&0&0&0\\
1&0&0&0&0&0\\
0&0&0&0&1&0\\
0&0&0&0&0&1\\
0&0&1&0&0&0\\
0&0&0&1&0&0\\
\end{pmatrix}.
\]
Using this Hermitian form, it is readily seen that $\mathrm{SU}_6(q)$ contains the matrix
\[
z:=\begin{pmatrix}
1&1&0&0\\
0&1&0&0\\
0&0&s&0\\
0&0&0&(\bar{s}^{-1})^T
\end{pmatrix},
\]
where $s\in\mathrm{SL}_2(q^2)$ is a Singer cycle having order $q^2+1$.
 The only $1$-dimensional subspace of $V$ stabilized by $z$ is $\langle (0,1,0,0,0,0)\rangle$, which is totally singular. Therefore $z$ has no conjugate in $\tilde H\cong \mathrm{GU}_5(q)$, via elements in $\mathrm{Aut}(\mathrm{SU}_6(q))$. By construction, $\order z=3(q^2+1)$. With a computation, using Lemma~\ref{aritme}, we obtain 
\begin{align*}\gcd (q^2+1,|\tilde K|)&=\gcd(q^2+1,q^3(q^6-1)(q^2-q+1)3)\\
&=\gcd(q^2+1,(q^6-1)(q^2-q+1))\\
&=\gcd(2,q-1)=2.
\end{align*}
Therefore, $z$ is not $\mathrm{Aut}(\mathrm{SU}_6(q))$-conjugate to an element of $\tilde K$.
\end{proof}

\subsection{Large dimensional unitary groups}\label{sec:largeunitary}

In this section we deal with large dimensional unitary groups $\mathrm{SU}_n(q)$ with  $n\ge 7$.

Let $t$ be a Bertrand number for $n$. As $n\ge 7$, $t$ is an odd prime with $\gcd(t,n)=1$, $t\ge 5$ and $n/2< t \leq n-2$. Recall that  if  $n\geq 8$, then the stronger inequality $n/2< t \leq n-3$ holds.

Moreover, from Lemma~\ref{aritme}~\eqref{eq:arithme2} and~\eqref{eq:arithme3}, we have $\gcd(q^{t}+1,q^{n-t}+(-1)^n)=q+1$.
In order to find the components of a  weak normal covering of $\mathrm{SU}_n(q)$, we consider a Bertrand element $z\in\mathrm{SU}_n(q)$ such that
\begin{itemize}
\item $\order z=\frac{(q^t
+1)(q^{n-t} +(-1)^n)}{\gcd(q^t+1,q^{n-t}+(-1)^n)}=\frac{(q^t
+1)(q^{n-t} +(-1)^n)}{q+1}$ (see Table~\ref{2}),
\item the action of $z$ on $V$ is of type $t\oplus \frac{n-t}{2}\oplus\frac{n-t}{2}$ if $n$ is odd and of type
$t\oplus (n-t)$ if $n$ is even. We write $V=V_t\perp W_1\perp W_2$ when $n$ is odd and $V=V_t\perp W$ when $n$ is even, where $\dim_{\mathbb{F}_{q^2}}V_t=t$, $\dim_{\mathbb{F}_{q^2}}W_1=\dim_{\mathbb{F}_{q^2}}W_2=(n-t)/2$, $\dim_{\mathbb{F}_{q^2}}W=n-t$ and $V_t,W_1,W_2,W$ are $z$-invariant submodules of $V$,  with $W$ and $V_t$ non-degenerate and $W_1$, $W_2$ totally singular.
\item $z$ induces a matrix of order $q^t+1$ on $V_t$, of order $q^{n-t}+1$ on $W$ and of order $q^{n-t}-1$ on both $W_1$ and $W_2$.
\end{itemize}
We first justify the existence of the Bertrand elements.

Let first $n$ be even.
Then both $t$ and $n-t$ are odd.  Pick a generator $\lambda$ for the subgroup of order $q+1$ in $\mathbb{F}_{q^2}^*$. Then, by Proposition \ref{sing-ord}$\,(3)$, these exist
Singer cycles $s_1\in\mathrm{GU}_t(q)$ and $s_2\in\mathrm{GU}_{n-t}(q)$ such that $\mathrm{det}(s_1)=\lambda$ and $\mathrm{det}(s_2)=\lambda^{-1}.$ Thus, by the embedding of $\mathrm{GU}_t(q) \perp \mathrm{GU}_{n-t}(q)$ in $\mathrm{GU}_n(q)$, we get an element $z\in \mathrm{SU}_n(q)$ having all the properties required to be a Bertrand element.

Let next $n$ be odd. The embedding of $\mathrm{GL}_{\frac{n-t}{2}}(q^2)$ in $\mathrm{GU}_{n-t}(q)$ determines a block matrix $x$ of order $q^{n-t}-1$ and determinant $b=a^{-q+1}$, where $\langle a \rangle=\mathbb{F}_{q^2}^*$. Since $b$ generates the subgroup of order $q+1$ of $\mathbb{F}_{q^2}^*$, by Proposition \ref{sing-ord}$\,(3)$, there exists a Singer cycle $s$ of $\mathrm{GU}_{t}(q)$ having determinant equal to $b^{-1}$. Thus the block matrix constructed using $x$ and $s$ gives the required Bertrand element $z$.

In the next result, we use~\cite{kl} for the labeling of the parabolic subgroups of $\mathrm{SU}_n(q)$.
\begin{lemma}\label{uni-z}
Let $M$ be a maximal subgroup of $\mathrm{SU}_n(q)$ with $n\ge 7$ containing a Bertrand element $z$ as in Table~$\ref{2}$ and satisfying the conditions above. Then one of the following holds
\begin{enumerate}
\item\label{uni-z:1} $M$ is of type $\mathrm{GU}_t(q)\perp \mathrm{GU}_{n-t}(q)$ and lies in class $\mathcal{C}_1$,
\item\label{uni-z:2} $n$ is odd,
$M$ is a parabolic subgroup of type $P_{(n-t)/2}$ and lies in class $\mathcal{C}_1$.
\end{enumerate}
\end{lemma}

\begin{proof}
As $t\ge 5$, from Zsigmondy's theorem and from Lemma \ref{primitivi}, we deduce that
\begin{equation}\label{5}
\varnothing \neq P_{2t}(q)\subseteq P_{t}(q^2).
\end{equation}
Let $r\in P_{2t}(q)$. Then
$r\mid q^t+1$ and $r\nmid q+1$. Hence $r\mid \frac{q^t+1}{q+1}$ and thus $r\mid \order z=\frac{q^t
+1}{q+1}(q^{n-t} +(-1)^n).$ Moreover, by \eqref{boundppd}, we have $r\geq 2t+1$.
In particular, $z$ is a $ppd(n,q^2;t)$-element with order divisible by an $r\in P_t(q^2)$ with $r\geq 2t+1.$
It follows that  $M$ is a $ppd(n,q^2;t)$-group and, by Theorem \ref{main},  we search
$M$ among the groups in the Examples~2.1--2.9 of~\cite{gpps}, with $n\geq 7$ and $r\geq 2t+1.$
From the definition of Bertrand number and from $n\ge 7$, we have either $n/2< t \le n-3$  or $(n,t)=(7,5).$  In any case we also have $\gcd(t,n)=1.$ Now, looking
at Tables~2-8 of~\cite{gpps}, it is easily checked that $r\neq 2 t+1$ since $t\leq n-2$. These facts rule out the groups of Examples~2.3, 2.4, 2.5, 2.6, 2.7, 2.8, 2.9, because they are
all given under at least one of the conditions:
\begin{itemize}
\item $t> n-3$ and
$n\neq 7,$ or
\item  $t=n-3$ even, or
\item $r = t+1,$ or
\item $r=2t+1.$
\end{itemize}
Hence we may reduce our considerations to Examples~2.1--2.2, that is, to $M\in \mathcal{C}_1$ or $M \in \mathcal{C}_5.$
\smallskip

\noindent\textsc{The subgroup $M$ lies in class $\mathcal{C}_1$. }We have two cases to consider: $M$ is the stabilizer of a proper totally singular subspace or of a proper non-degenerate subspace of $V$. We consider first the second possibility. Thus  $M$ is of type $\mathrm{GU}_m(q) \perp \mathrm{GU}_{n-m}(q)$, with $1\le m<n/2$.
Recall that the action of $z$ on $V$ is of type $t\oplus \frac{n-t}{2}\oplus \frac{n-t}{2}$ or $t\oplus n-t$ depending on whether $n$ is odd or even. When $n$ is even, $z$ leaves invariant exactly two proper subspaces of $V$, one having dimension $t$ and the other having dimension $n-t$; moreover, these subspaces are both non-degenerate. Therefore, if $z\in M$, then $t=n-m$, that is, $m=n-t$ and we obtain part~\eqref{uni-z:1}. Suppose $n$ odd. By definition,  $V$ has exactly three irreducible $\mathbb{F}_{q^2}\langle z\rangle$-submodules: namely, $W_1,W_2$ and $V_t$, where $V_t$ and $W_1\oplus W_2$ are non-degenerate and $W_1,W_2$ are totally singular. Therefore, if $z\in M$, then $t=n-m$, that is, $m=n-t$ and we obtain part~\eqref{uni-z:1}.
We consider now the case that $M$ is the stabilizer of a proper totally singular subspace of $V$. Thus  $M$ is a parabolic subgroup $P_m$, with $1\le m\le\lfloor n/2\rfloor$.  When $n$ is even, the two proper subspaces of $V$ left invariant by $z$ are both non-degenerate and hence  $z$ lies in no parabolic subgroup. Thus $n$ is odd.  The only totally singular proper subspaces of $V$ which are $z$-invariant are $W_1$ and $W_2$.  As $\dim_{\mathbb{F}_{q^2}} W_1=\dim_{\mathbb{F}_{q^2}} W_2=(n-t)/2$, we deduce  $m=(n-t)/2$ and we obtain part~\eqref{uni-z:2}.

\smallskip
\noindent\textsc{The subgroup $M$ lies in  class $\mathcal{C}_5$. }  By Lemma \ref{no-c5}, we have $M\leq \mathrm{GL}_n( q^{\delta/k})$, where
 $\delta=2$ and $k$ is a prime with $k\mid 2f$. When $k=2$, the possibility $M\in\mathcal{C}_5$  is ruled out by \eqref{5} and Lemma \ref{no-c5}\,(2). When $k\geq 3$, as  $kt\geq 15$, Lemma~\ref{no-c5}\,(2) applies because $P_{kt}(q^{2/k})\neq\varnothing$.
\end{proof}

\begin{proposition}\label{unitari}
For every $n\ge 7$, the weak normal covering number of $\mathrm{PSU}_n(q)$ is at least $3$.
\end{proposition}
\begin{proof}
As usual, we argue with $\mathrm{SU}_n(q)$. Let $\tilde H$ be a component of a weak normal covering of $\mathrm{SU}_n(q)$ containing a semisimple element having order $(q^n+1)/(q+1)$ and of type $n$ in its action on $V$ when $n$ is odd, and having order $q^{n-1}+1$ and of type $1\oplus (n-1)$ in its action on $V$ when $n$ is even. Moreover, let $\tilde K$ be a component containing the Bertrand element defined in Lemma~\ref{uni-z}.

Assume first $n$ even. From Lemma~\ref{malleu}, we have $\tilde H\cong\mathrm{GU}_{n-1}(q)$ and, from Lemma~\ref{uni-z}, we have that $\tilde K$ is of type $\mathrm{GU}_t(q)\perp \mathrm{GU}_{n-t}(q)$. Let $x\in \mathrm{SU}_n(q)$ be a semisimple element having order $(q^n-1)/(q+1)$ and having type $\frac{n}{2}\oplus \frac{n}{2}$ on $V$. Thus $V=W_1\perp W_2$, $\dim_{\mathrm{F}_{q^2}}W_1=\dim_{\mathrm{F}_{q^2}}W_2=n/2$ and $W_1$, $W_2$ are $x$-invariant. We also choose $x$ so that the matrix induced in its action on $W_1$ and on $W_2$ has order $(q^{n}-1)/(q+1)$. Now, the only proper $x$-invariant subspaces of $V$ are $W_1$ and $W_2$ and hence $x$ cannot have an $\mathrm{Aut}(\mathrm{SU}_n(q))$-conjugate in $\tilde{H}\cong\mathrm{GU}_{n-1}(q)$ because $n/2\notin\{1,n-1\}$. Similarly, $x$ cannot have an $\mathrm{Aut}(\mathrm{SU}_n(q))$-conjugate in $\tilde{K}$ because $t\ne n/2$. Therefore a weak normal covering of $\mathrm{SU}_n(q)$ contains at least one more component.

Assume $n$ odd. From Lemma~\ref{malleu}, we have $\tilde H\cong\mathrm{SU}_{n/k}(q^k).((q^k+1)/(q+1)).k$ for some prime divisor $k$ of $n$ and, from Lemma~\ref{uni-z}, we have that $\tilde K$ is a parabolic subgroup $P_{(n-t)/2}$ or is of type $\mathrm{GU}_t(q)\perp\mathrm{GU}_{n-t}(q)$. Let $x\in \mathrm{SU}_n(q)$ be a semisimple element having order $q^{n-2}+1$ and having type $1\oplus 1\oplus (n-2)$ on $V$. Thus $V=W_1\perp W_2\perp W$, $\dim_{\mathrm{F}_{q^2}}W_1=\dim_{\mathrm{F}_{q^2}}W_2=1$, $\dim_{\mathrm{F}_{q^2}}W=n-2$ and $W_1$, $W_2$, $W$ are $x$-invariant. We also choose $x$ so that the matrix induced in its action on $W_1$ and on $W_2$ has order $q+1$ and the matrix induced by $x$ on $W$ has order $q^{n-2}+1$. Now, the only proper $x$-invariant subspaces of $V$ are $W_1$, $W_2$ and $W$. Hence $x$ cannot have an $\mathrm{Aut}(\mathrm{SU}_n(q))$-conjugate in $\tilde{K}$ because $x$ does not stabilize a totally isotropic subspace of dimension $(n-t)/2$, nor a non-degenerate subspace of dimension $t$. Suppose $x$ is $\mathrm{Aut}(\mathrm{SU}_n(q))$-conjugate to an element in $\tilde H$. Then $q^{n-2}+1$ divides the order of $\tilde{H}$. From Zsigmondy's theorem, we may choose $q_{2(n-2)}\in P_{2(n-2)}(q)$. Now, the definition of primitive prime divisor yields that $q_{2(n-2)}$ is relatively prime to the order of $\mathrm{SU}_{n/k}(q^k).((q^k+1)/(q+1))$. Thus $q_{2(n-2)}$ divides $k$ and hence $q_{2(n-2)}\le k$. However,~\eqref{boundppd} yields $q_{2(n-2)}\ge 2(n-2)+1=2n-3$ and hence $k\ge 2n-3>n$, a contradiction. Therefore a weak normal covering of $\mathrm{SU}_n(q)$ contains at least one more component. Thus $\gamma_w(\mathrm{SU}_n(q))=\gamma_w(\mathrm{PSU}_n(q))\geq 3.$
\end{proof}

Now, the veracity of Table~\ref{00000} follows from the results in this section.

\section{Symplectic  groups}\label{sec:symplectic}

As usual, given a subgroup $X$ of $\mathrm{PSp}_n(q)$, we denote by $\tilde{X}$ its preimage under the natural projection $\mathrm{Sp}_n(q)\to \mathrm{PSp}_n(q)$. Exactly as for the unitary groups, since the automorphism group of $\mathrm{Aut}(\mathrm{Sp}_n(q))$ projects onto the automorphism group of $\mathrm{PSp}_n(q)$, we may work with  $\mathrm{Sp}_n(q)$. Extra care must to be taken when $n=4$ and $q$ is even, because graph-field automorphisms of $\mathrm{Sp}_4(q)$ do not act on the vector space $V$. 

\begin{lemma}[{{\cite[Theorem 1.1]{msw}}}]\label{malles} 
Let $n$ be an integer with $n\geq 4$, $(n,q)\ne (4,2)$ and let $M$ be a maximal subgroup
of $\mathrm{Sp}_n(q)$ containing  a Singer cycle.
Then one of the following holds
\begin{enumerate}
\item\label{malles:1} $M\cong\mathrm{Sp}_{n/k}(q^k).k$ for some prime number $k$  with $k\mid n$ and $n/k$ even, and $M\in \mathcal{C}_3$,
\item\label{malles:2} $nq/2$ is  odd and  $M \cong \mathrm{GU}_{n/2}(q).2$ is in  class $\mathcal{C}_3$,
\item\label{malles:3} $q$ is even and  $M \cong \mathrm{SO}_{n}^-(q)$ is in class $\mathcal{C}_8$,
\item\label{malles:4}$(n,q)=(8,2)$ and $M\cong\mathrm{PSL}_2(17)$ is in  class $\mathcal{S}$,
\item\label{malles:5}$(n,q)=(4,3)$ and $M\cong 2^{1+4}_-.A_5$ is in  class $\mathcal{C}_6$.
 \end{enumerate}
 \end{lemma}

\subsection{Small dimensional symplectic groups}\label{sec:smallsymp}
We start our analysis with small dimensional symplectic groups $\mathrm{PSp}_n(q)$ with $4\le n\le 8$. Observe that $\mathrm{Sp}_4(2)\cong S_6$ is not simple and that $\mathrm{PSp}_4(3)\cong\mathrm{PSU}_4(2)$.
 For the subgroup structure of $\mathrm{Sp}_4(q)$, $\mathrm{Sp}_6(q)$ and $\mathrm{Sp}_8(q)$ we use~\cite{bhr}.

\begin{lemma}\label{dimension4oddsymplectic}
Let $q$ be odd. The weak normal covering number of $\mathrm{PSp}_4(q)$  is at least $2$. Moreover, if $H$ and $K$ are maximal components of a weak normal $2$-covering of $\mathrm{PSp}_4(q)$, then we have $q=3$ and up to  $\mathrm{Aut(PSp}_4(q))$-conjugacy one of the following holds
\begin{enumerate}
\item $\tilde H\cong E_3^{1+2}:(2\times\mathrm{Sp_2(3)})\in\mathcal{C}_1$ and $\tilde K\cong \mathrm{Sp}_2(9):2\in\mathcal{C}_3$,
\item $\tilde H\cong E_3^{1+2}:(2\times\mathrm{Sp_2(3)})\in\mathcal{C}_1$ and $\tilde K\cong 2_{-}^{1+4}.A_5\in \mathcal{C}_6$.
\end{enumerate}

Each of the two weak normal $2$-coverings  of $\mathrm{PSp}_4(3)$ in $(1)$ and $(2)$ gives rise to a single normal $2$-covering.
\end{lemma}
\begin{proof}
It follows with the help of the computer algebra system \texttt{magma}~\cite{magma} that $\mathrm{Sp}_4(3)$ does admit weak normal $2$-coverings and the only weak normal $2$-coverings are reported above. These are also normal coverings. Therefore, for the rest of the proof, we suppose $q\ne 3$.

Let $\mu$ be a weak normal $2$-covering  of $\mathrm{PSp}_4(q)$ with maximal components. Let $\tilde H$ be a component of $\tilde \mu$
containing a Singer cycle of $\mathrm{Sp}_4(q)$. From Lemma~\ref{malles}, we have $$\tilde H\cong \mathrm{Sp}_2(q^2):2.$$

Write $q=p^f$, with $p$ an odd prime number and $f\geq 1$. Using the fact  that $q\geq 5$, a case-by-case analysis on the maximal subgroups of $\mathrm{Sp}_4(q)$ in~\cite[Tables~8.12,~8.13]{bhr} reveals that there are four $\mathrm{Sp}_4(q)$-conjugacy classes of maximal subgroups containing elements of order $p(q+1)$. Namely,
\begin{itemize}
\item[i)] $E_q^{1+2}:((q-1)\times \mathrm{Sp}_2(q))\in\mathcal{C}_1$,
\item[ii)]  $E_q^{3}:\mathrm{GL}_2(q)\in\mathcal{C}_1$,
\item[iii)]  $\mathrm{Sp}_2(q)^2:2\in \mathcal{C}_2$,
\item[iv)]  $\mathrm{GU}_2(q).2\in\mathcal{C}_3$.
\end{itemize}
As $\tilde H$ is in not in this list, the weak normal covering number of $\mathrm{Sp}_4(q)$ is at least $2$. Assume now, by contradiction, that $\gamma_w(\mathrm{PSp}_4(q))=2$. Then we have $\tilde\mu=\{\tilde H, \tilde K\}$, with $\tilde K$ containing an element of order $p(q+1)$.
In particular, $\tilde K$ must be in one of the above four possibilities i)-iv).

We now consider elements of order $p(q-1)$ in $\mathrm{Sp}_4(q)$. Recall that $\tilde H\cong\mathrm{Sp}_2(q^2):2$ and that $\mathrm{Sp}_2(q^2)\cong\mathrm{SL}_2(q^2)$. A non-identity unipotent element $u$ of $\mathrm{SL}_2(q^2)$ has centralizer of order $2q^2$, given by the product of the Sylow $p$-subgroup of $\mathrm{SL}_2(q^2)$ containing $u$ and of $Z(\mathrm{SL}_2(q^2))\cong C_2$.

From this we easily deduce that $\tilde H$ contains no elements of order $p(q-1)$.  Indeed assume that there exist $y\in \tilde H$ with $o(y)=p(q-1)$. Then $u:=y^{q-1}$ is a non-identity unipotent element of order $p$. Since $p$ is odd, we necessarily have that $u\in \mathrm{SL}_2(q^2)$. By the above argument, we get $|\cent{\mathrm{SL}_2(q^2)}{u}|=2q^2$ and thus $|\cent{\mathrm{\tilde H}}{u}|\mid 4q^2$. On the other hand, we have $\langle y\rangle\subseteq \cent{\mathrm{\tilde H}}{u}$ and thus $p(q-1)\mid 4q^2$. It follows that $q-1\mid 4$ and so $q\leq 5$. On the other hand it is directly checked that
 $\mathrm{SL}_2(25):2$ does not contain an element of order $5\cdot 4=20$.

Therefore, elements of order $p(q-1)$ all have an $\mathrm{Aut}(\mathrm{Sp}_4(q))$-conjugate in $\tilde K$.
We show now that $\mathrm{GU}_2(q).2$ admits no element of order $p(q-1)$. Let $\mathrm{GU}_2(q)$ be defined by the Hermitian form having matrix
\[
J:=\begin{pmatrix}
0&1\\
1&0
\end{pmatrix}
\]
and note that a non-identity unipotent element of order $p$ of $\mathrm{GU}_2(q)$ is conjugate to $u=\begin{pmatrix}
1&a\\
0&1
\end{pmatrix}$, for some $a\in  \mathbb{F}_{q^2}^*$ with $a^q+a=0.$
By an elementary computation one checks that, if $ \mathbb{F}_{q^2}^*=\langle t\rangle$, then
\begin{equation*}
\cent{\mathrm{GU}_2(q)}{u}\leq \left\{\begin{pmatrix}
\alpha&\beta\\
0&\alpha
\end{pmatrix}\ \mid\ \alpha\in \langle t^{q-1}\rangle, \ \beta\in  \mathbb{F}_{q^2}\right \}\leq \mathrm{GL}_2(q^2),
\end{equation*}
and thus
\begin{equation}\label{conto}
|\cent{\mathrm{GU}_2(q)}{u}|\mid (q+1)q^2.
\end{equation}
Assume now, by contradiction, that there exists $x\in \mathrm{GU}_2(q).2$ with $\order x=p(q-1)$. Then $x^2\in  \mathrm{GU}_2(q)$ has order $p(q-1)/2$ and thus there exists a non-identity unipotent element $u$ of order $p$ whose centralizer in $\mathrm{GU}_2(q)$ has size divisible by $p(q-1)/2$. By \eqref{conto}, we deduce $p(q-1)/2\mid (q+1)q^2$ and therefore $(q-1)/2\mid (q+1)$. It follows that $(q-1)/2\mid \gcd(q-1, q+1)=2$ and thus $q=5$. On the other hand it is directly checked that $\mathrm{GU}_2(5).2$ does not contain an element of order $5\cdot 4=20$.

Since $\mathrm{GU}_2(q).2$ has no elements of order $p(q-1)$, $\tilde K$ cannot be isomorphic to $\mathrm{GU}_2(q).2$ and hence we exclude this case from any further analysis.

Now, $\mathrm{Sp}_4(q)$ admits two types of elements having order $p(q+1)$: elements admitting eigenvalues in the ground field $\mathbb{F}_q$ and elements admitting no eigenvalues in the ground field $\mathbb{F}_q$.
This can be more easily seen by considering two symplectic forms. Using the symplectic form having matrix
\[
\begin{pmatrix}
0&1&0&0\\
-1&0&0&0\\
0&0&0&1\\
0&0&-1&0
\end{pmatrix},
\]
we can check with an easy computation that the matrices
\[
\begin{pmatrix}
1&1&0\\
0&1&0\\
0&0&s
\end{pmatrix},
\]
where $s\in \mathrm{SL}_2(q)$ is a Singer cycle, are symplectic of order $p(q+1)$ and have eigenvalues in the ground field $\mathbb{F}_q$. Whereas, using the symplectic form having matrix
\[
\begin{pmatrix}
0&0&1&0\\
0&0&0&1\\
-1&0&0&0\\
0&-1&0&0
\end{pmatrix},
\]
we can check with a computation that the $2\times 2$ block matrices
\[
\begin{pmatrix}
s&s\\
0&s
\end{pmatrix},
\]
where $s\in \mathrm{O}^-_2(q)$ is a Singer cycle are symplectic of order $p(q+1)$  and have no eigenvalues in the ground field $\mathbb{F}_q$.

However, the elements having order $p(q+1)$ in $\mathrm{Sp}_2(q)^2:2$ and in $E_{q}^{1+2}:((q-1)\times \mathrm{Sp}_2(q))$ are only of the first type, that is, admit eigenvalues in $\mathbb{F}_q$. On the other hand, the elements having order $p(q+1)$ in
$E_{q}^{3}:\mathrm{GL}_2(q)$ are only of the second type, that is, admit  no eigenvalues in $\mathbb{F}_q$. Since having eigenvalues in $\mathbb{F}_q$ is a property preserved by $\mathrm{Aut}(\mathrm{Sp}_4(q))$-conjugacy, we deduce that whatever one chooses for $\tilde K$, in the possibilities i),~ii),~iii) and~iv) above,
the $\mathrm{Aut}(\mathrm{Sp}_4(q))$-conjugates of $\tilde K$ contain only one type of elements of order $p(q+1)$, and the contradiction is found.
\end{proof}

We now deal with even characteristic $4$-dimensional symplectic groups. Here the result is rather different from Lemma~\ref{dimension4oddsymplectic} because of the peculiar subgroup lattice arising in  characteristic $2$. Recall that $\mathrm{PSp}_4(q)=\mathrm{Sp}_4(q)$, when $q$ is even.
We start with a technical lemma which follows from the work of Enomoto~\cite{enomoto} on the character table of $\mathrm{Sp}_4(q)$. However, to explain how the tables in~\cite{enomoto} can be used to deduce this lemma would take us too far astray. Therefore, here, we give a direct proof.

\begin{lemma}\label{technical}
Let $f$ be a positive integer and let $q:=2^f$. The following holds
\begin{enumerate}
\item\label{eq:technical1} $\mathrm{Sp}_4(q)$ contains a family $\mathcal{D}\neq \varnothing$ of elements $h$ having order $q+1$ and with $|\cent {\mathrm{Sp}_4(q)}h|=(q+1)^2$ if and only if $f>1$,
\item\label{eq:technical2} for every $h\in \mathcal{D}$, $|\nor{\mathrm{Sp}_4(q)}{\langle h\rangle}:\cent {\mathrm{Sp}_4(q)}h|\in \{2,4\}$,
\item\label{eq:technical3}when $f>2$, $\mathrm{Sp}_4(q)$ contains elements $h_1$ and $h_2$ belonging to $\mathcal{D}$ with  $|\nor{\mathrm{Sp}_4(q)}{\langle h_1\rangle}:\cent {\mathrm{Sp}_4(q)}{h_1}|=2$ and
 $|\nor{\mathrm{Sp}_4(q)}{\langle h_2\rangle}:\cent {\mathrm{Sp}_4(q)}{h_2}|=4$,
 \item\label{eq:technical4} when $f$ is even, each element $h\in \mathcal{D}$  which is  $\mathrm{Aut(Sp}_4(q))$-conjugate to an element of the maximal subgroup $\mathrm{Sp}_4(\sqrt{q})\in\mathcal{C}_5$ satisfies $|\nor{\mathrm{Sp}_4(q)}{\langle h\rangle}:\cent {\mathrm{Sp}_4(q)}{h}|=4 $.
 \end{enumerate}
\end{lemma}
\begin{proof}
We let the symplectic form of $\mathrm{Sp}_4(q)$ be defined by the matrix
\[
\begin{pmatrix}
0&1&0&0\\
1&0&0&0\\
0&0&0&1\\
0&0&1&0
\end{pmatrix}.
\]
 We first claim that
\begin{itemize}
\item[$(\dag)$]The centralizer of every element $h\in \mathrm{Sp}_4(q)$ of order $q+1$ has size divisible by $(q+1)^2$.
\item[$(\dag\dag )$]If the size of the centralizer of $h$ is exactly $(q+1)^2$, then $h$ is conjugate  to a matrix of the form
\begin{equation}\label{eq:matrix}
h_{\ell_1,\ell_2}=\begin{pmatrix}
\zeta^{\ell_1}&0\\
0&\zeta^{\ell_2}
\end{pmatrix},
\end{equation}
where $\zeta\in \mathrm{Sp}_2(q)$ is a Singer cycle, $\ell_1,\ell_2\in \{1,\dots,q\}$ and $\order {h_{\ell_1,\ell_2}}=\mathrm{lcm}\{\order{\zeta^{\ell_1}},\order{\zeta^{\ell_2}}\}=q+1$.\end{itemize}

Let $h\in\mathrm{Sp}_4(q)$ be of order $q+1$. Then $h$ is semisimple and hence $V$ decomposes as a direct sum of irreducible $\mathbb{F}_q\langle h\rangle$-modules. Thus we have four cases to consider:
\begin{itemize}
\item $V$ is irreducible,
\item $V=V_1\oplus V_2\oplus V_3\oplus V_4$ with $\dim_{\mathbb{F}_q}V_i=1,$ for $i\in\{1,\dots,4\}$,
\item $V=V_1\oplus V_2\oplus V_3$ with $\dim_{\mathbb{F}_q}V_1=\dim_{\mathbb{F}_q}V_2=1$ and $\dim_{\mathbb{F}_q}V_3=2$,
\item $V=V_1\oplus V_2$ with $\dim_{\mathbb{F}_q}V_1=\dim_{\mathbb{F}_q}V_2=2$.
\end{itemize}
The first case does not arise because in that case we would have $\order h\mid q^2+1$ and hence $q+1=\order h\mid \gcd(q^2+1,q+1)=1$, which is a contradiction. Similarly the
 second case does not arise because in that case we would have $\order h\mid q-1$ and hence $q+1=\order h\mid \gcd(q-1,q+1)=1$, which is a contradiction.  In the third case, by  Section~\ref{Huppert}, we have that $V_3$ is non-degenerate, $V_1\oplus V_2=V_3^\perp$ and $V_1$, $V_2$ are totally isotropic. Therefore the centralizer of $h$  has order $(q+1)|\mathrm{Sp}_2(q)|$, which is a proper multiple of $(q+1)^2$. Note that the factor $q+1$ arises from the centralizer of the Singer cycle induced on $V_3$, and $\mathrm{Sp}_2(q)$ arises from the centralizer  of the matrix induced on $V_1\oplus V_2$. Finally, in the fourth case $h$ is $\mathrm{Sp}_4(q)$-conjugate to a $2\times 2$-block matrix of the form
\[
\begin{pmatrix}
\xi&0\\
0&\mu
\end{pmatrix},
\]
where $\xi,\mu\in \mathrm{Sp}_2(q)$ have irreducible characteristic polynomials and $q+1=\mathrm{lcm}\{\order \xi,\order \mu\}$. Thus $\xi$ and $\mu$  are powers of suitable Singer cycles. Let $\zeta\in \mathrm{Sp}_2(q)$ be a fixed Singer cycle in  $\mathrm{Sp}_2(q)$.
Since all Singer cycles in $\mathrm{Sp}_2(q)$ are $\mathrm{Sp}_2(q)$-conjugate to some Singer cycle in $\langle \zeta\rangle$,  we have that
$h$ is $\mathrm{Sp}_4(q)$-conjugate to a matrix of the form \eqref{eq:matrix}.

Clearly,
\begin{equation}\label{boundmatrix}\cent {\mathrm{Sp}_4(q)}{h_{\ell_1,\ell_2}}\ge
\left\{\begin{pmatrix}
\zeta^i&0\\
0&\zeta^j
\end{pmatrix}\mid  i,j\in \{1,\dots,q+1\}
\right\}
\end{equation}
and the group on the right hand side has order $(q+1)^2$. This proves our claims~$(\dag)$ and $(\dag \dag)$.

\smallskip

Let now, for simplicity, be $h_\ell:=h_{1,\ell}$ and consider
$$\mathcal{E}:=\left\{h_\ell=\begin{pmatrix}
\zeta&0\\
0&\zeta^\ell
\end{pmatrix}
 \in \mathrm{Sp}_4(q) \mid\  \ell\in \{1,\dots,q\},\  \gcd(\ell,q+1)=1\right\}.$$

 We claim that a matrix $h_\ell\in \mathcal{E}$ has centralizer of size $(q+1)^2$ if and only if $\ell\notin\{1,q\}$. In order to show that, we identify the natural module $\mathbb{F}_q^2$ for  $\mathrm{Sp}_2(q)$ with $\mathbb{F}_{q^2}$ and take $\zeta$  to be the matrix representing  the multiplication by the $(q-1)^{\mathrm{th}}$ power of a generator of $\mathbb{F}_{q^2}^*$, again denoted by $\zeta$. When $\ell\in \{1,q\}$, we can find $A,B\in\mathrm{Sp}_2(q)$ with $\zeta^\ell=A^{-1}\zeta A$ and $\zeta^\ell=B\zeta B^{-1}$.
Indeed, when $\ell=1$ this is obvious because we may take $A=B=I$, and when $\ell=q$ we may take the matrix $B$ induced by the action on $\mathbb{F}_{q^2}$ of the Galois group of $\mathbb{F}_{q^2}/\mathbb{F}_q$ and $A=B^{-1}$. Then, the matrix
$$\begin{pmatrix}0&A\\B&0\end{pmatrix}$$
lies in $\cent {\mathrm{Sp}_4(q)}{h_\ell}$ and hence $|\cent {\mathrm{Sp}_4(q)}{h_\ell}|>(q+1)^2$.

Conversely, let $\ell\in\{1,\ldots,q\}\setminus \{1,q\}$ with $ \gcd(\ell,q+1)=1$. To prove that the equality in~\eqref{boundmatrix} is attained it suffices to show that $\langle
 e_1,e_2\rangle$ and $\langle e_3,e_4\rangle$ are the only proper $\mathbb{F}_q\langle h_\ell\rangle$-submodules of $V=\langle e_1,e_2,e_3,e_4\rangle$.  Suppose that $W$ is a proper $\mathbb{F}_q\langle h_\ell\rangle$-submodule different from $\langle
 e_1,e_2\rangle$ and $\langle e_3,e_4\rangle$. Observe that $W$ is a $2$-dimensional irreducible $\mathbb{F}_q\langle h_\ell\rangle$-submodule because $h_\ell$ has no eigenvalues in $\mathbb{F}_q$.
 Using the decomposition $V=\langle e_1,e_2\rangle\oplus \langle e_3,e_4\rangle$, for the computations that follow, we identify the relevant matrices of $\mathrm{Sp}_4(q)$ as $2\times 2$ block matrices and the vectors of $\mathbb{F}_q^4$ as block vectors $(x,\,y)\in (\mathbb{F}_{q}^2)^2\cong (\mathbb{F}_{q^2})^2$. Let $w\in W\setminus\{0\}$. By the irreducibility of $\langle e_1,e_2\rangle$ and of $\langle e_3,e_4\rangle$, we have $W\cap \langle
 e_1,e_2\rangle=0$ and $W\cap \langle e_3,e_4\rangle=0$. Thus we can write
 $$w=(x,\,y),$$
  with $x\in\langle e_1,e_2\rangle\setminus\{0\}$ and $y\in \langle e_3,e_4\rangle\setminus\{0\}$. We identify then $x$ and $y$ with elements in $\mathbb{F}_{q^2}^*$.

  Since $W$ has dimension $2$ and is irreducible, we have
 $$W=\mathrm{Span}_{\mathbb{F}_q}\{ w,h_\ell w\}=
 \mathrm{Span}_{\mathbb{F}_q}\left\{
(x,\,y),
 (\zeta x,\, \zeta^\ell y)\right\}.$$

Thus $h_\ell^2w\in \langle w,h_\ell w\rangle$ and hence, there exist $a,b\in\mathbb{F}_q$ with $h_\ell^2 w=aw+b h_\ell w$. We obtain the equalities
\begin{align*}
\zeta^2 x&=ax+b\zeta x,\\
\zeta^{2\ell} y&=ay+b\zeta^\ell y
\end{align*}
in the field $\mathbb{F}_{q^2}$.
By canceling $x$ and $y$, we deduce $\zeta^2-b\zeta-a=0$ and $\zeta^{2\ell}-b\zeta^\ell-a=0$. Therefore, $\zeta$ and $\zeta^\ell$ are roots of the polynomial $X^2-bX-a\in\mathbb{F}_q[x]$. Since the Galois group $\mathrm{Gal}(\mathbb{F}_{q^2}/\mathbb{F}_q)$ is cyclic of order $2$ generated by $\lambda\mapsto \lambda^q$, we deduce that the roots of $X^2-bX-a$ are $\zeta$ and $\zeta^q$. Therefore $\zeta^\ell\in \{\zeta,\zeta^q\}$ and $\ell\in \{1,q\}$, which is a contradiction.

We are now ready to prove part~\eqref{eq:technical1}. Indeed when $f=1$, we have $q=2$ and by  $(\dag \dag)$ every $h\in \mathrm{Sp}_4(2)$ with $\order h=3$ and centralizer of order $9$ is conjugate to a matrix of type $$h_\ell=\begin{pmatrix}
\zeta&0\\
0&\zeta^\ell
\end{pmatrix}$$
for a suitable Singer cycle $\zeta\in \mathrm{Sp}_2(2) $ and with $\ell\in\{1,2\}$. But we have seen above that those matrices have centralizer of size greater than $9$.

Let next  $f>1$. We then have
$$\mathcal{D}:=\left\{h_\ell\in \mathcal{E} \mid\  \ell\in \{1,\dots,q\}\setminus\{1,q\}\right\}\neq \varnothing$$
and, by the above considerations, $h_\ell\in \mathcal{D}$ implies $\order{h_\ell}=q+1$ and $|\cent {\mathrm{Sp}_4(q)}{h_\ell}|=(q+1)^2.$

\smallskip

We are now ready to prove part~\eqref{eq:technical2}. We explore the normalizers of $\langle h_\ell\rangle$ for  $ h_\ell\in \mathcal{D}$.
Inside $\mathrm{Sp}_2(q)$ the normalizer of the Singer cycle $\langle \zeta\rangle$ has order $2(q+1)$, where the ``$2$'' arises (as described above) from a Galois action. In  particular, $\zeta^\eta\in \{\zeta,\zeta^q=\zeta^{-1}\}$, for every $\eta\in \nor {\mathrm{Sp}_2(q)}{\langle\zeta\rangle}$. From this, it follows that
$$|\nor {\mathrm{Sp}_2(q)\perp\mathrm{Sp}_2(q)}{\langle h_\ell \rangle}:\cent {\mathrm{Sp}_4(q)}{h_\ell}|=2.$$
Since the elements in $\nor{\mathrm{Sp}_4(q)}{\langle h_\ell\rangle}$ permute the two subspaces left invariant by $h_\ell$, we deduce
$$|\nor{\mathrm{Sp}_4(q)}{\langle h_\ell\rangle}:\nor {\mathrm{Sp}_2(q)\perp\mathrm{Sp}_2(q)}{h_\ell}|\le 2$$
and hence~\eqref{eq:technical2} follows.

\smallskip

Before proving part~\eqref{eq:technical3} and~\eqref{eq:technical4}. We first claim that
\begin{itemize}
\item[$(\dag\dag\dag)$]
Let $\ell\in \{1,\ldots,q\}\setminus\{1,q\}$. Then $h_{\ell}\in \mathcal{D}$ satisfies $|\nor{\mathrm{Sp}_4(q)}{\langle h_\ell\rangle}:\cent {\mathrm{Sp}_4(q)}{h_\ell}|=4$ if and only if $\ell^2\equiv \pm 1 \ \mathrm{mod }(q+1)$.
\end{itemize}
From above, $|\nor{\mathrm{Sp}_4(q)}{\langle h_\ell\rangle}:\cent {\mathrm{Sp}_4(q)}{h_\ell}|=4$ if and only if there exists a matrix
\[
g:=\begin{pmatrix}
0&X\\
Y&0
\end{pmatrix}\in \mathrm{Sp}_4(q),
\]
for suitable $2\times 2$ matrices in $\mathrm{Sp}_2(q)$.  A computation yields that $Y^{-1}\zeta^\ell Y,\ X^{-1}\zeta X\in \langle \zeta\rangle= \langle \zeta^{\ell}\rangle$ and hence  $Y^{-1}\zeta Y,\ X^{-1}\zeta X\in\{\zeta,\zeta^{-1}\}$. This shows that $g^{-1}h_\ell g$ is one of the following four matrices
\[
\begin{pmatrix}
\zeta^\ell&0\\
0&\zeta
\end{pmatrix},
\begin{pmatrix}
\zeta^\ell&0\\
0&\zeta^{-1}
\end{pmatrix},
\begin{pmatrix}
\zeta^{-\ell}&0\\
0&\zeta
\end{pmatrix},
\begin{pmatrix}
\zeta^{-\ell}&0\\
0&\zeta^{-1}
\end{pmatrix}.
\]
An easy computation shows that at least one of these matrices is in $\langle h_\ell\rangle$ if and only if $\ell^2\equiv \pm 1 \ \mathrm{mod }(q+1).$

\smallskip 

To prove part~\eqref{eq:technical3} observe that when $q=4$ every element $\ell$ in $\{1,\ldots,q\}\setminus\{1,q\}$ with  $\gcd(\ell,q+1)=1$ satisfies $\ell^2\equiv \pm 1 \ \mathrm{mod }(q+1)$; however, when $q>4$, we may always choose  $\ell$ in $ \{1,\ldots,q\}\setminus\{1,q\}$ with $\gcd(\ell,q+1)=1$ and $\ell^2\not \equiv \pm 1 \ \mathrm{mod }(q+1)$.

\smallskip

Finally, we show part~\eqref{eq:technical4}. Let $f$ be even and let $h\in \mathcal{D}$ be $\mathrm{Sp}_4(q)$-conjugate to an element $x$ of the maximal subgroup $\mathrm{Sp}_4(\sqrt{q})\in\mathcal{C}_5$. Then  $\order x=q+1=(\sqrt{q})^2+1$  and, by Proposition \ref{sing-ord}, $x$ is a Singer cycle of $\mathrm{Sp}_4(\sqrt{q})$. Thus  $|\nor {\mathrm{Sp}_4(\sqrt{q})}{\langle x\rangle}:\langle x\rangle|=4$, where the ``$4$'' arises from the Galois action of $\mathrm{Gal}(\mathbb{F}_{(\sqrt{q})^4}/\mathbb{F}_{\sqrt{q}})$. Then
$$4\mid |\nor {\mathrm{Sp}_4(q)}{\langle h\rangle}:\langle h\rangle|=|\nor {\mathrm{Sp}_4(q)}{\langle h\rangle}:\cent {\mathrm{Sp}_4(q)}{h}|(q+1). $$
Since $q$ is even, this implies  $4\mid |\nor {\mathrm{Sp}_4(q)}{\langle h\rangle}:\cent {\mathrm{Sp}_4(q)}{h}|$ and thus, by \eqref{eq:technical2}, we get  $ |\nor {\mathrm{Sp}_4(q)}{\langle h\rangle}:\cent {\mathrm{Sp}_4(q)}{h}|=4$ which proves part~\eqref{eq:technical4}.
\end{proof}

\begin{lemma}\label{dimension4evensymplectic}
Let $f$ be a positive integer with $f\ge 2$ and let $q:=2^f$.  The weak normal covering number of $\mathrm{Sp}_4(q)=\mathrm{PSp}_4(q)$  is $2$.
Moreover, if $H$ and $K$ are maximal components of a weak normal $2$-covering of $\mathrm{Sp}_4(q)$, then up to $\mathrm{Aut}(\mathrm{Sp}_4(q))$-conjugacy one of the following holds
\begin{enumerate}
\item\label{case2} $H\cong\mathrm{SO}_4^-(q)$ and $K\cong\mathrm{SO}_4^+(q)$ are in class $\mathcal{C}_8$. In this case the weak normal covering gives rise, up to $\mathrm{Sp}_4(q)$-conjugacy, to the two normal coverings having components
\begin{itemize}
\item[(a)] \label{casea}$H\cong\mathrm{SO}_4^-(q)$ and $K\cong\mathrm{SO}_4^+(q);$
\item[(b)] \label{caseb}$H\cong\mathrm{Sp}_2(q^2):2\in\mathcal{C}_3$ and $K\cong\mathrm{Sp}_2(q)\mathrm{wr} \,2\in\mathcal{C}_2.$
\end{itemize}
\item\label{casesp1} $q=4$, $H\cong\mathrm{Sp}_2(16):2\in\mathcal{C}_3$ and $K\cong\mathrm{Sp}_4(2)\in\mathcal{C}_5$. In this case the weak normal covering does not give rise to normal coverings.
\end{enumerate}
 In particular, $\gamma(\mathrm{Sp}_4(q))=2.$
\end{lemma}
\begin{proof} When $q=4$, the proof follows with a computer computation. Therefore, for the rest of the proof we assume that $q>4.$

Let $H$ be a maximal component of a weak normal covering $\mu$ of $\mathrm{Sp}_4(q)$ of minimum size, with $H$ containing a Singer cycle. From Lemma~\ref{malles}, up to $\mathrm{Sp}_4(q)$-conjugacy, we have
 $$H\cong\mathrm{SO}_4^-(q)\in\mathcal{C}_8 \,\,\mathrm{or}\,\,H\cong\mathrm{Sp}_2(q^2):2\in\mathcal{C}_3.$$
Observe that a graph automorphism $\gamma$ of $\mathrm{Sp}_4(q)$ fuses  $\mathrm{SO}_4^-(q)$ and $\mathrm{Sp}_2(q^2):2$, see~\cite[Table~8.14]{bhr}.
Therefore, replacing $H$ with a suitable $\mathrm{Aut}(\mathrm{Sp}_4(q))$-conjugate, we may suppose that $H\cong\mathrm{SO}_4^-(q)$.
From Lemma~\ref{technical}, the group $\mathrm{Sp}_4(q)$ contains an element $y$ having order $q+1$ and with $|\cent {\mathrm{Sp}_4(q)}{y}|=(q+1)^2$. Now  assume, by contradiction, that $H\cong\mathrm{SO}_4^-(q)\cong \mathrm{SL}_2(q^2):2$ contains an element $y$  with $\order y=q+1$ and  $|\cent {\mathrm{Sp}_4(q)}{y}|=(q+1)^2$. Since $q^2-1\mid |\cent {\mathrm{SL}_2(q^2)}{y}|$, we then have $q^2-1\mid (q+1)^2$ and thus $q-1\mid q+1$, against $\gcd(q-1,q+1)=1$. Hence we deduce that $\mu$ contains a second maximal component $K$, with $y\in K.$ A case-by-case analysis on the maximal subgroups of
$\mathrm{Sp}_4(q)$ in~\cite[Table~$8.14$]{bhr} reveals that, up to $\mathrm{Sp}_4(q)$-conjugacy, there are three possibilities for $K$. Namely
\begin{itemize}
\item[i)] $K\cong\mathrm{Sp}_2(q)\mathrm{wr }\, 2\in\mathcal{C}_2$,
\item[ii)]  $K\cong\mathrm{SO}_4^+(q)\in\mathcal{C}_8$,
\item[iii)]  $f$ is even and $K\cong\mathrm{Sp}_4(\sqrt{q})\in \mathcal{C}_5$.
\end{itemize}
Since the graph automorphism $\gamma$ of $\mathrm{Sp}_4(q)$ fuses $\mathrm{SO}_4^+(q)$ and $\mathrm{Sp}_2(q)\mathrm{wr }\, 2$, up to  $\mathrm{Aut}(\mathrm{Sp}_4(q))$-conjugacy there are only two possibilities for $K$.

Consider first the case $f$ even and $K\cong\mathrm{Sp}_4(\sqrt{q})\in\mathcal{C}_5$. From Lemma~\ref{technical}, we see that, when $q>4$, not every element of order $q+1$ and centralizer of cardinality $(q+1)^2$ is $\mathrm{Aut}(\mathrm{Sp}_4(q))$-conjugate to an element of $K$. Thus $H,K$ are not components of a weak normal $2$-covering. Thus we exclude from any further consideration the choice $K\cong\mathrm{Sp}_4(\sqrt{q})\in\mathcal{C}_5$.

The choice $K\cong\mathrm{SO}_4^+(q)$ leads, by~\cite{Dye}, to a weak normal covering which is also a normal covering. Moreover, since $\gamma$ simultaneously fuses  $\mathrm{SO}_4^-(q)$ and $\mathrm{Sp}_2(q^2):2$ as well as $\mathrm{SO}_4^+(q)$ and $\mathrm{Sp}_2(q)\mathrm{wr }\, 2$, the weak normal covering in \eqref{case2} gives rise to at least  the two distinct normal coverings (a) and (b).

At this point, it suffices to show that the subgroups $\mathrm{SO}_4^-(q),\mathrm{Sp}_2(q)\mathrm{wr}2$  are not the components of a normal $2$-covering. Indeed, once this is established, the same conclusion follows for the subgroups
$\mathrm{SO}_4^+(q), \mathrm{Sp}_2(q^2):2$.
Now let $s$ be a Singer cycle of $\mathrm{GL}_2(q)$ and consider the matrix
\[
x:=\begin{pmatrix}
s&0\\
0&(s^{-1})^T
\end{pmatrix}.
\]
The matrix $x$ lies in $\mathrm{Sp}_4(q)$ for a suitable non-degenerate symplectic form, has order $q^2-1$ and action $2\oplus 2$. In particular, $x$ has no eigenvalues in the ground field $\mathbb{F}_q$. Assume, by contradiction, that  $x$ is $\mathrm{Sp}_4(q)$-conjugate to an element $y\in \mathrm{Sp}_2(q)\mathrm{wr}2$. Since $q$ is even, then $y$ lies in the base group $\mathrm{Sp}_2(q)\perp\mathrm{Sp}_2(q)$ so that $y$ has the form $s_1\perp s_2$, where $s_1,s_2\in\mathrm{Sp}_2(q)$
with $\langle s_i\rangle$  irreducible cyclic subgroups of $\mathrm{Sp}_2(q)$. Thus $\order {s_i}\mid q+1$ for $i\in\{1,2\}$ and hence also $q^2-1=\order y\mid q+1$, against $q>4.$

Assume next, by contradiction, that  $x$ is $\mathrm{Sp}_4(q)$-conjugate to an element $y\in \mathrm{SO}_4^-(q)$. Using  the structure of the semisimple elements in $\mathrm{SO}_4^-(q)$, we see that $\order y= q^2-1$ implies that $y$ has eigenvalues in the ground field $\mathbb{F}_q$, a contradiction.
\end{proof}

We now examine the group $\mathrm{PSp}_6(q)$.

\begin{lemma}\label{dimension6symplectic}
The weak normal covering number of $\mathrm{PSp}_6(q)$  is at least $2$. Moreover, if $H$ and $K$ are maximal components of a weak normal $2$-covering of $\mathrm{PSp}_6(q)$, then up to $\mathrm{Aut}(\mathrm{PSp}_6(q))$-conjugacy one of the following holds
\begin{enumerate}
\item\label{case11}$q$ is even, $ H\cong \mathrm{SO}_6^-(q)$ and $ K\cong\mathrm{SO}_6^+(q)$ are in class $\mathcal{C}_8$,
\item\label{case22} $q=3^f$ for some $f\geq 1$, $\tilde H\cong\mathrm{Sp}_2(q)\perp \mathrm{Sp}_4(q)\in\mathcal{C}_1$ and $\tilde K\cong\mathrm{Sp}_2(q^3):3\in\mathcal{C}_3$.
\end{enumerate}
The coverings in \eqref{case11} and \eqref{case22}  give both rise to a single normal covering.
\end{lemma}
\begin{proof}
When $q\in\{2,3,4\}$,  the result follows with a computation with the computer algebra system \texttt{magma}~\cite{magma}. Therefore, for the rest of the argument, we suppose $q\ge 5$.

Set $\kappa:=\gamma_w(\mathrm{PSp}_6(q))$ and let $H$ be a maximal component of a weak normal $\kappa$-covering $\mu$ of $\mathrm{PSp}_6(q)$, with $H$ containing a Singer cycle $x$. From Lemma~\ref{malles}, up to $\mathrm{Sp}_6(q)$-conjugacy, one of the following holds
\begin{enumerate}
\item[i)] $\tilde H\cong \mathrm{Sp}_2(q^3):3\in\mathcal{C}_3$,
\item[ii)]  $\tilde H\cong\mathrm{GU}_3(q).2\in\mathcal{C}_3$, with $q$ odd,
\item[iii)]  $\tilde H\cong \mathrm{SO}_6^-(q)\in\mathcal{C}_8$, with $q$ even.
\end{enumerate}

From~\cite[Tables~8.28,~8.29]{bhr}, we see that there exist two $\mathrm{Sp}_6(q)$-conjugacy classes of maximal subgroups  containing an element having order $(q+1)(q^2+1)/\gcd(q+1,q^2+1)=(q+1)(q^2+1)/ \gcd(2,q-1)$ and these are as follows
\begin{itemize}
\item[a)] $\mathrm{Sp}_2(q)\perp \mathrm{Sp}_4(q)\in\mathcal{C}_1$,
\item[b)] $\mathrm{SO}_6^+(q)\in\mathcal{C}_8$, with $q$ even.
\end{itemize}
This already shows that $\kappa\ge 2$, because none of the possibilities for $\tilde H$  is in this list.
Assume now that there exists a weak normal $2$-covering of $\mathrm{PSp}_6(q)$ with maximal components. By the above arguments, the corresponding  weak normal $2$-covering of $\mathrm{Sp}_6(q)$ is given by $\tilde \mu=\{\tilde H,\tilde K\}$, where $\tilde H$ must be in one of the three possibilities i)-iii) and $\tilde K$  must be in one of the two possibilities a)-b).

Let $z$ be an element of $\mathrm{Sp}_6(q)$ having order $q^3-1$.
 From~\cite[Tables~8.28,~8.29]{bhr} and from the fact that $q\ge 5$, we see that there exist four $\mathrm{Sp}_6(q)$-conjugacy classes of maximal subgroups containing an element $z$ having order $q^3-1$ and these are as follows
\begin{itemize}
\item[$\alpha$)] $E_q^6:\mathrm{GL}_3(q)\in\mathcal{C}_1$,
\item[$\beta$)] $\mathrm{Sp}_2(q^3):3\in\mathcal{C}_3$,
\item[$\gamma$)] $\mathrm{GL}_3(q).2\in\mathcal{C}_2$, with $q$ odd,
\item[$\delta$)] $\mathrm{SO}_6^+(q)\in\mathcal{C}_8$, with $q$ even.
\end{itemize}
Assume first that $\tilde K$ is of type
$$\mathrm{Sp}_2(q)\perp \mathrm{Sp}_4(q).$$
Now, $\tilde K$ contains no elements of order $q^3+1$ or $q^3-1$ and hence $x$ and $z$ are $\mathrm{Aut}(\mathrm{Sp}_6(q))$-conjugate to elements in $\tilde H$. By consulting the possibilities in the lists i)-iii) and $\alpha)$-$\gamma$), we deduce that $\mathrm{Sp}_2(q^3):3$ is the only possible group in both lists and hence  $\tilde H$ is of type $$\mathrm{Sp}_2(q^3):3.$$  Write $q=p^f$, where $p$ is a prime number and $f\geq 1$. Suppose that $p\ne 3$. Observe that the non-identity unipotent elements of $\tilde K$ centralize a subspace of dimension at least $2$ of $V$, whereas the non-identity unipotent elements of $\tilde H$ are contained in $\mathrm{Sp}_2(q^3)$ because $p\ne 3$ and hence centralize a subspace of dimension at least $3$ of $V$. As $\mathrm{Sp}_6(q)$ has unipotent elements $u$ with $\dim_{\mathbb{F}_q}V_1(u)=1$, as can be easily verified considering the possible Jordan forms of a $6\times 6$ symplectic matrix over $\mathbb{F}_q$, $u$ is not $\mathrm{Aut}(\mathrm{Sp}_6(q))$-conjugate to an element of $\tilde H$ or of $\tilde K$. Hence $\tilde H,\tilde K$  are not components of a normal $2$-covering. This contradiction says that  we have $$p=3,\  \tilde{H}\cong\mathrm{Sp}_2(q^3):3\ \mathrm{and}\ \tilde K\cong\mathrm{Sp}_2(q)\perp \mathrm{Sp}_4(q),\,\,\mathrm{or \,}q\mathrm{\  is\  even\  and }\  K\cong\mathrm{SO}_6^+(q).$$

When $q$ is a power of $3$, we obtain the weak normal covering in~\eqref{case22}. This is indeed a genuine normal $2$-covering of $\mathrm{Sp}_6(q)$, see~\cite{blw}.

Assume now $q$ even and $K\cong\mathrm{SO}_6^+(q)$. Now, $K$ contains no elements of order $q^3+1$ and hence $x$ is $\mathrm{Aut}(\mathrm{Sp}_6(q))$-conjugate to an element in $H$. By consulting the possibilities above, we deduce that $H$ is of type $\mathrm{Sp}_2(q^3):3$ or $\mathrm{SO}_6^-(q)$. In the second case, we get the normal covering in~\eqref{case11}, by ~\cite{Dye}. Assume now that $H$ is of type $\mathrm{Sp}_2(q^3):3$. Let $w\in\mathrm{Sp}_6(q)$ having order $(q-1)(q^2+1)$ and  action type $1\oplus 1\oplus 4$.
Clearly, $w$ cannot be $\mathrm{Aut}(\mathrm{Sp}_6(q))$-conjugate to an element of $K\cong\mathrm{Sp}_2(q^3):3$. Moreover, consulting the semisimple elements in the orthogonal group $K\cong\mathrm{SO}_6^+(q)$ (see~\cite{buturlakin} or~\cite{GLS3}), we see that $w$ is also not $\mathrm{Aut}(\mathrm{Sp}_6(q))$-conjugate to an element of $K$. Thus the pairing $\mathrm{SO}_6^+(q),\mathrm{Sp}_2(q^3):3$ does not give rise to a weak normal $2$-covering.
\end{proof}

Finally, we deal with $8$-dimensional symplectic groups.

\begin{lemma}\label{dimension8symplectic}
The weak normal covering number of $\mathrm{PSp}_8(q)$  is at least $2$. Moreover, if $H$ and $K$ are maximal components of a weak normal $2$-covering of $\mathrm{PSp}_8(q)$, then $q$ is even and up to $\mathrm{Aut}(\mathrm{Sp}_8(q))$-conjugacy $H\cong\mathrm{SO}_8^+(q)$ and $K\cong\mathrm{SO}_8^-(q)$ are in  class $\mathcal{C}_8$. This covering gives rise to a single normal covering.
\end{lemma}
\begin{proof}
When $q\in \{2,3\}$,  the result follows with a computation with the computer algebra system \texttt{magma}~\cite{magma}. Therefore, for the rest of the argument, we suppose $q\ge 4$.

Set $\kappa:=\gamma_w(\mathrm{PSp}_8(q))$ and let $H$ be a maximal component of a weak normal $\kappa$-covering $\mu$, with $\tilde H$ containing a Singer cycle $x$. From Lemma~\ref{malles}, up to $\mathrm{Sp}_8(q)$-conjugacy, $\tilde H$ is one of the following
\begin{itemize}
\item[i)] $\mathrm{Sp}_4(q^2):2\in\mathcal{C}_3$,
\item[ii)]  $\mathrm{SO}_8^-(q)\in\mathcal{C}_8$, with $q$ even.
\end{itemize}
Observe that, from~\cite[Tables~8.48,~8.49]{bhr}, we infer that in both possibilities\footnote{Note that in~\cite{bhr}, the normalizer is denoted with ``Stab''.}
\begin{equation}\label{eq:jolly}
\mathrm{Aut}(\tilde{G})={\bf N}_{\mathrm{Aut}(\tilde{G})}(\tilde{H})\tilde{G}.
\end{equation}
 This already shows that $\kappa\ge 2$.

 From~\cite[Tables~8.48,~8.49]{bhr} and from $q\ge 4$, we see that the $\mathrm{Sp}_8(q)$-conjugacy classes  of maximal subgroups containing an element having order $$\frac{(q+1)(q^3-1)}{\gcd(q+1,q^3-1)}=\frac{(q+1)(q^3-1)}{\gcd(2,q-1)}$$ are as follows
\begin{itemize}
\item[a)] $\mathrm{Sp}_2(q)\perp \mathrm{Sp}_6(q)\in\mathcal{C}_1$,
\item[b)]  $E_q^{6+6}:(\mathrm{GL}_3(q)\times \mathrm{Sp}_2(q))\in\mathcal{C}_1$,
\item[c)]  $\mathrm{SO}_8^-(q)$, with $q$ even, in class $\mathcal{C}_8$.
\end{itemize}
 Let now $\mu$ be a weak normal $2$-covering of $\mathrm{PSp}_8(q)$ with maximal components. By the above arguments, the corresponding  weak normal $2$-covering of $\mathrm{Sp}_8(q)$ is given by $\tilde \mu=\{\tilde H,\tilde K\}$, where $\tilde H$ must be in one of the two possibilities i)-ii) and $\tilde K$  must be in one of the three possibilities a)-c).

From~\cite[Tables~8.48,~8.49]{bhr} and from  $q\ge 4$,  we see that the $\mathrm{Sp}_8(q)$-conjugacy classes of maximal subgroups containing an element having order $q^4-1$ and having type $4\oplus 4$ are as follows
\begin{enumerate}
\item $E_q^{10}:\mathrm{GL}_4(q)\in\mathcal{C}_1$,
\item $\mathrm{Sp}_4(q^2):2\in\mathcal{C}_3$,
\item $\mathrm{GL}_4(q).2$, with $q$ odd, in  class $\mathcal{C}_2$,
\item $\mathrm{GU}_4(q).2$, with $q$ odd, in class $\mathcal{C}_3$,
\item $\mathrm{SO}_8^+(q)$, with $q$ even, in  class $\mathcal{C}_8$.
%\item $E_q^7:((q-1)\times \mathrm{Sp}_6(q))$, with $q$ even, in  class $\mathcal{C}_1$,
%\item $E_q^{3+8}:(\mathrm{GL}_2(q)\times \mathrm{Sp}_4(q))$, with $q$ even, in class $\mathcal{C}_1$,
%\item $\mathrm{Sp}_2(q)\perp \mathrm{Sp}_6(q)$, with $q$ even, in  class $\mathcal{C}_1$.
\end{enumerate}
This is not hard to check, but rather tedious. To exclude the subgroups in the Aschbacher class $\mathcal{S}$ it is more convenient and time efficient to observe that they do not contain elements having order $q^4-1$. Moreover, when $q$ is even, to exclude the subgroups $E_q^7:((q-1)\times \mathrm{Sp}_6(q))$, $E_q^{3+8}:(\mathrm{GL}_2(q)\times\mathrm{Sp}_4(q))$ and $\mathrm{Sp}_2(q)\perp \mathrm{Sp}_6(q)$ in the Aschbacher  class $\mathcal{C}_1$, it suffices to observe that none of the elements having order $q^4-1$ in these groups have type $4\oplus 4$.

Assume first that $$\tilde K\cong E_q^{6+6}:(\mathrm{GL}_3(q)\times \mathrm{Sp}_2(q)).$$ As $\tilde K$ contains no elements of order $q^4+1$ or $q^4-1$, $\tilde H$ contains elements having order $q^4+1$ and $q^4-1$. By consulting the lists of options i)-ii) and (1)-(5), we deduce that $\tilde H$ is of type $\mathrm{Sp}_4(q^2):2$. Now, consider elements $t\in \mathrm{Sp}_8(q)$ with $\order t=(q-1)(q^3+1)/\gcd(q-1,q^3+1)=(q-1)(q^3+1)/\gcd(2,q-1)$. Since $E_q^{6+6}:(\mathrm{GL}_3(q)\times \mathrm{Sp}_2(q))$ and $\mathrm{Sp}_4(q^2):2$ do not contain elements having this order, we deduce that $t$ is contained in no $\mathrm{Aut}(\mathrm{Sp}_8(q))$-conjugate of $\tilde H$ or $\tilde K$. Therefore, a weak normal $2$-covering cannot make use of  $E_q^{6+6}:(\mathrm{GL}_3(q)\times \mathrm{Sp}_2(q))$.

Assume now that $$\tilde K=\mathrm{Sp}_2(q)\perp\mathrm{Sp}_6(q).$$Therefore, $\tilde H$ is of type $\mathrm{Sp}_4(q^2):2$, or $\mathrm{SO}_8^-(q)$ with $q$ even. By considering the unipotent elements of $\mathrm{Sp}_8(q)$, we see that when $q$ is odd $\tilde H$ cannot be $\mathrm{Sp}_4(q^2):2$. Thus $q$ is even. Now, consider the non-degenerate symplectic  form on $\mathbb{F}_q^4$ having matrix
\[
J:=\begin{pmatrix}
0&0&0&1\\
0&0&1&0\\
0&1&0&0\\
1&0&0&0
\end{pmatrix}
\]
and let $\mathrm{Sp}_4(q)$ be the symplectic group with respect to this form. Now, using the matrix $J$, it can be easily verified that
\[
u:=\begin{pmatrix}
1&1&1&0\\
0&1&1&0\\
0&0&1&1\\
0&0&0&1
\end{pmatrix}\in\mathrm{Sp}_4(q)
\]
and that $u$ has order $4$. Moreover, $u$ has a unique $u$-invariant $2$-dimensional subspace (namely, the subspace spanned by the last two basis vectors)
 and that this subspace is totally isotropic with respect to $J$. Now, let $s$ be a Singer cycle of $\mathrm{Sp}_4(q)$ and set $$g:=u\oplus s\in\mathrm{Sp}_8(q).$$ Clearly, $\order g=4(q^2+1)$. Now, the element $g$ has no $\mathrm{Aut}(\mathrm{Sp}_8(q))$-conjugate in $\tilde{K}=\mathrm{Sp}_2(q)\perp \mathrm{Sp}_6(q)$ because $g$ fixes no non-degenerate $2$-subspace of $V$. Suppose that $g$ has an $\mathrm{Aut}(\mathrm{Sp}_8(q))$-conjugate in $\mathrm{Sp}_4(q^2):2$ and recall that $\mathrm{Sp}_4(q^2):2$ is a field extension subgroup. Without loss of generality, we may suppose that $g\in \mathrm{Sp}_4(q^2):2$.
 Now, when we see $g=u\oplus s$ as an element of $\mathrm{Sp}_4(q^2):2$, $s$ becomes a semisimple element fixing a $2$-dimensional non-degenerate subspace over $\mathbb{F}_{q^2}$. From this, we get
$${\bf C}_{\mathrm{Sp}_4(q^2)}(s)\cong \mathrm{Sp}_2(q^2)\perp\langle s\rangle.$$
Observe that the centralizer of $s$ in $\mathrm{Sp}_4(q^2):2$ is contained in $\mathrm{Sp}_4(q^2)$, because the ``:2'' on top acts via Galois conjugation and hence it cannot centralize the matrix $s$ being a Singer cycle of $\mathrm{Sp}_2(q^2)$. Thus
$$u\in{\bf C}_{\mathrm{Sp}_4(q^2):2}(s)= {\bf C}_{\mathrm{Sp}_4(q^2)}(s)\subseteq \mathrm{Sp}_2(q^2)\perp\langle s\rangle$$
and hence $ u\in \mathrm{Sp}_2(q^2)\cong\mathrm{SL}_2(q^2)$. However, non-identity unipotent elements of $\mathrm{SL}_2(q^2)$ have order $2$ and not $4$. This contradiction shows that the element $g$ has no $\mathrm{Aut}(\mathrm{Sp}_8(q))$-conjugate in $\mathrm{Sp}_4(q^2):2$. It remains to consider the case that $\tilde{H}\cong\mathrm{SO}_8^-(q)$. For dealing with this case, we turn to an element $y$ having order $q^4-1$ and type $4\oplus 4$. Since neither $\tilde{H}$ nor $\tilde{K}$ is one of the groups in~(1)-(5), we deduce that $y$ has no $\mathrm{Aut}(\mathrm{Sp}_8(q))$-conjugate in $\tilde{H}$ or in $\tilde{K}$. Hence this choice of $\tilde{H}$ and $\tilde{K}$ does not give rise to a weak normal $2$-covering.
 This whole paragraph has shown that a weak normal $2$-covering cannot make use of  $\mathrm{Sp}_2(q)\perp\mathrm{Sp}_6(q)$.

The only remaining possibility is that

$$\tilde K\cong\mathrm{SO}_8^-(q)\hbox{ and }q \hbox{ even}.$$ Since $\mathrm{Sp}_8(q)$ and $\mathrm{SO}_8^-(q)$ are isospectral we need to argue in more detail. Observe that the group $\tilde{H}$ appears in~(1)-(5) and, since $q$ is even, we may exclude (3) and~(4). Thus $\tilde{H}$ is one of the following groups:
\begin{equation}\label{eq25}\mathrm{Sp}_4(q^2):2,\,E_q^{10}:\mathrm{GL}_4(q) \hbox{ and }\mathrm{SO}_8^+(q).
\end{equation}

Finally we turn to the elements having order $q^3+1$ in $\mathrm{Sp}_8(q)$ and type $2\oplus 6$. The elements having type $2\oplus 6$ are not
conjugate to an element of $\mathrm{SO}_8^-(q)$ because they do not preserve a  non-degenerate quadratic form of ``minus type''. Therefore the elements of order $q^3+1$  and of type $2\oplus 6$ are all $\mathrm{Aut}(\mathrm{Sp}_8(q))$-conjugate to elements of $\tilde H$. Using elementary order considerations, from~\eqref{eq25}, we see that the only possibility for $\tilde H$ is $\mathrm{SO}_8^+(q)$. Now, by  Dye~\cite{Dye}, for $q$ even,  $\tilde{K}\cong\mathrm{SO}_8^-(q)$ and $\tilde H\cong\mathrm{SO}_8^+(q)$ do give rise to a normal covering of $\mathrm{Sp}_8(q)$.
\end{proof}

\subsection{Large dimensional symplectic groups}\label{sec:largesymp}
In this section we deal with large dimensional symplectic groups $\mathrm{Sp}_n(q)$ with $n\ge 10$, $n$ even. We consider the Bertrand number $t$ for the integer $n/2\geq 5$  (see Section \ref{bernum}). Thus for $n\geq 16$ we have $\frac{n}{4}<t\le \frac{n}{2}-3$ and $t$ is prime; for $n=14$, we have $t=5$; for $n=12$, we have $t=4$; for $n=10$ we have $t=3$. Note that, for  every $n\geq 10$, $\frac{n}{4}<t\le \frac{n}{2}-2$ holds. Moreover, $t$ is a prime, except when $n=12$.
Recall also that, for every $n\ge 10$, we have  $t\nmid n/2$.
 Moreover, since $\gcd(n/2,t)=1$ when $n\ne 12$ and $\gcd(n/2,t)=2$ when $n=12$,
 from Lemma~\ref{aritme}~\eqref{eq:arithme2} we get
\begin{align}\label{stupid}
\gcd(q^{t}+1,q^{\frac{n}{2}-t}+1)&=
\begin{cases}
\gcd(2,q-1)&\textrm{if }n=12 \textrm{ or }n/2 \textrm{ odd},\\
q+1&\textrm{if }n\ne 12 \textrm{ and }n/2 \textrm{ even}.\\
\end{cases}
\end{align}
In order to find the components of a  weak normal covering of $\mathrm{Sp}_n(q)$, we consider a Bertrand element $z\in\mathrm{Sp}_n(q)$ such that
\begin{itemize}
\item $\order z=\frac{(q^t
+1)(q^{\frac{n}{2}-t} +1)}{\gcd(q^t+1,q^{\frac{n}{2}-t}+1)}$ (see Table~\ref{2}),
\item the action of $z$ on $V$ is of type $2t\oplus (n-2t)$ and we write $V=U\perp W$, where $\dim_{\mathbb{F}_{q}}U=2t$, $\dim_{\mathbb{F}_{q}}W=n-2t$ and $U,W$ are irreducible $\mathbb{F}_q\langle z\rangle$-modules,
\item $z$ induces a matrix of order $q^t+1$ on $U$ and of order $q^{\frac{n}{2}-t}+1$ on $W$.
\end{itemize}

Note that the existence of Bertrand elements is an immediate consequence of the embedding of $\mathrm{Sp}_{2t}(q) \perp \mathrm{Sp}_{n-2t}(q)$ in $\mathrm{Sp}_{n}(q)$.

%\begin{lemma}\label{symplecticcounts}
%The Bertrand element $z\in \mathrm{Sp}_n(q)$ defined in Table~$\ref{2}$ is a $ppd(n,q;2t)$-element if and only if $(n/2,q)\ne (5,2)$. Moreover, if $(n/2,q)\ne (5,2)$, then $z$ is a strong $ppd(n,q;2t)$-element.

%When $(n/2,q)\ne (5,2)$ and $(n/2-t,q)\ne (3,2)$, for each prime $q_{2t}\in P_{2t}(q)$ and for each prime $q_{n-2t}\in P_{n-2t}(q)$, $q_{2t}q_{n-2t}$ divides $\order z$.
%\end{lemma}
%\begin{proof}
%The element $z$ is a $ppd(n,q;2t)$-element as long as $q^{2t}-1$ admits primitive prime divisors. It readily follows from the definition of Bertrand numbers and from Zsigmondy's theorem that $P_{2t}(q)\ne\varnothing$ as long as $(t,q)\ne (3,2)$, that is, $(n/2,q)=(5,2)$.

%When $n=6$, we have $t=4$ and $\gcd(q^t+1,q^{n/2-t}+1)=\gcd(q^4+1,q^2+1)=\gcd(q-1,2)$, where the
% last equality follows from a computation. Thus $\order z=(q^4+1)(q^2+1)/\gcd(q-1,2)$. From this explicit description
%of $\order z$, the lemma follows.

%Suppose $n\ne 6$ and $(n/2,q)\ne (5,2)$. In particular, $t$ is odd and $\gcd(t,n)=1$.  It follows from Lemma~\ref{aritme}~\eqref{eq:arithme5} that $\gcd(q^t+1,q^{n/2-t}+1)$ divides $\gcd(t,q+1)$. From this it easily follows that $z$ is a strong $ppd(n,q;2t)$-element and that each prime in $P_{n-2t}(q)$ divides $\order z$.
%\end{proof}

\begin{lemma}\label{bertrand-sp} If $M$ is a
maximal subgroup of $\mathrm{Sp}_n(q)$ with $n\ge 10$ containing a Bertrand
 element, then one of the following holds
\begin{enumerate}
\item\label{bertrand-sp1} $M \cong \mathrm{Sp}_{2t}(q) \perp \mathrm{Sp}_{n-2t}(q)\in\mathcal{C}_1$,
\item\label{bertrand-sp2} $n/2$ is  even, $n\neq 12$, $q$ is odd  and  $M \cong \mathrm{GU}_{n/2}(q).2\in\mathcal{C}_3$,
\item\label{bertrand-sp3} $q$ is even and $M \cong \mathrm{SO}^{+}_{n}(q)\in\mathcal{C}_8$,
\item\label{bertrand-sp4} $n=12$  and $M \cong \mathrm{Sp}_{6}(q^2).2\in\mathcal{C}_3$.
\end{enumerate}
\end{lemma}
\begin{proof}
Observe that $P_{2t}(q)\ne \varnothing$ as long as $(t,q)\ne (3,2)$, that is, $(n,q)\ne(10,2)$.
When $(n,q)=(10,2)$, we have verified the veracity of this result with the help of a computer; the maximal subgroup $M$ is isomorphic to either $\mathrm{Sp}_6(2)\perp \mathrm{Sp}_4(2)$ as in~\eqref{bertrand-sp1}, or $\mathrm{SO}_{10}^+(2)$ as in~\eqref{bertrand-sp3}. Therefore, for the rest of the proof, we suppose that $(n,q)\ne (10,2)$. From~\eqref{stupid}, we immediately see that
$$z \textrm{  is a  strong}\   ppd(n,q;2t)\textrm{-element}.$$

Since $2t\leq n-4,$ by Theorem~\ref{main},
$M$ belongs to one of the Aschbacher classes $\mathcal C_i$, with $i\in\{1,2,3,5,8\}$ or to
$\mathcal S$ as described in Example 2.6 a) of~\cite{gpps}. Note that the  class $\mathcal{C}_6$ is excluded using $2t\leq n-4$, see the last assertion in Theorem~\ref{main}.
 We now divide the proof depending on the Aschbacher class of $M$.

\smallskip

\noindent\textsc{The subgroup $M$ lies in class $\mathcal{C}_1$. }Here $M$ is the stabilizer of a proper totally isotropic subspace or of a non-degenerate subspace of $V$. As usual we use the notation in~\cite{kl}. By definition of $z$, the only $z$-invariant subspaces of $V$ are $U$ and $W$. As $U$ and $W$ are non-degenerate and $\dim_{\mathbb{F}_q}U=2t$, we deduce $M\cong\mathrm{Sp}_{2t}(q)\perp\mathrm{Sp}_{n-2t}(q)$.

\smallskip

\noindent\textsc{The subgroup $M$ lies in  class $\mathcal{C}_2$. }From~\cite[Section~$4.2$]{kl}, $M$ is of type $\mathrm{Sp}_{m}(q) \mathrm{wr}\, S_\ell$ for some even divisor $m$ of $n$ with $n=m\ell$ and $\ell\ge 2$, or of type $\mathrm{GL}_{n/2}(q).2$ with $q$ odd. The detailed structure of $M$ is described in~\cite[Propositions~$4.2.5$ and~$4.2.10$]{kl}.

Suppose $M$ is of type $\mathrm{Sp}_{m}(q) \mathrm{wr} S_\ell$. Let $r\in P_{2t}(q)$. Since $r$ divides the order of $M$ and $\gcd(r,q)=1$, we have that either $r$ divides $q^{2i}-1$, for some $i\in \{1,\ldots,m/2\}$, or $r\le \ell$. In the first case, by Lemma \ref{primitivi}\,(3),
we deduce $2t\mid 2i$ and thus
$t\le i\le m/2=n/(2\ell)\le n/4$, contradicting the fact that $t>n/4$. In the second case, as $r$ is a primitive prime divisor of $q^{2t}-1$, from~\eqref{boundppd}, we have $r>2t$. Hence $\ell \ge r>2t>n/2$, contradicting the fact that $\ell\le n/2$.

Suppose $M$ is of type $\mathrm{GL}_{n/2}(q).2$. Now, it is easily seen, arguing as in the previous paragraph, that no element in $P_{2t}(q)$ divides the order of $M$ and hence this case does not arise.

\smallskip

\noindent\textsc{The subgroup $M$ lies in  class $\mathcal{C}_3$. }From~\cite[Section~$4.3$]{kl}, $M$ is of type $\mathrm{Sp}_{n/r}(q^r).r$ for some prime divisor $r$ of $n/2$ or of type $\mathrm{GU}_{n/2}(q).2$ with $q$ odd. The detailed structure of $M$ is described in~\cite[Propositions~$4.3.7$ and~$4.3.10$]{kl}.

Suppose first that $M$ is of type $\mathrm{Sp}_{n/r}(q^{r}).r$. If $n=12$, then $t=4$, $r\in \{2,3\}$ and $$\order z=\frac{(q^4+1)(q^2+1)}{\gcd(2,q-1)}.$$ It is immediately checked that
$\mathrm{Sp}_{6}(q^{2}).2$ contains a conjugate of $z$, while $\mathrm{Sp}_{4}(q^{3}).3$ does not, which produces (4). Assume next that $n\neq 12$. Then $t$ is a prime. Let $s\in P_{2t}(q)$. Since $s$ divides the order of $M$,  we have that $s$ divides $q^{2ri}-1$, for some $i\in \{1,\ldots,n/(2r)\}$, or $s=r$. In the first case,  by Lemma \ref{primitivi}\,(3), we deduce $2t\mid 2ri$ and hence $t\mid ri$. We show that $\gcd(t,r)=1$. Assume the contrary. Then, since both $r$ and $t$ are prime, we have $t=r\mid n/2$, a contradiction. Thus we have $t\mid i$, which gives $t\leq n/4$, a contradiction.
It follows that $s=r$. By ~\eqref{boundppd}, we  then get $r=s>2t>n/2$, contradicting the fact that $r$ divides $n/2$.

Suppose next that $M$ is of type $\mathrm{GU}_{n/2}(q).2$. Now, from~\cite{GLS3} or from~\cite{buturlakin1}, $\mathrm{GU}_{n/2}(q)$ has elements of order $(q^{t} +1)(q^{\frac{n}{2}-t} +1)/\gcd(q^{t}
+1, q^{\frac{n}{2}-t} +1)$  if and only if $n/2-t$ is odd. Recall now that $t$ is an odd prime when $n\ne 12$ and $t=4$ when $n=12$. Thus $n/2-t$ is odd only when $n/2$ is even with $n\neq 12$. Thus part~\eqref{bertrand-sp2} holds.

\smallskip

\noindent\textsc{The subgroup $M$ lies in  class $\mathcal{C}_5$. }This case is ruled out by Lemma ~\ref{no-c5}\,(3), because $z$ is a strong $ppd(n,q;2t)$-element and $n\geq 10$.
\smallskip

\noindent\textsc{The subgroup $M$ lies in  class $\mathcal{C}_8$. }From~\cite[Section~$4.8$]{kl}, $q$ is even and $M$ is of type $\mathrm{SO}_n^\pm(q)$. Now, from~\cite{GLS3} or from~\cite{buturlakin}, $\mathrm{SO}_n^-(q)$ has no element having order $\order z=\frac{(q^{t} +1)(q^{\frac{n}{2}-t} +1)}{\gcd(q^{t}
+1, q^{\frac{n}{2}-t} +1)}$ and action type $2t\oplus (n-2t)$. Thus $M$ is of type $\mathrm{SO}_n^+(q)$ and part~\eqref{bertrand-sp3} holds.
\smallskip

\noindent\textsc{The subgroup $M$ lies in class $\mathcal{S}$. }From Example 2.6 a) of~\cite{gpps}, we obtain $$A_m\leq M \le S_m \times \Z{ \mathrm{Sp}_{n}(q)},$$ with $m\in\{
 n+1,\ n+2\}$.

 Thus the symmetric group $S_m$ contains an element having order
 $$\frac{\order z}{\gcd(2,q-1)}=\frac{(q^t+1)(q^{\frac{n}{2}-t}+1)}{\gcd(q^t+1,q^{\frac{n}{2}-t}+1)\gcd(2,q-1)}.$$

 We first deal with the case $n=12$. Thus $t=4$ and
 $$\order z=\frac{(q^4+1)(q^2+1)}{\gcd(2,q-1)}.$$
 Thus
 $\order z\ge q^6/2$. The maximal element order of $S_{14}$ is $84$ and hence $84\ge q^6/4$. This yields $q=2$. However, when $q=2$, $\order z=85$ and hence also this case does not arise. For the rest of our argument, we may suppose that $n\ne 12$. In particular, $t$ is an odd prime and $t\nmid n/2$.
Then~\eqref{stupid} yields that $\gcd(q^t+1,q^{\frac{n}{2}-t}+1)\le q+1$. Hence
 \begin{align*}
 \order z\ge \frac{q^{n/2}}{q+1}.
 \end{align*}
 Let $o$ be the maximal element order of an element in $S_m$.
 From~\cite{landau} and~\cite[Theorem~2]{landau1}, we have
 \begin{align*}
  \log o\le \sqrt{m\log m}\left(1+\frac{\log(\log (m))-a}{2\log (m)}\right),
 \end{align*}
 where $\log(m)$ denotes the logarithm of $m$ to the base $e$ and $a:=0.975$.
 Therefore
 \begin{align}\label{equation1}
 \log \left(\frac{q^{n/2}}{2(q+1)}\right)\le\sqrt{m\log m}\left(1+\frac{\log(\log (m))-a}{2\log (m)}\right).
 \end{align}
 This inequality holds true only when
 \begin{itemize}
 \item $q=4$ and $n\in \{10,12\}$, or
 \item $q=3$ and $n\in \{10,12,14,16, 18\}$, or
 \item $q=2$ and $n\le 46$.
 \end{itemize}For these remaining cases, we have computed explicitly $\order z$ and the order of the elements of
 $S_{n+2}$ and we have verified that in no case $\frac{\order z}{\gcd(2,q-1)}$ is the order of a permutation in
 $S_{n+2}$.
\end{proof}

\begin{proposition}\label{simplettici}
For every $n\ge 10$, the weak normal covering number of $\mathrm{PSp}_n(q)$ is at least $3$, unless $q$ is even.
When $q$ is even, if $H$ and $K$ are maximal components of a weak normal $2$-covering of $\mathrm{PSp}_n(q)=\mathrm{Sp}_n(q)$, then up to $\mathrm{Aut}(\mathrm{PSp}_n(q))$-conjugacy, we have  $H\cong\mathrm{SO}^-_{n}(q)$ and  $K\cong \mathrm{SO}^{+}_{n}(q)$. Such weak normal covering gives rise to a single normal covering.
\end{proposition}
\begin{proof}
When  $(n,q)\in \{(10,2),(14,2)\}$, the proof follows with a computer computation, therefore, for the rest of the proof we exclude the case $(n,q)\in\{ (10,2),(14,2)\}$.
 Let $\tilde H$ be a component of a weak normal covering  of $\mathrm{Sp}_n(q)$ containing a Singer cycle. Thus,  $\tilde H$ is one of the groups in  parts~\eqref{malles:1},~\eqref{malles:2} and~\eqref{malles:3} of Lemma~\ref{malles}. From~\cite[Table~$3.5$C]{kl}, we see that in each of these three possibilities we have
\begin{equation*}
\mathrm{Aut}(\mathrm{Sp}_n(q))={\bf N}_{\mathrm{Aut}(\mathrm{Sp}_n(q))}(\tilde{H})\mathrm{Sp}_n(q).
\end{equation*}
Therefore, $\tilde{H}$ on its own cannot give rise to a weak normal covering of $\mathrm{Sp}_n(q)$. Thus  $\gamma_w(\mathrm{Sp}_n(q))\ge 2$.
 Let now $\mu$ be a weak normal $2$-covering of $\mathrm{PSp}_n(q)$ with maximal components and with $\tilde{H}\in \tilde{\mu}$.

Now, the maximal
  subgroups of $\mathrm{Sp}_n(q)$
containing a Bertrand element are described in Lemma~\ref{bertrand-sp}.  By the above arguments, the corresponding  weak normal $2$-covering $\tilde \mu$ of $\mathrm{Sp}_n(q)$ is given by $\tilde \mu=\{\tilde H,\tilde K\}$, where $\tilde K$  must be in one of the  possibilities described in Lemma~\ref{bertrand-sp}.

Let $y\in \mathrm{Sp}_n(q)$ be  of order $$\frac{(q^{\frac{n}{2}-1}
+1)(q+1)}{\gcd(q^{\frac{n}{2}-1} +1,q+1)}$$  and action of type $(n-2)\oplus 2.$
In particular, $V$ has only two proper $y$-invariant subspaces, say $U$ and $W$, $V=U\perp W$, $\dim_{\mathbb{F}_q}U=n-2$ and $\dim_{\mathbb{F}_q}W=2$.  In particular, $y$ has no $\mathrm{Aut}(\mathrm{Sp}_n(q))$-conjugate in $\mathrm{Sp}_{2t}(q)\perp\mathrm{Sp}_{n-2t}(q)$ because $2t\notin\{2,n-2\}$.

Observe that $P_{n-2}(q)\ne\varnothing$ and let  $r\in P_{n-2}(q)$. If $r$ divides the order of $\mathrm{Sp}_{n/k}(q^k).k$, then either $r$ divides $q^{2ik}-1$ for some $i\in \{1,\ldots,n/(2k)\}$ or $r=k$. However, both possibilities are impossible. Indeed, the first possibility yields $2ik\ge n-2$ and hence $i=n/(2k)$. However, when $i=n/(2k)$, we have $q^{2ik}-1=q^n-1$ and $\gcd(q^{n}-1,q^{n-2}-1)=q^2-1$ by Lemma~\ref{aritme}~\eqref{eq:arithme1}. However, this yields that $r$ divides $q^2-1$, which is a contradiction. Similarly, when $k=r$, we contradict~\eqref{boundppd}.

The unitary group $\mathrm{GU}_{n/2}(q).2$ contains an element $\mathrm{Aut}(\mathrm{Sp}_n(q))$-conjugate to $y$ (that is, of order $(q^{n/2-1}+1)(q+1)/\gcd(q^{n/2-1}+1,q+1)$ and of action type $2\oplus (n-2)$) only when $n/2-1$ is odd, that is, $n/2$ is even.

When $q$ is even and $\varepsilon\in\{-,+\}$, the orthogonal group $\mathrm{SO}_{n}^\varepsilon(q)$ contains an element $\mathrm{Aut}(\mathrm{Sp}_n(q))$-conjugate to $y$ only when $\varepsilon=+$, because $y$ lies in $\mathrm{SO}_{n-2}^-(q)\perp\mathrm{SO}_2^-(q)$.

Suppose now that $\tilde{H}$ is as in part~\eqref{malles:2} of Lemma~\ref{malles}, that is, $\tilde{H}=\mathrm{GU}_{n/2}(q).2$, with $nq/2$ odd. The condition on $n$ and on $q$ and Lemma~\ref{bertrand-sp} yield $\tilde{K}=\mathrm{Sp}_{2t}(q)\perp\mathrm{Sp}_{n-2t}(q)$. However, neither $\tilde{H}$ nor $\tilde{K}$ contain an $\mathrm{Aut}(\mathrm{Sp}_n(q))$-conjugate of $y$, which is a contradiction. Thus, $\tilde{H}$ cannot be as in part~\eqref{malles:2} of Lemma~\ref{malles}. Using this fact and using the fact that an $\mathrm{Aut}(\mathrm{Sp}_n(q))$-conjugate of $y$ lies in $\tilde{H}$ or in $\tilde{K}$, by combining the possibilities for $\tilde{H}$ in Lemma~\ref{malles} with the possibilities for $\tilde{K}$ in Lemma~\ref{bertrand-sp}, we obtain that one of the following possibilities occurs
\begin{enumerate}
\item\label{largesp1}$\tilde{H}\cong\mathrm{Sp}_{n/k}(q^k).k$ and $\tilde{K}\cong\mathrm{GU}_{n/2}(q).2$, with $n/2$ even, $n\ne 12$ and $q$ odd,
\item\label{largesp2}$\tilde{H}\cong\mathrm{Sp}_{n/k}(q^k).k$ and $\tilde{K}\cong\mathrm{SO}_n^+(q)$, with $q$ even,
\item\label{largesp3}$\tilde{H}\cong\mathrm{SO}_{n}^-(q)$ and $\tilde{K}\cong\mathrm{SO}_n^+(q)$, with $q$ even,
%\item\label{largesp4}$\tilde{H}=\mathrm{SO}_{12}^-(q)$ and $\tilde{K}=\mathrm{Sp}_6(q^2).2$, with $q$ even.
\end{enumerate}
We now consider each possibility in turn.

Case~\eqref{largesp3} is the Dye normal covering of $\mathrm{Sp}_n(q)$ appearing in the statement of our proposition and hence it requires no further comment. As usual it gives rise to a unique $\mathrm{Sp}_n(q)$-class of normal coverings. To exclude further cases, we use an element similar to $y$. Let $y'\in \mathrm{Sp}_n(q)$ be  of order $$\frac{(q^{\frac{n}{2}-2}
+1)(q^2+1)}{\gcd(q^{\frac{n}{2}-2} +1,q^2+1)}$$  and action of type $(n-4)\oplus 4.$
In particular, $V$ has only two proper $y'$-invariant subspaces, say $U$ and $W$, $V=U\perp W$, $\dim_{\mathbb{F}_q}U=n-4$ and $\dim_{\mathbb{F}_q}W=4$.

Observe that $P_{n-4}(q)\ne\varnothing$ because we are excluding the case $(n,q)=(10,2)$. Let  $r\in P_{n-4}(q)$. If $r$ divides the order of $\mathrm{Sp}_{n/k}(q^k).k$, then either $r$ divides $q^{2ik}-1$ for some $i\in \{1,\ldots,n/(2k)\}$ or $r=k$. However, both possibilities are impossible. Indeed, the first possibility yields $2ik\ge n-4>n/2$ and hence $i=n/(2k)$. However, when $i=n/(2k)$, we have $q^{2ik}-1=q^n-1$ and $\gcd(q^{n}-1,q^{n-4}-1)=q^{\gcd(4,n)}-1$ by Lemma~\ref{aritme}~\eqref{eq:arithme1}. However, this yields that $r$ divides $q^4-1$, which is a contradiction because $r\in P_{n-4}(q)$. Similarly, when $k=r$, we contradict~\eqref{boundppd}.

The unitary group $\mathrm{GU}_{n/2}(q).2$ contains an element $\mathrm{Aut}(\mathrm{Sp}_n(q))$-conjugate to $y'$ (that is, of order $(q^{n/2-2}+1)(q^2+1)/\gcd(q^{n/2-2}+1,q^2+1)$ and of action type $(n-4)\oplus 4$) only when $n/2-2$ is odd, that is, $n/2$ is odd.

%When $q$ is even and $\varepsilon\in\{-,+\}$, the orthogonal group $\mathrm{SO}_{n}^\varepsilon(q)$ contains an element $\mathrm{Aut}(\mathrm{Sp}_n(q))$-conjugate to $y'$ only when $\varepsilon=+$, because $y'$ lies in $\mathrm{SO}_{n-4}^-(q)\perp\mathrm{SO}_4^-(q)$.

Using our considerations on $y'$, we have excluded Case~\eqref{largesp1} and hence we are left to discuss Case~\eqref{largesp2}.

Let $y''\in \mathrm{Sp}_n(q)$ be  of order $$\frac{(q^{\frac{n}{2}-1}
-1)(q+1)}{\gcd(q^{\frac{n}{2}-1} -1,q+1)}$$  and action of type $(n/2-1)\oplus (n/2-1)\oplus 2.$

Here we aim to prove that $y''$ has no $\mathrm{Aut}(\mathrm{Sp}_n(q))$-conjugate in $\mathrm{Sp}_{n/k}(q^k).k$. We argue by contradiction and we suppose, without loss of generality, that $y''\in \mathrm{Sp}_{n/k}(q^k).k$.
Observe that $P_{n/2-1}(q)\ne\varnothing$, because we are excluding the case $(n,q)=(14,2)$. Let  $r\in P_{n/2-1}(q)$. If $r$ divides the order of $\mathrm{Sp}_{n/k}(q^k).k$, then either $r$ divides $q^{2ik}-1$ for some $i\in \{1,\ldots,n/(2k)\}$ or $r=k$. The first possibility yields $r\mid \gcd(q^{2ik}-1,q^{n/2-1}-1)=q^{\gcd(2ik,n/2-1)}-1$, by Lemma~\ref{aritme}. Since $r\in P_{n/2-1}(q)$, we must have $\gcd(n/2-1,2ik)=n/2-1$ and hence
$$2ik=\ell\left(\frac{n}{2}-1\right),$$
for some $\ell\in\mathbb{N}\setminus\{0\}$. As $2ik\le n$, we obtain $\ell\in \{1,2\}$. When $\ell=1$, we get $2ik=n/2-1$ and $k$ divides $n/2-1$. As $k$ divides also $n/2$, we have $k\mid \gcd(n/2-1,n/2)=1 $, which is a contradiction. Analogously, when $\ell=2$, we get $2ik=n-2$ and $i=n/(2k)-1/k$. Since $n/2k$ is an integer, we obtain another contradiction.
 Similarly, when $k=r$,~\eqref{boundppd} yields $k=r\ge n/2-1+1=n/2$ and hence $r=k=n/2$. Summing up, $y''\in \mathrm{Sp}_2(q^{n/2}).(n/2)$ and hence
 \begin{align*}
 q^{\frac{n}{2}-1}-1&=\gcd\left(q^{\frac{n}{2}-1}-1,|\mathrm{Sp}_2(q^{n/2}).(n/2)|\right)\\
 &\le \gcd(q^{\frac{n}{2}-1}-1,|\mathrm{Sp}_2(q^{n/2})|)\cdot\frac{n}{2}\\
 &=\gcd(q^{\frac{n}{2}-1}-1,q^n-1)\cdot\frac{n}{2}\le (q^{\gcd(n/2-1,n)}-1)\frac{n}{2}\\
 &=(q^2-1)\frac{n}{2}.
\end{align*}
However, this inequality is never satisfied.

When $q$ is even and $\varepsilon\in\{-,+\}$, the orthogonal group $\mathrm{SO}_{n}^\varepsilon(q)$ contains an element $\mathrm{Aut}(\mathrm{Sp}_n(q))$-conjugate to $y''$ only when $\varepsilon=-$, because $y''$ lies in $\mathrm{SO}_{n-2}^+(q)\perp\mathrm{SO}_2^-(q)$.
Using our considerations on $y''$, we finally exclude Case~\eqref{largesp2}.
\end{proof}

Now, the veracity of Table~\ref{000=0=0} follows from the results in this section.

\section{Odd dimensional orthogonal  groups}\label{sec:oddorthogonal}
In this section we are concerned with $\mathrm{P}\Omega_n(q)$ with $nq$ odd and $n\ge 7$. Recall that $\mathrm{P}\Omega_n(q)=\Omega_n(q)$.
Note that the case $n=3$ is considered in Section~\ref{sec:linear}  because  $\Omega_3(q)\cong \mathrm{PSL}_2(q)$ and the case $n=5$ is considered in Section~\ref{sec:symplectic} because $\Omega_5(q)\cong \mathrm{PSp}_4(q)$.

\begin{lemma}[{{\cite[Theorem 1.1]{msw}}}] \label{msw-odd}
Let $M$ be a maximal subgroup of $\Omega_{n}(q)$ containing
a semisimple element having order $(q^{\frac{n-1}{2}}+1)/2$ and action type $1\oplus (n-1)$. Then, one of the following holds
\begin{enumerate}
\item\label{eq:msw-sp1}$M\cong\Omega_{n-1}^-(q).2\in\mathcal{C}_1$,
\item\label{eq:msw-sp2}$M\cong S_9\in\mathcal{S}$, $n=7$ and $q=3$.
\end{enumerate}
\end{lemma}

In earlier sections, in our analysis of linear, unitary and symplectic groups, we have used mostly semisimple elements. For the orthogonal groups $\Omega_n(q)$, it is instead convenient to make a massive use of unipotent elements. Therefore, before proving our main result for $\Omega_n(q)$, we collect here some basic remarks about those elements that we need in the sequel.

Recall that $u\in \mathrm{GL}_n(q)$ is \textit{\textbf{regular unipotent}}, if $u-I_n$ has rank $n-1$, where $I_n$ denotes the $n\times n$ identity matrix. The odd dimensional orthogonal group $\Omega_n(q)$ contains regular unipotent elements. Indeed, the existence of regular unipotent elements in $\mathrm{SO}_n(q)$ follows from~\cite[Proposition~$5.7.1$]{carter}; however, since $|\mathrm{SO}_n(q):\Omega_n(q)|=2$ (except when $n=1$) and since $q$ is odd, we deduce that the regular unipotent elements of $\mathrm{SO}_n(q)$ all lie in $\Omega_n(q)$. 

Let $u$ be a regular unipotent element. Clearly, the Jordan form of $u$ consists of a unique Jordan block, that is, $u$ is conjugate via an element of $\mathrm{GL}_n(q)$ to the matrix
\[
\begin{pmatrix}
1& 0&\cdots  &      & &\cdots &    0  \\
1& 1& 0&& & &    \vdots  \\
0& 1& 1&     0& & &        \\
\vdots&\ddots&\ddots&\ddots&&&\\
\vdots&&&\ddots&\ddots&\ddots&\vdots\\
0&\cdots&&0&1&1&0\\
0&\cdots&&\cdots&0&1&1\\
\end{pmatrix}.
\]
Using this representative of the conjugacy class of $u$ in $\mathrm{GL}_n(q)$, an elementary computation shows that, for every $i\in \{0,\ldots,n\}$, there exists a unique $\mathbb{F}_q\langle u\rangle$-submodule $V_i$ of $V=\mathbb{F}_q^n$ with $\dim_{\mathbb{F}_q}(V_i)=i$. Indeed, if we let $e_1,\ldots,e_n$ be the standard basis of $V$, we have $V_0=0$ and $V_i=\langle e_1,\ldots,e_i\rangle$ for $i\in \{1,\ldots,n\}$. In particular, the  $n+1$ $u$-invariant subspaces $V_0,\ldots,V_n$ of $V$ form the flag
$$0=V_0<V_1<V_2<\cdots<V_{n-1}<V_n=V$$
and $u$ has a unique $u$-invariant subspace of dimension $1$ given by $V_1$.

Suppose now that $u$ preserves the non-degenerate quadratic form $Q$ on $V$. As $V_1$ is $u$-invariant, so is $V_1^\perp$. Since $Q$ is non-degenerate,
$\dim_{\mathbb{F}_q}(V_1^\perp)=\dim_{\mathbb{F}_q}(V)-\dim_{\mathbb{F}_q}(V_1)=n-1$. Thus
$V_1^\perp=V_{n-1}$ and hence $V_1\le V_1^{\perp}$. Therefore, $V_1$ is totally singular.

\subsection{Small odd dimensional orthogonal groups}\label{sec:smalloddorthogonal}
We start our analysis with small odd dimensional orthogonal groups. For the subgroup structure of $\Omega_7(q)$ and $\Omega_9(q)$ we use~\cite{bhr}.

\begin{lemma}\label{dimension7oddorthogonal}
The weak normal covering number of $\Omega_7(q)$ is at least $3$.
\end{lemma}
\begin{proof}
When $q=3$, the result follows with a computation with the computer algebra system \texttt{magma}~\cite{magma}. Therefore, for the rest of the argument, we suppose $q\ge 5$.

Let $\mu$ be a weak normal covering of $\Omega_7(q)$ of minimum cardinality and maximal components. Let $H\in \mu$ with $H$ containing a semisimple element of order $(q^3+1)/2$ and action type $1\oplus 6$. From Lemma~\ref{msw-odd}, we have that
 $$H\cong\Omega_6^-(q).2.$$

A case-by-case analysis on the maximal subgroups of $\Omega_7(q)$ in~\cite[Tables~8.39,~8.40]{bhr} reveals that there are two $\Omega_7(q)$-conjugacy classes of maximal subgroups  containing elements of order $q^3-1$. Namely,
\begin{enumerate}
\item $E_{q}^{6+6}:\frac{1}{2}\mathrm{GL}_3(q)\in\mathcal{C}_1$,
\item $\Omega_6^+(q).2\in\mathcal{C}_1$.
\end{enumerate}
In particular, since $H$ is not in this list, we deduce that there exists a further component $K\in \mu$ and that $K$ is in one of the above two possibilities.

Now, let $u\in \Omega_3(q)$ be a regular unipotent element, let $s\in \Omega_4^-(q)$ be a semisimple element of order $(q^2+1)/2$ and having type $4$ on $\mathbb{F}_q^4$ and
let $z=u\oplus s$.   By construction, $z\in\Omega_7(q)$ with respect to a suitable non-degenerate quadratic form. Since $u$ has order $p$, we have ${\bf o}(z)=p(q^2+1)/2$.

As we observed above, the only $1$-dimensional subspace of $\mathbb{F}_q^3$ left invariant by $u$  is totally singular.
Therefore, the only $1$-dimensional subspace of $V$ left invariant by $z$ is
totally singular and hence $z$ is not $\mathrm{Aut}(\Omega_7(q))$-conjugate to elements in $\Omega_6^+(q).2$ or $\Omega_6^-(q).2$, because these groups stabilize a $1$-dimensional non-singular subspace. However, no element in $E_{q}^{6+6}:\frac{1}{2}\mathrm{GL}_3(q)$ has order divisible by $(q^2+1)/2$: this can be seen using Lemma~\ref{aritme} and using the fact that $q\ge 5$ is odd.
Thus $\gamma_w(\Omega_7(q))\ge 3$.
\end{proof}

\begin{lemma}\label{dimension9oddorthogonal}
The weak normal covering number of $\Omega_9(q)$ is at least $3$.
\end{lemma}
\begin{proof} Let $\mu$ be a weak normal covering of $\Omega_9(q)$ of minimum cardinality and maximal components. Let
$H\in \mu$ with $H$ containing a semisimple element of order $(q^4+1)/2$ and action type $1\oplus 8$. From Lemma~\ref{msw-odd}, we have that $$H\cong\Omega_8^-(q).2.$$

Now, let $u\in \Omega_3(q)$ be a regular unipotent element, let $s\in \Omega_6^-(q)$ be a semisimple element of order $(q^3+1)/2$ and having type $6$ on $\mathbb{F}_q^6$ and
let $y=u\oplus s\in\mathrm{GL}_9(q)$.   By construction, $y\in\Omega_9(q)$ with respect to a suitable non-degenerate quadratic form. Since $u$ has order $p$, we have ${\bf o}(y)=p(q^3+1)/2$.

 A case-by-case analysis on the maximal subgroups of $\Omega_9(q)$ in~\cite[Tables~8.58,~8.59]{bhr} reveals that there are two $\Omega_9(q)$-conjugacy classes of maximal subgroups containing a conjugate of $y$. Namely,
\begin{enumerate}
\item $E_{q}^{7}:\left(\frac{q-1}{2}\times \Omega_7(q)\right).2\in \mathcal{C}_1$,
\item $(\Omega_3(q)\times \Omega_6^-(q)).2^2\in\mathcal{C}_1$.
\end{enumerate}
In particular, since $H$ is not in this list, there  exists a further component $K\in \mu$  such that $K$ is in one of the above two possibilities. Suppose $\mu=\{H,K\}$.

Let $u'$ be a regular unipotent element of $\Omega_9(q)$. As we observed above, the only $1$-dimensional subspace of $V=\mathbb{F}_q^9$ left invariant by $u'$  is totally singular. Every unipotent element $v$ in $(\Omega_3(q)\times \Omega_6^-(q)).2^2$ lies in $\Omega_3(q)\times \Omega_6^-(q)$ and hence satisfies $\dim_{\mathbb{F}_q}\cent V v\ge 2$. In particular, $u'$ has no $\mathrm{Aut}(\Omega_9(q))$-conjugate in $H$ (because $u'$ does not fix a non-degenerate vector) or in $(\Omega_3(q)\times \Omega_6^-(q)).2^2$. We deduce that
$$K\cong E_{q}^{7}:\left(\frac{q-1}{2}\times \Omega_7(q)\right).2.$$

Now, let $z$ be a semisimple element of $\Omega_9(q)$ belonging to the maximal subgroup $\Omega_8^+(q).2$ and having type $4\oplus 4\oplus 1$. This element can be explicitly constructed by taking  the quadratic form $Q$ preserved by $\Omega_9(q)$ and  letting $z\in \Omega_9(q)$ be a semisimple element fixing a vector $v\in V$ with $Q(v)$ not a square in $\mathbb{F}_q$ and with $z$ inducing on $\langle v\rangle^\perp$ a semisimple matrix having type $4\oplus 4$. Now, $z$ lies in no conjugate of $H\cong \Omega_8^-(q).2$ because $z$ does not fix any non-zero vector $w$ with $Q(w)$ a square in $\mathbb{F}_q$. As we are assuming that $\mu=\{H,K\}$ is a weak normal covering of $\Omega_{9}(q)$, $z$ is $\mathrm{Aut}(\Omega_9(q))$-conjugate to an element of $K$. Since $z$ is semisimple, $z$ is $\mathrm{Aut}(\Omega_9(q))$-conjugate to a semisimple element in $(\frac{q-1}{2}\times \Omega_7(q)).2$. However, since $z$ has type $4\oplus 4\oplus 1$, $z$ preserves no $7$-dimensional subspace of $V$.
\end{proof}

\subsection{Large odd dimensional orthogonal groups}\label{sec:largeoddorthogonal}
In this section, we deal with large odd dimensional orthogonal groups $\Omega_n(q)$ with $n\ge 11$. We consider the Bertrand number $t$ for the integer $\frac{n-1}{2}\geq 5$  (see Section \ref{bernum}). Thus for $n\geq 17$ we have $\frac{n-1}{4}<t\le \frac{n-1}{2}-3$ and $t$ is an odd prime; for $n=15$, we have $t=5$; for $n=13$, we have $t=4$; for $n=11$ we have $t=3$. In particular $t$ is always an odd prime apart the case $n=13$.
Moreover, for  every $n\geq 11$,  we have $\frac{n-1}{4}<t\le \frac{n-1}{2}-2$ and  $t\nmid \frac{n-1}{2}$.
Thus, from Lemma~\ref{aritme}~\eqref{eq:arithme2}, we have
\begin{align}\label{ber-unit}
\gcd(q^t+1,q^{\frac{n-1}{2}-t}+1)=
\begin{cases}
q+1&\textrm{if }\frac{n-1}{2} \textrm{ is even and }n\neq 13,\\
2&\textrm{if }\frac{n-1}{2} \textrm{ is odd or } n=13.\\
\end{cases}
\end{align}

We consider a Bertrand element $z$ of $\mathrm{\Omega}_n(q)$ such that
\begin{itemize}
\item $\order z=\frac{(q^t+1)(q^{\frac{n-1}{2}-t}+1)}{\gcd(q^t+1,q^{\frac{n-1}{2}-t}+1)}$ (see Table~\ref{2}),
\item the action of $z$ on $V$ is of type $1\oplus 2t\oplus (n-2t-1)$ and we write $V=U_1\perp U_2\perp U_3$, where $\dim_{\mathbb{F}_q}U_1=1$, $\dim_{\mathbb{F}_q}U_2=2t$, $\dim_{\mathbb{F}_q}U_3=n-2t-1$ and $U_1,U_2,U_3$ are irreducible $\mathbb{F}_q\langle z\rangle$-modules,
\item $z$ induces a matrix of order $q^t+1$ on $U_2$ and a matrix of order $q^{\frac{n-1}{2}-t}+1$ on $U_3$.
\end{itemize}
The existence of Bertrand elements is easily proved. Consider a Singer cycle $s_{2t}\in \mathrm{SO}^-_{2t}(q)$ and 
 and a Singer cycle $s_{n-2t+1}\in \mathrm{SO}^-_{n-2t-1}(q)$. Then by Proposition \ref{spinor-ber-prop},  we have  $s_{2t}\oplus s_{n-2t+1}\in\mathrm{\Omega}_{n-1}^+(q)$ so that $z:=I_1\oplus s_{2t}\oplus s_{n-2t+1}\in\mathrm{\Omega}_{n}(q)$ is a Bertrand element of $\mathrm{\Omega}_n(q)$.

Observe that, since $q$ is
odd and both $2t$ and $n-2t-1$ are at least $4$, the sets $P_{2t}(q)$ and $P_{n-2t-1}(q)$ are both non-empty. Moreover, since $2t\neq n-2t-1$, Lemma~\ref{primitivi}~$(2)$ implies that $P_{2t}(q)\cap P_{n-2t-1}(q)=\varnothing.$ 
In particular, for every $a\in P_{2t}(q)$ and $b\in P_{n-2t-1}(q)$, the product $ab$ divides $\order z$. In particular
$$z \textrm{ is a strong }ppd(n,q;2t)\textrm{-element}.$$
\begin{lemma}\label{bertrand-odd}
If $M$ is a maximal subgroup of $\Omega_n(q)$ with $n\ge 11$ containing a Bertrand element, then one of the following holds
\begin{enumerate}
\item\label{bertrand-odd:1} $M$ is of type $\mathrm{O}_{2t+1}(q)\perp \mathrm{O}^{-}_{n-2t-1}(q)\in\mathcal{C}_1$,
\item\label{bertrand-odd:2} $M$ is of type $\mathrm{O}_{n-2t}(q)\perp \mathrm{O}^{-}_{2t}(q)\in\mathcal{C}_1$,
\item\label{bertrand-odd:3} $M$ is of type $\mathrm{O}_{n-1}^+(q)\in\mathcal{C}_1$.
\end{enumerate}
\end{lemma}
\begin{proof} Let $z$ be a Bertrand element of $\Omega_n(q)$ with $n\ge 11.$
By Theorem~\ref{main}, Lemma~\ref{no-c5}\,(3)
and~\cite[Table~3.5D]{kl}, $M$ belongs to one of the Aschbacher classes
$\mathcal C_i$, $i\in \{1,2,3\}$ or to $\mathcal{S}$ and is described in
Example 2.6 a) of \cite{gpps}.
\smallskip

\noindent\textsc{The subgroup $M$ lies in class  $\mathcal{C}_1$. } Suppose that $M$ is the stabilizer of a totally singular subspace of $V$ of dimension $m$ with $1\le m\le (n-1)/2$, that is, $M$ is a parabolic subgroup $P_m$. The element $z$ acts irreducibly on $U_1$, $U_2$ and $U_3$ and $U_1,U_2,U_3$ are non-degenerate of dimension $1$, $2t$ and $n-2t-1$ respectively. Hence $U_1,U_2,U_3$ are pairwise non-isomorphic and they are the only irreducible $\mathbb{F}_q\langle z\rangle$-submodules of $V$. If $z\in M$, then $z$ fixes a non-trivial totally singular subspace $U$ and hence $U_i\le U$, for some $i$. However, this contradicts the fact that $U_i$ is non-degenerate.

Suppose now that $M$ is the stabilizer of a non-degenerate subspace $U$ of $V$ of dimension $2k+1$, that is, $M$ is of type
$\mathrm{O}_{2k+1}(q) \perp\mathrm{O}^{\epsilon}_{n-2k-1}( q)$, with $\epsilon\in\{+,-\}$.  We  refer to~\cite[Proposition~$4.1.6$]{kl} for the precise  structure of $M$. 

Since $U$ must be the sum of irreducible $\mathbb{F}_q\langle z\rangle$-submodules of $V$, and the only possible dimensions for an irreducible $\mathbb{F}_q\langle z\rangle$-submodule of $V$ are $1$, $2t$ and $n-2t-1$, then the only possibilities for the dimension of $U$ are $2k+1=1$, $2k+1=2t+1$, $2k+1=n-2t$, that is, $k\in \{0,t,\frac{n-2t-1}{2}\}$.
Thus in order to reach ~\eqref{bertrand-odd:1}-\eqref{bertrand-odd:3}, we just need to show that when $k=0$ the only possibility is $\epsilon=+$, while for the other two cases the only possibility is 
$\epsilon=-$. 

We distinguish the three cases. Suppose first that $k=0$. Now, $z\in \mathrm{O}_1(q)\perp\mathrm{O}^\epsilon_{n-1}(q)$. The $(n-1)$-dimensional subspace of $V$ fixed by $z$ is $U_2\oplus U_3$. Observe that $z$ restricted to $U_2$ is the Singer cycle $s_{2t}$ and $z$ restricted to $U_3$ is the Singer cycle $s_{n-2t-1}$. In particular, the quadratic form restricted to $U_2$ and to $U_3$ must be of minus type. Therefore, the quadratic form induced on $U_2\oplus U_3$ is of plus type and hence $\epsilon=+$. Suppose next that $k=t$. Now, $z\in \mathrm{O}_{2t+1}(q)\perp\mathrm{O}^\epsilon_{n-2t-1}(q)$. The $(n-2t-1)$-dimensional subspace of $V$ fixed by $z$ is $U_3$. Observe that $z$ restricted to $U_3$ is the Singer cycle $s_{n-2t-1}$. In particular, the quadratic form restricted to $U_3$ must be of minus type and hence $\epsilon=-$. The argument when $t=(n-2t-1)/2$ is similar to the previous case and thus omitted.

\smallskip

\noindent\textsc{The subgroup $M$ lies in class $\mathcal{C}_2$. } These are the groups in the Example
2.3 of \cite{gpps}. Thus  $M \le
\mathrm{GL}_1(q) \mathrm{wr} S_n$. It follows that, for every $a\in P_{2t}(q)$ and $b\in P_{n-2t-1}(q)$, the product $ab$ divides the order of an element
in $S_n$. Since $a$ and $b$ are distinct, that requires $n \ge a + b$ and hence, by \eqref{boundppd}, $n\ge 2t +1 +n -2t =n +1$, which is a
contradiction.
\smallskip

\noindent\textsc{The subgroup $M$ lies in class $\mathcal{C}_3$. }These groups are described in~\cite[Example 2.4]{gpps}. Recall that $z$ is a $ppd(n,q;2t)$ element.  Since $n\neq 2t+1$, we consider only the case b) of
Example 2.4 in~\cite{gpps}. Hence $M\leq \mathrm{GL}_{n/b}(q^b).b,$ where  $b$ is a divisor of  $\gcd(n, 2t)$ with $b\ge 2$. As $n$ is odd, $b$ is odd and hence $b=t\mid n$ so that $n=kt$, for some positive integer $k$.
 Since $(n-1)/4<t< (n-1)/2$, we have that $k=3$ divides $n$ and $b=t=n/3$. Thus $n\geq 15,$ so that $t$ is an odd prime and, by Proposition 4.3.17 in \cite{kl}, we have
$$M\cong \Omega_3( q^t).\,t$$ and $M$ has size $\frac{q^{t}(q^{2t}-1)\,t}{2}$. Assume, by contradiction, that $z\in M.$ Then $\order z\mid |M|$ and thus
\begin{equation}\label{bfact1}
\frac{q^{\frac{n-1}{2}-t}+1}{\gcd(q^t+1,q^{\frac{n-1}{2}-t}+1)}\mid \frac{q^t-1}{2} t.
\end{equation}
 By Lemma \ref{aritme}\,(3), recalling that $t$ is odd, we have that
 \begin{equation*}\label{bfact2}
\gcd(q^{\frac{n-1}{2}-t}+1, q^t-1)=2.
\end{equation*}
Moreover, by \eqref{ber-unit}, we have that $2\mid \gcd(q^t+1,q^{\frac{n-1}{2}-t}+1)$.
Hence, by \eqref{bfact1},
we deduce that
\begin{equation}\label{bfact3}
\frac{q^{\frac{n-1}{2}-t}+1}{\gcd(q^t+1,q^{\frac{n-1}{2}-t}+1)}\mid t.
\end{equation}
Recalling that $t=n/3$ and that $\gcd(q^t+1,q^{\frac{n-1}{2}-t}+1)\leq q+1$,
it follows that
\begin{equation}\label{bfact4}
\frac{q^{\frac{n-3}{6}}+1}{q+1}\leq \frac{n}{3}.
\end{equation}
Now if $n\geq 27$ we have
$$\frac{q^{\frac{n-3}{6}}+1}{q+1}\geq q^{\frac{n-3}{6}-2}(q-1)\geq 2q^{\frac{n-15}{6}}\geq 2\cdot 3^{\frac{n-15}{6}}>\frac{n}{3},$$
which  contradicts \eqref{bfact4}.

It remains to consider the cases $n\in \{15,21\}.$
 Let $n=15$. Then $t=5$ and \eqref{bfact3} becomes
 \begin{equation*}\label{bfact3-15}
\frac{q^{2}+1}{2}\mid 5,
\end{equation*}
which gives $q=3$. But then we have $\order z=5(3^5+1)$ and $M\cong \Omega_3( 3^5).\,5.$ Since $\Omega_3( 3^5)$ has size $\frac{3^{5}(3^{10}-1)}{2}$, we have that $z\notin \Omega_3( 3^5)$ and that a power of $z$ is a generator of $\mathrm{Gal}(\mathbb{F}_{3^5}/\mathbb{F}_{3})=\langle \alpha\rangle.$ Now $\cent{M}{\alpha}\cong \Omega_3( 3).\,5$ and since $z$ centralizes $\alpha$ we reach the contradiction $5(3^5 +1)\mid 15(3^2-1)$. 

Let finally $n=21$. Then $t=7$ and \eqref{bfact3} becomes
 \begin{equation*}\label{bfact3-21}
\frac{q^{3}+1}{q+1}\mid 7,
\end{equation*}
which gives $q=3$. But then we have $\order z=7(3^7+1)$ and $M\cong \Omega_3( 3^7).\,7.$ Since $ \Omega_3( 3^7)$ has size $\frac{3^{7}(3^{14}-1)}{2}$, we have that $z\notin \Omega_3( 3^7)$ and that a power of $z$ is a generator of $\mathrm{Gal}(\mathbb{F}_{3^7}/\mathbb{F}_{3})=\langle \alpha\rangle.$ Now $\cent{M}{\alpha}\cong \Omega_3( 3).\,7$ and since $z$ centralizes $\alpha$ we reach the contradiction $7(3^7 +1)\mid 21(3^2-1)$.

\smallskip

\noindent\textsc{The subgroup $M$ lies in class $\mathcal{S}$. }Let $p$ be the characteristic of $\mathbb{F}_q$. The maximal subgroups
$M$ in Example 2.6 a) satisfy  $ M \le S_m ,$ with $m\in\{n+1,n+2\}$. 

From~\eqref{ber-unit} and from a computation, we have  $$\order z>\frac{q^{\frac{n-1}{2}}}{q+1}.$$ Observe that $\order z$ is the order of an element in $S_m$.
Arguing as in Lemma \ref{bertrand-sp},  we obtain
\begin{align}\label{equation2}
 \log \left(\frac{q^{\frac{n-1}{2}}}{q+1}\right)\le\sqrt{m\log m}\left(1+\frac{\log(\log (m))-0.975}{2\log (m)}\right).
 \end{align}
 This inequality holds true only when $n\in \{11,13,15,17,19\}$ and $q=3$. We have checked those remaining cases  with a computer and in no case $S_{n+2}$ contains an element having order $\order z$.
\end{proof}

\begin{proposition}\label{ortogonali dispari}For every
$n\ge 11$, the weak normal covering number of $\Omega_n(q)$ is at least $3$.
\end{proposition}
\begin{proof}
Assume, by contradiction, that  $H$ and $K$ are maximal components of a weak
normal $2$-covering of $\Omega_{n}(q)$. By Lemma~\ref{msw-odd}, we have
 $H\cong\Omega_{n-1}^-(q).2$ and $K$ is given by one of the possibilities in Lemma~\ref{bertrand-odd}.
 
 Suppose first that $K$ is of type $\mathrm{O}_{2t+1}(q)\perp
\mathrm{O}^{-}_{n-2t-1}(q)$ or $\mathrm{O}_{n-2t}(q) \perp \mathrm{O}^{-}_{2t}(q)$. Without loss of generality, we may suppose that the orthogonal form for $\Omega_n(q)$ is given by the matrix 
\[
J:=\begin{pmatrix}
0_{\frac{n-1}{2}}&I_{\frac{n-1}{2}}&0\\
I_{\frac{n-1}{2}}&0_{\frac{n-1}{2}}&0\\
0&0&1
\end{pmatrix},
\]
where $0_{\frac{n-1}{2}}$ and $I_{\frac{n-1}{2}}$ are the $(n-1)/2\times (n-1)/2$ zero and identity matrix respectively. We let $e_1,\ldots,e_n$ be the canonical bases of $V=\mathbb{F}_q^n$. Now, given a Singer cycle $s$ for $\mathrm{GL}_{(n-1)/2}(q)$, we get
\[
y':=\begin{pmatrix}
s&0&0\\
0&(s^{-1})^T&0\\
0&0&1
\end{pmatrix}\in \mathrm{SO}_n(q).
\]

Set $y:=y'^2$ and observe that  $y$ is a semisimple element in $\mathrm{\Omega}_n(q)$.
We claim that 
\begin{itemize}
\item[$(\dag)$]  $U_1:=\langle e_1,\ldots,e_{(n-1)/2}\rangle$, $U_2:=\langle e_{(n+1)/2},\ldots,e_{n-1}\rangle$ and $U_3:=\langle e_n\rangle$ are  the only irreducible $\mathbb{F}_q\langle y\rangle$-submodules of $V$.
\end{itemize}
 Indeed, $s^{q-1}$ is a Singer cycle for $\mathrm{SL}_{(n-1)/2}(q)$ so that it acts irreducibly on $U_1$.
Since the order of $((s^{-1})^T)^{q-1}$ is equal to the order of $s^{q-1}$, by Proposition \ref{sing-ord}, we have that also $((s^{-1})^T)^{q-1}$ is a Singer cycle  for $\mathrm{SL}_{(n-1)/2}(q)$ so that it acts irreducibly on $U_2$.

Since $q$ is odd, we have that
 $\langle s^2\rangle \geq \langle s^{q-1}\rangle$ and thus $s^2$ acts irreducibly on $U_1$ and, similarly, $(s^{-2})^T$ acts irreducibly on $U_2$. As $\dim_{\mathbb{F}_q}(U_3)=1$, it follows that  $U_1,U_2,U_3$ are irreducible $\mathbb{F}_q\langle y\rangle$-submodules of $V$. It remains to prove the uniqueness. To this end, let $W$ be an irreducible $\mathbb{F}\langle y\rangle$-submodule of $V$. If $W\cap U_3\ne 0$, then $U_3=W$. Assume then $W\cap U_3=0$. Now, $W\cap U_3^\perp$ is an $\mathbb{F}_q\langle y\rangle$-submodule; since $W$ is irreducible, we deduce $W\le U_3^\perp=U_1\oplus U_2$, or $W\cap U_3^\perp=0$.
  The latter case is impossible because it implies $\dim_{\mathbb{F}_q}(W)=1$; but it can be readily seen that the only eigenvalue of $y$ in $\mathbb{F}_q$ is $1$ and its eigenspace is $U_3$. Thus we are left with $W\le U_1\oplus U_2$.  If $W\cap U_i\ne 0$, for some $i\in \{1,2\}$, then $W=U_i$ because $W\cap U_i$ is a submodule of the irreducible module $W$ and of the irreducible module $U_i$. Assume then $W\cap U_1=W\cap U_2=0$. This implies
 $$U_1\cong \frac{U_1}{W\cap U_1}\cong \frac{U_1\oplus W}{U_1}\leq \frac{U_1\oplus U_2}{U_1}\cong U_2.$$
 Since the module $U_2$ is irreducible and $U_1\neq 0$, this implies that 
 $U_1$ and $U_2$ are isomorphic $\mathbb{F}_q\langle y\rangle$-modules. Let $\varphi:U_1\to U_2$ be a $\mathbb{F}_q\langle y\rangle$-isomorphism, that is, $\varphi(uz)=\varphi(u)z$  holds for all $ u\in U_1$ and $z\in \langle y\rangle$. Represent now $\varphi$, with respect to the basis  $(e_1,\ldots,e_{(n-1)/2})$ of $U_1$ and to the basis $(e_{(n+1)/2},\ldots,e_{n-1})$ of $U_2$, via the matrix $B\in\mathrm{GL}_{(n-1)/2}(q)$. Then we have
 \begin{equation}\label{equaB}
 B^{-1}s^{2}B=(s^{-2})^{T}.
 \end{equation}
We know that, up to conjugacy, the Singer cycle  $s\in \mathrm{GL}_{(n-1)/2}(q)$ is given by $\pi_a$ where $a$ is  a generator of $\mathbb{F}_{q^{(n-1)/2}}^*$. In particular, the eigenvalues of $s$ in the field $\mathbb{F}_{q^{(n-1)/2}}$ form the set $\{a, a^q,\dots, a^{q^{\frac{n-1}{2}-1}}\}$.
Thus  the eigenvalues of $s^{2}$ in $\mathbb{F}_{q^{(n-1)/2}}$ form the set $$\{a^2,a^{2q},\ldots,a^{2q^{\frac{n-1}{2}-1}}\},$$ and the eigenvalues of $(s^{-2})^T$ in $\mathbb{F}_{q^{(n-1)/2}}$ form the set $$\{a^{-2},a^{-2q},\ldots,a^{-2q^{\frac{n-1}{2}-1}}\}.$$ 
Now,~\eqref{equaB} implies that 
$$\{a^2,a^{2q},\ldots,a^{2q^{\frac{n-1}{2}-1}}\}=\{a^{-2},a^{-2q},\ldots,a^{-2q^{\frac{n-1}{2}-1}}\}.$$
Therefore, there exists $i\in \{1,\ldots,(n-3)/2\}$ such that  $a^2=a^{-2q^i}$, that is, $a^{2(q^i+1)}=1$. This implies that $q^{(n-1)/2}-1$ divides $2(q^i+1)$ so that, by Lemma \ref{aritme}, we deduce $$q^{(n-1)/2}-1=\gcd(q^{(n-1)/2}-1,2(q^i+1))\leq 2(q^{\gcd(i,(n-1)/2) } +1).$$ Since $i<(n-1)/2$, we have that $\gcd(i,(n-1)/2)\leq (n-1)/4$ and thus we get $$q^{(n-1)/2}-1\leq 2(q^{(n-1)/4}+1),$$ which implies $q^{(n-1)/4}\leq 3$. This is impossible because $q^{(n-1)/4}\geq 3^{10/4}.$ This concludes the proof of our claim~$(\dag)$.

\smallskip

 It is readily seen that the quadratic form induced by $J$ on the subspace of $V=\mathbb{F}_q^n$ spanned by $e_1,\dots, e_{n-1}$ has Witt defect $0$. We have proven above that $U_1,U_2,U_3$ are the only irreducible $\mathbb{F}_q\langle y\rangle$-submodules of the semisimple module $V$. Therefore, the only proper $\mathbb{F}_q\langle y\rangle$-submodules of the semisimple module $V$ are
$$U_1,U_2,U_3,U_1\oplus U_2,U_1\oplus U_3,U_2\oplus U_3.$$ 
Recalling that $U_1$ and $U_2$ are totally isotropic of dimension $(n-1)/2>1$, we deduce that the only proper non-degenerate $\mathbb{F}_q\langle y\rangle$-submodules of the semisimple module $V$ are $U_1\oplus U_2$ having dimension $n-1$ and $U_3$ having dimension $1$. What is more, the quadratic form induced on the only $(n-1)$-dimensional $y$-invariant subspace of $V$, that is on $U_1\oplus U_2$,  is of Witt defect 0.  From this, it follows that $y$ is not $\mathrm{Aut}(\Omega_n(q))$-conjugate to an element  in $H=\Omega_{n-1}^-(q).2$. Furthermore, recalling $t\le \frac{n-1}{2}-2$, we immediately check that
 $1,n-1\notin\{2t+1,n-2t-1,2t,n-2t\}$ and thus  $y$ is not $\mathrm{Aut}(\Omega_n(q))$-conjugate to an element  in $K$.

Finally, suppose that $K\cong\Omega_n^+(q).2$.  Here we argue using regular unipotent elements.
Let $u$ be a regular unipotent element of $\Omega_n(q)$. Recall that the only $1$-dimensional subspace of $V=\mathbb{F}_q^n$ left invariant by $u$  is totally singular.
 Since the elements in  $H\cong\Omega_n^-(q).2$ and in $K\cong\Omega_n^+(q).2$ all fix some non-degenerate $1$-dimensional subspace of $V$, we deduce that $u$ has no $\mathrm{Aut}(\Omega_n(q))$-conjugate in $H$ or in $K$.
\end{proof}

Summing up, when $n\ge 7$, $\mathrm{P}\Omega_n(q)=\Omega_n(q)$ has weak normal covering number at least $3$.

\subsection{An auxiliary result}
We include here a rather technical lemma which will only be used in the proof of Theorem~\ref{main theoremgeneral} for even characteristic symplectic groups.
\begin{lemma}\label{aggiunta}
Let $n\ge 6$ be even and let $q$ be an even prime power. Then $\mathrm{Sp}_n(q)$ contains two regular unipotent elements $u_+$ and $u_-$ such that
\begin{enumerate}
\item\label{aggiunta1}$u_+$ and $u_-$ lie in distinct $\mathrm{Aut}(\mathrm{Sp}_n(q))$-conjugacy classes,
\item\label{aggiunta2}$u_+$ preserves a quadratic form of Witt defect $0$, $u_-$ preserves a quadratic form of Witt defect $1$, thus $u_+\in\mathrm{SO}_n^+(q)$ and $u_-\in\mathrm{SO}_n^-(q)$ with respect to these quadratic forms,
\item\label{aggiunta3}$u_+$ has no $\mathrm{Aut}(\mathrm{Sp}_n(q))$-conjugate in $\mathrm{SO}_n^-(q)$ and $u_-$ has no $\mathrm{Aut}(\mathrm{Sp}_n(q))$-conjugate in $\mathrm{SO}_n^+(q)$,
\item\label{aggiunta4}$u_+\in\mathrm{SO}_n^+(q)\setminus \Omega_n^+(q)$ and $u_-\in\mathrm{SO}_n^-(q)\setminus\Omega_n^-(q)$.
\end{enumerate}
\end{lemma}
\begin{proof}
Let $m:=-n/2$ and let 
\[
J:=\begin{pmatrix}
 & & & &1\\
 & & &1& \\
  & &\iddots&& \\
   &1& && \\
   1& & && \\
\end{pmatrix}\in\mathrm{GL}_{m}(q),
\]
where the elements off the diagonal are equal to $0$. We let $\mathrm{Sp}_n(q)$ be the symplectic group with respect to the symplectic form having matrix
\[\mathcal{J}:=
\begin{pmatrix}
0&J\\
J&0
\end{pmatrix}.
\]

We construct the matrices $u_+$ and $u_-$ explicitly. We let
\[
A_1:=\begin{pmatrix}
1&1&\cdots&\cdots&1\\
 &1&1&\cdots&1\\
 &&\ddots&\ddots&\vdots\\
 && &1&1\\
 && &&1
\end{pmatrix},
A_2:=
\begin{pmatrix}
1&1&& &\\
 &1&1& &\\
 &&\ddots&\ddots&\\
 && &1&1\\
 && &&1
\end{pmatrix}
\in\mathrm{GL}_{m}(q),\]
where as usual unmarked entries are tacitly assumed to be $0$. Moreover, we let
\[
j_1:=\begin{pmatrix}
1\\
\vdots\\
1
\end{pmatrix},
j_2:=\begin{pmatrix}
1\\
\vdots\\
1\\
0
\end{pmatrix}\in\mathbb{F}_q^{m}.
\]
Now we let
\[
B_+:=\begin{pmatrix}
j_1&0&\cdots&0
\end{pmatrix},
B_-:=\begin{pmatrix}
j_1&j_2&0&\cdots&0
\end{pmatrix}\in\mathrm{Mat}\left(m\times m,q\right).
\]
Finally, we define
\[
u_+:=
\begin{pmatrix}
A_1&B_+\\
0&A_2
\end{pmatrix},
u_+:=
\begin{pmatrix}
A_1&B_-\\
0&A_2
\end{pmatrix}\in\mathrm{SL}_n(q).
\]
For instance, when $n=8$, we have
\[
u_+=
\begin{pmatrix}
1&1&1&1&1&0&0&0\\
0&1&1&1&1&0&0&0\\
0&0&1&1&1&0&0&0\\
0&0&0&1&1&0&0&0\\
0&0&0&0&1&1&0&0\\
0&0&0&0&0&1&1&0\\
0&0&0&0&0&0&1&1\\
0&0&0&0&0&0&0&1
\end{pmatrix},
u_-=
\begin{pmatrix}
1&1&1&1&1&1&0&0\\
0&1&1&1&1&1&0&0\\
0&0&1&1&1&1&0&0\\
0&0&0&1&1&0&0&0\\
0&0&0&0&1&1&0&0\\
0&0&0&0&0&1&1&0\\
0&0&0&0&0&0&1&1\\
0&0&0&0&0&0&0&1
\end{pmatrix}.
\]
Observe that $u_+$ and $u_-$ consists of a unique Jordan block an hence $u_+$ and $u_-$ are regular unipotent elements. 

Using the matrix defining the symplectic form for $\mathrm{Sp}_n(q)$, it is easy to verify that $u_+,u_-\in\mathrm{Sp}_n(q)$. From~\cite[Proposition~7.1.9]{DeFranceschi} or~\cite[Proposition~5.1.7]{DeFranceschi1}, it can be verified that $u_+$ and $u_-$ lie in distinct $\mathrm{Sp}_n(q)$-conjugacy classes. Let $\phi:\mathbb{F}_q\to\mathbb{F}_q$ be defined by $x\mapsto x^2$, $\forall x\in \mathbb{F}_q$. Now $\phi$ induces a field automorphism on $\mathrm{Sp}_n(q)$ which, by abuse of notation, we still denote by $\phi$. As $\mathrm{Aut}(\mathrm{Sp}_n(q))=\mathrm{Sp}_n(q)\rtimes\langle\phi\rangle$ and as $\phi$ centralizes $u_+$ and $u_-$, we deduce that $u_+$ and $u_-$ are in distinct $\mathrm{Aut}(\mathrm{Sp}_n(q))$-conjugacy classes. In particular, part~\eqref{aggiunta1} holds true.

Now let $E_{11}$ be the $m\times m$-matrix having $1$ in position $(1,1)$ and $0$ anywhere else. Let 
\[
Q_+:=
\begin{pmatrix}
0&J\\
0&E_{11}
\end{pmatrix}\in \mathrm{Mat}(n\times n,q).
\]
Since $Q_++Q_+^T=\mathcal{J}$ is the matrix defining the symplectic form for $\mathrm{Sp}_n(q)$, we see that $Q_+$ defines a non-degenerate quadratic form polarizing to $\mathcal{J}$. We still denote with $Q_+$ this quadratic form. Observe that
$$Q_+(x_1,\ldots,x_{2m})=x_1x_{2m}+x_{2}x_{2m-2}+\cdots+x_{m}x_{m+1}+x_{m}^2.$$
Since the $\mathbb{F}_q$-subspace spanned by the standard vectors $e_1,\ldots,e_m$ is totally singular for $Q_+$, we see that $Q_+$ has Witt defect zero. Let $\mathrm{SO}_n+(q)$ be the special orthogonal group with respect to $Q_+$. Thus $\mathrm{SO}_n^+(q)\le \mathrm{Sp}_n(q)$. A matrix computation shows that $u_+$ preserves $Q_+$ and hence $u_+\in Q_+$.

Now let $a\in\mathbb{F}_q$ such that the degree $2$ polynomial $T^2+T+a\in\mathbb{F}_q[T]$ is irreducible. Now let $E_{mm}$ be the $m\times m$-matrix having $1$ in position $(m,m)$ and $0$ anywhere else. Let 
\[
Q_-:=
\begin{pmatrix}
E_{mm}&J\\
0&aE_{11}
\end{pmatrix}\in \mathrm{Mat}(n\times n,q).
\]
Since $Q_-+Q_-^T=\mathcal{J}$ is the matrix defining the symplectic form for $\mathrm{Sp}_n(q)$, we see that $Q_-$ defines a non-degenerate quadratic form polarizing to $\mathcal{J}$. We still denote with $Q_-$ this quadratic form. Observe that
$$Q_-(x_1,\ldots,x_{2m})=x_1x_{2m}+x_{2}x_{2m-1}+\cdots+x_{m-1}x_{m+2}+x_{m}x_{m+1}+x_m^2+ax_{m+1}^2.$$
In this case, we see that $Q_-$ has Witt defect $1$. Let $\mathrm{SO}_n^-(q)$ be the special orthogonal group with respect to $Q_-$. Thus $\mathrm{SO}_n^-(q)\le \mathrm{Sp}_n(q)$. A matrix computation shows that $u_-$ preserves $Q_-$ and hence $u_-\in Q_-$. In particular, part~\eqref{aggiunta2} holds true.

Parts~\eqref{aggiunta3} and~\eqref{aggiunta4} follow from~\cite[Proposition~7.1.9]{DeFranceschi} or~\cite[Proposition~5.1.9]{DeFranceschi1}.
\end{proof}
\section{Orthogonal groups with Witt defect $1$}\label{sec:orthogonal-}
In this section, we deal with even dimensional orthogonal groups $\mathrm{P}\Omega_n^-(q)$ of Witt defect $1$, for $n\ge 8$.% Here, we may work with $\Omega_n^-(q)$. 

Note that the cases $n\in\{2,4,6\}$ are not considered here because   $\Omega_2^-(q)\cong C_{\frac{q+1}{\gcd(2,q-1)}}$ is abelian, while $\mathrm{P}\Omega_4^-(q)\cong \mathrm{PSL}_2(q^2)$ and $\mathrm{P}\Omega_6^-(q)\cong \mathrm{PSU}_4(q)$ are treated in the previous sections.

We begin with some preliminary lemmas.
\begin{lemma}\label{lemma:new2022+} Let  $s$  be a Singer cycle for $\mathrm{GL}_{\ell}(q)$, with $\ell\geq 1$ and $q$ odd, and let 
\[
x_\ell:=\begin{pmatrix}
s&0\\
0&(s^{-1})^T\\
\end{pmatrix}\in \mathrm{SO}_{2\ell}^+(q).
\]
Then $\order {x_\ell}=q^\ell-1$ and $x_\ell\notin\Omega_{2\ell}^+(q)$.
\end{lemma}
\begin{proof} We describe $\mathrm{SO}_{2\ell}^+(q)$ with respect to the  orthogonal  form having matrix
\[
J:=\begin{pmatrix}
0_{\ell}&I_{\ell}\\
I_{\ell}&0_{\ell}\\
\end{pmatrix},
\]
with respect to the basis $(e_1,\ldots,e_{\ell}, e_{\ell+1},\dots, e_{2\ell})$ of $V=\mathbb{F}_q^{\ell}$. Let $W:=\langle e_1,\ldots,e_{\ell}\rangle$ and $U:=\langle e_{\ell+1},\dots, e_{2\ell} \rangle$. Note that $x_\ell$ fixes both the subspaces $W$ and $U$. By Lemma 2.7.2 in \cite{kl}, we have that $x_\ell\in \Omega_{2\ell}^+(q)$ if and only if $\mathrm{det}_{W}(x_\ell)$ is a square in  $\mathbb{F}_q.$ Now, by \eqref{piadet}, we have 
$$\mathrm{det}_{W}(x_\ell)=\mathrm{det}_{\mathbb{F}_q}(s)=a^{\frac{q^\ell-1}{q-1}},$$
where $\langle a\rangle=\mathbb{F}_{q^\ell}^\ast.$ Note that $b:=a^{\frac{q^\ell-1}{q-1}}$ generates $\mathbb{F}_{q}^\ast$. Now recall that, since $q$ is odd,  the subgroup $(\mathbb{F}_q^\ast)^2$ has index $2$ in $\mathbb{F}_q^\ast$. Thus $b\notin (\mathbb{F}_q^\ast)^2$  and hence $x_\ell\notin\Omega_{2\ell}^+(q)$.
\end{proof}

\begin{lemma}\label{lemma:new2022}Let $n$ be an even positive integer and let $1\le \ell\le n/2-1$. Then $\Omega_n^-(q)$ contains an element having order
$$\frac{(q^\ell+1)(q^{\frac{n}{2}-\ell}-1)}{\gcd\left(q^\ell+1,q^{\frac{n}{2}-\ell}-1\right)}$$
and action type $2\ell\oplus \left(\frac{n}{2}-\ell\right)\oplus\left(\frac{n}{2}-\ell\right)$.
\end{lemma}
\begin{proof} Let $s$ be a Singer cycle of $\mathrm{SO}_{2\ell}^-(q)$; in particular, ${\bf o}(s):=q^\ell+1$. Let $x_{\frac{n}{2}-\ell}\in\mathrm{SO}_{n-2\ell}^+(q)\setminus \Omega_{n-2\ell}^-(q)$ having order ${\bf o}(x_{\frac{n}{2}-\ell})=q^{n/2-\ell}-1$ and action type $\left(\frac{n}{2}-\ell\right)\oplus\left(\frac{n}{2}-\ell\right)$ as described in Lemma \ref{lemma:new2022+}.
Set $x:=s\oplus x_{\frac{n}{2}-\ell}\in \mathrm{SO}_n^-(q)$. Then $x$ has the required order and action type and we only need to show that $x\in\Omega_n^-(q)$.
When $q$ is even, as ${\bf o}(x)$ is odd and $|\mathrm{SO}_n^-(q):\Omega_n^-(q)|$ is a power of $2$, we deduce $x\in \Omega_n^-(q)$. Suppose $q$ odd. Observe that $s\in \mathrm{SO}_{2\ell}^-(q)\setminus\Omega_{2\ell}^-(q)$ because the order of a Singer cycle in $\Omega_{2\ell}^-(q)$ for $q$ odd is $(q^\ell+1)/2$. 
Hence the spinor norm of $s$ and $x_{\frac{n}{2}-\ell}$ is non-trivial and hence the spinor norm of $x=s\oplus x_{\frac{n}{2}-\ell}$ is trivial, that is, $x\in \Omega_n^-(q)$.
\end{proof}
\begin{lemma} [{{\cite[Theorem 1.1]{msw}}}] \label{msw-} 
Let $n$ be an even integer with $n\ge 8$ and let $M$ be a maximal subgroup of $\Omega^-_{n}(q)$ containing a
 Singer cycle. Then one of the following holds
\begin{enumerate}
  \item\label{eq:msw-1}
$M\cong \Omega^-_{n/s}(q^s).s$ is in class $\mathcal{C}_3$ and $s$ is a prime divisor
of $n/2 $ with $n/s\ge 4$,
 \item\label{eq:msw-2} $n/2$ is odd and $M$ is of type $\mathrm{GU}_{n/2}(q)\in\mathcal{C}_3$.
\end{enumerate}
\end{lemma}
The detailed structure of $M$ in part~\eqref{eq:msw-2} can be found in~\cite[Proposition~$4.3.18$]{kl}.

\subsection{Small even dimensional orthogonal groups of Witt defect 1}\label{sec:smallevenorthogonal1}
We start our analysis with small even dimensional orthogonal groups of Witt defect $1$. For the subgroup structure of $\mathrm{P}\Omega_8^-(q)$, $\mathrm{P}\Omega_{10}^-(q)$ and $\mathrm{P}\Omega_{12}^-(q)$, we use~\cite{bhr}.

\begin{lemma}\label{dimension8orthogonal-}
The weak normal covering number of $\mathrm{P}\Omega_8^-(q)$ is at least $3$.
\end{lemma}
\begin{proof}
As usual, we work with $\Omega_8^-(q)$. Consider a weak normal covering $\tilde{\mu}$ of $\Omega_8^-(q)$ of minimum cardinality and maximal components.
Let $\tilde H$ be a component of $\tilde{\mu}$ containing a Singer cycle. From Lemma~\ref{msw-}, we get $$\tilde H\cong \Omega_4^-(q^2).2\cong \mathrm{PSL}_2(q^4).2.$$

Let $y$ be an element of $\Omega_8^-(q)$ having order $(q^3-1)(q+1)/\gcd(q^3-1,q+1)=(q^3-1)(q+1)/\gcd(2,q-1)$ and action type $2\oplus 3\oplus 3$. The existence of such an element is guaranteed by Lemma \ref{lemma:new2022}.

Moreover, we have that $V=U_1\oplus U_2\oplus U_3$, with $U_1,U_2,U_3$ irreducible $\mathbb{F}_q\langle y\rangle$-modules, $\dim_{\mathbb{F}_q}U_1=2$,  $\dim_{\mathbb{F}_q}U_2=\dim_{\mathbb{F}_q}U_3=3$, the orthogonal form of $\Omega_8^-(q)$ restricted to $U_1$ is of Witt defect $1$ and restricted to $U_2\oplus U_3$ is of Witt defect $0$ where $U_2$ and $U_3$ are totally isotropic. A case-by-case analysis on the maximal subgroups of $\Omega_8^-(q)$ in~\cite[Tables~8.52,~8.53]{bhr} reveals that there are two $\Omega_8^-(q)$-conjugacy classes of maximal subgroups $M$ containing a conjugate of $y$. Namely,
\begin{enumerate}
\item $M\cong E_{q}^{3+6}:\left(\frac{1}{\gcd(2,q-1)}\mathrm{GL}_3(q)\times\Omega_2^-(q)\right).\gcd(2,q-1)\in \mathcal{C}_1$,
\item $M\cong (\Omega_2^-(q)\times \Omega_6^+(q)).2^{\gcd(2,q-1)}\in\mathcal{C}_1$.
\end{enumerate}
In particular, since $\tilde H$ is not in this list we deduce $\gamma_w(\Omega_8^-(q))\ge 2$. Let $\tilde K$ be a component of $\tilde{\mu}$ containing a conjugate of $y$. Thus, $\tilde K$ must be in one of the above two possibilities.

Let $z\in \Omega_8^-(q)$ having order $\frac{(q^3+1)(q-1)}{\gcd(2,q-1)}$ and action type $6\oplus 1\oplus 1$. The existence of such an element is guaranteed by Lemma \ref{lemma:new2022}.

The group $\tilde H$ does not contain elements having this order because $q^3+1$ does not divide $|\tilde H|$.  This can be easily verified when $q\ne 2$ by showing that a primitive prime divisor of $q^6-1$ does not divide $|\tilde H|=2(q^8-1)$ and it can be checked directly when $q=2$.
Moreover, none of the two possibilities above for $\tilde K$ contains elements of order $q^3+1$: this can be easily verified when $q\ne 2$ by showing that a primitive prime divisor of $q^6-1$ does not divide $|\tilde K|$ and it can be checked directly when $q=2$. Therefore $\gamma_w(\mathrm{P}\Omega_8^-(q))\ge 3$.
\end{proof}

\begin{lemma}\label{dimension10orthogonal-}
The weak normal covering number of $\mathrm{P}\Omega_{10}^-(q)$ is at least $3$.
\end{lemma}
\begin{proof}
We have verified the veracity of this lemma with a computer when $q\in \{2,3\}$ and hence in the rest of the proof we assume $q\ge 4$.  Let $\tilde \mu$ be a weak normal covering of $\Omega_{10}^-(q)$ of minimum cardinality and maximal components and let $\tilde H$ be a  component containing a Singer cycle. From Lemma~\ref{msw-} and by \cite{bhr}, we obtain
$$\tilde H\cong \left(\frac{q+1}{\gcd(2,q+1)}\circ \mathrm{SU}_5(q)\right).\gcd(q+1,5).$$

Let $z$ be an element of $\Omega_{10}^-(q)$ such that
\begin{itemize}
\item $\order z=(q^4+1)(q-1)/\gcd(q^4+1,q-1)=(q^4+1)(q-1)/\gcd(2,q-1)$,
\item the action of $z$ on $V$ is of type $ 1\oplus 1\oplus8$,
\item $V=U_1\oplus U_2\oplus U_3$, with $U_1,U_2,U_3$ irreducible $\mathbb{F}_q\langle z\rangle$-modules, $\dim_{\mathbb{F}_q}U_1=\dim_{\mathbb{F}_q}U_2=1$,  $\dim_{\mathbb{F}_q}U_3=8$, the orthogonal form of $\Omega_{10}^-(q)$ restricted to $U_1\oplus U_2$ is non-degenerate of Witt defect $0$ (with $U_1$ and $U_2$ totally singular) and restricted to $U_3$ is non-degenerate of Witt defect $1$.
\end{itemize}
The existence of such an element is guaranteed by Lemma \ref{lemma:new2022}.

%\textcolor{blue}{ We show that $\Omega_{10}^-(q)$ does contain such an element. Define
% $z:=s_8\oplus x\in \mathrm{SO}_{10}^-(q)$ where $s_8\in \mathrm{SO}_8^-(q)$ is a Singer cycle and $x\in \mathrm{SO}_2^+(q)\cong C_{q-1}.(2,q)$ is an element of order $q-1$. Let first $q$ be odd. 
 %Then $s_8\notin \Omega_8^-(q)$ and  $x\notin \Omega_2^+(q)\cong  C_{\frac{q-1}{2}}$. As a consequence the spinor norm on $z$ is trivial and thus $z\in \Omega_{10}^-(q)$. Let next $q$ be even. Then $s_8\in \Omega_8^-(q)$ and $x\in \Omega_2^+(q)\cong C_{q-1}$ so that again $z\in \Omega_{10}^-(q)$. Since the action of $x$ is reducible we have that the action of $z$ is of type $8\oplus 1\oplus 1$. Moreover, we have $\order z=\frac{(q^4+1)(q-1)}{\gcd(2,q-1)}$.}

We prove  that $\tilde{H}$ contains no element having order divisible by $q^4+1$. Actually, we prove that the order of $\tilde H$ is not divisible by a primitive prime divisor
\ $r\in P_8(q)$. Clearly $r$ does not divide the size of $\frac{q+1}{\gcd(2,q+1)}\circ \mathrm{SU}_5(q)$. Moreover, by \eqref{boundppd}, we have that $r\geq 9$ so that $r\nmid \gcd(q+1,5).$ Thus $r\nmid |\tilde H|.$

 Therefore, no $\mathrm{Aut}(\Omega_{10}^-(q))$-conjugate of $z$ lies in $\tilde H$.

Assume that $\gamma_w(\Omega_{10}^-(q))=2$ so that $\tilde \mu=\{\tilde H,\tilde K\},$ where
$\tilde K$ is  a maximal component containing an $\mathrm{Aut}(\Omega_{10}^-(q))$-conjugate of $z$. Our task is reaching a contradiction.

Using the fact that $q\ge 4$, a case-by-case analysis on the maximal subgroups of $\Omega_{10}^-(q)$ in~\cite[Tables~8.68,~8.69]{bhr} reveals that the maximal subgroups $M$ having order divisible by $q^4+1$ are, up to $\Omega_{10}^-(q)$-conjugacy, as follows
\begin{enumerate}
\item $M\cong E_{q}^{8}:\left(\frac{q-1}{\gcd(2,q-1)}\times\Omega_8^-(q)\right).\gcd(2,q-1)\in\mathcal{C}_1$,
\item $q$ is even and $M\cong \mathrm{Sp}_8(q)\in\mathcal{C}_1$
\item $q$ is odd and $M\cong \Omega_9(q).2\in\mathcal{C}_1$,
\item $M\cong (\Omega_2^+(q)\times \Omega_8^-(q)).2^{\gcd(2,q-1)}\in \mathcal{C}_1$,
\item $q\equiv 1\pmod 4$ and $M\cong \Omega_5(q^2).2\in\mathcal{C}_3$,
\item $q\equiv 3\pmod 4$ and $M\cong 2\times \Omega_5(q^2).2\in\mathcal{C}_3$.
\end{enumerate}
In particular, $\tilde{K}$ is one of these groups.

Let $x$ be an element of $\Omega_{10}^-(q)$ such that
\begin{itemize}
\item ${\bf o}(x)= (q^3-1)(q^2+1)/\gcd(2,q-1)$ and $x$ has action type $4\oplus 3\oplus 3$;
\item $V=U_1\perp (U_2\oplus U_3)$, with $U_1,U_2,U_3$ irreducible $\mathbb{F}_q\langle x\rangle$-modules, $\dim_{\mathbb{F}_q}U_1=4$, $\dim_{\mathbb{F}_q}U_2=\dim_{\mathbb{F}_q}U_3=3$, the orthogonal form of $\Omega_{10}^-(q)$ restricted to $U_1$ is of Witt defect $1$ and restricted to $U_2\oplus U_3$ is of Witt defect $0$ where $U_2$ and $U_3$ are totally isotropic;
\item $x$ induces on $U_2$ and $U_3$ a matrix of order $q^3-1$ and on $U_1$ a Singer cycle of order $q^2+1$.
\end{itemize}
The existence of such an element is guaranteed by Lemma \ref{lemma:new2022}.

Now, $x$ does not fix non-zero totally isotropic subspaces of dimension less or equal to $ 2$ and hence $x$ has no $\mathrm{Aut}(\Omega_{10}^-(q))$-conjugate in
\begin{itemize}
\item $M\cong E_{q}^{8}:\left(\frac{q-1}{\gcd(2,q-1)}\times\Omega_{8}^-(q)\right).\gcd(2,q-1)\in\mathcal{C}_1$, or in
\item $M\cong \mathrm{Sp}_{8}(q)\in\mathcal{C}_1$ with $q$ even.
\end{itemize}
Similarly, $x$ does not fix a non-degenerate subspace of dimension less or equal to $2$ and hence $x$ has no $\mathrm{Aut}(\Omega_{10}^-(q))$-conjugate in
\begin{itemize}
\item $M\cong \Omega_{9}(q).2\in\mathcal{C}_1$ with $q$ odd, or in
\item $M\cong (\Omega_2^+(q)\times \Omega_{8}^-(q)).2^{\gcd(2,q-1)}\in \mathcal{C}_1$.
\end{itemize}

In the remaining two possibilities for $\tilde K$ we have that $|\tilde K|\in \{2(q^4-1)(q^8-1), 4(q^4-1)(q^8-1)\}.$
We show that $q^3-1$ does not divide $4(q^4-1)(q^8-1)$, so that   $\tilde K$ does not contain $x$ or any of its conjugates.
Let $r\in P_3(q)$. Then, by \eqref{boundppd}, $ r\geq 5$ so that $r\nmid 4(q^4-1)(q^8-1).$  By the above considerations on $r\in P_3(q)$, we also have that $r\nmid (q^2-1)(q^4-1)(q^5+1)$. Moreover, by Lemma ~\ref{aritme}, we also have $r\nmid q^3+1$ and hence $x$ has no $\mathrm{Aut}(\Omega_{10}^-(q))$-conjugate in $\tilde{H}$.

 Therefore $\gamma_w(\mathrm{P}\Omega_{10}^-(q))\ge 3$.
\end{proof}

\begin{lemma}\label{dimension12orthogonal-}
The weak normal covering number of $\mathrm{P}\Omega_{12}^-(q)$ is at least $3$.
\end{lemma}
\begin{proof}
We have verified the veracity of this lemma with a computer when $q\in \{2,3\}$ and hence in the rest of the proof we assume $q\ge 4$.
As usual we may work with $\Omega_{12}^-(q)$. Let $\tilde \mu$ be a weak normal covering of $\Omega_{12}^-(q)$ of minimum cardinality and maximal components and let $\tilde H$ be a  component containing a Singer cycle. 
 From Lemma~\ref{msw-}, we obtain
$$\tilde H\cong \Omega_4^-(q^3).3\,\hbox{ or }\,\tilde H\cong \Omega_6^-(q^2).2.$$

Let $z\in \Omega_{12}^-(q)$ having order $\frac{(q^5+1)(q-1)}{\gcd(2,q-1)}$ and action type $10\oplus 1\oplus 1$. The existence of such an element is guaranteed by Lemma \ref{lemma:new2022}.

The group $\tilde H$ does not contain elements having this order because $q^5+1$ does not divide $|\tilde H|$: this can be easily verified using the structure of $\tilde{H}$ and  considering a primitive prime divisor of $q^{10}-1$.
In particular, $\gamma_w(\Omega_{12}^-(q))\ge 2$.

Let $\tilde K$ be a component of $\tilde \mu$ containing an $\mathrm{Aut}(\Omega_{12}^-(q))$-conjugate of $z$.
Using the fact that $q\ge 4$, a case-by-case analysis on the maximal subgroups of $\Omega_{12}^-(q)$ in~\cite[Tables~8.84,~8.85]{bhr} reveals that the maximal subgroups $M$ having order divisible by $q^5+1$ are as follows
\begin{enumerate}
\item $M\cong E_{q}^{10}:\left(\frac{q-1}{\gcd(2,q-1)}\times\Omega_{10}^-(q)\right).\gcd(2,q-1)\in\mathcal{C}_1$,
\item $q$ is even and $M\cong \mathrm{Sp}_{10}(q)\in\mathcal{C}_1$
\item $q$ is odd and $M\cong \Omega_{11}(q).2\in\mathcal{C}_1$,
\item $M\cong (\Omega_2^+(q)\times \Omega_{10}^-(q)).2^{\gcd(2,q-1)}\in \mathcal{C}_1$.
\end{enumerate}
In particular, $\tilde{K}$ is one of these groups and in all cases $\tilde{K}$ is in the Aschbacher class $\mathcal{C}_1$. Furthermore, in the first two cases $\tilde{K}$ fixes a non-zero totally isotropic subspace of dimension at most $2$ and in the last two cases $\tilde{K}$ fixes a non-degenerate subspace of dimension at most $2$.

Assume $\gamma_w(\Omega_{12}^-(q))=2$. Let $y$ be an element of $\Omega_{12}^-(q)$ having order $$\frac{(q^3-1)(q^3+1)}{\gcd(q^3-1,q^3+1)}=\frac{q^6-1}{\gcd(2,q-1)},$$ action type $6\oplus 3\oplus 3$ and decomposing $V=U_1\perp (U_2\oplus U_3)$, with $U_1,U_2,U_3$ irreducible $\mathbb{F}_q\langle y\rangle$-modules, $\dim_{\mathbb{F}_q}U_1=6$, $\dim_{\mathbb{F}_q}U_2=\dim_{\mathbb{F}_q}U_3=3$, the orthogonal form of $\Omega_{12}^-(q)$ restricted to $U_1$ has Witt defect $1$ and restricted to $U_2\oplus U_3$ has Witt defect $0$, where $U_2$ and $U_3$ are totally isotropic.
The existence of such an element is guaranteed by Lemma \ref{lemma:new2022}.

Now, $y$ does not fix non-zero subspaces of dimension $\le 2$ and hence $y$ has no $\mathrm{Aut}(\Omega_{12}^-(q))$-conjugate in $\tilde{K}$. On the other hand, by the usual considerations on a primitive prime divisor $r$ of $q^6-1$,  we see that $r$ divides the order of $y$ but does not divide the order of $\Omega_6^-(q^2).2$. Note that in the critical case for Zsigmondy theorem, that is, $q^6-1$ has no primitive prime divisors,  we have that $\order y=63$ and $7\nmid |\Omega_6^-(4).2|.$
Since $\gamma_w(\Omega_{12}^-(q))=2$, we have $$\tilde{H}\cong \Omega_4^-(q^3).3.$$

Let $x$ be an element of $\Omega_{12}^-(q)$ having order $$\frac{(q^5-1)(q+1)}{\gcd(q^5-1,q+1)}=\frac{(q^5-1)(q+1)}{\gcd(2,q-1)},$$ action type $5\oplus 5\oplus 2$ and decomposing $V=U_1\perp (U_2\oplus U_3)$, with $U_1,U_2,U_3$ irreducible $\mathbb{F}_q\langle x\rangle$-modules, $\dim_{\mathbb{F}_q}U_1=2$, $\dim_{\mathbb{F}_q}U_2=\dim_{\mathbb{F}_q}U_3=5$, the orthogonal form of $\Omega_{12}^-(q)$ restricted to the non-degenerate submodule  $U_1$ is of Witt defect $1$ and $x$ induces on $U_1$ a Singer cycle of order $q+1$, while  the orthogonal form of $\Omega_{12}^-(q)$ restricted to $U_2\oplus U_3$ is of Witt defect $0$, where $U_2$ and $U_3$ are totally isotropic and $x$ induces on them a matrix of order $q^5-1$. The existence of such an element is guaranteed by Lemma \ref{lemma:new2022}.

Now, $x$ does not fix non-zero totally isotropic subspaces of dimension at most  $ 2$ and hence $x$ has no $\mathrm{Aut}(\Omega_{12}^-(q))$-conjugate in
\begin{itemize}
\item $M\cong E_{q}^{10}:\left(\frac{q-1}{\gcd(2,q-1)}\times\Omega_{10}^-(q)\right).\gcd(2,q-1)\in\mathcal{C}_1$, or in
\item $M\cong \mathrm{Sp}_{10}(q)\in\mathcal{C}_1$ with $q$ even.
\end{itemize}
Similarly, $x$ does not fix a non-degenerate subspace of dimension $1$ and hence $x$ has no $\mathrm{Aut}(\Omega_{12}^-(q))$-conjugate in
\begin{itemize}
\item $M\cong \Omega_{11}(q).2\in\mathcal{C}_1$ with $q$ odd.
\end{itemize}
Finally, the only $2$-dimensional non-degenerate subspace of $V$ fixed by $x$ is of ``minus type'' by construction and hence $x$  has no $\mathrm{Aut}(\Omega_{12}^-(q))$-conjugate in
\begin{itemize}
\item $M\cong (\Omega_2^+(q)\times \Omega_{10}^-(q)).2^{\gcd(2,q-1)}\in \mathcal{C}_1$.
\end{itemize}
Therefore, regardless of the structure of $\tilde{K}$, we have shown that  $x$  has no $\mathrm{Aut}(\Omega_{12}^-(q))$-conjugate in $\tilde{K}$.

On the other hand, the group $\tilde H\cong \Omega_4^-(q^3).3$ does not contain elements having the order  of $x$, because $q^5-1$ does not divide $|\tilde H|$ by  the usual considerations on a primitive prime divisor of $q^5-1$.
 Therefore $\gamma_w(\mathrm{P}\Omega_{12}^-(q))\ge 3$.
\end{proof}

\subsection{Large even dimensional orthogonal groups of Witt defect 1}\label{sec:largeevenorthogonal1}

%\begin{lemma}\label{phi2}
%Let $m$ be an integer with $m\ge 7$  and $m\neq 9$. Then there exists an odd integer $\ell$ with %$1<\ell<m/2$ and $\gcd(m,\ell)=1$.
%\end{lemma}
%\begin{proof}
%We let $\Phi$ denote the Euler's totient function.
% Assume, by contradiction, that  there exists no odd integer $\ell$ with $1<\ell<m/2$ and $\gcd(\ell,m)=1$. Then, for every $x $ with $1<x<m/2$ and coprime to $m$, we have that  $x$ is a power of $2.$ Indeed if there is $x $ with $1<x<m/2$, coprime to $m$ and divisible by some odd prime $p$, then we also have $1<p<m/2$ and $p$ coprime to $m$, against our assumption. It follows that those numbers are at most $\lfloor\log_2(m/2)\rfloor$.
%Look now to the numbers $y$ with $m/2<y<m$ coprime to $m$. Since $1<m-y<m/2$ and $m-y$ is coprime to $m$, we have that $m-y$ is a power of $2$. It follows that those $y$ are also at most $\lfloor\log_2(m/2)\rfloor$. Therefore, $\Phi(m)\le 2(\lfloor\log_2(m)-1\rfloor)+1=2\lfloor\log_2(m)\rfloor-1$. From~\cite{HW}, we have $\Phi(m)\ge \sqrt{m/2}$. However, the inequality  $\sqrt{m/2}\le 2\lfloor\log_2(m)\rfloor-1$  is valid only when $m\le 614$. When $m\le 614$, we have verified the veracity of the statement with the help of a computer.
%\end{proof}
%Note that $m\leq 6$ and $m=9$ are genuine exceptions in Lemma \ref{phi2}.
We start this section with a technical lemma; in what follows, we only need a few special cases of Lemma~\ref{o-c3}.  

\begin{lemma}\label{o-c3}
Let $m$ be an integer with $m\ge 6$, let $\ell$ be odd with $1\le \ell\le m$ and $\gcd(m,\ell)=1$, and let $x\in \Omega_{2m}^-(q)$ with $$\order x=\frac{(q^{\ell}-1)(q^{m-\ell}+1)}{\gcd(q^\ell-1,q^{m-\ell}+1)}.$$ If  $x\in M$ with $M$ a maximal subgroup of  $\Omega_{2m}^-(q)$ in class $\mathcal{C}_3$, then $m$ is odd, $\ell< m/2$ and $M\cong\Omega_m(q^2).2$. 
\end{lemma}
Note that the existence of at least one  $x\in\Omega_{2m}^-(q)$ of the above order is guaranteed by Lemma~$\ref{lemma:new2022}$. 
\begin{proof}
From Lemma~\ref{aritme}~\eqref{eq:arithme3}, we have $\gcd(q^\ell-1,q^{m-\ell}+1)=\gcd(2,q-1)$ because $\ell$ is odd. Therefore,
$$\order x=\frac{(q^{\ell}-1)(q^{m-\ell}+1)}{\gcd(2,q-1)}.$$

 Suppose that $x\in M$, where $M$ is a maximal subgroup of $\Omega_{2m}^-(q)$ in the Aschbacher class $\mathcal{C}_3$. From~\cite{kl}, we have three cases to consider:
\begin{itemize}
\item $m$ is odd and $M$ is of type $\mathrm{GU}_m(q)$,
\item $M\cong \Omega_{2m/s}^-(q^s).s$, where $s$ is a prime divisor of $m$ with $m/s\ge 2$,
\item $m$ is odd and $M\cong \Omega_m(q^2).2$.
\end{itemize}
We show that the first two cases cannot occur and, if the third case does occur, then $\ell< m/2$. We divide the proof into two cases, depending on whether $\ell>m/2$ or $\ell<m/2$. Observe that the case $\ell=m/2$ is excluded because $\gcd(m,\ell)=1$.

\bigskip
\noindent\textsc{Case $\ell>m/2$.}
\smallskip

\noindent As $\ell$ is odd and $\ell>m/2\geq 3$, from Zsigmondy's theorem, $P_{\ell}(q)\ne \varnothing$ and hence $x$ is a strong $ppd(2m,q;\ell)$-element. Let $r\in P_{\ell}(q)$. Then $r$ divides $|M|$ and, by ~\eqref{boundppd}, $r\geq \ell+1\geq 5.$

Suppose first that $m$ is odd and $M$ is of type $\mathrm{GU}_m(q)$. As above, we have
$$M\cong \frac{q+1}{a}\cdot \mathrm{PSU}_m(q).\gcd(q+1,m),$$
where $a:=\gcd(q+1,4)$. Thus
\begin{itemize}
\item $r$ divides $q^k+1$, for some odd $k$ with $1\le k\le m$, or
\item $r$ divides $q^k-1$ for some even $k$ with $1\le k\le m$.
\end{itemize} In the first case, by Lemma~\ref{aritme}~\eqref{eq:arithme3}, $r$ divides $\gcd(q^\ell-1,q^k+1)=\gcd(2,q-1)$, against $r\geq 5.$ In the second case, by Lemma~\ref{primitivi}~\eqref{primitivi3}, $\ell$ divides $k$. As $\ell$ is odd and $k$ is even, we deduce that $2\ell$ divides $k$. Thus $m<2\ell\le k\le m$, which is a contradiction. Therefore, this case does not arise.

Suppose next  that $M\cong \Omega_{2m/s}^-(q^s).s$, for some prime divisor  $s$ of $m$ with $m/s\ge 2$. Thus
\begin{itemize}
\item $r$ divides $q^m+1$, or
\item $r$ divides $q^{2sk}-1$ for some $k$ with $1\le k\le m/s-1$, or
\item $r=s$.
\end{itemize}
If $r$ divides $q^m+1$, then using the fact that $\ell$ is odd, by Lemma~\ref{aritme}~\eqref{eq:arithme3}, we deduce that $r$ divides $\gcd(q^\ell-1,q^k+1)=\gcd(2,q-1)$, which is impossible.

If $r$ divides $q^{2sk}-1$, then by Lemma \ref{primitivi}\,(3), we have $\ell \mid 2sk$. As $\ell$ is odd, we get $\ell\mid sk$.
Since $sk<m$ and $\ell>m/2$, we get $sk=\ell$. Thus $s\mid \gcd(m,\ell)=1$, which is a contradiction.

Suppose finally that $M\cong \Omega_m(q^2).2$ with $m$ odd. Since $r\ne 2$ and since $r$ is relatively prime to $q$, we have that
\begin{center}
\item $r$ divides $q^{4k}-1$ for some $k$ with $1\le k\le (m-1)/2$.
\end{center}
Then, by Lemma \ref{primitivi}\,(3), we have $\ell \mid 4k$. As $\ell$ is odd, we get $\ell\mid k$ and hence $m/2<\ell\le k\le (m-1)/2$, which is clearly a contradiction.

\bigskip
\noindent\textsc{Case $\ell<m/2$.}
\smallskip

\noindent As $m\ge 6$ and $1\le \ell<m/2$, we have  $2m-2\ell\geq m+1\geq 7$  and thus, from Zsigmondy's theorem, $P_{2m-2\ell}(q)\ne \varnothing$ and hence $x$ is a strong $ppd(2m,q;2m-2\ell)$-element. Let $r\in P_{2m-2\ell}(q)$. Then $r$ divides $|M|$.

Suppose first that $m$ is odd and $M$ is of type $\mathrm{GU}_m(q)$. From~\cite[Proposition~$4.3.18$]{kl}, we have
$$M\cong \frac{q+1}{a}\cdot \mathrm{PSU}_m(q).\gcd(q+1,m),$$
where $a:=\gcd(q+1,4)$. Thus 
\begin{itemize}
\item $r$ divides $q^k+1$, for some odd $k$ with $1\le k\le m$, or
\item $r$ divides $q^k-1$ for some even $k$ with $1\le k\le m$.
\end{itemize} 
 In the first case, as $r$ is a primitive prime divisor of $q^{2m-2\ell}-1$, by Lemma \ref{primitivi}\,(3), we get $2m-2\ell\mid 2k$ so that the even number $m-\ell$ divides the odd number $k$, which is absurd.

In the second case, using again the fact that $r$ is a primitive prime divisor of $q^{2m-2\ell}-1$, we get  $k\ge 2m-2\ell>m$, against $k\le m.$

Therefore, this case does not arise.

Suppose next that $M\cong \Omega_{2m/s}^-(q^s).s$, for some prime divisor  $s$ of $m$ with $m/s\ge 2$. Thus
\begin{itemize}
\item $r$ divides $q^m+1$, or
\item $r$ divides $q^{2sk}-1$ for some $k$ with $1\le k\le m/s-1$, or
\item $r=s$.
\end{itemize}
 When $r$ divides $q^m+1$, note that  by Lemma \ref{primitivi}\,(3), we have $2m-2\ell\mid 2m$ and thus $m-\ell\mid m$. Since $m-\ell>m/2$, this implies $m-\ell=m$, which is clearly impossible.

 When $r$ divides $q^{2sk}-1$, observe that by Lemma \ref{primitivi}\,(3), we have $2m-2\ell\mid 2sk$ and thus $m-\ell\mid sk\leq m$. Since $m-\ell>m/2$, this implies $m-\ell=sk$. As $s$ divides $m$, we deduce that $s$ divides $\ell$. However, as $\gcd(\ell,m)=1$, this yields $s=1$, contradicting the fact that $s$ is a prime number.

 When $r=s$, from~\eqref{boundppd}, we get $s\ge 2(m-\ell)+1>m+1$ against the fact that $s\leq m$.

Finally, when $m$ is odd and $M\cong\Omega_{m}(q^2).2$, we reach the conclusion of the
lemma.

\end{proof}

\begin{proposition}\label{ortogonali-}
For every $n\ge 14$, the weak normal covering number of $\mathrm{P}\Omega^-_{n}(q)$ is at least $3$.
\end{proposition}
\begin{proof}
%When $(n,q)=(14,2),$ the proof follows with a computer computation.

As usual we may argue with $\Omega_n^-(q)$.
 Let $\tilde H$ be a maximal component of a weak normal covering of $\Omega_n^-(q)$ containing a Singer cycle. From Lemma~\ref{msw-}, we have that $\tilde H$  is in class $\mathcal{C}_3$ and
\begin{itemize}
\item $\tilde H\cong\Omega_{n/s}^-(q^s).s$ for some prime divisor $s$ of $n/2$ with $n/s\ge 4$, or
\item $\tilde H$ is of type $\mathrm{GU}_{\frac{n}{2}}(q)$ and $n/2$ is odd.
\end{itemize}

Let  $x\in \Omega_n^-(q)$  with
$$\order x={\displaystyle\frac{(q^{\frac{n}{2}-1} +1)(q-1)}{\gcd(q^{\frac{n}{2}-1}+1,q-1)}}=
{\displaystyle\frac{(q^{\frac{n}{2}-1} +1)(q-1)}{\gcd(2,q-1)}}$$
and action type $(n-2)\oplus 1\oplus 1$ on the natural module $V$.  The existence of such an element is guaranteed by Lemma \ref{lemma:new2022}.

For later, we need to observe that, by Section \ref{Huppert}, the quadratic form of $\Omega_n^-(q)$ restricted to the $(n-2)$-dimensional $x$-invariant subspace of $V$ has Witt defect $1$.
As $\tilde H$ is  in class $\mathcal{C}_3$, Lemma~\ref{o-c3} applied with $\ell:=1$ yields that $\tilde H$ contains no elements having order $\order x$. 

Therefore, there exists a second component, $\tilde K$ say, containing an $\mathrm{Aut}(\Omega_n^-(q))$-conjugate of $x$. This shows $\gamma_w(\Omega_n^-(q))\ge 2$. For the rest of the proof, we suppose that $\gamma_w(\Omega_n^-(q))=2$ and that $\tilde H,\tilde K$ are the maximal  components of a weak normal $2$-covering.

As $n\ge 14$, $P_{n-2}(q)\ne \varnothing$ and hence
\begin{align*}
\tilde{K} \textrm{ is a }ppd(n,q;n-2)\textrm{-group.}
\end{align*}

Let $\ell$ be a prime number with $n/4<\ell<n/2$.
Observe that the existence of $\ell$ is guaranteed by Bertrand's postulate  and that obviously $\ell$ is odd. Let $y\in \Omega_n^-(q)$ with
$$\order y=\frac{(q^{\frac{n}{2}-\ell}+1)(q^\ell-1)}{\gcd(q^{\frac{n}{2}-1},q^\ell-1)}=\frac{(q^{\frac{n}{2}-\ell}+1)(q^\ell-1)}{\gcd(2,q-1)}$$
and type $\ell\oplus\ell\oplus (n-2\ell)$. The existence of such an element is guaranteed by Lemma \ref{lemma:new2022}.

As $\tilde H$ is in class $\mathcal{C}_3$ and $\ell>n/4$, it follows from Lemma~\ref{o-c3} that $y$  does not have a conjugate in $\tilde H$. Hence $y$ has an $\mathrm{Aut}(\Omega_n^-(q))$-conjugate  in $\tilde K$. From Zsigmondy's theorem, we have that $P_{\ell}(q)\ne\varnothing$, $y$ is a strong $ppd(n,q;\ell)$-element and 
\begin{align*}
\tilde{K} \textrm{ is a }ppd(n,q;\ell)\textrm{-group.}
\end{align*}
As $n-\ell>3$, by Theorem~\ref{main}, Lemma~\ref{no-c5}\,(3) and Table $3.5.$ F in
\cite{kl}, $\tilde K$ belongs to one of the classes $\,\mathcal C_i\,$ for some
$i\in \{1,2,3\}$ or to $\mathcal S$ and it is described in the Example~2.6 a) of~\cite{gpps}.

\smallskip

\noindent\textsc{The subgroup $\tilde K$ lies in class $\mathcal{C}_1$. }Here $\tilde K$ is the stabilizer of a totally singular or of a non-degenerate subspace of $V$. 

Recall $V=U_1\oplus U_2\oplus U_3$, with $U_1,U_2,U_3$ irreducible $\mathbb{F}_q\langle x\rangle$-submodules of $V$, $\dim_{\mathbb{F}_q}U_1=n-2$, $\dim_{\mathbb{F}_q}U_2=\dim_{\mathbb{F}_q}U_3=1$  and with the restriction of the quadratic form on $U_1$ having Witt defect $1$.

Now, we use the fact that $\tilde K$ contains an $\mathrm{Aut}(\Omega_n^-(q))$-conjugate of $x$. From the paragraph above, we deduce that one of the following holds
\begin{itemize}
\item $\tilde K$ is a parabolic subgroup $P_1\in\mathcal{C}_1$,
\item $q$ is even and $\tilde K\cong\mathrm{Sp}_{n-2}(q)\in\mathcal{C}_1$,
\item $q$ is odd and $\tilde K$ is of type $\mathrm{O}_1(q)\perp\mathrm{O}_{n-1}(q)\in\mathcal{C}_1$,
\item $\tilde K$ is of type $\mathrm{O}_2^+(q)\perp\mathrm{O}_{n-2}^-(q)\in\mathcal{C}_1$.
\end{itemize}
 However, since $\ell,2\ell,n-2\ell,n-\ell\notin\{1,2,n-1,n-2\}$, $y$ stabilizes no subspaces of $V$ of dimension $1$ or $2$. Therefore, this case does not arise.
\smallskip

\noindent\textsc{The subgroup $\tilde K$ lies in class $\mathcal{C}_2$. } These are the groups in the Example
2.3 of \cite{gpps}. Hence $\tilde K\leq \mathrm{GL}_1(q) \mathrm{wr} S_n$  is a $ppd(n,q;n-2)$-group of order divisible by the unique primitive prime divisor of $q^{n-2}-1$ given by $n-1\geq 13$.
Moreover, by  Proposition 4.2.15 in \cite{kl},  we have that $q\equiv 3\pmod 4$
 is a prime, $n\equiv 2 \pmod4$ and $\tilde K\leq E_2^{n}.\,S_n.$ Let $n=2+4k$ and note that, since $n\geq 14$, we have $k\geq 3.$ By the fact that $\tilde K$ contains $x$, it follows that $\tilde K$ contains an element of order $q^{\frac{n}{2}-1} +1$. Obviously $n-1\mid q^{\frac{n}{2}-1} +1=q^{2k} +1.$ Now note that  $q^{2k} +1\equiv 3^{2k} +1\pmod 4 \equiv 2\pmod 4$ so that $2$ is the maximum power of $2$ dividing  $q^{2k} +1.$
Moreover, $q^{2k} +1>2(n-1)=8k+2$ because we  have
 $$q^{2k} +1\geq 3^{2k} +1>8k+2,$$
 for every  $k\geq 3.$

 Thus there exists an odd prime $r$ with $r(n-1)\mid q^{2k} +1$, and hence such that $r(n-1)\mid  \order x$. Hence we also have an element of order $r(n-1)$ in $E_2^{n}.\,S_n.$ It follows that there is an element of such an order in $S_n$. Recall that $n-1$ is prime, see~\cite[Example~$2.3$]{gpps}.  When $r\neq n-1$, this implies $n\geq r+n-1\geq n+2$, which is a contradiction; when $r=n-1$, this implies $(n-1)^2\leq n$, which is again a contradiction.

\smallskip

\noindent\textsc{The subgroup $\tilde K$ lies in class $\mathcal{C}_3$. } Here we use the fact that $\tilde K$ contains an $\mathrm{Aut}(\Omega_n^-(q))$-conjugate of $y$. We conclude that this case cannot arise by Lemma~\ref{o-c3}, because $\ell>n/4$.

\smallskip

\noindent\textsc{The subgroup $\tilde K$ lies in class $\mathcal{S}$.} Since $\tilde K$ is described in  Example 2.6 a) of~\cite{gpps}, we have $$ A_m\le \tilde K \le S_m
\times \Z{\Omega_n^-(q)},$$ with $m\in\{ n+1,n+2\}$. 
In particular,  since ${\bf Z}(\Omega_n^-(q))$ divides $2$ by~\cite[Table 2.1.D]{kl}, we deduce $x^2\in S_m$. In particular,  $S_m$ contains an element having order
 $$\order {x^2}=\frac{{\bf o}(x)}{2}=\frac{(q^{\frac{n}{2}-1}+1)(q-1)}{4},$$
 when $q$ is odd, and having order  $$\order {x^2}=\order x=(q^{\frac{n}{2}-1}+1)(q-1),$$
 when $q$ is even. In both cases, when $q\neq 3$, we have that  $S_m$ contains an element of order at least $q^{\frac{n}{2}-1}+1$; when $q=3$, we have that $S_m$ contains an element of order $(q^{\frac{n}{2}-1}+1)/2$.
 Set $c:=1$ when $q\ne 3$ and $c:=2$ when $q=3$.
 Arguing as in the symplectic case, we deduce
\begin{align}\label{equation22}
 \log\left(\frac{q^{\frac{n}{2}-1}+1}{c}\right)\le\sqrt{m\log m}\left(1+\frac{\log(\log (m))-0.975}{2\log (m)}\right).
 \end{align}
 This inequality holds true only when $n\le 36$ and $q=2$, or $q=3$ and $n\le 16$. We have checked these cases  with a computer and, when $m\in \{n+1,n+2\}$, $S_m$ contains no elements having order $2^{\frac{n}{2}-1}+1$ when $14\le n\le 36$ and having order $(3^{\frac{n}{2}-1}+1)/2$ when $14\le n\le 16$. Therefore, also this case does not arise.
\end{proof}

Summing up, when $n\ge 8$, $\mathrm{P}\Omega_n^-(q)$ has weak normal covering number at least $3$.

\section{Orthogonal groups with Witt defect $0$}\label{sec:orthogonal+}
In this section, we deal with even dimensional orthogonal groups $\mathrm{P}\Omega_n^+(q)$ of Witt defect $0$, with $n\geq 8$. Note that the case $n=6$ is considered in Section~\ref{sec:unitary}  because  $\mathrm{P}\Omega_6^+(q)\cong \mathrm{PSU}_4(q)$. The cases $n\in \{2,4\}$ are not considered because the groups $\mathrm{P}\Omega_4^+(q)$ and $\mathrm{P}\Omega_2^+(q)$ are not simple.

When $n\ne 8$, we may work with $\Omega_n^+(q)$ because the automorphism group of $\Omega_n^+(q)$ projects onto the automorphism group of $\mathrm{P}\Omega_n^+(q)$. The case $n=8$ is rather special. Triality automorphisms of $\mathrm{P}\Omega_8^+(q)$ cannot be lifted to  automorphisms of $\Omega_8^+(q)$ (unless $q$ is even, because in that case $\Omega_8^+(q)=\mathrm{P}\Omega_8^+(q)$). Moreover, triality automorphisms do not act on $V$ and hence, for instance, an element and its image under a triality automorphism might have a rather different action on $V$.

Before embarking in the proofs of our main result, we define a family of elements in $\Omega_n^+(q)$.
Given an even integer $m$ with $2\leq m\leq n/2$, we consider the embedding of
$\mathrm{SO}^-_{m}(q)\perp \mathrm{SO}^-_{n-m}(q)$ in $\mathrm{SO}^{+}_{n}(q)$. Let $s_m\in \mathrm{SO}^-_{m}(q)$ and $s_{n-m}\in \mathrm{SO}^-_{n-m}(q)$ be Singer cycles 
 having order $q^{m/2}+1$ and $q^{n/2-m/2}+1$, respectively. By Proposition~\ref{spinor-ber-prop}, we have that 
  $x:=s_m\oplus s_{n-m}\in\Omega_n^+(q)$ has action type $m\oplus(n-m)$ and
$${\bf o}(x)=\frac{(q^{\frac{m}{2}}+1)(q^{\frac{n}{2}-\frac{m}{2}}+1)}{\gcd(q^{\frac{m}{2}}+1,q^{\frac{n}{2}-\frac{m}{2}}+1)}.$$
We will need the element $x$ defined above with $m=2$ in Lemma \ref{msw+} and with $m=4$ in Lemma \ref{max-y}.
Note that, by Lemma \ref{aritme}, when $m=2$, we have
$${\bf o}(x)=\frac{(q+1)(q^{\frac{n}{2}-1}+1)}{\gcd(q+1,q^{\frac{n}{2}-1}+1)}=
\begin{cases}
q^{\frac{n}{2}-1}+1&\textrm{if }\frac{n}{2} \textrm{ is even},\\
\frac{(q+1)(q^{\frac{n}{2}-1}+1)}{\gcd(2,q-1)}&\textrm{if }\frac{n}{2} \textrm{ is odd},\\
\end{cases}
$$
while, when $m=4$, we have
$${\bf o}(x)=\frac{(q^2+1)(q^{\frac{n}{2}-2}+1)}{\gcd(q^2+1,q^{\frac{n}{2}-2}+1)}=\frac{(q^2+1)(q^{\frac{n}{2}-2}+1)}{\gcd(2,q-1)}.
$$

\begin{lemma}\label {msw+} 
Let $M$ be a  maximal subgroup of $\Omega^+_{n}(q)$ with $n \ge 10$ containing an element $x$ of order
$\frac{(q+1)(q^{\frac{n}{2}-1}+1)}{\gcd(q+1, q^{\frac{n}{2}-1}+1)}$ and action type
$2\oplus (n-2)$.
Then, one of the following holds
\begin{enumerate}
\item\label {msw+1} $M\cong(\Omega_2^-( q)\perp\Omega^-_{n-2}(q)).2^{\gcd(2,q-1)}\in\mathcal{C}_1$,
%\item\label {msw+2} $q=3$ and  $M\cong \Omega_{n-1}(3).2\in\mathcal{C}_1$,
\item\label {msw+3} $n/2$ is even and $M$ is of type $\mathrm{GU}_{n/2}(q).2\in\mathcal{C}_3$.
%\item\label {msw+4} $nq/2$ is odd and $M$ is of type $\Omega_{n/2}(q^2).2\in\mathcal{C}_3$.
\end{enumerate}
\end{lemma}
For the detailed structure of $M$ in part~\eqref{msw+3} see~\cite[Propositions~$4.3.18$,~$4.3.20$]{kl}. In particular, in~\eqref{msw+3}, we stress here the fact that $$|M|\hbox{ divides } |\mathrm{GU}_{n/2}(q).2|.$$
\begin{proof}We start by observing that~\cite[Theorem~$1.1$]{msw}
 describes the maximal subgroups of $\Omega_n^+(q)$ containing a low-Singer  cycle $s$,
 that is, an element of $\Omega_n^+(q)$ of order $q^{n/2-1}+1$ and action $(n-2)\oplus 1\oplus 1$ inducing a Singer cycle belonging to $\mathrm{O}_{n-2}^-(q)$ on its invariant subspace of dimension $n-2$ and the identity matrix on its invariant subspace of dimension $2$. Clearly both $x$ in the statement of this lemma and $s$ in \cite[Theorem 1.1]{msw} are strong $ppd(n,q;n-2)$-elements. Now, the proof of \cite[Theorem 1.1]{msw} just uses this fact. Hence we reach  the same conclusion for the maximal subgroups containing $s$ and those containing $x$: namely, they belong to $\mathcal{C}_1\cup \mathcal{C}_3$. Now, using  \cite[Table II]{msw}, we have that either part~\eqref{msw+3} holds or
 \begin{itemize}
 \item[$(\dag)$]$nq/2$ is odd and $M$ is of type $\Omega_{n/2}(q^2).2\in\mathcal{C}_3$. 
 \end{itemize} 
Furthermore, with a direct inspection on $\mathcal{C}_1$, we get  \eqref{msw+1}. Therefore, to conclude the proof we need to exclude the case described in~$(\dag)$: to do this we use the action type of $x$.

Suppose $nq/2$ is odd, $M$ is of type $\Omega_{n/2}(q^2).2\in\mathcal{C}_3$ and $M$ contains an element $x$  of order
$\frac{(q+1)(q^{\frac{n}{2}-1}+1)}{\gcd(q+1, q^{\frac{n}{2}-1}+1)}$ and action type
$2\oplus (n-2)$. As $nq/2$ is odd, from Lemma~\ref{aritme}, we have $${\order x}=\frac{(q+1)(q^{\frac{n}{2}-1}+1)}{2}.$$
Let $x':=x^2$ and let  $M_0$ be the subgroup of $M$ with $|M:M_0|=2$ and with $M_0\cong\Omega_{n/2}(q^2)$. Observe that $x'\in M_0$. Now, $x'$ fixes a $1$-dimensional  non-degenerate $\mathbb{F}_{q^2}$-subspace $U$ of $\mathbb{F}_{q^2}^{n/2}$. Let $(M_0)_U$ be the stabilizer in $M_0$ of $U$. In particular, $x'\in (M_0)_U$. The action of $x'$ on $U$ induces a matrix having order $(q+1)/2$ because the element $x$ induces by hypothesis a matrix having order $q+1$ on $U$. From~\cite[Proposition~4.1.6]{kl},  we have $(M_0)_U\cong \Omega_{n/2-1}^-(q^2).2$.  Therefore, the elements of $(M_0)_U$ induce in their action on $U$ elements having order $1$ or $2$.  Thus, we must have
$$\frac{q+1}{2}=2,$$
that is, $q=3$. Now that we have pinned down the exact value of $q$, we use the explicit description of $M$ in~\cite[Lemma~5.3.5]{burness}.  Indeed, since $q=3$, we have $q\equiv -1\pmod 4$ and hence~\cite[Lemma~5.3.5]{burness} gives $M\cong \mathrm{SO}_{n/2}(q^2)$. As $x\in M_U$ and $M_U\cong (\mathrm{SO}_1(q^2)\perp\mathrm{SO}_{n/2-1}^-(q^2)).2$, we deduce that the action of $x$ on $U$ induces a matrix of order $2$. However, this gives $q+1=2$, which is clearly a contradiction. 

From~\cite[Proposition~4.1.6]{kl}, we have $(\mathrm{GO}_{n/2}(q^2))_U\cong \Omega_{n/2-1}^-(q^2).2$. Therefore, the elements of $(\mathrm{GO}_{n/2}(q^2))_U$ induce in their action on $U$ elements having order $1$ or $2$. Thus, we must have
$$\frac{q+1}{2}=2,$$
that is, $q=3$. Now that we have pinned down the exact value of $q$, we use the explicit description of $M$ in~\cite[Lemma~5.3.5]{burness}.  Indeed, since $q=3$, we have $q\equiv -1\pmod 4$ and hence~\cite[Lemma~5.3.5]{burness} gives $M\cong \mathrm{SO}_{n/2}(q^2)$. 
As $x\in M_U$ and $M_U\cong (\mathrm{SO}_1(q^2)\perp\mathrm{SO}_{n/2-1}^-(q^2)).2$, we deduce that the action of $x$ on $U$ induces a matrix of order $2$. However, this gives $q+1=2$, which is clearly a contradiction.
\end{proof} 
In order to explain the notation in the rest of this section, we recall that for $n$ a positive integer, $[n]$ denotes an arbitrary group of order $n.$
\begin{lemma}\label{max-y}
Let $M$ be a maximal subgroup of $\Omega^{+}_{n}(q)$ with  $n \ge 10$  containing an element $y$ of order
$$\frac{(q^2+1)(q^{\frac{n}{2}-2} +1)}{\gcd( q^2 +1,q^{\frac{n}{2}-2}+1)}=\frac{(q^2+1)(q^{\frac{n}{2}-2}+1)}{\gcd(2,q-1)}$$
 and action type $4\oplus
(n-4)$. Then, one of the following holds
\begin{enumerate}
\item\label{max-y1} $M\cong(\Omega^{-}_4(q) \perp \Omega^{-}_{n -4} (q)).2^{\gcd(2,q-1)}\in\mathcal{C}_1$,
\item\label{max-y2} $n/2$ is even and $M\cong\Omega^+_{n/2}( q^2).[4]\in\mathcal{C}_3$.
\end{enumerate}
\end{lemma}
Note that the existence of such an element $y$ in $\Omega^{+}_{n}(q)$ is guaranteed by Proposition \ref{spinor-ber-prop}.
\begin{proof}
When $(n,q)= (10,2),$ an inspection in~\cite{atlas} shows that $M=(\Omega^{-}_4(2) \perp \Omega^{-}_{6} (2)).2$. Assume then that  $(n,q)\ne (10,2)$.
Then $P_{n-4}(q)\neq \varnothing$ and $y$ is a strong $ppd(n,q;n-4)$-element. By Theorem~\ref{main},  $M$ belongs to one of
the Aschbacher classes $\mathcal C_i$,  for $i\in\{1,\,2,\,3,\, 5\}$ or to $\mathcal{S}$
and in such case it is described in Example 2.6 a) of~\cite{gpps}.

\smallskip

\noindent\textsc{The subgroup $M$ lies in  class $\mathcal{C}_1$. }We use~\cite[Propositions~$4.1.6$,~$4.1.7$ and~$4.1.20$]{kl} for the structure of $M$.

The $\mathbb{F}_q\langle y\rangle$-module $V$ decomposes as the sum of two irreducible submodules, say $U_1$ and $U_2$. Now, $\dim_{\mathbb{F}_q}U_1=4$, $\dim_{\mathbb{F}_q}U_2=n-4$ and the quadratic form preserved by $\Omega_n^+(q)$ restricted to both $U_1$ and $U_2$ is non-degenerate with Witt defect  $1$. Using this information and~\cite{kl}, it is readily seen that $M\cong(\Omega^{-}_4(q) \perp \Omega^{-}_{n -4} (q)).2^{\gcd(2,q-1)}$ and part~\eqref{max-y1} holds.

\smallskip

\noindent\textsc{The subgroup $M$ lies in class $\mathcal{C}_2$. }Here $M$ is described in Example 2.3 of~\cite{gpps}. Thus
$$M \le \mathrm{O}_1(q) \,\mathrm{wr}\, S_n.$$
Let $r\in P_{n-4}(q)$ and $s\in P_{4}(q)$ and observe that $n-4\neq 4$ implies $r\ne s$. Moreover, we have $\gcd(rs,|\mathrm{O}_1(q)|)=1$. As $rs$ divides $\order y$, we deduce that $S_n$ contains an element of order $rs$. Therefore, from~\eqref{boundppd} and from the fact that $r$ and $s$ are distinct primes, we deduce $n\ge r+s>(n-4)+4=n$, which is a contradiction.

\smallskip

\noindent\textsc{The subgroup $M$ lies in  class $\mathcal{C}_3$. }We use~\cite[Propositions~$4.3.14$,~$4.3.18$ and~$4.3.20$]{kl} for the structure of $M$. One of the following holds
\begin{itemize}
\item $n/2$ is even and $M$ is of type $\mathrm{GU}_{n/2}(q)$,
\item $M$ is of type $\mathrm{O}_{n/s}^+(q^s)$ for some prime divisor $s$ of $n$, $n/s$ even and $n/s\ge 4$,
\item $qn/2$ is odd and $M$ is of type $\mathrm{O}_{n/2}(q^2)$.
\end{itemize}

We show that the first possibility cannot arise. Indeed pick $r\in P_{n-4}(q)$ and note that $r\neq p$, where $p$ is the characteristic of $\mathbb{F}_q$. As $r$ divides $\order y$, $r$ divides also $|M|$. But $r$ cannot divide the factors of $|M|$ of type $q^i-1$ for $i $ even with $i\in \{2,\dots, n/2\}$ because, since $n\geq 10$, we have $n/2<n-4$. Moreover, $r$ cannot divide the factors of $|M|$ of type $q^i+1$ for $i $ odd with $i\in \{1,\dots, n/2-1\}$ because this implies $n-4\mid 2i$. But since $n \equiv 0\pmod 4$, we have that $4\mid n-4$ and thus $2\mid i$, against $i$ odd.

 Also the third possibility does not arise. Indeed, if $M$ is of type $\mathrm{O}_{n/2}( q^2)$ with $n/2$ odd, then
$$\pi(|M|)=\pi\left(p \cdot \prod_{i=1}^{\frac{n-2}{4}}(q^{4i} - 1)\right).
$$
 Let $r\in P_{n-4}(q)$. As before $r\neq p$ and $r$ divides $|M|$. Hence $n-4$ divides $4i$, for some $i\in \{1,\ldots,(n-2)/4\}$. Now $n \equiv 2\pmod 4$, so that also $n-4 \equiv 2\pmod 4$. It follows that $n-4\mid 2i$, which implies $n-4\leq (n-2)/2$ and hence $n\le 6$. However, this is a contradiction.

 Finally, suppose $M$ is of type $\mathrm{O}_{n/s}^+(q^s)$ for some prime divisor $s$ of $n$, $n/s$ even and $n/s\ge 4$. Then 
 $$\pi(|M|)=\pi\left(s\cdot p\cdot (q^{\frac{n}{2}}-1)\cdot \prod_{i=1}^{\frac{n}{2s}-1}(q^{2is} - 1)\right).
$$

Pick $r\in P_{n-4}(q)$. By  $n\geq 10$ and by \eqref{boundppd}, $r\nmid s\, p(q^{\frac{n}{2}}-1)$. If $r$ divides $q^{2is}-1$ for some $i\in \{1,\ldots,n/(2s)-1\}$ then $n-4\le 2is\le n-2s$ and hence $s=2$. In particular, $n/2$ is even and $M=\Omega^+_{n/2}( q^2).[4]$, that is, part~\eqref{max-y2} is satisfied.

\smallskip

\noindent\textsc{The subgroup $M$ lies in  class $\mathcal{C}_5$. } This case is ruled out by Lemma \ref{no-c5}\,(3).
\smallskip

\noindent\textsc{The subgroup $M$ lies in  class $\mathcal{S}$. } Here $M$ is described in  Example 2.6 a) of~\cite{gpps} and thus
$$A_m\le M \le S_m \times \Z {\Omega_n^+(q)},$$
with $m\in\{ n +1,n+2\}$.
Thus the symmetric group $S_m$ contains an element having order
 $$\frac{\order y}{\gcd(2,q-1)}=\frac{(q^2+1)(q^{\frac{n}{2}-2}+1)}{\gcd(2,q-1)^2}.$$
 Arguing as in the symplectic case, we deduce
  \begin{align*}
 \log \left(\frac{(q^2+1)(q^{\frac{n}{2}-2}+1)}{\gcd(2,q-1)^2}\right)\le\sqrt{m\log m}\left(1+\frac{\log(\log (m))-0.975}{2\log (m)}\right).
 \end{align*}
For $n\ge 10$, this inequality holds true only when $q=2$ and $n\le 32$, or $q=3$ and $n\le 12$.
 For these cases, we have computed explicitly $\order y$ and the order of the elements of
 $S_{n+2}$ and we have verified that in no case $\frac{\order y}{\gcd(2,q-1)}$ is the order of a permutation in
 $S_{n+2}$.
\end{proof}
%%%%%%%%%%%%%%%%
%%%%%%%%%%%%%%%%%%%

\begin{proposition}\label{ortogonali +}
For every $n\ge 10$, the weak normal covering number of $\mathrm{P}\Omega_n^+(q)$ is at least $3$.
\end{proposition}
\begin{proof}
As $n\ne 8$, we may argue with $\Omega_n^+(q)$. When $(n,q)=(14,2)$, the veracity of the statement is confirmed with a computer computation; therefore, for the rest of our argument, we suppose $(n,q)\ne (14,2)$.

Let $\tilde H$ be a maximal component of a weak normal covering of $\Omega_n^+(q)$ containing an $\mathrm{Aut}(\Omega_n^+(q))$-conjugate of the  element $x$ as in Lemma~\ref{msw+}. Similarly, let $\tilde K$ be a maximal component of a weak normal covering of $\Omega_n^+(q)$ containing an $\mathrm{Aut}(\Omega_n^+(q))$-conjugate of the element $y$ as in Lemma~\ref{max-y}. It is readily seen in those lemmas that $\tilde H$ and $ \tilde K$ are not  $\mathrm{Aut}(\Omega_n^+(q))$-conjugate and hence $\gamma_w(\Omega_n^+(q))\ge 2$. Assume now, by contradiction, that $\gamma_w(\Omega_n^+(q))=2$. Thus a weak normal $2$-covering for $\Omega_n^+(q)$ is given by
$\tilde \mu=\{\tilde H,\tilde K\},$ where $\tilde H$ is one of the maximal subgroups appearing in Lemma~\ref{msw+} and $\tilde K$ is one of the maximal subgroups appearing in Lemma \ref{max-y}.

Let $g\in \Omega_n^+(q)$ with
$\order g=\frac{q^{\frac{n}{2}-1}-1}{\gcd(2,q-1)}$ and having type $(n/2-1)\oplus (n/2-1)\oplus 1\oplus 1$. This element $g$ arises from the embedding of $(\mathrm{O}_{n-2}^+(q)\perp\mathrm{O}_2^+(q))\cap \Omega_n^+(q)$ in $\Omega_n^+(q)$ and hence $V=U_1\oplus U_2\oplus U_3\oplus U_4$, where $U_1,U_2,U_3,U_4$ are irreducible $\mathbb{F}_q\langle g\rangle$-modules, $\dim_{\mathbb{F}_q}U_1=\dim_{\mathbb{F}_q}U_2=n/2-1$, $\dim_{\mathbb{F}_q}U_3=\dim_{\mathbb{F}_q}U_4=1$, $U_1,U_2,U_3,U_4$ are totally isotropic, $U_1\oplus U_2$ is non-degenerate and the orthogonal form preserved by $\Omega_n^+(q)$ restricted to $U_1\oplus U_2$ is of Witt defect $0$. 

Let $r$ be a primitive prime divisor of $q^{n/2-1}-1$: the existence of $r$ is guaranteed by Zsigmondy's theorem and by the fact that $(n,q)\ne (14,2)$. Observe that $r$ does not divide the order of $(\Omega_4^-(q)\perp\Omega_{n-4}^-(q)).2^{\gcd(2,q-1)}$ and of $\Omega_{n/2}^+(q^2).[4]$ when $n/2$ is even. Therefore, $g$ is not $\mathrm{Aut}(\Omega_{n}^+(q))$-conjugate to an element of $\tilde K.$

When $n/2$ is even, $r$ is also relatively prime to the order of $\mathrm{GU}_{n/2}(q).2$ and hence $\tilde H$ is not as in part~\eqref{msw+3} of Lemma~\ref{msw+}.

The only $2$-dimensional subspace of $\mathbb{F}_q^n$ left invariant by $g$ is $U_1\oplus U_2$. Since the quadratic form induced on $U_1\oplus U_2$ has Witt defect $0$, we deduce that $\tilde H$ is not as in part~\eqref{msw+1} of Lemma~\ref{msw+}. Therefore, $nq/2$ is odd and $\tilde H$ is of type $\Omega_{n/2}(q^2).2\in\mathcal{C}_3$. From Lemma~\ref{max-y}, we deduce that $\tilde K\cong(\Omega_4^-(q)\perp\Omega_{n-4}^-(q)).2^2$.

By~\cite{wall}, $\Omega_n^+(q)$ contains a unipotent element $u$  having Jordan blocks of size $n-1$ and $1$. Suppose that $u$ has an $\mathrm{Aut}(\Omega_n^+(q))$-conjugate in $\tilde H$. Since $q$ is odd, $\Omega_{n/2}(q^2)$ contains a unipotent element having Jordan blocks of size $n-1$ and $1$. Let the Jordan blocks of  $u$, viewed as a matrix of $\mathrm{GL}_{n/2}(\mathbb{F}_{q^2})$ have sizes $\ell_1,\ell_2,\ldots,\ell_s$. Then the Jordan blocks of  $u$, viewed as a matrix of $\mathrm{GL}_{n}(\mathbb{F}_{q})$ have sizes $\ell_1,\ell_1,\ell_2,\ell_2,\ldots,\ell_s,\ell_s$. However, this contradicts the fact that the Jordan blocks of  $u$, viewed as a matrix of $\mathrm{GL}_{n}(\mathbb{F}_{q})$ have sizes $n-1$ and $1$. Therefore,  $u$ has an $\mathrm{Aut}(\Omega_n^+(q))$-conjugate in $\tilde K\cong (\Omega_4^-(q)\perp\Omega_{n-4}^-(q)).2^2$. However, this is incompatible with the Jordan blocks of $u$.
\end{proof}

It remains to deal with the eight dimensional orthogonal group $\mathrm{P}\Omega_8^+(q)$.

\subsection{Eight dimensional orthogonal groups with Witt defect $0$} We recall that we have to take extra care in this case. In particular, the action of the automorphism group of $\mathrm{P}\Omega_8^+(q)$ on $\mathrm{P}\Omega_8^+(q)$ does not preserve the Jordan form of its elements. Thus we will be mainly arguing with $\mathrm{P}\Omega_8^+(q)$ and with element orders, rather than with $\Omega_8^+(q)$ and with action types. Nevertheless, it is is more convenient to give the structure of maximal subgroups of $\Omega_8^+(q)$, rather than of their projective image in $\mathrm{P}\Omega_8^+(q)$.

From the discussion at the beginning of the section, we have that, there exist semisimple elements $a,b\in\mathrm{SO}_8^+(q)$  such that
\begin{itemize}
\item $s:=ab\in\Omega_8^+(q)$,
\item $\order s=\order a=\order b=q^2+1$,
\item $s$ has action type $4\oplus 4$,
\item $\dim_{\mathbb{F}_q}\cent V a=\dim_{\mathbb{F}_q}\cent V b=4$, $V=\cent V a\perp \cent V b$ and
\item the quadratic form preserved by $\Omega_8^+(q)$ restricted to $\cent Va $ and $\cent V b$ is non-degenerate and of Witt defect $1$.
\end{itemize}

The projective image $\bar s$ of $s$ in $\mathrm{P}\Omega_8^+(q)$ has order $(q^2+1)/\gcd(2,q-1)$, because $\langle s\rangle$ contains the matrix $-I$, generator of $\Z{\Omega_8^+(q)}$.

We claim that, when $q>3$, we may choose $a$ and $b$ such that $a$ and $b$ are not in the same $\Omega_8^+(q)$-conjugacy class. The proof is very similar to the proof of Lemma~\ref{technical}, here we only give a sketch of the proof. Indeed, $|\nor {\Omega_8^+(q)}{\langle a\rangle}:\cent {\Omega_8^+(q)}{\langle a\rangle}|=4$ and the action of $\nor {\Omega_8^+(q)}{\langle a\rangle}$ on $\langle a\rangle$ is the cyclic action of the Galois group of $\mathbb{F}_{q^4}/\mathbb{F}_q$. In particular, as long as $4<\phi(\order a)$, we may always find $b$ not in the same $\Omega_8^+(q)$-conjugacy class of $a$. It follows from an easy computation that, if $\phi(x)\le 4$, then $x\le 12$.  Note now that  the only numbers $q^2+1$ that are less than or equal to $12$ arise when $q\in \{2,3\}$.

Suppose now that $q>3$  and let $s,a,b$ be as in the previous paragraphs and with $a$ and $b$ not in the same $\Omega_8^+(q)$-conjugacy class. A computation yields that
${\bf C}_{\mathrm{P}\Omega_8^+(q)}(\bar{s})$ is contained in the projective image of $(\mathrm{O}_4^-(q)\perp \mathrm{O}_4^-(q))\cap \Omega_8^+(q)$  in $\mathrm{P}\Omega_8^+(q)$. Hence
$${\bf C}_{\mathrm{P}\Omega_8^+(q)}(\bar{s})=\frac{(\langle a\rangle\times\langle b\rangle)\cap \Omega_8^+(q)}{{\bf Z}(\Omega_8^+(q))}$$
and
$$|\cent {\mathrm{P}\Omega_8^+(q)}{\bar{s}}|=\frac{(q^2+1)^2}{\gcd(2,q-1)}=
\begin{cases}
(q^2+1)^2 &\textrm{when }q\textrm{ is even,}\\
\frac{(q^2+1)^2}{2}&\textrm{when }q \textrm{ is odd.}
\end{cases}
$$

Using~\cite[Table~8.50]{bhr}, when $q>3$, we have selected in Table~\ref{o7q^2+1} the maximal subgroups $\tilde H$ of $\Omega_8^+(q)$ such that $H\le \mathrm{P}\Omega_8^+(q)$ contains a semisimple element $\bar s$ having order $(q^2+1)/\gcd(2,q-1)$ (regardless of the action of its lift $s$ to $\Omega_8^+(q)$ on $V$).

\begin{table}[ht]
\begin{tabular}{c|c|c||c|c|c}
\toprule[1.5pt]
\multicolumn{3}{c||}{$q$ odd, $q>3$}&\multicolumn{3}{|c}{ $q$ even, $q>2$} \\
\midrule[1.5pt]
Structure& class&nr. $\Omega_8^+(q)$-classes&Structure&class&nr. $\Omega_8^+(q)$-classes\\
\midrule[1.5pt]
$E_q^6:\left(\frac{q-1}{2}\times\Omega_6^+(q)\right).2$&$\mathcal{C}_1$&$1$&
$E_q^6:\left((q-1)\times\Omega_6^+(q)\right)$&$\mathcal{C}_1$&$1$\\
$E_q^{6}:\frac{1}{2}\mathrm{GL}_4(q)$&$\mathcal{C}_1$&$2$&
$E_q^{6}:\mathrm{GL}_4(q)$&$\mathcal{C}_1$&$2$\\
$2\times \Omega_7(q)$&$\mathcal{C}_1$&$2$& $\mathrm{Sp}_6(q)$&$\mathcal{C}_1$&$1$\\
$2^\cdot\Omega_7(q)$&$\mathcal{S}$&$4$& $\mathrm{Sp}_6(q)$&$\mathcal{S}$&$2$\\

$(\Omega_2^+(q)\times \Omega_6^+(q)).[4]$&$\mathcal{C}_1$&$1$&
$(\Omega_2^+(q)\times \Omega_6^+(q)).2$&$\mathcal{C}_1$&$1$\\

$\mathrm{SL}_4(q).\frac{q-1}{2}.2$&$\mathcal{C}_2$&$2$&
$\mathrm{SL}_4(q).(q-1).2$&$\mathcal{C}_2$&$2$\\

$(\Omega_2^-(q)\times \Omega_6^-(q)).[4]$&$\mathcal{C}_1$&$1$&
$(\Omega_2^-(q)\times \Omega_6^-(q)).2$&$\mathcal{C}_1$&$1$\\

$\mathrm{SU}_4(q).\frac{q+1}{2}.2$&$\mathcal{C}_3$&$2$&
$\mathrm{SU}_4(q).(q+1).2$&$\mathcal{C}_3$&$2$\\

$\Omega_4^-(q)^2.[4].S_2$&$\mathcal{C}_2$&1&
$\Omega_4^-(q)^2.2.S_2$&$\mathcal{C}_2$&1\\

$\Omega_4^+(q^2).[4]$&$\mathcal{C}_3$&2&
$\Omega_4^+(q^2).[4]$&$\mathcal{C}_3$&2\\

$2\times \Omega_8^-(\sqrt{q})$&$\mathcal{C}_5$&2&
$\Omega_8^-(\sqrt{q})$&$\mathcal{C}_5$&1\\

$2^\cdot \Omega_8^-(\sqrt{q})$&$\mathcal{S}$&4&
$\Omega_8^-(\sqrt{q})$&$\mathcal{S}$&2\\

$(\Omega_3(q)\times  \Omega_5(q)).[4]$&$\mathcal{C}_1$&2&&&\\
$(\mathrm{Sp}_2(q)\circ  \mathrm{Sp}_4(q)).2$&$\mathcal{C}_4$&4&&&\\
$q=5$, $2^\cdot \mathrm{Suz}(8)$&$\mathcal{S}$&8&&&\\

\bottomrule[1,5pt]
\end{tabular}
\caption{Maximal subgroups $\tilde H$ of $\Omega_8^+(q)$ such that $H$  contains a semisimple element of order $(q^2+1)/\gcd(2,q-1)$} \label{o7q^2+1}
\end{table}

Using Table~\ref{o7q^2+1}, it can be shown that the only maximal subgroups $\tilde H$ of $\Omega_8^+(q)$ such that $H\le\mathrm{P}\Omega_8^+(q)$ contains a semisimple element $\bar s$ with $\order {\bar s}=(q^2+1)/\gcd(2,q-1)$ and $|\cent {\mathrm{P}\Omega_8^+(q)}{\bar s}|=(q^2+1)^2/\gcd(2,q-1)$ are given in Table~\ref{o7q^2+11}.

\begin{table}[ht]
\begin{tabular}{c|c|c||c|c|c}
\toprule[1.5pt]
\multicolumn{3}{c||}{$q$ odd, $q>3$}&\multicolumn{3}{|c}{ $q$ even, $q>2$}\\
\midrule[1.5pt]
Structure& class&nr. $\Omega_8^+(q)$-classes&Structure&class&nr. $\Omega_8^+(q)$-classes\\
\midrule[1.5pt]
$\Omega_4^-(q)^2.[4].S_2$&$\mathcal{C}_2$&1&
$\Omega_4^-(q)^2.2.S_2$&$\mathcal{C}_2$&1\\
$\Omega_4^+(q^2).[4]$&$\mathcal{C}_3$&2&
$\Omega_4^+(q^2).[4]$&$\mathcal{C}_3$&2\\
$2\times \Omega_8^-(\sqrt{q})$&$\mathcal{C}_5$&2&
$\Omega_8^-(\sqrt{q})$&$\mathcal{C}_5$&1\\
$2^\cdot \Omega_8^-(\sqrt{q})$&$\mathcal{S}$&4&
$\Omega_8^-(\sqrt{q})$&$\mathcal{S}$&2\\
$q=5$, $2^\cdot \mathrm{Suz}(8)$&$\mathcal{S}$&8&&&\\

\bottomrule[1,5pt]
\end{tabular}
\caption{Maximal subgroups $\tilde H$ of $\Omega_8^+(q)$ such that $H$ contains a semisimple element of order $(q^2+1)/\gcd(2,q-1)$ and having small centralizer} \label{o7q^2+11}
\end{table}

We are now ready to prove our final lemma.

\begin{lemma}\label{dimension8orthogonal+}
The weak normal covering number of $\mathrm{P}\Omega_8^+(q)$ is at least $2$. Moreover, if $H$ and $K$ are maximal subgroups  of a weak normal $2$-covering of $\mathrm{P}\Omega_8^+(q)$, then $q\in \{2,3\}$ and, up to $\mathrm{Aut}(\mathrm{P}\Omega_8^+(q))$-conjugacy, $H$ and $K$ are one of the examples in Table~$\ref{000====}$.
\end{lemma}
\begin{proof}
When $q\in \{2,3,5\}$, we have verified the veracity of this lemma with a computer. Therefore, for the rest of the proof we suppose $q\notin \{2,3,5\}$.

Let $H$ be a maximal component of a weak normal $2$-covering of $\mathrm{P}\Omega_8^+(q)$ containing an element having order $(q^2+1)/\gcd(2,q-1)$ and having centralizer of order $(q^2+1)^2/\gcd(2,q-1)$. Then $\tilde H$ is described in Table~\ref{o7q^2+11}. In particular, the order of $H$ is given in Table~\ref{o7o7o7}.

\begin{table}[!h]
\begin{tabular}{c|c||c|c}
\toprule[1.5pt]
\multicolumn{2}{c||}{$q$ odd, $q>3$}&\multicolumn{2}{|c}{ $q$ even, $q>2$} \\
\midrule[1.5pt]
Order $H$& comments& Order $H$&comments\\
\midrule[1.5pt]
$2q^4(q^4-1)^2$&&
$4q^4(q^4-1)^2$&\\
$q^6(q-1)(q^3-1)(q^4+1)$&$q$ a square&
$q^6(q-1)(q^3-1)(q^4+1)$&$q$ a square\\
$29120$&$q=5$&&\\
\bottomrule[1,5pt]
\end{tabular}
\caption{Possible orders of $H$} \label{o7o7o7}
\end{table}

From the embedding of $\mathrm{O}_6^-(q)\perp\mathrm{O}_2^-(q)$ in $\mathrm{O}_8^+(q)$, using Proposition \ref{spinor-ber-prop}, we deduce that $\Omega_8^+(q)$ contains a semisimple element $s_-$ having order $q^3+1$ and  $\mathrm{P}\Omega_8^+(q)$ contains an element $\bar{s}_-$ having order $(q^3+1)/\gcd(2,q-1)$. Similarly, from the embedding of $\mathrm{O}_6^+(q)\perp\mathrm{O}_2^+(q)$ in $\mathrm{O}_8^+(q)$, 
using Lemma \ref{lemma:new2022+}, 
we deduce that $\Omega_8^+(q)$ contains a semisimple  element $s_+$ having order $q^3-1$ and $\mathrm{P}\Omega_8^+(q)$ contains an element $\bar{s}_+$ having order $(q^3-1)/\gcd(2,q-1)$. It is readily seen from Table~\ref{o7o7o7} that $|H|$ is not divisible by a primitive prime divisor of $q^6-1$ and hence $H$ does not contain an $\mathrm{Aut}(\mathrm{P}\Omega_8^+(q))$-conjugate of $\bar{s}_-$. Observe now that, if $H$ contains an $\mathrm{Aut}(\mathrm{P}\Omega_8^+(q))$-conjugate of $\bar{s}_+$, then $(q^3-1)/\gcd(2,q-1)$ divides $|H|$ and hence $\tilde{H}$ is isomorphic to either $2^.\Omega_8^-(\sqrt{q})$  or $2\times\Omega_8^-(\sqrt{q})$ when $q$ is odd, or to $\Omega_8^-(\sqrt{q})$ when $q$ is even. However, these groups do not contain elements having order as large as $(q^3-1)/\gcd(2,q-1)=(\sqrt{q}^6-1)/\gcd(2,\sqrt{q}-1)$. Therefore, $H$ does not contain  an $\mathrm{Aut}(\mathrm{P}\Omega_8^+(q))$-conjugate of $\bar{s}_+$.

The previous paragraph shows that $\gamma_w(\mathrm{P}\Omega_8^+(q))\ge 2$. We now argue by contradiction and we suppose that $\gamma_w(\mathrm{P}\Omega_8^+(q))=2$. Let $K$ be the second component in a weak normal $2$-covering of $\mathrm{P}\Omega_8^+(q)$, with $K$ maximal. By the above analysis, we get that $\tilde K$ contains elements having order $q^3-1$ and  $q^3+1$. Therefore, by Lemma~\ref{aritme}, $|\tilde K|$ is divisible by
$$
\frac{(q^3-1)(q^3+1)}{\gcd(q^3-1,q^3+1)}=
\frac{q^6-1}{\gcd(2,q-1)}.$$ A direct inspection on the maximal subgroups of $\Omega_8^+(q)$ in~\cite[Table~$8.50$]{bhr} reveals that $\tilde K$ is one of the groups in Table~\ref{o7q^2+111}.

\begin{table}[ht]
\begin{tabular}{c|c|c||c|c|c}
\toprule[1.5pt]
\multicolumn{3}{c||}{$q$ odd, $q>3$}&\multicolumn{3}{|c}{ $q$ even,  $q>2$} \\
\midrule[1.5pt]
Structure& class&nr. $\mathrm{O}_8^+(q)$-classes&Structure&class&nr. $\mathrm{O}_8^+(q)$-classes\\
\midrule[1.5pt]
$2\times \Omega_7(q)$&$\mathcal{C}_1$&$2$& $\mathrm{Sp}_6(q)$&$\mathcal{C}_1$&$1$\\
$2^\cdot\Omega_7(q)$&$\mathcal{S}$&$4$& $\mathrm{Sp}_6(q)$&$\mathcal{S}$&$2$\\
\bottomrule[1,5pt]
\end{tabular}
\caption{Maximal subgroups of size divisible by $(q^6-1)/\gcd(2,q-1)$} \label{o7q^2+111}
\end{table}

Suppose first that $q$ is even. In particular, $\Omega_8^+(q)=\mathrm{P}\Omega_8^+(q)$, $H=\tilde H$ and $K=\tilde K\cong \mathrm{Sp}_6(q)$.  Observe that the triality automorphism of $\Omega_8^+(q)$ fuses the three $\Omega_8^+(q)$-classes of maximal subgroups isomorphic to $\mathrm{Sp}_6(q)$. We now consider particular elements of order $q^3+1$. We consider the elements $g$ of order $q^3+1$ and with the property that $V$ has a $6$-dimensional irreducible $\mathbb{F}_q\langle g\rangle$-submodule $W$ with $g$ inducing on $W$ a matrix of order $q^3+1$.
These elements fall into two types: the elements having two eigenvalues in $\mathbb{F}_q$ (and hence having type $6\oplus 1\oplus 1$ on $V$) and the elements having no eigenvalues in $\mathbb{F}_q $ (and hence having type $6\oplus 2$ on $V$). The elements of the first type are covered by the $\Omega_8^+(q)$-class of $\mathrm{Sp}_6(q)$ in class $\mathcal{C}_1$. Therefore the remaining type is covered by the remaining two $\Omega_8^+(q)$-classes of $\mathrm{Sp}_6(q)$ in  class $\mathcal{S}$. Let us denote by $o_1$ the number of elements of the first type and $o_2$ the number of elements of the second type. We claim that $$o_2=o_1q.$$ Indeed, any element $g$ of the first type decomposes $V$ as the direct sum $V=W\oplus W^\perp$, where $\dim_{\mathbb{F}_q}(W) =6$, $\dim_{\mathbb{F}_q}(W^\perp)=2$ and the quadratic forms induced on $W$ and on $W^\perp$ are of Witt defect 1. As $\Omega_2^-(q)$ has order $q+1$, we can use the same $g$ to construct $(q+1)-1=q$ distinct elements of the second type. This proves our claim.

As the $\Omega_8^+(q)$-class of $\mathrm{Sp}_6(q)\in \mathcal{C}_1$ covers $o_1$ elements of order $q^3+1$, the remaining two classes cover at most $2o_1$ elements of order $q^3+1$. Observe that here we are using the fact that the triality automorphism of $\Omega_8^+(q)$ fuses the three $\Omega_8^+(q)$-classes of maximal subgroups isomorphic to $\mathrm{Sp}_6(q)$. Therefore, these two classes cover all elements of the second type only if
$$2o_1\ge o_2=o_1q,$$
that is, $q\le 2$, which is a contradiction. This argument explain also why the case $q=2$ is special and had to be dealt with a computer.

Suppose now that $q$ is odd. The argument above can be applied also when $q$ is odd, the only difference is that $q+1$ is replaced by $(q+1)/2$. Therefore we have a putative weak normal covering only if $2\ge (q-1)/2$, that is, $q\le 5$.
 As we mentioned in the opening paragraph of this proof, the cases $q\in \{3,5\}$ have been dealt with a computer.
\end{proof}

The proof of Lemma~\ref{dimension8orthogonal+} is necessarily technical because the weak normal covering number of $\mathrm{P}\Omega_8^+(q)$ is not that far from $2$. Indeed, $\mathrm{P}\Omega_8^+(q)$ has a weak normal covering having three components, see Theorem~\ref{covering3}. Therefore,  in view of Lemma~\ref{dimension8orthogonal+}, $\gamma_w(\mathrm{P}\Omega_8^+(q))=3$ when $q\ge 4$ and $\gamma_w(\mathrm{P}\Omega_8^+(q))=2$ when $q\le 3$. To prove Theorem~\ref{covering3}, we first need a preliminary observation concerning $\Omega_4^+(q)$.
\begin{lemma}\label{covering4}
Let $\tilde g\in \Omega_4^+(q)$ having order divisible by the characteristic $p$ of $\mathbb{F}_q$. Then $\tilde g$ in its action on the orthogonal space $\mathbb{F}_q^4$ fixes a $2$-dimensional totally isotropic subspace.
\end{lemma}
\begin{proof}
Let $G:=\mathrm{SL}_2(q)\times \mathrm{SL}_2(q)$ and let $V$ be the $4$-dimensional vector space over $\mathbb{F}_q$ consisting of the $2\times2$ matrices with coefficients in $\mathbb{F}_q$. The group $G$ has a linear action on $V$ by setting
$$(a,b)\cdot v:=a^{-1}vb,$$
for each $(a,b)\in G$ and $v\in V$. Let $\tilde G$ be the group induced by the linear action of $G$ on $V$. 

Now, $V$ is endowed of a natural quadratic form $Q:V\to \mathbb{F}_q$ by setting
$$Q(v):=\det v=x_{11}x_{22}-x_{12}x_{21},\qquad \forall v=\begin{pmatrix}x_{11}&x_{12}\\x_{21}&x_{22}\end{pmatrix}\in G.$$
In particular, $v\in V$ is totally singular if and only if $\det v=0$.

The linear action of $G$ on $V$ preserves the quadratic form $Q$ because
$$Q((a,b)\cdot v)=Q(a^{-1}vb)=\det(a^{-1}vb)=\det(v)=Q(v).$$
In particular, $\tilde G\le \mathrm{SO}_4(q)$. It is not hard to verify that $\tilde G\le \Omega_4^+(q)$. Since $|\tilde G|=|G|/2=q^2(q^2-1)/2=|\Omega_4^+(q)|$, we deduce $\tilde G=\Omega_4^+(q)$. We use this model of $\Omega_4^+(q)$ to prove this lemma.

Let $\tilde g\in\tilde G$ be an element having order divisible by the characteristic of $\mathbb{F}_q$. In particular, $\tilde g$ is the projection of an element $g\in G$. Replacing $g$ by a suitable conjugate, we may suppose that
\[
g=\left(\begin{pmatrix}1&a\\0&1\end{pmatrix},c\right)\hbox{ or }g=\left(c,\begin{pmatrix}1&a\\0&1\end{pmatrix}\right),
\]
for some $a\in\mathbb{F}_q$ and some $b\in\mathrm{SL}_2(q)$.
In the first case, $g$ fixes the $2$-dimensional totally isotropic subspace
\[
\left\langle
\begin{pmatrix}
1&0\\
0&0
\end{pmatrix},
\begin{pmatrix}
0&1\\
0&0
\end{pmatrix}
\right\rangle.
\]
In the second case, $g$ fixes the $2$-dimensional totally isotropic subspace
\[
\left\langle
\begin{pmatrix}
0&1\\
0&0
\end{pmatrix},
\begin{pmatrix}
0&0\\
0&1
\end{pmatrix}
\right\rangle.\qedhere
\]
\end{proof}
\begin{theorem}\label{covering3}The group $\mathrm{P}\Omega_8^+(q)$ admits a weak normal covering having cardinality $3$.
\end{theorem}
\begin{proof}
Let $H$ be the stabilizer of a $1$-dimensional totally isotropic subspace of $V=\mathbb{F}_q^8$, let $J$ be the stabilizer of a $2$-dimensional non-degenerate subspace of $V$ of minus type, and let $K$ be the stabilizer of a $4$-dimensional non-degenerate subspace of $V=\mathbb{F}_q^8$ of minus type. We verify, using the triality,  that
$H, J,K$ are the components of
 a weak normal covering of $\mathrm{P}\Omega_8^+(q)$.

Let $g\in \mathrm{P}\Omega_8^+(q)$. We show that $g$ has an $\mathrm{Aut}(\mathrm{P}\Omega_8^+(q))$-conjugate in $H$, $J$ or $K$.

We consider the group $\Omega_8^+(q)$ and we let $\tilde g\in\Omega_8^+(q)$ projecting to $g$. Since $\Omega_8^+(q)$ admits no Singer cycle, $\tilde g$ does not act irreducibly on $V$ and hence $\tilde g$ fixes a non-zero proper subspace  $W$ of $V$ with $1\le \dim(W)\le 4$. We choose $W$ so that $W$ is an irreducible $\mathbb{F}_q\langle\tilde g\rangle$-submodule of $V$. In particular, $W$ is either a totally isotropic subspace of $V$, or $W$ is a non-degenerate subspace of $V$.

Assume that $W$ is a non-degenerate subspace of $V$. Observe that $\dim W$ is even and that the quadratic form induced by $V$ on $W$ has Witt defect $1$, because $\tilde g$ acts irreducibly on $W$ and because odd dimensional orthogonal groups $\Omega_\ell(q)$ with $\ell\ge 3$ and even dimensional orthogonal groups of plus type do not contain Singer cycles. In particular, $g$ is $\mathrm{P}\Omega_8^+(q)$-conjugate to an element of $J$ or $K$.

Assume that $W$ is a totally isotropic subspace of $V$. If $\dim W=1$, then $g$ is $\mathrm{P}\Omega_8^+(q)$-conjugate  to $H$. Observe now that the triality automorphism of $\mathrm{P}\Omega_8^+(q)$ maps stabilizers of $1$-dimensional totally isotropic subspaces of $V$ to either of the two families of maximal totally isotropic subspaces of $V$. Therefore, if $\dim W=4$, then $g$ is $\mathrm{Aut}(\mathrm{P}\Omega_8^+(q))$-conjugate to $H$. Suppose $\dim W=3$. From the polar geometry associated to $V$~\cite[Section~6.4]{Cameron}, we see that $g$ fixes a $4$-dimensional totally isotropic subspace. In particular, also in this case, $g$ is $\mathrm{Aut}(\mathrm{P}\Omega_8^+(q))$-conjugate to $H$. Suppose $\dim W=2$. Now, $\tilde g$ fixes $W^\perp$ and hence $\tilde g$ induces a linear action on the $4$-dimensional vector space $W^\perp/W$ having Witt defect zero: the fact that $W^\perp/W$ has the same type as $V$ follows from the discussion on polar spaces in~\cite[Section~6]{Cameron}. Clearly, $\tilde g$ is not semisimple and it does induce an element on $W^\perp/W$ having order divisible by the characteristic of $\mathbb{F}_q$: indeed, if $\tilde g$ induces a semisimple element on $W^\perp/W$, then $\tilde g$ is semisimple because $\tilde g$ acts on $W$ and on $V/W^\perp$ irreducibly. Therefore, by Lemma~\ref{covering4}, $\tilde{g}$ fixes a $2$-dimensional totally isotropic subspace of $W^\perp/W$. Hence $\tilde g$ fixes a $4$ dimensional totally isotropic subspace of $V$. In particular, arguing as above, $g$ is $\mathrm{Aut}(\mathrm{P}\Omega_8^+(q))$-conjugate to $H$.
\end{proof} 
Now, the veracity of Table~\ref{000====} follows from the results in this section.

\section{Proofs of the main theorems}\label{sec:final}

Theorem  \ref{main theorem}, Theorem \ref{main theorem1} and Corollary \ref{corollary} follow now immediately by inspection of the  Tables \ref{00}-\ref{000====}.

We devote the rest of the section to the proof of Theorem~\ref{main theorem2}. We start with an elementary observation. Recall that a group $Y$ is said to be \textit{\textbf{supersolvable}} if there exists a normal series
$1=Y_0\unlhd Y_1\unlhd \cdots \unlhd Y_{\ell-1}\unlhd Y_{\ell}=Y$
such that each quotient $Y_i/Y_{i-1}$ is cyclic.
\begin{lemma}\label{lemma-super}Let $Y$ be a supersolvable group, let $V$ and $Z$ be subgroups of $Y$ with $V<Z$ and $Z\unlhd Y$. Then
$$\bigcup_{y\in Y}V^y\subsetneq Z.$$
\end{lemma}
\begin{proof}
We prove it by induction on $|Y|$. As $Y$ is a finite group, using~\cite[Lemma~$2$ parts~(ii) and~(iii)]{baer},
we see that there exists a normal series $$1=Y_0\unlhd Y_1\unlhd \cdots \unlhd Y_{\ell-1}\unlhd Y_{\ell}=Y$$
of $Y$, with $Y_i/Y_{i-1}$ cyclic of prime order for each $i\in \{1,\ldots,\ell\}$ and with $Z=Y_\kappa$, for some $\kappa\in \{1,\ldots,\ell\}$. We write $\bar{Y}:=Y/Y_1$ and we use the ``bar'' notation for the projection of $Y$ to $\bar{Y}$.

Now, $\bar{V}\le \bar{Z}\unlhd\bar{Y}$. Suppose first $\bar{V}\ne \bar{Z}$. As $\bar{Y}$ is a quotient of $Y$ and as $Y$ is supersolvable, by~\cite[page~150]{robinson}, we deduce that $\bar{Y}$ is supersolvable. Since $|\bar{Y}|<|Y|$, our inductive hypothesis gives
$$\bigcup_{\bar{y}\in \bar{Y}}
\bar{V}^{\bar{y}}\subsetneq \bar{Z}.$$
From this, it immediately follows that
 $$\bigcup_{y\in Y}V^y\subsetneq Z.$$
 Suppose next that $\bar{V}=\bar{Z}$.
 Since $Y_1\le Z$, we deduce $Z=VY_1$. Let $p$ be the order of $Y_1$ and assume, by contradiction, that $\bigcup_{y\in Y}V^y= Z$. Since $Y_1$ is cyclic included in $Z$, there exists $y\in Y$ with $Y_1^y\le V$. However, since $Y_1\unlhd Y$, we get $Y_1=Y_1^y\le V$ and thus $Z=V$, contrary to our hypothesis.
\end{proof}
Next, we need a rather technical result.
\begin{lemma}\label{l:technicalS4}
Let $U:=B\times C$ be a finite group with $B\cong S_4$ and with $C$ cyclic of order $f$. Let $V$, $Z$ and $Y$ be subgroups of $U$ with $V< Z\unlhd Y\le U$ and
$$Z=\bigcup_{y\in Y}V^y.$$
Then the following holds:
\begin{enumerate}
%\item $Y=B_1\times C_1$, where $B_1\le B$, $C_1\le C$, $B_1=B$ or $B_1=B'$,
\item $Z=K\times C_1$, where $K$ is the Klein subgroup of $B$ and $C_1\le C$,
\item $V=K_0\times C_1$, where $K_0$ is a subgroup of order $2$ of $K$.
\end{enumerate}
\end{lemma}
\begin{proof}
Since $C=Z(U)$ and since $Z=\bigcup_{y\in Y}V^y$, we get
\begin{align}\label{eq:pablo1}
C\cap V=C\cap Z.
\end{align}
Now, consider the subgroups of $U$ given by $Y_0=YC$, $Z_0=ZC$ and $V_0=VC$. Observe that $Y_0/C$ is isomorphic to a subgroup of $U/C\cong B\cong S_4$. Moreover,
$$Z_0=\bigcup_{y\in Y_0}V_0^y.$$
Now, a direct inspection on the subgroups of $S_4$ reveals that either $Z_0=V_0$, or

\begin{itemize}
\item $Y_0=B_1\times C$, where $B_1\cong A_4$ or $B_1\cong S_4$,
\item $Z_0=K\times C$, where $K$ is the Klein subgroup of $B$, and
\item $V_0=K_0\times C$, where $K_0$ is a subgroup of order $2$ of $K$.
\end{itemize}
The first possibility yields
\begin{align}\label{eq:pablo2}
VC=V_0=Z_0=ZC.
\end{align}
However,~\eqref{eq:pablo1} and~\eqref{eq:pablo2} yield $V=Z$, which is a contradiction. Thus
$$VC=K_0\times C\hbox{ and }ZC=K\times C.$$

Let $\pi:Z\to C$ be the natural projection and let $C_1$ be the image of $\pi$. To conclude the proof it suffices to show that $C_1\le V$ (indeed, this implies $Z=K\times C_1$ and $V=K_0\times C_1$). Let $c\in C_1$ with $C_1=\langle c\rangle$ and let $k_1,k_2,k_3$ be the three involutions in $K$. Since $c$ lies in the image of $\pi$, there exists $z\in Z$ with $c=\pi(z)$. Now, as $Z\le K\times C$, we have $z=k_ic$, for some $i\in \{1,2,3\}$. Without loss of generality, we may suppose that $i=1$. As $Z$ is a union of $Y$-conjugates, we may assume, replacing $z$ with a suitable $Y$-conjugate, that $z\in V$. Let $y\in Y$ with $k_1^y=k_2$. Observe that the existence of $y$ is guaranteed by the fact that $YC/C$ contains a subgroup isomorphic to $A_4$. Now,
$$z^y=k_1^yc^y=k_2c.$$
Thus $k_2c\in Z$ and hence $k_3=(k_1c)(k_2c)^{-1}\in Z$. In particular, $k_i\in Z$, for every $i\in \{1,2,3\}$. Thus $c=(k_1c)k_1^{-1}\in Z$. Now, $c$ has a $Y$-conjugate in $V$ and hence $c\in V$.
\end{proof}
\begin{proof}[Proof of Theorem~$\ref{main theorem2}$]
Let $X$ be an almost simple group with socle $G$, let $A_X:=\mathrm{Aut}(X)$ and $A_G:=\mathrm{Aut}(G)$.
We claim that
\begin{equation}\label{chain-alm}
G\le X\unlhd A_X\le A_G.
\end{equation}
Since $G$ is characteristic in $X$ we have the homomorphism $A_X\rightarrow A_G$ defined by $\varphi\mapsto\varphi_{\mid G}$. We show that this homomorphism is injective. Indeed, assume $\varphi_{\mid G}=id_{\mid G}$ for some $\varphi\in A_X$. Then for every $x\in G$, and $y\in X$ we have  $$x^y=\varphi(x^y)=\varphi(x)^{\varphi(y)}=x^{\varphi(y)}$$ so that $x^{y^{-1}\varphi(y)}=x$. Consider now the homomorphism $X\rightarrow A_G$ associating with $z\in X$ the restriction $(\sigma_z)_{\mid G}$ to $G$ of the  conjugation $\sigma_z$ via $z$. Assume that the kernel of this homomorphism is not trivial. Then it contains the only minimal normal subgroup of $X$, that is, it contains $G$. As a consequence $G$ would be abelian, a contradiction. Hence, coming back to our equality $x^{y^{-1}\varphi(y)}=x,$ we see it as $(\sigma_{y^{-1}\varphi(y)})_{\mid G}=id_{G}$  and thus we deduce $y^{-1}\varphi(y)=1$, that is, $\varphi(y)=y$ which means $\varphi=id_{X}.$ Hence we have $A_X\le A_G$. Finally note that, since $Z(X)=1$, we have $X=\mathrm{Inn}(X)\unlhd A_X.$ That shows \eqref{chain-alm}.
 In particular $A_X$ is almost simple with socle $G$ and $G$ is characteristic in every term of the chain \eqref{chain-alm}.

Suppose now that part~(1) of Theorem~\ref{main theorem2} does not hold. Then $\gamma_w(X)=1$. Let $\mu:=\{H\}$ be a weak normal $1$-covering of $X$ for some $H<X$. We need to show that $G=\mathrm{P}\Omega_8^+(q)$, $q$ is odd, $A_X$ contains a triality automorphism of $G$, $X$ does not contain a triality automorphism of $G$ and $G\le H$.

As
\begin{equation}\label{ausiliaria}X=\bigcup_{a\in A_X}H^a,\end{equation}
we get
\begin{equation*}
G=G\cap \bigcup_{a\in A_X}H^a=\bigcup_{a\in A_X}(G\cap H)^a\subseteq \bigcup_{a\in A_G}(G\cap H)^a\subseteq G
\end{equation*}
and hence
\begin{equation}\label{ausiliaria2}
G=\bigcup_{a\in A_G}(G\cap H)^a.
\end{equation}

Assume $G\nleq H$, that is, $G\cap H$ is a proper subgroup of $G$. From~\eqref{ausiliaria2}, $\{G\cap H\}$ is a weak normal $1$-covering of $G$ and $\gamma_w(G)=1$, contradicting Theorem~\ref{main theorem}. Therefore $$G\le H.$$

 Set $Y:=A_X/G$, $Z:=X/G$ and $V:=H/G$ and observe that $V<Z\unlhd Y$. Since $G$ is characteristic in $A_X$ and thus also normal in $A_X$, we obtain that $G^a=G$, for each $a\in A_X$. Thus, the elements  $a\in A_X$ induce a group automorphism on the quotient $X/G=Z$ by setting $(xG)^a=x^aG$, for every $x\in X$. In particular, we have a group action of $A_X$ on $X/G=Z$, that is,  a group homomorphism $A_X\to \mathrm{Aut}(Z)$. Observe that $G$ is contained in the kernel of this group homomorphism and hence we have a group action of $A_X/G=Y$ on $Z$ by setting $(xG)^{aG}=x^aG$, for every $x\in X$ and for every $a\in A_X$. As $(xG)^{aG}=(xG)^a$ for every $x\in X$ and for every $a\in A_X$, by quotienting the left and the right hand side of~\eqref{ausiliaria} with $G$, we obtain
$$Z=\frac{X}{G}=\bigcup_{a\in A_X}\left(\frac{H}{G}\right)^a=
\bigcup_{a\in A_X}\left(\frac{H}{G}\right)^{aG}=
\bigcup_{a\in Y}V^a.$$

Now, Lemma~\ref{lemma-super} shows that $Y$ is not supersolvable. Since subgroups of supersolvable groups are supersolvable (see~\cite[page 150]{robinson}), we deduce that  $A_G$ is not supersolvable. A direct inspection on the outer automorphisms of the non-abelian simple groups reveals that $G=\mathrm{P}\Omega_8^+(q)$ and  $q$ is odd. We postpone the proof that $A_X=\mathrm{Aut}(X)$ contains a triality automorphism, that $X$ does not contain a triality automorphism and that $\mathrm{PO}_8^+(q)\leq X$ to Remark~\ref{remark:final}, where we give slightly more information on $X$ and $H$.
\end{proof}
\begin{remark}\label{remark:final}
{\rm Let $X$ be an almost simple group with socle $G$ and  let $A_X:=\mathrm{Aut}(X)$. Suppose $\gamma_w(X)=1$ and let $\mu:=\{H\}$ be a weak normal $1$-covering of $X$.

Using Lemma~\ref{l:technicalS4} and the notation in the proof above, we can give a much more detailed information on $X$, $G$ and $H$ arising in part $(2)$ of Theorem~\ref{main theorem2}.

From~\cite[page~38]{kl}, we have
$$\mathrm{Out}(G)=\mathrm{Out}(\mathrm{P}\Omega_8^+(q))=B\times C,$$
where $B\cong S_4$ is the subgroup generated by the inner-diagonal and graph outer-automorphisms of $G$ and $C$ is cyclic of order $f$ and  is the subgroup generated by the field outer-automorphisms. Here, $q=p^f$, where $p$ is an odd prime number.

Set $U:=B\times C$ and observe that $V<Z\unlhd Y\le U$, where $V=H/G$, $Z=X/G$ and $Y=A_X/G$. In particular, Lemma~\ref{l:technicalS4} yields a very detailed information on the various possibilities for $Y$, $Z$ and $V$, and hence in turn for $A_X$, $X$ and  $H$.  From this, using \cite[Chapter 4]{GLS3}, it immediately follows that $A_X$ contains a triality automorphism, $X$ does not contain a triality automorphism and $H$ contains an $A_X$-conjugate of $\mathrm{PO}_8^+(q)$.
As the $A_X$-conjugates of $H$ cover $X$, we deduce   $\mathrm{PO}_8^+(q)\le X$.
}
\end{remark}

\section{Almost simple groups having socle a sporadic simple group}\label{section:new}
In this section we prove the correctness of Table~\ref{tableSporadic}. Our proof is mainly computational.

Let $G$ be an almost simple group having socle a sporadic simple group.
Except for the Baby Monster  and for the Monster, the computations are straightforward because in each case the character table and the list of the maximal subgroups of $G$ are stored in \textrm{GAP} and in \texttt{magma}. Therefore, for these groups, we may compute the list of all the permutation characters of the faithful primitive actions of $G$. Using this list, $\gamma(G)$ is simply the minimum number of permutation characters $\pi_1,\ldots,\pi_{\ell}$ such that the class function
$\pi_1+\cdots+\pi_\ell$ does not vanish in any conjugacy class. Similarly, using this list, we may easily determine the edges of the invariably generating graph $\Lambda(G)$. Indeed, given two conjugacy classes $C_1$ and $C_2$ of $G$, $\{C_1,C_2\}$ is an edge of $\Lambda(G)$ if and only if, for some $c_1\in C_1$ and $c_2\in C_2$, $$\pi(c_1)\pi(c_2)=0,$$
for every permutation character $\pi$ as above.

Now, with $\Lambda(G)$, we can use standard algorithms to compute the clique number $\kappa(G)$. 

The computation for the Baby Monster $B$ is slightly more involved. All the permutation characters of $B$ acting on the right cosets of a maximal subgroup are stored in \texttt{GAP}, except for the action of $B$ on the cosets of a maximal subgroup $M$ of type $(2^2\times F_4(2)):2$. This omission from the \texttt{GAP} library is due to the fact that the conjugacy fusion of some of the elements of $M$ in $B$ remains a mystery. Therefore, except for the action of $B$ on the right cosets of $(2^2\times F_4(2)):2$, we may compute all the remaining permutation characters of primitive faithful actions. We have computed the permutation character of the action of $B$ on the right cosets of $(2^2\times F_4(2)):2$ using the \texttt{GAP} command \texttt{PossibleClassFusions} as instructed in~\cite[Proposition~$3.4$]{burness2}. Despite the ambiguity on the fusion of the conjugacy classes of $(2^2\times F_4(2)):2$ in $B$, the command \texttt{PossibleClassFusions} shows that all of these possibilities give rise to the same permutation character. Now, with all of these permutation characters available, we may argue as in all other cases.

We now deal with the case that $G$ is the Monster group. From~\cite[Section~$3.6$]{Wilson} and~\cite{Wilson2}, we see that the classification, up to isomorphism  and up to conjugacy, of the maximal subgroups of $G$ is complete except for a few open cases. In particular, if $M$ is a maximal subgroup of $G$, then either $M$ is in~\cite[Section~$3.6$]{Wilson}, or the socle of $M$ is $\mathrm{PSL}_2(13)$ or $\mathrm{PSL}_2(16)$. However, things are not that simple because the conjugacy classes of the known maximal subgroups of $G$ are not yet fully understood. Moreover, to the best of our knowledge, we currently know only six permutation characters of primitive faithful actions of $G$: namely, the permutation characters for the actions of $G$ on the right cosets of maximal subgroups of type $2.\mathrm{B}$, $2^{1+24}.\mathrm{Co}_1$, $3.\mathrm{Fi}_{24}$, $2^2.{}^2E_6(2):S_3$ and $2^{5+10+20}.(S_3\times L_5(2))$ are determined in~\cite{breuer,breuer2}, and of type $3^{1+12}.2\mathrm{Suz}.2$ is determined in~\cite{Wilson3}.  Therefore, for the Monster $G$, we have argued differently.

Given a positive integer $x$, we let $\delta(x)$ be the set of divisors of $x$ and, given a set of positive integers $X$, we let $\Delta(X)=\bigcup_{x\in X}\delta(x)$.
In Table~\ref{tableaux}, we have collected some information on the maximal subgroups of $G$. In the first column we have given the name of the maximal subgroups of $G$ (up to the small ambiguities explained above) and in the second column we have given a set of integers $X$ such that $\Delta(X)$ is the set of element orders of the corresponding maximal subgroup. For instance, in the first row of Table~\ref{tableaux}, we see that the the second column is
$$\{ 32, 36, 38, 40, 44, 46, 48, 50, 54, 56, 60, 62, 66, 68, 70, 78, 84, 94, 104, 110 \}$$
and hence the order of the elements in $2.B$ are all of the divisors of the elements in this set.
The symbol $\Delta(X)$ help us to keep the notation short, so that we do not have to list all of the element orders of a maximal subgroup, but only some. In the last two lines, we have put a label next to the groups $L_2(13).2$ and $L_2(16).4$ because at the moment it is not known whether the Monster has maximal subgroups having socle $L_2(13)$ or $L_2(16)$. Moreover, in case that the Monster does contain such maximal subgroups, it is not clear what is the exact structure and the number of conjugacy classes. This ambiguity will play no role in our argument that follows.
\begin{center}
\begin{longtable}{c|c}
\caption[Maximal subgroups of the Monster and their element orders]{Maximal subgroups of the Monster and their element orders}\label{tableaux}\\
\toprule[1.5pt]
subgroup& element orders\\
\midrule[1.5pt]
\endfirsthead
\endhead

%\hline &Continues on next page \textcolor{red}{questa riga non serve mi pare...} \\ \hline
\endfoot
$2.B$&\{32, 36, 38, 40, 44, 46, 48, 50, 54, 56, 60, 62, 66,\\
& 68, 70, 78, 84, 94, 104, 110\}\\
$2^{1+24}.Co_1$&\{32, 36, 40, 48, 52, 56, 60, 66, 70, 78, 84, 88, 92\}\\
$3.Fi_{24}$&\{34, 36, 40, 45, 46, 48, 51, 54, 60, 66, 69, 70, 78, 84, 87, 105\}\\
$2^2.{}^2E_6(2):S_3$&\{32, 36, 38, 40, 48, 52, 56, 60, 66, 68, 70, 84\}\\
$2^{10+16}.O_{10}^+(2)$&\{32, 34, 36, 40, 45, 48, 51, 56, 60, 62, 84\}\\
$2^{2+11+22}.(M_{24}\times S_3)$&\{32, 40, 48, 56, 60, 66, 69, 84, 88, 92\}\\
$3^{1+12}.2Suz.2$&\{36, 40, 45, 48, 54, 56, 60, 66, 78, 84\}\\
$2^{5+10+20}.(S_3\times L_5(2))$&\{32, 40, 48, 56, 60, 62, 84, 93\}\\
$S_3\times Th$&\{24, 36, 38, 54, 57, 60, 62, 78, 84, 93\}\\
$2^{3+6+12+18}.(L_3(2)\times 3S_6)$&\{32, 40, 48, 56, 60, 70, 84, 105\}\\
$3^8.O_8^-(3).2_3$&\{24, 27, 36, 39, 40, 41, 42, 45, 52, 56, 60\}\\
$(D_{10}\times HN).2$&\{24, 36, 38, 40, 44, 45, 50, 60, 70, 84, 95, 105, 110\}\\
$(3^2:2\times O_8^+(3)).S_4$&\{24, 36, 40, 54, 56, 60, 78, 84, 104\}\\
$3^{2+5+10}.(M_{11}\times 2S_4)$&\{24, 36, 40, 45, 54, 60, 66, 88\}\\
$3^{3+2+6+6}.(L_3(3)\times SD_{16})$&\{24, 36, 54, 78, 104\}\\
$5^{1+6}:2J_2:4$&\{24, 40, 50, 56, 60, 70\}\\
$(7:3\times He):2$&\{48, 51, 56, 60, 70, 84, 105, 119\}\\
$(A_5\times A_{12}):2$&\{22, 24, 33, 36, 40, 45, 55, 60, 70, 84, 105\}\\
$5^{3+3}.(2\times L_3(5))$&\{24, 50, 60, 62\}\\
$(A_6\times A_6\times A_6).(2\times S_4)$&\{9, 24, 40, 60\}\\
$(A_5\times U_3(8):3_1):2$&\{24, 28, 36, 38, 42, 45, 57, 60, 95, 105\}\\
$5^{2+2+4}:(S_3\times GL_2(5))$&\{24, 50, 60\}\\
$(L_3(2)\times S_4(4):2).2$&\{40, 48, 51, 56, 60, 68, 70, 84, 105, 119\}\\
$7^{1+4}:(3\times 2S_7)$&\{24, 56, 60, 70, 84\}\\
$(5^2:[2^4]\times U_3(5)).S_3$&\{24, 40, 60, 70, 84\}\\
$(L_2(11)\times M_{12}).2$&\{24, 40, 60, 66, 88, 110\}\\
$(A_7\times (A_5\times A_5):2^2):2$&\{ 24, 40, 60, 70, 84, 105\}\\
$5^4:(3\times 2L_2(25)):2_2$&\{20, 24, 30, 78\}\\
$7^{2+1+2}.GL_2(7)$&\{42, 48, 56\}\\
$M_{11}\times A_6.2^2$&\{24, 30, 40, 66, 88, 110\}\\
$(S_5\times S_5\times S_5):S_3$&\{18, 24, 40, 60\}\\
$(L_2(11)\times L_2(11)):4$&\{24, 55, 60, 66\}\\
$13^2:2L_2(13).4$&\{12, 52, 56\}\\
$(7^2:(3\times 2A_4)\times L_2(7)).2$&\{24, 84\}\\
$(13:6\times L_3(3)).2$&\{24, 78, 104\}\\
$13^{1+2}:(3\times 4S_4)$&\{24, 52, 78\}\\
$L_2(71)$&\{35, 36, 71\}\\
$L_2(59)$&\{29, 30, 59\}\\
$11^2:(5\times 2A_5)$&\{20, 30, 55\}\\
$L_2(29):2$&\{28, 29, 30\}\\
$7^2:SL_2(7)$&\{6, 8, 14\}\\
$L_2(19):2$&\{18, 19, 20\}\\
$41:40$&\{40,41\}\\
$L_2(41)$&\{20, 21, 41\}\\
$L_2(13).2$ $^{\ast}$&\{12,13,14\}\\
$L_2(16).4$ $^{\ast}$&\{8,10,12,15,17\}\\
\bottomrule[1.5pt]
\end{longtable}
\end{center}
Using the information in Table~\ref{tableaux}, we now construct an auxiliary graph. We let $\Gamma$ be the graph whose vertices are the orders of the elements of the Monster and where two vertices $x$ and $y$ are declared to be adjacent, if there exists no maximal subgroup $M$ containing elements having order $x$ and $y$. It can be verified, by \texttt{magma}, that $\{41,59,71, 87, 92, 93, 94, 95,119\}$ is a clique of $\Gamma$ and hence
\begin{equation}\label{monster:1}\kappa(G)\ge 9.
\end{equation}
Indeed, let $x_1,\dots, x_9$ be the orders of elements forming a clique of size $9$ in $\Gamma$ and consider $g_i\in G$ such that $\order {g_i}=x_i$. The conjugacy classes $C_i=g_i^G$ are nine distinct vertices in $\Lambda(G)$. Pick $c_i\in C_i$ and $c_j\in C_j$ for $i\neq j.$ Then $\langle c_i, c_j\rangle=G$, otherwise $\langle c_i, c_j\rangle$ would be contained in a maximal subgroup $M$ of $G$. Since $\order {c_i}=x_i$ and $\order {c_j}=x_j$, by $\Gamma$ definition, there is no edge incident to $x_i, x_j$ in  $\Gamma$, against the fact  that they are part of a clique.

Next, we have computed the permutation characters of $G$ acting on the cosets of its maximal subgroups
$2.\mathrm{B}$, $2^{1+24}.\mathrm{Co}_1$, $3.\mathrm{Fi}_{24}$ and $2^{5+10+20}.(S_3\times L_5(2))$. Recall that these are some of the permutation characters of primitive faithful actions of $G$ already available in the literature. Using these characters, we have checked that, the elements of the Monster not belonging to some conjugate of these four maximal subgroups have order $$41,\,57,\,59,\,71,\,95,\,119.$$
Therefore, to obtain a normal covering of the Monster, using as components $2.\mathrm{B}$, $2^{1+24}.\mathrm{Co}_1$, $3.\mathrm{Fi}_{24}$ and $2^{5+10+20}.(S_3\times L_5(2))$, we need to cover these remaining elements. Consulting Table~\ref{tableaux}, we see that using $$(A_5\times U_3(8).3).2,\,(L_3(2)\times S_4(4)).2,\,L_2(71),\, L_2(59),\, 41:40,$$
we cover all the remaining elements.

 Therefore, we obtain a normal covering of the Monster with $9$ subgroups and hence
\begin{equation}\label{monster:2}\gamma(G)\le 9.
\end{equation}
From \eqref{kappa-gamma}, ~\eqref{monster:1} and~\eqref{monster:2}, we finally deduce $\gamma(G)=\kappa(G)=9$.

For sporadic simple groups $G$ admitting outer automorphisms, the computations for $\gamma_w(G)$, $\kappa_w(G)$, $\gamma(\mathrm{Aut}(G))$ and $\kappa(\mathrm{Aut}(G))$ are entirely similar. Indeed, we may compute $\gamma(\mathrm{Aut}(G))$ and $\kappa(\mathrm{Aut}(G))$ arguing exactly as above because the list of the maximal subgroups and the character table of $\mathrm{Aut}(G)$ is available and stored in \texttt{GAP}. For computing $\gamma_w(G)$, we determine the list of all the permutation characters of the actions of $\mathrm{Aut}(G)$ acting on the right cosets of $M$, where $M$ is a maximal subgroup of $G$. Then, $\gamma_w(G)$ is simply the minimum number of such characters $\pi_1,\ldots,\pi_\ell$ such that the class function $\pi_1+\cdots+\pi_\ell$ does not vanish in $G$. Similarly, using this list of permutation characters, we may also determine the edges of the $\mathrm{Aut}$-invariably generating graph. Indeed, two conjugacy classes $x_1^{\mathrm{Aut}(G)}$ and $x_2^{\mathrm{Aut}(G)}$, with $x_1,x_2\in G$, are adjacent in the $\mathrm{Aut}$-invariably generating graph if and only if $\pi(x_1)\pi(x_2)=0$, for every permutation character as above. Now, using built-in algorithms we may compute the clique number $\kappa_w(G)$ of this graph.

\section{Dropping the maximality}\label{sec:appendix}
In this section, we use Theorem~\ref{main theorem1} to classify the weak normal $2$-coverings and the normal $2$-coverings of the non-abelian simple groups. This means that we drop out the hypothesis about the maximality of the components and give the description of all the possible components in a weak normal $2$-covering of $G$, for $G$ a non-abelian simple group such that $\gamma_w(G)=2.$

\begin{theorem}\label{main theoremgeneral}
Let $G$ be a finite non-abelian simple group and let $\mu=\{H,K\}$ be  a weak normal $2$-covering of $G$. Then  the pair $(H,K)$ appears in {\rm Tables~\ref{00:Alt}--\ref{00:orthogonal}}, up to  $\mathrm{Aut}(G)$-conjugacy. Moreover, $\mu$ gives rise to at least a normal $2$-covering of $G$ if and only if in the corresponding row of the fifth column of those tables appears a number greater than $0.$
\end{theorem}

From Theorem~\ref{main theoremgeneral}, we obtain the following rather remarkable corollary.
\begin{corollary}\label{cor:21}
Let $G$ be a finite non-abelian simple group and let $\mu=\{H,K\}$ be  a weak normal $2$-covering of $G$. Then either $H$ or $K$ is maximal in $G$, or $G\cong\mathrm{PSL}_3(4)$. When $G\cong\mathrm{PSL}_3(4)$, there exists a unique weak normal $2$-covering $\{H,K\}$, with $H$ and $K$ not maximal in $G$, see~Table~\ref{00:linear}.
\end{corollary}
We adopt the same notation as in Notation~\ref{notation} for reading Tables~\ref{00:Alt}--\ref{00:orthogonal}. In these tables, when listing $H$ and $K$, in the case that either $H$ or $K$ is not maximal in $G$, we also give some embeddings in certain overgroups. Observe that in Table~\ref{00:orthogonal}, we are not listing all weak normal $2$-coverings, because $\mathrm{P}\Omega_8^+(2)$ has $60$ distinct $\mathrm{Aut}(\mathrm{P}\Omega_8^+(2))$-conjugacy classes of weak normal $2$-coverings and $\mathrm{P}\Omega_8^+(3)$ has $2019$ distinct $\mathrm{Aut}(\mathrm{P}\Omega_8^+(3))$-conjugacy classes of weak normal $2$-coverings. It is quite remarkable that, although $\mathrm{P}\Omega_8^+(2)$ and $\mathrm{P}\Omega_8^+(3)$ have so many weak normal $2$-coverings up to $\mathrm{Aut}$-conjugacy, it can be checked with a computer that in each weak normal $2$-covering $\{H,K\}$ either $H$ or $K$ is maximal. Actually more is true. Using the notation from~\cite{atlas}, let $x$ be an element of order $9$ in the class $9J$ of $\mathrm{P}\Omega_8^+(3)$. Then $\{\Omega_7(3),\langle x\rangle\}$ is a weak normal $2$-covering of $\mathrm{P}\Omega_8^+(3)$. In particular, for every subgroup $H$ of $\mathrm{P}\Omega_8^+(3)$ containing $x$, $\{\Omega_7(3),H\}$ is a weak normal $2$-covering. It turns out that, up to $\mathrm{Aut}(\mathrm{P}\Omega_8^+(3))$-conjugacy, there are $2011$ weak normal $2$-coverings of this type. 

We now give  the proof of Theorem~\ref{main theoremgeneral}. Let $G$ be a non-abelian simple group and let $\mu=\{H,K\}$ be a weak normal $2$-covering of $G$. Let $H_M$ and $K_M$ be maximal subgroups of $G$ with $H\le H_M$ and $K\le K_M$. Then $\{H_M,K_M\}$ is a weak normal $2$-covering of $G$ with maximal components and, by Theorem~\ref{main theorem1}, the group $G$ and the pair $(H_M,K_M)$ appear in Tables~\ref{00}--\ref{000====}.

 Except when $(G,H_M,K_M)$ is part of an infinite family of weak normal $2$-coverings, we may simply use the computer algebra system \texttt{magma} to prove Theorem~\ref{main theoremgeneral}. Hence, for proving Theorem~\ref{main theoremgeneral}, it suffices to consider the following cases:
\begin{enumerate}
\item\label{ddd1}$G=G_2(q)$, $H_M\cong\mathrm{SL}_3(q).2$, $K_M\cong\mathrm{SU}_3(q).2$, $q\ge 4$ and $q$ even;
\item\label{ddd2}$G=F_4(q)$, $H_M\cong\,{^3}D_4(q).3$, $K_M\cong \mathrm{Spin}_9(q)\cong 2.\Omega_9(q)$, $q=3^a$, $a\ge 1$;
\item\label{ddd3}$G=\mathrm{PSL}_2(q)$, $H_M\cong D_{q+1}$, $K_M\cong E_q:((q-1)/2)$, $q>9$  and $q$ odd;
\item\label{ddd4}$G=\mathrm{PSL}_2(q)$, $H_M\cong D_{2(q+1)}$, $K_M\cong E_q:(q-1)$, $q>4$ and $q$ even;
\item\label{ddd5}$G=\mathrm{PSL}_2(q)$, $H_M\cong D_{2(q+1)}$, $K_M\cong D_{2(q-1)}$, $q>4$ and $q$ even;
\item\label{ddd6}$G=\mathrm{PSL}_3(q)$, $H_M\cong \frac{q^2+q+1}{\gcd(3,q-1)}:3$, $K_M$ parabolic and $q\ne 4$;
\item\label{ddd7}$G=\mathrm{PSL}_4(q)$, $H_M\cong \frac{1}{d}\mathrm{SL}_2(q^2).(q+1).2$, $K_M\cong\frac{1}{d}E_q^3:\mathrm{GL}_3(q)$ and $d=\gcd(4,q-1)$;
\item\label{ddd8}$G=\mathrm{PSU}_3(q)$, $H_M\cong (q^2-q+1):3$, $K_M\cong\mathrm{GU}_2(q)$, $q=3^a$ and $a>1$;
\item\label{ddd9}$G=\mathrm{PSU}_4(q)$, $H_M\cong \frac{1}{d}\mathrm{GU}_3(q)$, $K_M\cong\frac{1}{d}E_q^4:\mathrm{SL}_2(q^2):(q-1)$, $d=\gcd(4,q+1)$ and $q\ge4$;
\item\label{ddd10}$G=\mathrm{Sp}_n(q)$, $H_M=\mathrm{SO}_n^-(q)$, $K_M=\mathrm{SO}_n^+(q)$, $n\ge 4$ even, $q$ even and $(n,q)\ne (4,2)$;
\item\label{ddd11}$G=\mathrm{PSp}_6(q)$, $H_M=\frac{1}{2}(\mathrm{Sp}_2(q)\perp \mathrm{Sp}_4(q))$, $K_M=\frac{1}{2}\mathrm{Sp}_2(q^3):3$ and $q=3^a$. 
\end{enumerate}
If $H=H_M$ and $K=K_M$, then the proof follows from Theorem~\ref{main theorem1}. Therefore we may suppose that either 
\begin{equation}\label{tired}H<H_M\,\hbox{ or }\,K<K_M.
\end{equation}

\smallskip

Recall that for $X\leq G$, we use the notation $\tilde X$ to denote $\pi^{-1}(X).$ By Section~\ref{cove-classic-simple}, we have
the weak normal $2$-covering of $\tilde G$ given by $\tilde \mu=\{\tilde H, \tilde K\}$. Moreover,  $\tilde H_M:=\pi^{-1}(H_M)$, $\tilde K_M:=\pi^{-1}(K_M)$ are maximal subgroups of $\tilde G$ containing $\tilde H$ and $\tilde K$ respectively and $\tilde \mu_M:=\{\tilde H_M, \tilde K_M\}$ is a weak normal $2$-covering of $\tilde G$ with maximal components.

\smallskip

\noindent\textsc{Linear groups: Cases~\eqref{ddd3},~\eqref{ddd4},~\eqref{ddd5},~\eqref{ddd6} and~\eqref{ddd7}.}

\smallskip

\noindent \textbf{Suppose $G=\mathrm{PSL}_2(q)$ and we are in Case~\eqref{ddd3} or~\eqref{ddd4}.} Set $d:=\gcd(2,q-1)$. Here, we have $$H\le H_M\cong D_{2(q+1)/d}\,\hbox{ and }\,K\le K_M\cong E_q:((q-1)/d).$$ Now, $G$ contains an element $x$ of order $(q+1)/d$ and an element $y$ of order $(q-1)/d$. As $K_M$ contains no element having order $(q+1)/d$ and $H_M$ contains no element having order $(q-1)/d$, up to replacing $x$ and $y$ with suitable $\mathrm{Aut}(G)$-conjugates, we deduce 
\begin{align*}
C_{(q+1)/d}&\cong\langle x\rangle\le H\le H_M={\bf N}_G(\langle x\rangle)\cong D_{2(q+1)/d},\\
C_{(q-1)/d}&\cong\langle y\rangle\le K\le K_M\cong E_q:((q-1)/d).
\end{align*}
Observe that $\langle x\rangle$ is a maximal subgroup of $H_M$ and $\langle y\rangle$ is a maximal subgroup of $K_M$. It follows that  $H\in \{\langle x\rangle,H_M\}$ and $K\in \{\langle y\rangle,K_M\}$.

Assume now $q$ odd. Then, $H$ contains no non-identity unipotent elements. Therefore, $K$ contains a non-identity unipotent element $u$. Thus $\langle y\rangle<\langle y,u\rangle\le K$. From the maximality of $\langle y\rangle$ in $K_M$, we get $K_M=\langle y,u\rangle= K$. Now,~\eqref{tired} gives $H=\langle x\rangle$. It is easy to verify that $\langle x\rangle$ and $K_M\cong E_q:((q-1)/2)$ is a normal $2$-covering of $G$.

Assume now $q$ even. If $H$ contains an involution $u$, then $H=\langle x,u\rangle={\bf N}_G(\langle x\rangle)=H_M\cong D_{2(q+1)}$ so that, again by ~\eqref{tired}, we get $K=\langle y\rangle$ and it is readily seen that $\{ H_M, \langle y\rangle\}$ is a normal $2$-covering of $G$. If $H$ contains no involution, then $H=\langle x\rangle\cong C_{q+1}$, $K$ must contain some involutions $u$ and $K=\langle y,u\rangle=K_M\cong E_q:(q-1)$. Again, it is readily seen that $\{\langle x\rangle, K_M\}$ is a normal $2$-covering of $G$.

\medskip

\textbf{Suppose $G=\mathrm{PSL}_2(q)$ and we are in Case~\eqref{ddd5}.} Here, $q$ is even and we have $H\le H_M\cong D_{2(q+1)}$ and $K\le K_M\cong D_{2(q-1)}$. As above, $G$ contains an element $x$ of order $q+1$ and an element $y$ of order $q-1$ and we have 
\begin{align*}
\langle x\rangle&\le H\le H_M={\bf N}_G(\langle x\rangle)\cong D_{2(q+1)},\\
\langle y\rangle&\le K\le K_M={\bf N}_G(\langle y\rangle)\cong D_{2(q-1)}.
\end{align*}
Since either $H$ or $K$ contains elements of order $2$, we have $H=H_M$ or $K=K_M$. If $K<K_M$, we obtain $H=H_M\cong D_{2(q+1)}$ and $K=\langle y\rangle\cong C_{q-1}$. Note that the covering $\{H_M,\langle y\rangle \}$ has already been considered in the paragraph above.
If $K=K_M$, we instead obtain one additional normal $2$-covering with $H=\langle x\rangle\cong C_{q+1}$.
\medskip

\textbf{Suppose $G=\mathrm{PSL}_3(q)$ and we are in Case~\eqref{ddd6}.} Here it is more convenient to work with $\tilde G=\mathrm{SL}_3(q)$ and the covering $\tilde \mu=\{\tilde H, \tilde{K}\}$. We have $\tilde H\le \tilde{H}_M\cong (q^2+q+1):3$ and $\tilde{K}\le \tilde{K}_M\cong E_q^2:\mathrm{GL}_2(q)$ is parabolic. Now, $\tilde{G}$ contains an element $x$ of order $q^2+q+1$ and an element $y$ of order $q^2-1$. As $\tilde{K}_M$ contains no element having order $q^2+q+1$ and $\tilde{H}_M$ contains no element having order $q^2-1$, up to replacing $x$ and $y$ with suitable $\mathrm{Aut}(\tilde{G})$-conjugates, we deduce 
\begin{align*}
C_{q^2+q+1}\cong \langle x\rangle&\le \tilde{H}\le \tilde{H}_M={\bf N}_{\tilde{G}}(\langle x\rangle) \cong (q^2+q+1):3,\\\,C_{q^2-1}\cong\langle y\rangle&\le \tilde{K}\le \tilde{K}_M\cong E_q^2:\mathrm{GL}_2(q).
\end{align*}
Observe that $C_{q^2+q+1}$ is a maximal subgroup of $\tilde{H}_M$. Thus $\tilde{H}\in \{\langle x\rangle,{\bf N}_{\tilde G}(\langle x\rangle)\}$. 

Assume first that $q$ is not a power of $3$, or that $\tilde{H}=\langle x\rangle\cong C_{q^2+q+1}$. Observe that $\tilde{H}$ has no non-identity unipotent elements. Therefore, all non-identity unipotent elements of $\tilde{G}$ have an $\mathrm{Aut}(\tilde{G})$-conjugate in $\tilde{K}$. Let $\tilde{P}\cong E_q^2$ be the Fitting subgroup of $\tilde{K}_M$. Note that the non-identity unipotent elements in $\tilde P$ are not regular. Since $\tilde{G}$ contains regular unipotent elements, $\tilde{K}$ contains a regular unipotent element  $u$ and $u\in \tilde{K}_M\setminus \tilde{P}$. Now, using the subgroup structure of $\mathrm{GL}_2(q)$ and using the fact that $\langle u,y\rangle\le \tilde{K}$, we deduce that $\tilde{K}_M=\langle y,u,\tilde{P}\rangle$, 
 that is, $\tilde{K}\tilde{P}=\tilde{K}_M$. Since $\mathrm{GL}_2(q)$ acts irreducibly on $E_q^2$, we deduce that $\mathrm{GL}_2(q)$ is a maximal subgroup of $E_q^2:\mathrm{GL}_2(q)$. Therefore, either $\tilde K=\tilde{K}_M\cong E_q^2:\mathrm{GL}_2(q)$ or $\tilde{K}\cong \mathrm{GL}_2(q)$ . In the former case, routine considerations of the possible Jordan forms of the elements in $\tilde{G}=\mathrm{SL}_3(q)$ show that $\{\langle x\rangle,\tilde{K}=\tilde{K}_M\}$ is a normal $2$-covering of $\tilde{G}$. In the latter case, $\tilde{K}$ is a complement of $\tilde{P}$ in $\tilde{K}_M$. Since all complements of $\tilde{P}$ in $\tilde{K}_M$ are $\tilde{K}_M$-conjugate, see~\cite{CPS}\footnote{Strictly speaking, the work in~\cite{CPS} deals with the Lie group $\mathrm{SL}_n(q)$, however, the results there can be used also for deducing the corresponding result for $\mathrm{GL}_n(q)$.}, we may suppose that
$$\tilde{K}=\left\{\begin{pmatrix}\det(A)^{-1}&0\\ 0&A\end{pmatrix}\mid A\in\mathrm{GL}_2(q)\right\}.$$
However, we see that $\tilde{K}$ contains no regular unipotent elements, which is a contradiction. Therefore no $2$-covering arises in this case.

Assume now that $q$ is a power of $3$ and that $\tilde{H}={\bf N}_{\tilde{G}}(\langle x\rangle)\cong C_{q^2+q+1}:C_3$. When $q=3$, the proof follows with a computer computation; therefore for the rest of the argument we suppose $q>3$. This will help for avoiding some degenerate cases. Here the argument is rather similar to the argument above, only slightly more delicate. By~\eqref{tired}, we have $\tilde{K}<\tilde{K}_M$. It can be shown that the non-identity unipotent elements of $\tilde{H}$ are regular. In particular, all non-identity non-regular unipotent elements of $\tilde{G}$ have an $\mathrm{Aut}(\tilde{G})$-conjugate in $\tilde{K}$. Thus, $\tilde{K}$ contains a non-identity non-regular unipotent element. As above, let $\tilde{P}$ be the Fitting subgroup of $\tilde{K}_M$. Observe that $\langle y\rangle\cong C_{q^2-1}$ acts by conjugation transitively on the non-identity elements of $\tilde{P}$. As a consequence, if there exists a non-identity non-regular unipotent element $u\in\tilde{P}\cap \tilde{K}$, then $\tilde{P}\le \tilde{K}$. Using this observation we now locate the non-identity non-regular unipotent elements in $\tilde{K}$ with respect to $\tilde{P}$. 

We first suppose that there exists a non-identity element $u\in\tilde{P}\cap\tilde{K}$. 
Then $\tilde{K}\ge \langle \tilde P,y\rangle\cong E_q^2:(q^2-1)$. All the elements in $E_q^2:(q^2-1)$ having order $q-1$ are $\mathrm{Aut}(\tilde{G})$-conjugate to an element in $\langle y^{q+1}\rangle$  and hence they have two distinct eigenvalues $b$ and $b^{-2}$ in $\mathbb{F}_q$ and their matrices are conjugate to 
\[
\begin{pmatrix}
b^{-2} &0&0\\
0&b&0\\
0&0&b
\end{pmatrix},
\]
with $b\in \mathbb{F}_q^\ast$ having order $q-1$. Since $\tilde{G}$ contains  elements having order $q-1$ and having three distinct eigenvalues in $\mathbb{F}_q$ (here we are using $q>3$) and since $\tilde{H}$ has no elements of order $q-1$, we deduce that $\tilde{K}$ contains a representative of each $\mathrm{Aut}(\tilde{G})$-conjugacy class of elements having order $q-1$ and having three distinct eigenvalues in $\mathbb{F}_q$. Let $\lambda$ be a generator of the multiplicative group $\mathbb{F}_{q}^\ast$ and consider
\[
z':=\begin{pmatrix}
1&0&0\\
0&\lambda&0\\
0&0&\lambda^{-1}
\end{pmatrix}.
\]
Since $q$ is a power of $3$ with $q>3$, it is not hard to verify that $$(\dag)\qquad\{1,\lambda,\lambda^{-1}\}\cap \{-1,-\lambda,-\lambda^{-1}\}=\varnothing.$$ Let $z\in \tilde K$ with $z$ $\mathrm{Aut}(\tilde{G})$-conjugate to $z'$.
Observe that $z$ has three distinct eigenvalues $1,\lambda,\lambda^{-1}$. 
 We now use the subgroup structure of $\mathrm{GL}_2(q)\cong \tilde{K}_M/\tilde{P}$ to prove that $\tilde{K}$ cannot contain such an element. There exists only one proper subgroup of $\mathrm{GL}_2(q)$ properly containing a cyclic subgroup of order $q^2-1$, that is, the normalizer of this subgroup of order $q^2-1$. In other words, as $\tilde{P}\le\tilde{K} $, we have $\tilde{K}\geq \langle y,u,z\rangle\cong E_q^2:((q^2-1).2)$. We identify $(q^2-1).2$ with the group $\mathbb{F}_{q^2}^\ast\rtimes \langle\phi\rangle$, where $\phi:\mathbb{F}_{q^2}\to\mathbb{F}_{q^2}$ is the Frobenius map $x\mapsto x^q$. The elements in $\mathbb{F}_{q^2}^\ast$ having order $q-1$ lie in $\mathbb{F}_{q}^\ast$ and hence have a unique eigenvalue with multiplicity $2$. Therefore our putative element $z$, interpreted as an element of $\mathbb{F}_{q^2}^\ast\rtimes\langle\phi\rangle$, can be written as $\lambda^\ell\phi$, for some $\ell\in \{1,\ldots,q^2-1\}$. Let $\mu_1$ and $\mu_2$ be the two distinct eigenvalues of $\lambda^\ell\phi$. Now,$$\lambda^\ell \phi\lambda^\ell\phi=\lambda^\ell\lambda^{\ell q}=\lambda^{\ell(q+1)}\in\mathbb{F}_q.$$ Therefore, $\lambda^\ell \phi$ squares to an element in $\mathbb{F}_{q}^\ast$ and hence $\mu_1^2=\mu_2^2$. This implies $\mu_2\in \{\mu_1,-\mu_1\}$, contradicting~$(\dag)$.

The only remaining case is that all the non-identity non-regular unipotent elements in $\tilde{K}$ lie in $\tilde{K}\setminus\tilde{P}.$ Thus such elements must be $\mathrm{Aut}(\tilde{G})$-conjugate to an element of $ \tilde{H}$. It follows that $3\mid |\tilde{H}|$ so that 
 $ \tilde{H}=  \tilde{H}_M\cong (q^2+q+1):3$ and, by ~\eqref{tired}, $\tilde K<\tilde K_M$.
As we have seen above (when dealing with the case $q$ not a power of $3$), this  yields $\tilde{K}\cong \mathrm{GL}_2(q)$. Therefore, $\tilde{K}$ is a complement of $\tilde{P}$ in $\tilde{K}_M$. Since all complements of $\tilde{P}$ in $\tilde{K}$ are conjugate, see~\cite{CPS}, we may suppose that
$$\tilde{K}=\left\{\begin{pmatrix}\det(A)^{-1}&0\\ 0&A\end{pmatrix}\mid A\in\mathrm{GL}_2(q)\right\}.$$
Routine considerations on the possible Jordan forms of the elements of $\tilde G=\mathrm{SL}_3(q)$ yield that $H={\bf N}_{\tilde{G}}(\langle x\rangle)\cong (q^2+q+1):3$ and $\tilde{K}\cong \mathrm{GL}_2(q)$ is a normal $2$-covering.

\medskip

\textbf{Suppose $G=\mathrm{PSL}_4(q)$ and we are in Case~\eqref{ddd7}.} The case $q=2$ is of no concern because $\mathrm{PSL}_4(2)\cong A_8$. The weak normal coverings of $\mathrm{PSL}_4(3)$ have been found with the auxiliary help of a computer. Therefore, for the rest of this case, we suppose $q\ge 4$. Here it is more convenient to work with $\tilde G=\mathrm{SL}_4(q)$. We have $\tilde H\le \tilde{H}_M\cong \mathrm{SL}_2(q^2).(q+1).2$ and $\tilde{K}\le \tilde{K}_M\cong E_q^3:\mathrm{GL}_3(q)$ is parabolic. The argument is very similar to the argument dealing with $\mathrm{PSL}_3(q)$, but with some minor differences. Now, $\tilde{G}$ contains an element $x$ of order $q^3+q^2+q+1$ and an element $y$ of order $q^3-1$. As $\tilde{K}_M$ contains no element having order $q^3+q^2+q+1$ and $\tilde{H}_M$ contains no element having order $q^3-1$, up to replacing $x$ and $y$ with suitable $\mathrm{Aut}(\tilde{G})$-conjugates, we deduce 
\begin{align*}
C_{q^3+q^2+q+1}\cong \langle x\rangle&\le \tilde{H}\le \tilde{H}_M\cong\mathrm{SL}_2(q^2).(q+1).2,\\
C_{q^3-1}\cong\langle y\rangle&\le \tilde{K}\le \tilde{K}_M\cong E_q^3:\mathrm{GL}_3(q).
\end{align*}

Assume first that $q$ is odd. We claim that $\tilde K=\tilde{K}_M$. Since $2(q+1)$ is relatively prime to $q$, the non-identity unipotent elements of $\tilde{H}$ all lie in $\mathrm{SL}_2(q^2)$. The Jordan form of the non-identity unipotent elements of $\mathrm{SL}_2(q^2)$, viewed as $2\times 2$-matrices over $\mathbb{F}_{q^2}$, is
\[
\begin{pmatrix}
1&1\\
0&1
\end{pmatrix}.
\]

These matrices, viewed as $4\times 4$-matrices over $\mathbb{F}_q$, have Jordan form
\[
\begin{pmatrix}
1&1&0&0\\
0&1&0&0\\
0&0&1&1\\
0&0&0&1
\end{pmatrix}.
\]

In particular, $\tilde H$ has no regular unipotent elements. Therefore $\tilde{K}$ contains a regular unipotent element $u$. Let $\tilde P\cong E_q^3$ be the Fitting subgroup of $\tilde{K}_M$. Using the subgroup structure of $\mathrm{GL}_3(q)$, which can be deduced from the subgroup structure of $\mathrm{SL}_3(q)$ in~\cite{bhr},  it is not hard to verify that one of the following holds:
\begin{itemize}
\item $\langle \tilde{P},y,u\rangle=\tilde{K}_M$,
\item $q$  is a power of $3$ and $\langle \tilde{P},y,u\rangle/\tilde{P}\cong (q^3-1):3$.
\end{itemize} 
In the first case, as we have seen similarly a few times in the case of $\tilde{G}=\mathrm{SL}_3(q)$, we deduce that either $\tilde{K}=\tilde{K}_M$ or $\tilde{K}\cong \mathrm{GL}_3(q)$ is a complement of $\tilde{P}$ in $\tilde{K}$.
 In the first possibility our claim is proved. In the second possibility, using~\cite{CPS}, we have that all the complements of $\tilde{P}$ in $\tilde{K}_M$ are conjugate and hence $\tilde{K}$ is conjugate to
$$\left\{\begin{pmatrix}\det(A)^{-1}&0\\ 0&A\end{pmatrix}\mid A\in\mathrm{GL}_3(q)\right\}.$$
Since $\mathrm{GL}_3(q)\le\tilde{G}=\mathrm{SL}_4(q)$ has no element having Jordan form consisting only of one block, we deduce that $u$ cannot lie in a conjugate to $\mathrm{GL}_3(q)$ and hence we reach a contradiction. 

In the second case,  that is when $q$  is a power of $3$ and $\langle \tilde{P},y,u\rangle/\tilde{P}\cong (q^3-1):3$, the same argument gives that either $\tilde{K}=\tilde{K}_M$ and our claim is proved, or $\tilde{K}=\langle y,u\rangle\cong E_q^3:(q^3-1):3$. If $\tilde{K}=\langle y,u\rangle$, we prove that we reach a contradiction. Recall that $q>3$. Let $z\in\tilde{G}=\mathrm{SL}_4(q)$, having type $2\oplus 1\oplus 1$, with $\mathbb{F}_q^4=V_1\oplus V_2\oplus V_3$ with $z$ inducing a matrix $s$ of order $q^2-1$ on $V_1$, the $1\times 1$ matrix $(\mathrm{det} (s))^{-1}$ on $V_2$ and the $1\times 1$  matrix $1$ on $V_3$. %Then  $V_1,V_2,V_3$ are the only  $\mathbb{F}_q\langle z\rangle$-submodules.
   We show that no $\mathrm{Aut}(\tilde{G})$-conjugate of $z$ in neither in $\tilde{H}_M$ nor in $\tilde{K}.$ For $\tilde{K}$ this comes immediately from  the fact that $\order z=q^2-1\nmid |\tilde{K}|$. Assume next, by contradiction, that  $z\in\tilde{H}_M=(\mathrm{GL}_2(q^2).2)\cap \mathrm{SL}_4(q)$.  Then $z^2\in \mathrm{GL}_2(q^2)$. Observe now that $s^2$ acts irreducibly  on $V_1$. Indeed, by~\cite[Lemma~2.4]{bl}, the minimal polynomial of $s^2$ is irreducible of degree $r\mid 2$, where $r\in \mathbb{N}$ is minimum such that $\frac{q^2-1}{q^r-1}\mid 2.$ The possibility $r=1$ implies the contradiction $q+1\leq 2$, so that $r=2$.  Moreover $$\order{(\mathrm{det} (s))^{-2}}=\frac{q-1}{2}\neq 1,$$ because $q\neq 3$. As a consequence, the action of  $z^2$  is  of type $2\oplus 1\oplus 1$ and $1$ is an eigenvalue of $z^2$ of multiplicity one. Since no matrix in $\mathrm{GL}_2(q^2)$ can have $1$ as eigenvalue of multiplicity one, we have reached a contradiction.
  
   Therefore,  for every $q$ odd, we have proved our claim, that is, $\tilde{K}=\tilde{K}_M$.

Again, under the assumption that $q$ is odd, we claim that $\mathrm{SL}_2(q^2).(q+1)= \tilde{H}$; recall $q>3$. From~\eqref{tired}, it suffices to show that $\mathrm{SL}_2(q^2).(q+1)\le \tilde{H}$. Let $w$ be an element of order $q^2-1$ in $\tilde{G}=\mathrm{SL}_4(q)$, having type $2\oplus 2$ and such that $\mathbb{F}_q^4$ has only two distinct irreducible $\mathbb{F}_q\langle z\rangle$-subspaces. To construct such an element $w$ it suffices to take a Singer cycle $s$ of $\mathrm{GL}_2(q)$  and consider $w=s\oplus (s^{-1})^T$. Now, by order considerations, $w$ has no $\mathrm{Aut}(\tilde{G})$-conjugate in $\tilde{K}$ and hence $w$ has an $\mathrm{Aut}(\tilde{G})$-conjugate in $\tilde{H}$. Without loss of generality, we may suppose that $w\in \tilde{H}$. Using the subgroup structure of $\mathrm{SL}_2(q^2)$ and recalling that $q^2\neq 9$, it is not hard to verify that $\langle x^{q+1},w\rangle\cong\mathrm{SL}_2(q^2)$. Therefore, $\tilde{H}\ge \langle x,w\rangle\ge \mathrm{SL}_2(q^2).(q+1)$.

Routine computations with the possible Jordan forms of the elements of $\mathrm{SL}_4(q)$ show that $\tilde{H}\cong\mathrm{SL}_2(q^2).(q+1)$ and $\tilde{K}\cong E_q^3:\mathrm{GL}_3(q)$ is a normal $2$-covering.

\medskip

Assume next that $q$ is even. Since $q+1$ is relatively prime to $q$, the non-identity unipotent elements of $\tilde{H}$ all lie in $\mathrm{SL}_2(q^2).2$. 
It can be verified that the Jordan forms of the non-identity unipotent elements of  $\mathrm{SL}_2(q^2).2$ are
\[
\begin{pmatrix}
1&1&0&0\\
0&1&0&0\\
0&0&1&1\\
0&0&0&1
\end{pmatrix},
\begin{pmatrix}
1&1&0&0\\
0&1&1&0\\
0&0&1&1\\
0&0&0&1
\end{pmatrix}. 
\]
Indeed, the Jordan form of the non-identity unipotent elements of $\mathrm{SL}_2(q)$, viewed as $2\times 2$-matrices over $\mathbb{F}_{q^2},$ is
\[
\begin{pmatrix}
1&1\\
0&1
\end{pmatrix}.
\]
These matrices, viewed as $4\times 4$-matrices over $\mathbb{F}_q$, have Jordan form consisting of two Jordan blocks of cardinality $2$. Now, let $\phi:\mathbb{F}_{q^2}\to\mathbb{F}_{q^2}$ be the Galois automorphism defined by $x\mapsto x^q$, $\forall x\in\mathbb{F}_{q^2}$. We now let $u\phi\in \mathrm{SL}_2(q^2).2$ be a $2$-element and we study the Jordan form of $u\phi$, when viewed as an element of $\mathrm{SL}_4(q)$. Replacing $u$ by a suitable $\mathrm{SL}_2(q^2)$-conjugate, we may suppose that
\[
\begin{pmatrix}
1&a\\
0&1
\end{pmatrix},
\]
for some $a\in\mathbb{F}_{q^2}$. We suppose that, for some $a\in\mathbb{F}_{q^2}$, the Jordan form of $u\phi$, viewed as an element of $\mathrm{SL}_4(q)$ does not have two Jordan blocks of size $2$. Therefore, $u\phi$ has either two Jordan blocks of size $1$ or $u\phi$ has a Jordan block of size at least $3$. In the first, case $u\phi$ has order $2$ and, in the second case, $u\phi$ has order $4$. Observe that
\[
(u\phi)^2=uu^\phi=
\begin{pmatrix}
1&a+a^q\\
0&1
\end{pmatrix}.
\]
In particular, $(u\phi)^2$ viewed as a matrix of $\mathrm{SL}_4(q)$ is either the identity matrix or has two Jordan blocks of size $2$. In the second case, $u\phi$ is a Jordan block of cardinality $4$, as we claimed above. In the first case, $a+a^q=0$, that is, $a\in\mathbb{F}_q$. Now, let $\varepsilon_1,\varepsilon_2$ be an $\mathbb{F}_{q^2}$-basis of $\mathbb{F}_{q^2}^2$ and let $\lambda,\lambda^q$ be a Galois basis for $\mathbb{F}_{q^2}/\mathbb{F}_q$. Observe that $\lambda\varepsilon_1,\lambda^q\varepsilon_1,\lambda\varepsilon_2,\lambda^q\varepsilon_2$ is an $\mathbb{F}_q$-basis of $\mathbb{F}_{q^2}^2. $With respect to this basis the Jordan form of $u\phi$ is
\[
\begin{pmatrix}
0&1&0&a\\
1&0&a&0\\
0&0&0&1\\
0&0&1&0
\end{pmatrix}.
\]
This matrix has two Jordan blocks of size $2$.

In particular, those Jordan forms have two blocks of size $2$ or just one block of size $4$. Note also that we have one more Jordan form than in the case of $q$ odd. Thus, we need extra care to treat this case.
Now, $\tilde H$ has no unipotent elements having two Jordan blocks of sizes $3$ and $1$. Therefore $\tilde{K}$ contains a unipotent element $u$ having two Jordan blocks of sizes $3$ and $1$. Let $\tilde P\cong E_q^3$ be the Fitting subgroup of $\tilde{K}_M$. Using the subgroup structure of $\mathrm{GL}_3(q)$ it is not hard to verify that $\langle \tilde{P},y,u\rangle=\tilde{K}_M$. 

We deduce that either $\tilde{K}=\tilde{K}_M$ or $\tilde{K}\cong \mathrm{GL}_3(q)$ is a complement of $\tilde{P}$ in $\tilde{K}$.
With~\cite{CPS}, we deduce that all complements of $\tilde{P}$ in $\tilde{K}_M$ are conjugate. Summing up, either
$$\tilde{K}=\tilde{K}_M\cong E_q^3:\mathrm{GL}_3(q)\hbox{ or }\mathrm{GL}_3(q)\cong \tilde{K}<\tilde{K}_M\cong E_q^3:\mathrm{GL}_3(q),$$
where in the second possibility we may assume that $\tilde{K}$ is the standard complement of $\mathrm{GL}_3(q)$ in $E_q^3:\mathrm{GL}_3(q)$.

We claim that $\mathrm{SL}_2(q^2).(q+1)\le  \tilde{H}$.  Let $w$ be an element of order $q^2-1$ in $\tilde{G}=\mathrm{SL}_4(q)$, having type $2\oplus 2$ and such that $\mathbb{F}_q^4$ has only two distinct irreducible $\mathbb{F}_q\langle w\rangle$-submodules. To construct such an element $w$ it suffices to take a Singer cycle $s$ of $\mathrm{GL}_2(q)$ and consider $w=s\oplus (s^{-1})^T$. Now, $w$ has no $\mathrm{Aut}(\tilde{G})$-conjugate in $\tilde{K}$ and hence $w$ has an $\mathrm{Aut}(\tilde{G})$-conjugate in $\tilde{H}$. Without loss of generality, we may suppose that $w\in \tilde{H}$. Using the subgroup structure of $\mathrm{SL}_2(q^2)$ it is not hard to verify that $\langle x^{q+1},w\rangle\cong\mathrm{SL}_2(q^2)$. Therefore, $\tilde{H}\ge \langle x,w\rangle\cong \mathrm{SL}_2(q^2).(q+1)$.

We claim that $\tilde K=\tilde K_M\cong E_q^3:\mathrm{GL}_3(q)$. We argue by contradiction and we suppose that $\tilde K<\tilde K_M$ and hence, from what we have seen before, $\tilde K\cong \mathrm{GL}_3(q)$. Now, let $$A=\begin{pmatrix}a_{11}&a_{12}\\a_{21}&a_{22}\end{pmatrix}\in\mathrm{GL}_2(q)$$ be a Singer cycle of order $q^2-1$ and let $a:=\det(A)\in\mathbb{F}_q$. From~\eqref{piadet}, $a$ is a generator of the multiplicative group $\mathbb{F}_q^\ast$ and hence $\order a=q-1$. Since $q$ is even, there exists $b\in\mathbb{F}_q$ with $b^2=a^{-1}$. Let
\[
g:=\begin{pmatrix}
b&1&0&0\\
0&b&0&0\\
0&0&a_{11}&a_{12}\\
0&0&a_{21}&a_{22}\\
\end{pmatrix}\in \mathrm{SL}_4(q).
\]
It is easy to verify that $g$ has no $\mathrm{Aut}(\tilde G)$-conjugate in $\tilde K\cong\mathrm{GL}_3(q)$. Therefore, $g$ has an $\mathrm{Aut}(\tilde G)$-conjugate in $\tilde H\le \tilde H_M\cong \mathrm{SL}_2(q).(q+1).2$. Replacing $g$ by a suitable $\mathrm{Aut}(\tilde G)$-conjugate, we may suppose that $g\in \tilde H_M\cong \mathrm{SL}_2(q^2).(q+1).2=(\mathrm{GL}_2(q^2)\cap\mathrm{SL}_4(q)).2$. Now, $\order g=2(q^2-1)$ and 
 \[g^{q^2-1}=
\begin{pmatrix}
1&1&0&0\\
0&1&0&0\\
0&0&1&0\\
0&0&0&1
\end{pmatrix}\in \tilde H_M.
\]
As we have seen above, $g^{q^2-1}$ is not the Jordan form of an element in $(\mathrm{GL}_2(q^2)\cap\mathrm{SL}_4(q)).2$: the Jordan form of the non-identity unipotent elements in this group consists of either two Jordan blocks of size $2$ or a single Jordan block of size $4$. This shows that $\tilde K$ cannot be isomorphic to $\mathrm{GL}_3(q)$.
 
Routine computations with the possible Jordan forms of the elements of $\mathrm{SL}_4(q)$ show that $\tilde{H}\cong\mathrm{SL}_2(q^2).(q+1)$ and $\tilde{K}\cong E_q^3:\mathrm{GL}_3(q)$ is a normal $2$-covering. 

\medskip

\noindent\textsc{Unitary groups: Cases~\eqref{ddd8} and~\eqref{ddd9}.}

\smallskip

\noindent Suppose $G=\mathrm{PSU}_3(q)$ and we are in Case~\eqref{ddd8}. Observe that $G\cong\tilde{G}$, because $q$ is a power of $3$ and $\gcd(3,q+1)=1$. Here, we have $H\le H_M\cong (q^2-q+1):3$ and $K\le K_M\cong \mathrm{GU}_2(q)$. Now, $G$ contains an element $x$ of order $q^2-q+1$ and an element $y$ of order $q^2-1$. As $K_M$ contains no element having order $q^2-q+1$ and $H_M$ contains no element having order $q^2-1$, up to replacing $x$ and $y$ with suitable $\mathrm{Aut}(G)$-conjugates, we deduce 
$$C_{q^2-q+1}\cong \langle x\rangle\le H\le H_M={\bf N}_{G}(\langle x\rangle) \cong (q^2-q+1):3$$ and $$\,C_{q^2-1}\cong\langle y\rangle\le K\le K_M\cong \mathrm{SU}_2(q).(q+1).$$
Observe that $\langle x\rangle$ is a maximal subgroup of $H_M$. In particular, $H\in \{\langle x\rangle,{\bf N}_{G}(\langle x\rangle)\}$.

Assume $H=\langle x\rangle\cong C_{q^2-q+1}$. Then all the unipotent elements of $G$ are $\mathrm{Aut}(G)$-conjugate to an element of $K$, however this is impossible because $K_M$ contains no regular unipotent elements. Thus $H=H_M={\bf N}_G(\langle x\rangle)$. The non-identity unipotent elements of $H=H_M$ are regular, and hence $H$ contains no non-identity non-regular unipotent elements. Therefore, $K$ contains at least one of these elements, say $u$. In particular $3$ divides $|K|$. Now, $|\mathrm{GU}_2(q):\mathrm{SU}_2(q)|=q+1$, $\mathrm{GU}_2(q)/\mathrm{SU}_2(q)$ is cyclic of order $q+1$ and $\mathrm{SU}_2(q)\cong \mathrm{SL}_2(q)$. Using the fact that $K$ contains an element of order $q^2-1$ and an element of order $3$ and using the subgroup structure of $\mathrm{SL}_2(q)$, we deduce that either $K=K_M$ or 
$K=P.(q+1)$, where $P\cong\langle y^{q+1},u\rangle$ is a parabolic subgroup of $\mathrm{SL}_2(q)$. Recalling~\eqref{tired},  we obtain $K\cong P.(q+1)$. In particular, $|K|=q(q^2-1)$. However, in this case, we run into trouble with the elements of order $q+1$. Indeed, as $q>2$, $G$ has two types of elements of order $q+1$: elements having $2$ eigenvalues in $\mathbb{F}_{q^2}$ and elements having $3$ distinct eigenvalues in $\mathbb{F}_{q^2}$, see the proof of Lemma \ref{dimension3unitary}. It can be verified that the group $K$ has no elements of the second type;  therefore we have no more additional weak normal $2$-coverings.

\medskip

Suppose $G=\mathrm{PSU}_4(q)$ and we are in Case~\eqref{ddd9}. Here it is more convenient to work with $\tilde G=\mathrm{SU}_4(q)$. We have $\tilde H\le \tilde{H}_M\cong \mathrm{GU}_3(q)$ and $\tilde{K}\le \tilde{K}_M\cong E_q^4:\mathrm{SL}_2(q^2):(q-1)$ is parabolic. Now, $\tilde{G}$ contains an element $x$ of order $q^3+1$ and an element $y$ of order $(q^4-1)/(q+1)$. Using the fact that $q\geq 4$ we deduce that $\tilde{K}_M$ contains no element having order $q^3+1$ and $\tilde{H}_M$ contains no element having order $(q^4-1)/(q+1)$, up to replacing $x$ and $y$ with suitable $\mathrm{Aut}(\tilde{G})$-conjugates, we deduce 
$$C_{q^3+1}\cong \langle x\rangle\le \tilde{H}\le \tilde{H}_M\cong\mathrm{GU}_3(q)$$
and $$C_{(q^4-1)/(q+1)}\cong\langle y\rangle\le \tilde{K}\le \tilde{K}_M\cong E_q^4:\mathrm{SL}_2(q^2):(q-1).$$

We claim that 
$$\tilde{K}=\tilde{K}_M.$$
As $\tilde H\le \tilde{H}_M\cong\mathrm{GU}_3(q)$, $\tilde H$ contains no regular unipotent elements. Therefore, $\tilde K$ contains a regular unipotent element $u$. Let $\tilde P$ be the Fitting subgroup of $\tilde{K}_M$.  Observe that $\order {y^{q-1}}=q^2+1$ is the order of a Singer cycle in $\mathrm{SL}_2(q^2)$ so that, by Proposition \ref{sing-ord}, it is indeed a Singer cycle of  $\mathrm{SL}_2(q^2)$. Using the subgroup structure of $\mathrm{SL}_2(q^2)$, it can be verified that either $\langle \tilde{P},u,y^{q-1}\rangle/\tilde{P}\cong\mathrm{SL}_2(q^2)$ or $\langle \tilde{P},u,y^{q-1}\rangle/\tilde{P}\cong \langle y^{q-1}\rangle.2 $ and $q$ is even.  We show that the second possibility is ridiculous. Indeed, if $\langle \tilde{P},u,y^{q-1}\rangle/\tilde{P}\cong \langle y^{q-1}\rangle.2$, then a field automorphism of $\mathrm{SL}_2(q^2)$, when viewed as an element of $\mathrm{SL}_4(q)$, has Jordan form consisting only of one Jordan block, because so does $u$. However, it follows with an explicit computation fixing a basis that a field automorphism consists of two Jordan blocks of size $2$.

Thus $\tilde{K}\tilde{P}\ge \langle \tilde P,u,y^{q-1}\rangle\cong E_q^4:\mathrm{SL}_2(q^2)$. As $y\in \tilde K$, we have $\tilde K \tilde P=\tilde{K}_M$. 
 In particular, $\tilde{K}$ acts by conjugation irreducibly on $\tilde{P}$. Therefore, either $\tilde{K}=\tilde{K}_M$ or $\tilde{K}\cap \tilde{P}=1.$ In the first case, our claim is proved. In the second case, $\tilde{K}\cong \mathrm{SL}_2(q^2):(q-1)$ is a complement of $\tilde{P}$ in $\tilde{K}_M$. By~\cite{CPS}, $\tilde{K}$ is conjugate to the standard complement of $\mathrm{SL}_2(q^2):(q-1)$ in $E_q^4:\mathrm{SL}_2(q^2):(q-1)$. However, we obtain a contradiction because this group contains no regular unipotent elements.

We now claim that
$$\tilde{H}=\tilde{H}_M$$
and hence~\eqref{tired} is not satisfied.
We use the subgroup structure of $\mathrm{SU}_3(q)$ and the fact that $\tilde{H}\cap \mathrm{SU}_3(q)$ contains the Singer cycle $x^{q+1}$. Now, with the restriction that $q\ge 4$, the overgroups of $\langle x^{q+1}\rangle$ in $\mathrm{SU}_3(q)$ consist of only three groups
$$\langle x^{q+1}\rangle<{\bf N}_{\mathrm{SU}_3(q)}(\langle x^{q+1}\rangle)\cong (q^2-q+1).3<\mathrm{SU}_3(q).$$
Therefore, $\tilde{H}=\langle x\rangle$, or $\tilde{H}={\bf N}_{\tilde{G}}(\langle x\rangle)$, or $\tilde{H}=\tilde{H}_M$. However, the last possibility is excluded by~\eqref{tired}  and by the fact that $\tilde K=\tilde{K}_M$. To conclude our analysis and reach a contradiction we look at elements having order $q+1$. The elements of order $q+1$ in $\tilde{G}=\mathrm{SU}_4(q)$ are of two types: 
\begin{itemize}
\item they have $4$ distinct eigenvalues in $\mathbb{F}_{q^2}$, or
\item they have $3$ (or less) distinct eigenvalues in $\mathbb{F}_{q^2}$.
\end{itemize}
The elements of order $q+1$ in $\tilde{H}$ when $\tilde{H}=\langle y\rangle$ or $\tilde{H}={\bf N}_{\tilde{G}}(\langle y\rangle)$ cannot have four distinct eigenvalues because these elements arise as the $q^{2}-q+1$ power of $x$. Similarly, the elements of order $q+1$ in $\tilde{K}=E_q^4:\mathrm{SL}_2(q^2):(q-1)$ cannot have four distinct eigenvalues.

Summing up, when $\tilde{G}=\mathrm{SU}_4(q)$, the only weak normal $2$-covering has maximal components.

\smallskip

\noindent\textsc{Exceptional groups: Cases~\eqref{ddd1} and~\eqref{ddd2}.}

\smallskip

\noindent Suppose we  are in Case~\eqref{ddd2}, that is, $G=F_4(q)$, $H_M\cong {}^3D_4(q).3$, $K_M\cong \mathrm{Spin}_9(q)\cong 2.\Omega_9(q)$, $q=3^a$ and $a\ge 1$. When $q=3$, we have verified that the only weak normal $2$-covering of $G$ uses maximal components. Assume now $q>3$. Now, $G$ contains elements of order $q^4 - q^2 + 1$
and elements of order $(q + 1)(q^3-1)$. Let $x,y\in G$ with $\order x=q^4-q^2+1$ and $\order y=(q+1)(q^3-1)$. As $K_M$ contains no element having order $q^4-q^2+1$ and $H_M$ contains no element having order $(q+1)(q^3-1)$, we deduce that $H$ contains elements of order $q^4-q^2+1$ and $K$ contains elements of order $(q+1)(q^3-1)$. Hence $(q^4-q^2+1)\le H\le H_M\cong {}^3D_4(q).3$ and $((q+1)(q^3-1))\le K\le K_M\cong 2.\Omega_9(q)$. 

We use the subgroup structure of ${}^3D_4(q).3$, see~\cite{bhr}. The overgroups of $\langle x\rangle$ in ${}^3D_4(q)$ consist of only three groups 
$$\langle x\rangle<{\bf N}_{{}^3D_4(q)}(\langle x\rangle)\cong (q^4-q^2+1).[4]<{}^3D_4(q).$$
In particular, $H\cap {}^3D_4(q)$ is one of these groups. Now, the unipotent elements of $2.\Omega_9(q)$ have order $1,3,9$ and the unipotent elements of ${}^3D_4(q)$ have also order $1,3,9$. However, in ${}^3D_4(q).3\setminus {}^3D_4(q)$ there are unipotent elements having order $27$. This shows that $H$ contains a unipotent element $u$ having order $27$. Thus $u^3\in {}^3D_4(q)$ has order $9$ and $\langle x,u^3\rangle\le H$. Therefore, $\langle x,u^3\rangle\cong {}^3D_4(q)$ and $H_M=H=\langle x,u\rangle={}^3D_4(q).3$.

We use the subgroup structure of $2.\Omega_9(q)$, see~\cite{bhr}. Now, $2.\Omega_9(q)$ contains an element $z$ of order $q^2 + 1$. As $H=H_M$ contains no element having order $q^2+1$, we deduce that $K$ contains elements of order $q^2+1$. In particular, we may suppose, replacing $z$ if necessary, that $z\in K$ and hence $\langle y,z\rangle\leq K$. In particular, $|K|$ is divisible by $(q^2+1)(q^2-1)(q^3-1)/2$. Using the subgroup structure of $2.\Omega_9(q)$ we deduce that $K=K_M\cong 2.\Omega_9(q)$.

Summing up, when $G=F_4(q)$, the only weak normal $2$-covering has maximal components.

\medskip

Suppose now that we are in Case \eqref{ddd1}.
Here, $G=G_2(q)$, $q\ge 4$ and $q$ is even. When $q=4$, we have used a computer to establish the result. Therefore, for the rest of the argument we suppose that $q\ge 8$. Here, we have $H\le H_M\cong \mathrm{SL}_3(q).2$ and $K\le K_M\cong\mathrm{SU}_3(q).2$. Now, $H_M$ contains elements having order $q^2+q+1$ and $K_M$ contains elements of order $q^2-q+1$. As $K_M$ contains no element having order $q^2+q+1$ and $H_M$ contains no element having order $q^2-q+1$, we deduce that $H$ contains elements of order $q^2+q+1$ and $K$ contains elements of order $q^2-q+1$. Hence $(q^2+q+1)\le H\le H_M\cong \mathrm{SL}_3(q).2$ and $(q^2-q+1)\le K\le K_M\cong \mathrm{SU}_3(q).2$.

% Now, $G_2(q)$ contains an element having order $4\cdot (q-1)$. Since $\mathrm{SU}_3(q).2$ contains no elements having this order, we deduce that these elements lie in $H$ and hence $H=\mathrm{SL}_3(q).2$. 

Using~\cite{carter}, we see that there exists a $1$ to $1$ correspondence between the conjugacy classes of maximal tori in $G_2(q)$ and the conjugacy classes in the Weyl group $W$ of $G_2(q)$. As $W\cong D_{12}$ is dihedral of order $12$, we see that $W$ has two distinct conjugacy classes of reflections, and these are preserved by field automorphisms of $G_2(q)$. These two conjugacy classes correspond to maximal tori of order $q^2-1$ in $G_2(q)$. This yields that, if $y_1\in \mathrm{SL}_3(q).2$ has order $q^2-1$ and $y_2\in \mathrm{SU}_3(q).2$ has order $q^2-1$, then $\langle y_1\rangle$ and $\langle y_2\rangle$ are not $\mathrm{Aut}(G)$-conjugate. In particular, both $H$ and $K$ contain an element having order $q^2-1$. This yields
$\mathrm{SL}_3(q)\le H\le H_M\cong \mathrm{SL}_3(q).2$ and $\mathrm{SU}_3(q)\le K\le K_M\cong \mathrm{SU}_3(q).2$.

In the process of computing the character table for $G_2(q)$, Enomoto and Yamada~\cite[Section~1]{enomoto1} have shown that there exists two conjugacy classes of cyclic unipotent subgroups of order $8$, with representatives $\langle u_1\rangle$ and $\langle u_2\rangle$. Now, as $\mathrm{SL}_3(q).2$ and $\mathrm{SU}_3(q).2$ contain elements having order $8$, we deduce, replacing $u_1$ with $u_2$ if necessary, that $H$ contains a $\mathrm{Aut}(G)$-conjugate $u_1'$ of $u_1$ and $K$ contains a $\mathrm{Aut}(G)$-conjugate $u_2'$ of $u_2$. Thus
$H=\langle \mathrm{SL}_3(q),u_1'\rangle=H_M\cong\mathrm{SL}_3(q).2$ and $K=\langle \mathrm{SU}_3(q),u_2'\rangle=K_M\cong\mathrm{SU}_3(q).2$.
 
% For $F_4(q)$, we see that $H$ must contain an element of order $q^4-q^2+1$, because $\gcd(q^4-q^2+1,|\mathrm{Spin}_9(q)|)=1$. The list of the maximal subgroups of ${^3}D_4(q)$ gives that either $H={^3}D_4(q)$, or $H={^3}D_4(q).3$, or $H=(q^4-q^2+1).3$ or $H=(q^4-q^2+1).[12]$. 

%For $G_2(q)$ it is easier, we have $H\le \mathrm{SL}_2(q).2$ and $K\le \mathrm{SU}_3(q).2$. The group $H$ must contain an element of order $q^2+q+1$ and $K$ must contain an element of order $q^2-q+1$ and hence $(q^2+q+1).[6]\le H\le\mathrm{SL}_3(q).2$ and $(q^2-q+1).2\le K\le \mathrm{SU}_3(q).2$. Now, $G_2(q)$ contains an element having order $4\cdot (q-1)$. Since $\mathrm{SU}_3(q).2$ contains no elements having this order, we deduce that these elements lie in $H$ and hence $H=\mathrm{SL}_3(q).2$. 

\smallskip

\noindent\textsc{Symplectic groups: Cases~\eqref{ddd10} and~\eqref{ddd11}.}

\smallskip

\noindent 
 We now consider Case~\eqref{ddd10}, with $n\ge 6$. When $q\in \{2,4\}$ and $n\in \{6,8\}$, we have determined the weak normal $2$-coverings with the help of a computer. In particular, when $n\in \{6,8\}$, we may suppose that $q\ge 8$.
 
 Here, we have $H\le H_M\cong \mathrm{SO}_n^-(q)$ and $K\le K_M\cong \mathrm{SO}_n^+(q)$.  Now, for each $\ell\in \{0,\ldots,n/2-1\}$, $G$ contains an element $x_\ell$ with 
 \[
 {\bf o}(x_\ell)
 =
 \begin{cases}
\frac{(q^\ell-1)(q^{n/2-\ell}+1)}{\gcd(q^\ell-1,q^{n/2-\ell}+1)}& \textrm{if } \ell>0,\\
q^{n/2}+1& \textrm{if } \ell=0.
\end{cases}
\] As $K_M$ contains no element $\mathrm{Aut}(G)$-conjugate to $x_\ell$ (one can see this for example using~\cite{guest}), replacing $x_\ell$ with a suitable $\mathrm{Aut}(G)$-conjugate, we have $x_\ell\in H$ for all $\ell\in \{0,\ldots,n/2-1\}$. In particular, $H$ contains the Singer cycle $x_0$ for $\Omega_{n}^-(q)$. By using Lemma~\ref{msw-} when $n\ge 8$ and~\cite{bhr} when $n=6$, we obtain that one of the following holds
\begin{itemize}
\item there exists a prime $s$ with $s$ dividing $n/2$, $n/s\ge 4$ such that 
$H\le \Omega_{n/s}^-(q^s).s$, 
\item $n/2$ is odd and $H\le \mathrm{GU}_{n/2}(q).2$, or
\item $\Omega_{n}^-(q)\le H$.
\end{itemize}
Here we are using the fact that, when $n=6$, $q\ge 8$. Now, it is not hard to check that it is impossible for $|\Omega_{n/s}^-(q^s).s|$ or for $|\mathrm{GU}_{n/2}(q).2|$ to be divisible by $(q^\ell-1)(q^{n/2-\ell}+1)/\gcd(q^\ell-1,q^{n/2-\ell}+1)$, for every $\ell\in \{1,\ldots,n/2-1\}$. Therefore, $$\Omega_{n}^-(q)\le H\le \mathrm{SO}_n^-(q)\cong\Omega_n^-(q).2.$$

Now, for each $\ell\in \{1,\ldots,n/2-1\}$, $G$ contains an element $y_\ell$ with $${\bf o}(y_\ell)=\frac{(q^\ell+1)(q^{n/2-\ell}+1)}{\gcd(q^\ell+1,q^{n/2-\ell}+1)}.$$ 
and type $2\ell\oplus (n-2\ell)$. As $H_M$ contains no element $\mathrm{Aut}(G)$-conjugate to $y_\ell$, replacing $y_\ell$ with a suitable $\mathrm{Aut}(G)$-conjugate, we have $y_\ell\in K$ for all $\ell\in \{1,\ldots,n/2-1\}$. In particular, $K$ contains $y_1$, which has type $2\oplus (n-2)$. By using Lemma~\ref{msw+} when $n\ge 10$ and~\cite{bhr} when $n\in\{6,8\}$, we obtain that one of the following holds
\begin{itemize}
\item $K\le (\Omega_{2}^-(q)\perp\Omega_{n-2}^-(q)).[4]$, or
\item $n/2$ is even and $K\le \mathrm{GU}_{n/2}(q).[4]$, or
\item $\Omega_n^+(q)\le K$.
\end{itemize}
Here we are using the fact that, when $n\in\{6,8\}$, $q\ge 8$. Now, it is not hard to check that it is impossible  for $|\mathrm{GU}_{n/2}(q).[4]|$ to be divisible by $(q^i+1)(q^{n/2-i}+1)/\gcd(q^i+1,q^{n/2-i}+1)$, for every $i\in \{1,\ldots,n/2-1\}$. Similarly, $|(\Omega_2^-(q)\perp\Omega_{n-2}^-(q)).[4]|$ is divisible by $(q^i+1)(q^{n/2-i}+1)/\gcd(q^i+1,q^{n/2-i}+1)$, for every $i\in \{1,\ldots,n/2-1\}$, if and only if $q\in \{2,4\}$. However, when $n\ge 8$, the group $\Omega_2^-(q)\perp\Omega_{n-2}^-(q)$ cannot contain $y_2$, which has type $4\oplus (n-4)$. Therefore, $$\Omega_{n}^+(q)\le K\le \mathrm{SO}_n^+(q)\cong\Omega_n^+(q).2.$$

We summarize here Lemma~\ref{aggiunta}. The group $G=\mathrm{Sp}_n(q)$ contains regular unipotent elements $u_+$ and $u_-$, with $u_+$ preserving a non-degenerate quadratic form having Witt index $0$ and polarizing the symplectic form on $G$ and with $u_-$ preserving a non-degenerate quadratic form having Witt index $1$ and polarizing the symplectic form on $G$. Furthermore, $u_+$ and $u_-$ are in distinct $\mathrm{Aut}(\mathrm{Sp}_n(q))$-conjugacy classes. Now, $\Omega_{n}^+(q)$ and $\Omega_n^-(q)$ do not contain regular unipotent elements and $u_+\in\mathrm{SO}_n^+(q)\setminus \Omega_n^+(q)$ and $u_-\in \mathrm{SO}_n^-(q)\setminus \Omega_n^-(q)$. Moreover, $u_+$ has no $\mathrm{Aut}(\mathrm{Sp}_n(q))$-conjugate in $\mathrm{SO}_n^-(q)$ and  $u_-$ has no $\mathrm{Aut}(\mathrm{Sp}_n(q))$-conjugate in $\mathrm{SO}_n^+(q)$. In particular, replacing $u_+$ and $u_-$ by suitable $\mathrm{Aut}(G)$-conjugates, we have that $u_-\in H\setminus\Omega_n^-(q)$ and $u_+\in K\setminus\Omega_n^+(q)$. Therefore,  $H=H_M\cong\mathrm{SO}_n^-(q)$  and $K=K_M\cong\mathrm{SO}_n^+(q)$.

\medskip

We now consider Case~\eqref{ddd10}, with $n= 4$. When $q=4$, we have determined the weak normal $2$-coverings with the help of a computer. Therefore, for the rest of the proof, we suppose $q\ge 8$.
 Here, we have $H\le H_M\cong \mathrm{SO}_4^-(q)$ and $K\le K_M\cong \mathrm{SO}_4^+(q)$. The group $G$ contains a Singer cycle $x$ having order $q^2+1$ and, in the proof of Lemma~\ref{dimension4evensymplectic}, we have proved that $G$ contains an element $x$ having order $q^2+1$ and an element $y$ having order $q+1$ and with $|{\bf C}_{G}(y)|=(q+1)^2$. 
  Moreover, we have shown that $y$ has no $\mathrm{Aut}(G)$-conjugate in $H_M\cong\mathrm{SO}_4^-(q)$. Therefore,
$$C_{q^2+1}\cong \langle x\rangle\le H\le H_M\cong\mathrm{SO}_4^-(q)\,\hbox{ and }\,C_{q+1}\cong \langle y\rangle\le K\le K_M\cong \mathrm{SO}_4^+(q).$$
Now, $\mathrm{SO}_4^-(q)\cong\mathrm{SL}_2(q^2).2$ and $\mathrm{SO}_4^+(q)\cong (\mathrm{SL}_2(q)\times \mathrm{SL}_2(q)).2$. By consulting the maximal subgroups of $\mathrm{SL}_2(q^2)$, we see that one of the following holds:
\begin{itemize}
\item $H\le {\bf N}_{\mathrm{SO}_4^-(q)}({\langle x\rangle})\cong (q^2+1).[4]$, or
\item $\Omega_4^-(q)\le H$.
\end{itemize}
Similarly, by consulting the maximal subgroups of $\mathrm{SL}_2(q)\times\mathrm{SL}_2(q)$, 
we see that one of the following holds:\begin{itemize}
\item $C_{q+1}\times C_{q+1}\le K\le (D_{2(q+1)}\times D_{2(q+1)}).2$, or
\item $\Omega_4^+(q)\le K$.
\end{itemize}

The group $G=\mathrm{Sp}_4(q)$ contains regular unipotent elements $u_+$ and $u_-$, with $u_+$ preserving a non-degenerate quadratic form having Witt index $0$ and polarizing the symplectic form on $G$ and with $u_-$ preserving a non-degenerate quadratic form having Witt index $1$ and polarizing the symplectic form on $G$. Now, $\Omega_{4}^+(q)$ and $\Omega_4^-(q)$ do not contain regular unipotent elements and $\mathrm{SO}_n^+(q)\setminus \Omega_4^+(q)$ and $\mathrm{SO}_4^-(q)\setminus \Omega_4^-(q)$ both contain regular unipotent elements. Additionally, it can be verified that $u_+$ and $u_-$ are not conjugate by a graph or a field automorphism of $G$. In particular, replacing $u_+$ and $u_-$ by suitable $\mathrm{Aut}(G)$-conjugates, we have that $u_-\in H\setminus\Omega_n^-(q)$ and $u_+\in K\setminus\Omega_n^+(q)$.  Taking this into account, we have that one of the followings holds:
\begin{itemize}
\item $H={\bf N}_{\mathrm{SO}_4^-(q)}({\langle x\rangle})\cong (q^2+1).[4]$, or
\item $H=\mathrm{SO}_4^-(q)$.
\end{itemize}
Moreover,
\begin{itemize}
\item $(C_{q+1}\times C_{q+1}).4\le K\le (D_{2(q+1)}\times D_{2(q+1)}).2$, or
\item $K=\mathrm{SO}_4^+(q)$.
\end{itemize} 

We claim  that $H\cong (q^2+1).[4]$ and $K\cong\mathrm{SO}_4^+(q)$ does give rise to a weak normal $2$-covering. Indeed, since $\mathrm{SO}_4^+(q)$ and $\mathrm{SO}_4^-(q)$ is a normal $2$-covering, it suffices to show that each element $g$ in $\mathrm{SO}_4^-(q)\setminus H$ has an $\mathrm{Aut}(\mathrm{Sp}_4(q))$-conjugate in either $H$ or $K\cong\mathrm{SO}_4^+(q)$. Using the fact that $\mathrm{SO}_4^-(q)\cong\mathrm{SL}_2(q^2):2$, we see that $\order g$ is a divisor of $q^2+1$, or $q^2-1$, or $4$, or $2(q-1)$, or $2(q+1)$. Now, the argument is a case by case analysis using also the parametrization of the $\mathrm{Sp}_4(q)$-conjugacy classes in~\cite{enomoto}. We just give details when $\order g=2(q+1)$ and omit the order cases. Here, $g=su$, where $\order u=2$ and $\order s=q+1$. In particular, $s$ is a semisimple element and hence the vector space $V$ decomposes as the direct sum of non-degenerate symplectic spaces $V_1\oplus V_2$ with $\dim V_1=\dim V_2=2$. Observe that $u$ either fixes setwise $V_1$ and $V_2$ or $u$ swaps $V_1$ and $V_2$. Now, a graph automorphism of $\mathrm{Sp}_4(q)$ maps the $\mathcal{C}_8$-subgroup $K=\mathrm{SO}_4^+(q)$ to a $\mathcal{C}_2$-subgroup $(\mathrm{Sp}_2(q)\times\mathrm{Sp}_2(q)):2$ and hence $g$ has an $\mathrm{Aut}(\mathrm{Sp}_4(q))$-conjugate in $K$.

We claim  that the weak normal $2$-covering $H\cong (q^2+1).[4]$ and $K\cong\mathrm{SO}_4^+(q)$ does not give rise to normal $2$-coverings. We argue by contradiction. Here we need to be careful with the role played by the graph automorphism of $\mathrm{Sp}_4(q)$. We argue using elements having order $q^2-1$. As $H$ contains no elements having order $q^2-1$, we deduce that each element having order $q^2-1$ in $\mathrm{Sp}_4(q)$ has a $\mathrm{Sp}_4(q)$-conjugate in $K$. Let $g:=s\oplus s(-1)^T$, where $s\in\mathrm{GL}_2(q)$ is a Singer cycle and $g$ fixes a $2$-dimensional totally isotropic subspace of $\mathbb{F}_q^4$. Since $g$ has a $\mathrm{Sp}_4(q)$-conjugate in $H$, we deduce that $K$ cannot be in the Aschbacher class $\mathcal{C}_2$. Thus $K$ must be in the Aschbacher class $\mathcal{C}_8$, see~\cite[Table~8.15]{bhr}. Now, we suppose that the matrix defining the symplectic form for $\mathrm{Sp}_4(q)$ is
\[
J:=\begin{pmatrix}
0&1&0&0\\
1&0&0&0\\
0&0&0&1\\
0&0&1&0
\end{pmatrix}.
\]
Without loss of generality, we may suppose that $\mathrm{SO}_4^+(q)$ preserves the quadratic form
$$x_1x_2+x_3x_4.$$
Let $A\in\mathrm{SL}_2(q)$ be a Singer cycle having order $q+1$ and let $a\in\mathbb{F}_q^\ast$ with $\order a=q-1$. Set
\[
g:=\begin{pmatrix}A&0&0\\0&a&0\\0&0&a^{-1}\end{pmatrix}
.\]
Now, $g$ induces on $\langle e_1,e_2\rangle$ a matrix having order $q+1$ and hence the putative quadratic form of Witt index 0 preserved by $g$ must restrict to $\langle e_1,e_2\rangle$ to a quadratic form having Witt index $1$, because $\mathrm{O}_2^+(q)$ has no elements of order $q+1$. However, this is impossible.

Recall that, by~\eqref{tired}, we have $H<H_M$ or $K<K_M$. Clearly, when 
 $H={\bf N}_{\mathrm{SO}_4^-(q)}({\langle x\rangle})\cong (q^2+1).[4]$ and $(C_{q+1}\times C_{q+1}).4\le K\le (D_{2(q+1)}\times D_{2(q+1)}).2$,  we have no weak normal $2$-coverings, because we are not covering elements having order $q-1$. Finally, suppose that $H\cong\mathrm{SO}_4^-(q)$ and $(C_{q+1}\times C_{q+1}).4\le K\le (D_{2(q+1)}\times D_{2(q+1)}).2$. 
we argue by contradiction and we suppose that $H$ and $K$ do form a weak normal $2$-covering. As $K$ has no element of order $q-1$, we deduce that each element having order $q-1$ in $\mathrm{Sp}_4(q)$ has an $\mathrm{Aut}(\mathrm{Sp}_4(q))$-conjugate in $H$. We use the matrix $J$ in the previous paragraph for the symplectic form on $\mathrm{Sp}_4(q)$. Let $a,b\in\mathbb{F}_q^\ast$ with $\order a=\order b=q-1$ and with $\{a,a^{-1},b,b^{-1}\}$ having cardinality $4$. Observe that this is possible because $q\ge 4$.
Set
\[
g:=\begin{pmatrix}a&0&0&0\\0&a^{-1}&0&0\\0&0&b&0\\0&0&0&b^{-1}\end{pmatrix}
.\]
Recall again that under a graph automorphism $H$ is in the Aschbacher class $\mathcal{C}_3$ or $\mathcal{C}_8$. The elements of order $q-1$ in the $\mathcal{C}_3$-class $\mathrm{SL}_2(q^2):2$ cannot have four distinct eigenvalues. Thus $g$ must be conjugate, under an inner or field automorphism, to a maximal subgroup $\mathrm{SO}_4^-(q)$ in the Aschbacher class $\mathcal{C}_8$. 
Now, $\langle e_1,e_2\rangle$ and $\langle e_3,e_4\rangle$ are two $2$-dimensional non-degenerate subspaces of $\mathbb{F}_q^4$ which are mutually orthogonal. Moreover, on each subspace $g$ induces a matrix of order $q-1$. Therefore, our putative quadratic form having Witt index $1$ must restrict to both $\langle e_1,e_2\rangle$ and $\langle e_3,e_4\rangle$ to a quadratic form having Witt index $1$, because $\mathrm{O}_2^+(q)$ does not contain elements of order $q+1$.
Thus we have reached in this way a contradiction.
 
\medskip

 We now consider Case~\eqref{ddd11}. Here it is more convenient to work with $\tilde{G}=\mathrm{Sp}_6(q)$. We have $\tilde{H}\le\tilde{H}_M\cong \mathrm{Sp}_2\perp\mathrm{Sp}_4(q)$ and $\tilde{K}\le \tilde{K}_M\cong \mathrm{Sp}_2(q^3).3$.  Now, $\tilde{G}$ contains an element $x$ of order $q^2+1$ and an element $y$ of order $q^3+1$. As $\tilde{K}_M$ contains no element having order $q^2+1$ and $\tilde{H}_M$ contains no element having order $q^3+1$, up to replacing $x$ and $y$ with suitable $\mathrm{Aut}(\tilde{G})$-conjugates, we deduce 
$$C_{q^2+1}\cong \langle x\rangle\le \tilde{H}\le \tilde{H}_M\cong \mathrm{Sp}_2(q)\perp\mathrm{Sp}_4(q)$$ and $$C_{q^3+1}\cong\langle y\rangle\le \tilde{K}\le \tilde{K}_M\cong \mathrm{Sp}_2(q^3).3.$$

We claim that $\tilde{K}=\tilde{K}_M$. The group $\tilde{G}$ contains unipotent elements $u_1$  having one Jordan block of size $6$. Elements of this type are not in $\tilde{H}_M\cong \mathrm{Sp}_2(q)\perp\mathrm{Sp}_4(q)$ and hence $u_1$ has an $\mathrm{Aut}(\tilde{G})$-conjugate in $\tilde{K}$. Observe that $u_1$ has order $9$ and hence
$\langle y,u_1^3\rangle$ is isomorphic to a subgroup of $\mathrm{Sp}_2(q^3)$ having order divisible by $3(q^3+1)$. Considering the subgroup structure of $\mathrm{Sp}_2(q^3)$, we deduce $\langle u_1^3,y\rangle\cong\mathrm{Sp}_2(q^3)$. Hence $\tilde{K}=\langle u_1,y\rangle=\tilde{K}_M$.

We claim that $\tilde{H}=\tilde{H}_M$, contradicting~\eqref{tired}. Among the elements of order $q^2+1$ we can choose $x$ fixing  pointwise a $2$-dimensional subspace of $\mathbb{F}_q^6$ and hence $x\in \mathrm{Sp}_4(q)$. The group $\tilde{G}$ contains unipotent elements $u_2$  having three Jordan blocks of size $1$, $1$ and $4$ respectively. Elements of this type are not in $\tilde{K}=\tilde{K}_M$ because the non-identity unipotent elements in $\tilde{K}$ have Jordan blocks of type $6$, or $2\oplus 2\oplus 2$, or $3\oplus 3$. Hence $u_2$ has an $\mathrm{Aut}(\tilde{G})$-conjugate in $\tilde{H}$. Observe that $u_2$ has order $9$. 
Hence $\langle x,u_2\rangle$ is isomorphic to a subgroup of $\mathrm{Sp}_4(q)$ having order divisible by $9(q^2+1)$ and containing a cyclic subgroup of order $9$. Considering the subgroup structure of $\mathrm{Sp}_4(q)$, we deduce $\langle u_2,x\rangle\cong\mathrm{Sp}_4(q)$. To show that the remaining $\mathrm{Sp}_2(q)$ direct factor of $\tilde{H}_M$ is also contained in $\tilde{H}$ it suffices to consider elements $x'$ having order $(q^2+1)(q+1)/2$ and type $2\oplus 4$ (where the matrix induced on the two invariant subspaces has order $q+1$ and $q^2+1$, respectively) and $u_2'$ having order $9$ and consisting of two Jordan blocks of size $2,4$. These elements are not $\mathrm{Aut}(\tilde{G})$-conjugate to elements in $\tilde{K}$ and hence we may assume that $u_2',x'\in\tilde{H}$. Now, $\mathrm{Sp}_2(q)\perp\mathrm{Sp}_4(q)\cong\langle u_2,x,u_2',x'\rangle\le \tilde{H}$.

\begin{table}[ht]
\begin{tabular}{c|c|c|c|c|c}
\toprule[1.5pt]
Group& Comp. $H$&Comp. $K$&Comments&Normal&Nr. \\
\midrule[1.5pt]
$A_5$&$A_5\cap(S_2\times S_3)$&$D_{10}$&&1&5\\
                 &$A_4$&$D_{10}$&&1&\\
    &$A_5\cap(S_2\times S_3)$&$5<D_{10}$&&1&\\
    &$A_4$&$5<D_{10}$&&1&\\ 
                 &$A_3<A_4$&$D_{10}$&&1&\\ 

$A_6$&$A_6\cap(S_2\times S_4)$ &$A_5$ &&2&8\\
                 &$A_6\cap (S_3\mathrm{wr} S_2)$ &$A_5$ &&2&\\
&$D_8<A_6\cap(S_2\times S_4)$ &$A_5$ &$|H|=8$&0&\\
&$A_6\cap (S_3\mathrm{wr} S_2)$ &$D_{10}<A_5$ &&1&\\
&$A_6\cap(S_2\times S_4)$ &$D_{10}<A_5$ &&0&\\
&$C_4<D_8<A_6\cap(S_2\times S_4)$ &$A_5$ &$|H|=4$&0&\\
&$A_6\cap (S_3\mathrm{wr} S_2)$ &$5<D_{10}<A_5$ &&1&\\
&$A_6\cap(S_2\times S_4)$ &$5<D_{10}<A_5$ &&0&\\

$A_7$&$A_7\cap(S_2\times S_5)$&$\mathrm{SL}_3(2)$&&2&2\\
 &$A_7\cap(S_2\times S_5)$&$7:3<\mathrm{SL}_3(2)$&&1&\\

$A_8$&$A_8\cap(S_3\times S_5)$&$2^3:\mathrm{SL}_3(2)$&&2&4\\
 &$A_3\times A_5<A_8\cap(S_3\times S_5)$&$2^3:\mathrm{SL}_3(2)$&&2&\\
 &$\langle(1\,2\,3)(4\,5\,6\,7\,8),(2\,3)(5\,7\,8\,6)\rangle<A_8\cap(S_3\times S_5)$&$2^3:\mathrm{SL}_3(2)$&$|H|=60$&2&\\
&$A_3\times D_{10}<A_8\cap(S_3\times S_5)$&$2^3:\mathrm{SL}_3(2)$&$|H|=30$&2&\\

$A_9$&$A_9\cap(S_4\times S_5)$&$\mathrm{P}\Gamma\mathrm{L}_2(8)$&&0&1\\
\bottomrule[1.5pt]
\end{tabular}
\caption{Components of the weak normal $2$-coverings of non-abelian simple groups: alternating groups}\label{00:Alt}
\end{table}

\begin{table}[ht]
\begin{tabular}{c|c|c|c|c|c}
\toprule[1.5pt]
Group& Comp. $H$&Comp. $K$&Comments&Normal&Nr. \\
\midrule[1.5pt]
$M_{11}$&$M_8:S_3\cong 2^\cdot S_4$&$\mathrm{PSL}_2(11)$&notation from~\cite{atlas}&1&8\\
&$M_9:S_2\cong 3^2:Q_8.2$&$\mathrm{PSL}_2(11)$&&1&\\
&$M_{10}\cong A_6.2$&$\mathrm{PSL}_2(11)$&&1&\\

&$\langle(1\, 5\, 3\, 8\, 9\, 2\, 7\, 11)(4\, 10),(1\, 6\, 9)(2\, 7\, 8)(3\, 5\, 11) \rangle<M_9:S_2$&$\mathrm{PSL}_2(11)$&$|H|=72$&1&\\
&$Q_8.2<(M_{8}:S_3)\cap (M_9:S_2)\cap M_{10}$&$\mathrm{PSL}_2(11)$&$|H|=16$&1&\\
&$C_8<Q_8.2<(M_{8}:S_3)\cap (M_9:S_2)\cap M_{10}$&$\mathrm{PSL}_2(11)$&$|H|=8$&1&\\
&$M_8:S_3\cong 2^\cdot S_4$&$C_{11}: C_5<\mathrm{PSL}_2(11)$&notation from~\cite{atlas}&1&\\
&$M_{10}\cong A_6.2$&$C_{11}: C_5<\mathrm{PSL}_2(11)$&&1&\\

$M_{12}$&$M_{10}:2\cong A_6.2^2$&$\mathrm{PSL}_2(11)$&notation from~\cite{atlas}&0&3\\
&$M_{11}$&$2\times S_5$&&0&\\
&$M_{11}$&$2\times A_5<2\times S_5$&&0&\\
\bottomrule[1.5pt]
\end{tabular}
\caption{Components of the weak normal $2$-coverings of non-abelian simple groups: sporadic groups}\label{00:sporadic}
\end{table}

\begin{table}[ht]
\begin{tabular}{c|c|c|c|c|c}
\toprule[1.5pt]
Group& Comp. $H$&Comp. $K$&Comments&Normal&Nr. \\
\midrule[1.5pt]
$G_2(q)$&$\mathrm{SL}_3(q).2$&$\mathrm{SU}_3(q).2$ &$q\ge 4$, $q$ even&1&1\\
$G_2(2)'$&$\mathrm{PSL}_2(7)$&$4\cdot S_4$&$G_2(2)'\cong \mathrm{PSU}_3(3)$&1&4\\
&$\mathrm{PSL}_2(7)$&$3_+^{1+2}:8$&&1&\\
&$\mathrm{PSL}_2(7)$&$3:8<(4.S_4)\cap (3_{+}^{1+2}:8)$&&1&\\
&$7:3<\mathrm{PSL}_2(7)$&$4.S_4$&&1&\\

$G_2(3)$&$\mathrm{PSL}_2(13)$&$[q^5]:\mathrm{GL}_2(3)$&&0&2\\
&$\mathrm{PSL}_3(3):2$&$\mathrm{PSL}_2(8):3$&&0&\\
$^{2}G_2(3)'$&$D_{18}$&$D_{14}$&$^{2}G_2(3)'\cong \mathrm{PSL}_2(8)$&1&5\\
&$D_{18}$&$2^3:7$&&1&\\
&$C_9<D_{18}$&$2^3:7$&&1&\\
&$C_9<D_{18}$&$D_{14}$&&1&\\
&$D_{18}$&$C_7<D_{14}$&&1&\\
$^{2}F_4(2)'$&$2.[2^8]:5:4$&$\mathrm{PSL}_3(3):2$&&1&2\\
&$2.[2^8]:5:4$&$\mathrm{PSL}_2(25)$&&1&\\
$F_4(q)$&${}^3D_4(q).3$&$\mathrm{Spin}_9(q)$&$q=3^a$&1&1\\
\bottomrule[1.5pt]
\end{tabular}
\caption{Components of the weak normal $2$-coverings of non-abelian simple groups: exceptional groups}\label{00:exceptional}
\end{table}

\begin{table}[!h]
\begin{tabular}{c|c|c|c|c|c}
\toprule[1.5pt]
Group& Comp. $H$&Comp. $K$&Comments&Normal&Nr.\\
\midrule[1.5pt]
$\mathrm{PSL}_2(7)$&$S_4 $ &parabolic &&$2$&4\\
&$D_8<S_4 $ &parabolic &&$1$&\\
&$C_4<D_8<S_4 $ &parabolic &&$1$&\\
&$S_4 $ &$C_7<$ parabolic&&$2$&\\

 $\mathrm{PSL}_2(q)$&$D_{q+1}$ &parabolic &$q$ odd, $q>9$&$1$&2\\
&$C_{(q+1)/2}<D_{q+1}$ &parabolic &$q$ odd, $q>9$&$1$&\\
$\mathrm{PSL}_2(q)$                  &$D_{2(q+1)}$&parabolic&$q>4$, $q$ even&1&5\\
&$C_{q+1}<D_{2(q+1)}$&parabolic&$q>4$, $q$ even&1&\\
&$D_{2(q+1)}$&$C_{q-1}<$parabolic&$q>4$, $q$ even&1&\\
                 &$D_{2(q+1)}$&$D_{2(q-1)}$&$q>4$, $q$ even&1&\\
&$C_{q+1}<D_{2(q+1)}$&$D_{2(q-1)}$&$q>4$, $q$ even&1&\\

$\mathrm{PSL}_3(q)$&$\left(\frac{q^2+q+1}{\gcd(3,q-1)}\right):3$&$E_q^2:\mathrm{GL}_2(q)$ parabolic&$\gcd(3,q)=1$, $q>4$&2&$2$\\
&$\frac{q^2+q+1}{\gcd(3,q-1)}$&$E_q^2:\mathrm{GL}_2(q)$ parabolic&$\gcd(3,q)=1$, $q>4$&2&\\

$\mathrm{PSL}_3(3)$&$13:3$& $E_3^2:\mathrm{GL}_2(3)$ parabolic&&2&$4$\\
&$13$&$E_3^2:\mathrm{GL}_2(3)$ parabolic&&2&\\
&$13:3$&$E_3^2:(8:2)<E_3^2:\mathrm{GL}_2(3)$ &&2&\\
&$13:3$&$\mathrm{GL}_2(3)<E_3^2:\mathrm{GL}_2(3)$&&1&\\

$\mathrm{PSL}_3(q)$&$\left(\frac{q^2+q+1}{\gcd(3,q-1)}\right):3$&$E_q^2:\mathrm{GL}_2(q)$ parabolic&$\gcd(3,q)=3$, $q>3$&2&$3$\\
&$\frac{q^2+q+1}{\gcd(3,q-1)}$&$E_q^2:\mathrm{GL}_2(q)$ parabolic&$\gcd(3,q)=3$, $q>3$&2&\\
&$\frac{q^2+q+1}{\gcd(3,q-1)}:3$&$\mathrm{GL}_2(q)<E_q^2:\mathrm{GL}_2(q)$&$\gcd(3,q)=3$, $q>3$ &1&\\

$\mathrm{PSL}_3(4)$&$\mathrm{SL}_3(2)$&$A_6$&&0&$13$\\
&$\mathrm{SL}_3(2)$&parabolic&&6&\\
&$7:3<\mathrm{SL}_3(2)$&parabolic&&2&\\
&$\mathrm{SL}_3(2)$&$K<$ parabolic&$|K|=60$&0&\\
&$\mathrm{SL}_3(2)$&$K<$ parabolic&$|K|=60$&0&\\
&$\mathrm{SL}_3(2)$&$K<$ parabolic&$|K|=80$&0&\\
&$\mathrm{SL}_3(2)$&$K<$ parabolic&$|K|=160$&6&\\
&$7:3<\mathrm{SL}_3(2)$&$K<$ parabolic&$|K|=160$&2&\\
&$\mathrm{SL}_3(2)$&$5:2$&&0&\\
&$\mathrm{SL}_3(2)$&$5<5:2$&&0&\\
&$7:3<\mathrm{SL}_3(2)$&$A_6$&&0&\\
&$7<7:3<\mathrm{SL}_3(2)$&$A_6$&&0&\\
&$7<7:3<\mathrm{SL}_3(2)$&parabolic&&2&\\

$\mathrm{PSL}_4(q)$&$\frac{1}{d}\mathrm{SL}_2(q^2).(q+1).2$&$\frac{1}{d}E_q^3:\mathrm{GL}_3(q)$&$d:=\gcd(4,q-1)$, $q\ge 4$& $2$ & $2$\\
&$\frac{1}{d}\mathrm{SL}_2(q^2).(q+1)<\frac{1}{d}\mathrm{SL}_2(q^2).(q+1).2$&$\frac{1}{d}E_q^3:\mathrm{GL}_3(q)$&$d:=\gcd(4,q-1)$, $q\ge 4$& $2$ & \\
$\mathrm{PSL}_4(3)$&$\frac{1}{2}\mathrm{SL}_2(9).4.2$&$\frac{1}{2}E_3^3:\mathrm{GL}_3(3)$& & $2$ & $4$\\
&$\frac{1}{2}\mathrm{SL}_2(9).4<\frac{1}{2}\mathrm{SL}_2(9).4.2$&$\frac{1}{2}E_3^3:\mathrm{GL}_3(3)$& & $2$ & \\
&$(2\times A_5).2^2<\frac{1}{2}\mathrm{SL}_2(9).4.2$&$\frac{1}{2}E_3^3:\mathrm{GL}_3(3)$& & $0$ & \\
&$A_5.4<(2\times A_5).2^2<\frac{1}{2}\mathrm{SL}_2(9).4.2$&$\frac{1}{2}E_3^3:\mathrm{GL}_3(3)$&& $0$ & \\
\bottomrule[1.5pt]
\end{tabular}
\caption{Components of the weak normal $2$-coverings of non-abelian simple groups: linear groups
(For $\mathrm{PSL}_2(4)\cong A_5$,
$\mathrm{PSL}_2(5)\cong A_5$,
$\mathrm{PSL}_2(9)\cong A_6$, $\mathrm{PSL}_4(2)\cong A_8$, see Table~\ref{00:Alt}. For $\mathrm{PSL}_3(2)$, use $\mathrm{PSL}_3(2)\cong\mathrm{PSL}_2(7)$.)}\label{00:linear}
\end{table}

\begin{table}[ht]
\begin{tabular}{c|c|c|c|c|c}
\toprule[1.5pt]
Group& Comp. $H$&Comp. $K$&Comments&Normal&Nr.\\
\midrule[1.5pt]
$\mathrm{PSU}_3(q)$&$(q^2-q+1):3$&$\mathrm{GU}_2(q)$&$3=\gcd(q,3)$, $q>3$&1&$1$\\

$\mathrm{PSU}_3(3)$&$\mathrm{PSL}_2(7)$&$\mathrm{GU}_2(3)$&&1&$4$\\
&$\mathrm{PSL}_2(7)$&$E_3^{1+2}:8$& &1&\\
&$\mathrm{PSL}_2(7)$&$3:8<\mathrm{GU}_2(3)\cap E_3^{1+2}:8$&&1&\\
&$7:3<\mathrm{PSL}_2(7)$&$\mathrm{GU}_2(3)$&&1&\\

$\mathrm{PSU}_3(5)$&$A_7$&$\frac{1}{3}\mathrm{GU}_2(5)$&&0&3\\
&$A_7$&$\frac{1}{3}E_5^{1+2}:24$& &3&\\
&$A_7$&$4:8<\frac{1}{3}E_5^{1+2}:24$&$|K|=40$ &0&\\

$\mathrm{PSU}_4(q)$&$\frac{1}{d}\mathrm{GU}_3(q)$&$\frac{1}{d}E_q^4:\mathrm{SL}_2(q^2):(q-1)$&$d:=\gcd(4,q+1)$, $q\ge 4$&$1$&$1$\\

$\mathrm{PSU}_4(2)$&$\mathrm{GU}_3(2)$&$\mathrm{Sp}_4(2)$&&1&$3$\\
&$\mathrm{GU}_3(2)$&$E_2^4:\mathrm{SL}_2(4)$&&1&\\
&$\mathrm{GU}_3(2)$&$S_5<\mathrm{Sp}_4(2)$&$|K|=120$&1&\\

$\mathrm{PSU}_4(3)$&$A_7$&$\frac{1}{4}E_3^{1+4}:\mathrm{SU}_2(3):8$&&4&4\\
&$\frac{1}{4}\mathrm{GU}_3(3)$&$\frac{1}{4}E_3^4:\mathrm{SL}_2(9):2$&&1&\\
&$\mathrm{PSL}_3(4)$&$\frac{1}{4}E_3^{1+4}:\mathrm{SU}_2(3):8$&&2&\\

&$\frac{1}{4}\mathrm{GU}_3(3)$&$\mathrm{PSU}_4(2)$&&0&\\

$\mathrm{PSU}_6(2)$&$\mathrm{Sp}_6(2)$&$\mathrm{PGU}_5(2)$&&0&2\\
&$\mathrm{PSU}_4(3):2$&$\mathrm{PGU}_5(2)$&&0&\\

\bottomrule[1.5pt]
\end{tabular}
\caption{Components of the weak normal $2$-coverings of non-abelian simple groups: unitary groups}\label{00:unitary}
\end{table}

\begin{table}[ht]
\begin{tabular}{c|c|c|c|c|c}
\toprule[1.5pt]
Group& Comp. $H$&Comp. $K$&Comments&Normal&Nr.\\
\midrule[1.5pt]
$\mathrm{PSp}_4(3)$&$\frac{1}{2}E_3^{1+2}:(2\times\mathrm{Sp}_2(3))$&$\mathrm{PSp}_2(9):2$&&1&3\\
&$\frac{1}{2}E_3^{1+2}:(2\times \mathrm{Sp}_2(3))$&$2^4.A_5$&&1&\\
&$\frac{1}{2}E_3^{1+2}:(2\times\mathrm{Sp}_2(3))$&$S_5<\mathrm{PSp}_2(9):2$&&1&3\\
$\mathrm{Sp}_n(q)$& $\mathrm{SO}_n^-(q)$&$\mathrm{SO}_n^+(q)$&$n	\ge 6$, $q$ even&1&$1$\\
$\mathrm{Sp}_4(q)$& $\mathrm{SO}_4^-(q)$&$\mathrm{SO}_4^+(q)$& $q\ge 8$, $q$ even &2&$2$\\
& $(q^2+1).[4]<\mathrm{SO}_4^-(q)$&$\mathrm{SO}_4^+(q)$& &0&\\

$\mathrm{Sp}_4(4)$& $\mathrm{SO}_4^-(4)$&$\mathrm{SO}_4^+(4)$&&2&$8$\\
&$\mathrm{Sp}_2(16):2$&$\mathrm{Sp}_4(2)$&&0&\\
&$\mathrm{Sp}_2(16):2$&$\mathrm{Sp}_4(2)'<\mathrm{Sp}_4(2)$&&0&\\
& $\mathrm{SO}_4^-(4)$&$K<\mathrm{SO}_4^+(4)$&$|K|=200$&0&\\
& $\mathrm{SO}_4^-(4)$&$K<\mathrm{SO}_4^+(4)$&$|K|=100$&0&\\
&$\mathrm{Sp}_2(16):2$&$S_5<\mathrm{SO}_4^+(4)\cap \mathrm{Sp}_4(2)$&&0&\\
&$\mathrm{Sp}_2(16):2$&$5:4<S_5<\mathrm{SO}_4^+(4)\cap \mathrm{Sp}_4(2)$&&0&\\
&$7:4<\mathrm{SO}_4^-(4)$&$\mathrm{SO}_4^+(4)$&&0&\\
$\mathrm{PSp}_6(3^f)$&$\frac{1}{2}\left(\mathrm{Sp}_2(3^f)\perp\mathrm{Sp}_4(3^f)\right)$&$\frac{1}{2}\mathrm{Sp}_2(3^{3f}):3$&&1&1\\
\bottomrule[1.5pt]
\end{tabular}
\caption{Components of the weak normal $2$-coverings of non-abelian simple groups: symplectic groups
(For
$\mathrm{PSp}_4(2)'\cong A_6$, see Table~\ref{00})}\label{00:symplectic}
\end{table}

\begin{table}[ht]
\begin{tabular}{c|c|c|c|c|c}
\toprule[1.5pt]
Group& Comp. $H$&Comp. $K$&Comments&Normal&Nr.\\
\midrule[1.5pt]
$\mathrm{P}\Omega_8^+(2)$&-&-&omitted&0&60\\
$\mathrm{P}\Omega_8^+(3)$&-&-&omitted&0&2019\\
\bottomrule[1.5pt]
\end{tabular}
\caption{Components of the weak normal $2$-coverings of non-abelian simple groups: orthogonal groups}\label{00:orthogonal}
\end{table}

\subsection{Proof of Theorem~2 of Burness and Tong-Viet}\label{BTproof}
In this section, we use Theorem~\ref{main theoremgeneral} to prove a rather interesting result of Burness and Tong-Viet.

 Let $A$ be a transitive permutation group on a finite set $\Omega$. We recall that, an element $x\in A$ is a derangement if it acts fixed-point-freely on $\Omega$, that is, for every $\omega\in \Omega$, $\omega^x\ne \omega$.  We write $\mathcal{D}(A)$ for the set of derangements in $A$. If $H$ is a point stabilizer, then $x$ is a derangement if and only if $x^A\cap H=\varnothing$, where $x^A$ is the $A$-conjugacy class of $x$. Therefore,
$$A=\mathcal{D}(A)\cup \bigcup_{a\in A}H^a,$$
where this union is disjoint. 

In~\cite{BT}, Burness and Tong-Viet are interested in transitive permutation groups $A$ with the special property that, for some prime number $r$, every derangement is an $r$-element, that is, has order a power of $r$. In particular, if $R$ is a Sylow $r$-subgroup of $A$, then
$$A=\bigcup_{a\in A}R^a\cup \bigcup_{a\in A}H^a.$$
The authors are interested only in primitive permutation groups and they show that, in this case,  the socle of $A$ is either abelian or non-abelian simple. Moreover, when the socle of $A$ is non-abelian simple, the group $A$ is almost simple and  they give a complete classification of the possible choices for the primitive group $A$ and for the prime $r$. We report here their main result for almost simple groups.
\begin{theorem}\label{BGPROOF}
Let $A$ be a finite almost simple primitive permutation group with point
stabilizer $H$. Then every derangement in $A$ is an $r$-element for some fixed prime $r$ if
and only if $(A, H, r )$ is one of the cases in Table~$\ref{bgtable}$.\footnote{We are reporting Table~$\ref{bgtable}$ using the notation in our work.}
\end{theorem} 

\begin{table}[!h]
\begin{tabular}{ccccc}\hline
$A$&$H$&$r$&Conditions&Remarks\\\hline
$M_{11}$&$\mathrm{PSL}_2(11)$&2&&\\
$\mathrm{PSL}_2(8)$&$P_1$ or $D_{14}$&3&&$P_1$ parabolic\\
$\mathrm{PSL}_2(q)$&$P_1$&$r$&$q=2r^e-1$&$P_1$ parabolic\\
                   &$P_1$ or $D_{2(q-1)}$&$r$&$r=q+1$ Fermat prime&$P_1$ parabolic\\
                   &$D_{2(q+1)}$&$r$& $r=q-1$ Mersenne prime&\\
$\mathrm{PGL}_2(q)$&$P_1$&$2$&$q=2^{e+1}-1$ Mersenne prime&$P_1$ parabolic\\                   
$\mathrm{P}\Gamma\mathrm{L}_2(q)$&${\bf N}_G(D_{2(q+1)})$&$r$&$r=q-1$ Mersenne prime&\\  
$\mathrm{P}\Gamma\mathrm{L}_2(8)$&$P_1$ or ${\bf N}_G(D_{14})$&$3$&&$P_1$ parabolic\\  
$\mathrm{PSL}_3(q)$&$P_1$ or $P_2$&$r$&$q^2+q+1=\gcd(3,q-1)r$ or $q^2+q+1=3r^2$&$P_1,P_2$ parabolic\\  
\end{tabular}
\caption{Examples arising in Theorem~$\ref{BGPROOF}$}\label{bgtable}
\end{table}
\begin{proof}
From the paragraphs leading to the statement of Theorem~\ref{BGPROOF}, we see that given an almost simple primitive group $A$ with point stabilizer $H$, every derangement in $A$ is an $r$-element for a fixed prime $r$ if and only if $\{H,R\}$ is a normal $2$-covering of $A$, where $R$ is a Sylow $r$-subgroup. Note that $H$ is  a proper subgroup of $A$, because  $H$ is core-free in $A$, and $R$ is a proper subgroup of $A$, because $A$ is not an $r$-group.

Let $G$ be the socle of $A$ and observe that $G$ is a non-abelian simple group. Now, $H\cap G$ is a proper subgroup of $G$, because $H$ is a core-free subgroup of $A$. Moreover, $R\cap G$ is a proper subgroup of $G$, because $G$ is not an $r$-group. Furthermore,
\begin{align*}
G&=G\cap A=G\cap\left(\bigcup_{a\in A}H^a\cup\bigcup_{a\in A}R^a\right)=\bigcup_{a\in A}(H\cap G)^a\cap\bigcup_{a\in A}(R\cap G)^a\\
&=\bigcup_{a\in \mathrm{Aut}(G)}(H\cap G)^a\cap\bigcup_{a\in \mathrm{Aut}(G)}(R\cap G)^a
\end{align*}
and hence $\{H\cap G,R\cap G\}$ is a weak normal $2$-covering of $G$
appearing in Theorem~\ref{main theoremgeneral}, with the extra information that $R\cap G$ is a Sylow subgroup of $G$.

Now the proof follows with a direct inspection of {\rm Tables~\ref{00:Alt}--\ref{00:orthogonal}}. We do not give details of this computations, but they are all straightforward. Observe that the inspection of Table~\ref{00:orthogonal} is via a computer computation, because we have omitted in Table~\ref{00:orthogonal} structural information on the two components of the weak normal $2$-coverings.
\end{proof}

\section{Degenerate normal $2$-coverings}\label{sec:degenerateasdf}
In this section we aim to give some partial results concerning normal $2$-coverings of almost simple groups which, in some sense, we consider \textbf{\textit{degenerate}}. Let $A$ be an almost simple group with socle $G$ and let $\mu=\{H,K\}$ be a normal $2$-covering of $A$. We say that $\mu$ is degenerate if either $G\le H$ or $G\le K$. As we have discussed in Section~\ref{sec:introintro}, there are two cases to consider
\begin{itemize}
\item[DI:] $G\le H\cap K$, 
\item[DII:] replacing $H$ with $K$, if necessary, $H\ge G$ and $K\ngeq G$.
\end{itemize}

\subsection{The degenerate normal $2$-coverings of the first type}\label{primo-tipo}
To discuss the degenerate normal $2$-coverings of the first type we only need two preliminary remarks.
\begin{lemma}\label{lemma:nilpotent}
Let $X$ be a  finite nilpotent group. Then either $X$ is cyclic and hence $X$ has no normal coverings, or the normal covering number of $X$ is not $2$. 
\end{lemma}
\begin{proof}
We argue by contradiction and we suppose that there exists a non-cyclic nilpotent group $X$ having normal covering number  $2$.  Let $\{V,W\}$ be a normal $2$-covering of $X$. Without loss of generality, we may suppose that $V$ and $W$ are both maximal subgroups of $X$. As $X$ is nilpotent, $V\unlhd X$ and $W\unlhd X$. Therefore $V$ and $W$ are normal proper subgroups of $X$ with $X=V\cup W$, however this is impossible.
\end{proof}

\begin{lemma}\label{lemma:silly2}
Let $G$ be a non-abelian simple group. Then the outer automorphism group of $G$ is non-nilpotent if and only if $G$ is isomorphic to one of the following groups:
\begin{enumerate}
\item\label{parT1} $\mathrm{PSL}_n(q)$ with $\gcd(n,q-1)$ divisible by an odd prime,
\item\label{parT2}  $\mathrm{PSU}_n(q)$   with $\gcd(n,q+1)$ divisible by an odd prime,
\item\label{parT3}  $\mathrm{P}\Omega_8^+(q)$,
\item\label{parT4}  $E_6(q)$ with $3$ dividing $q-1$,
\item\label{parT5}  ${}^2E_6(q)$ with $3$ dividing $q+1$. 
\end{enumerate}
\end{lemma}
\begin{proof}
The structure of the outer automorphism group of the non-abelian simple groups is in~\cite[Chapter~2]{kl} for classical groups, in~\cite[Section~2.5]{GLS3} for exceptional groups and, for instance in ~\cite{atlas}, for alternating and sporadic groups. Using these references, we deduce that  the outer automorphism group of $G$ is not nilpotent if and only if $G$ is isomorphic to one of the groups listed in the statement of this theorem.
\end{proof}

\begin{theorem}\label{theorem:silly}
Let $A$ be an almost simple group with socle $G$ and let $\{H,K\}$ be a normal $2$-covering of $A$ with $G\le H\cap K$. Then $A/G$ is not nilpotent and $G$ is isomorphic to one of the groups in~\eqref{parT1}--~\eqref{parT5} of Lemma~$\ref{lemma:silly2}$.

Conversely, for each of these five cases, $\mathrm{Aut}(G)$ admits a normal $2$-covering $\{H,K\}$ with $G\le H\cap K$, that is, $\mathrm{Aut}(G)$ admits a degenerate normal $2$-covering of the first type.
\end{theorem}
\begin{proof}
Let $X:=A/G$, $V:=H/G$ and $W:=K/G$. Since $V$ and $W$ are proper subgroups of $X$, we have that $\{V,W\}$ are the components of a normal $2$-covering of $X$. It follows that  $X$ is not cyclic. Thus Lemma~\ref{lemma:nilpotent} implies that $X=A/G$ is not nilpotent. As a consequence  $\mathrm{Aut}(G)/G$ is not nilpotent and Lemma \ref{lemma:silly2} implies that $G$  is in one of the five cases ~\eqref{parT1}--~\eqref{parT5} of Lemma~$\ref{lemma:silly2}$.
Suppose now that $G$ is in one of the five cases ~\eqref{parT1}--~\eqref{parT5} of Lemma~$\ref{lemma:silly2}$. Using the same references appearing in the proof of Lemma~$\ref{lemma:silly2}$, it can be verified that $\mathrm{Aut}(G)$ admits a normal subgroup $N\geq G$ with $\mathrm{Aut}(G)/N$ isomorphic to a dihedral group $D$ of order $2p$, for some odd prime number $p$. For instance, when $G:=E_6(q)$, using~\cite[Theorem~2.5.12~(i)]{GLS3}, we see that $\mathrm{Aut}(G)/G$ has a quotient isomorphic to the dihedral group of order $2\cdot 3$. As $D$ has normal covering number $2$, we deduce that $\mathrm{Aut}(G)$ has also normal covering number $2$  and admits a degenerate normal $2$-covering of the first type.
\end{proof}
\subsection{The degenerate normal $2$-coverings of the second type}\label{secondo-tipo}
We now deal with the degenerate normal $2$-coverings of almost simple groups of the second type. A key tool to do this is the theory of \textrm{\textbf{Shintani descent}}, which we outline here following the work of Kawanaka~\cite[Section~2]{Ka}. We also refer to~\cite[Section~2]{BurGue} and the more recent improvements of Harper on the theory of Shintani descent~\cite{HarperScott}.

Let $X$ be a connected linear algebraic group defined over an algebraically closed field and let $\sigma:X\to X$ be a Frobenius morphism. Thus, by definition, $\sigma$ is a bijective endomorphism of algebraic groups with finite fixed point subgroup
$$X_\sigma:=\{x\in X\mid x^\sigma=x\}.$$
Let $e$ be a positive integer and set $\mathcal{G}:=X_{\sigma^e}$ and $\mathcal{H}:=X_{\sigma}\le \mathcal{G}\le X$. As $\mathcal{G}$ is $\sigma$-stable, the restriction $\sigma'$ of $\sigma:X\to X$ to $\mathcal{G}$ is an automorphism of $\mathcal{G}$ having order $e$. By abuse of notation, we identify $\sigma'$ with $\sigma$ when considering the action of $\sigma$ on $\mathcal{G}$. Let $\mathcal{A}:=\langle\sigma\rangle$ and consider the semidirect product $\mathcal{G}\rtimes \mathcal{A}$ with multiplication
$$\sigma^i s\cdot \sigma^j t=\sigma^{i+j}s^{\sigma^j}t,$$
for every $i,j\in\mathbb{Z}$ and $s,t\in \mathcal{G}$.

Let $s\in  \mathcal{G}$ and consider the element $\sigma s$ in the coset $\sigma \mathcal{G}=\mathcal{G}\sigma$ of $\mathcal{G}\rtimes \mathcal{A}$. Then  we have $(\sigma s)^2=\sigma^2 s^\sigma s$ and an inductive argument gives
\begin{equation}\label{orderelt}
(\sigma s)^e=s^{\sigma^{e-1}}s^{\sigma^{e-2}}\cdots s^\sigma s\in \mathcal{G}.
\end{equation}
By the Lang-Steinberg theorem~\cite[Theorem~2.1.1]{GLS3}, there exists $a\in X$ with $s=a^{-\sigma }a$. Using~\eqref{orderelt}, we find $a(\sigma s)^ea^{-1}\in X_\sigma=\mathcal{H}$. Indeed,
\begin{align*}
(a(\sigma s)^ea^{-1})^\sigma&=(a(s^{\sigma^{e-1}}s^{\sigma^{e-2}}\cdots s^\sigma s)a^{-1})^{\sigma}=a^\sigma s^{\sigma^e}s^{\sigma^{e-1}}\cdots s^{\sigma^2}s^\sigma a^{-\sigma}\\
&=as^{-1}s^{\sigma^e}s^{\sigma^{e-1}}\cdots s^{\sigma^2}s^\sigma sa^{-1}=as^{\sigma^{e-1}}\cdots s^{\sigma^2}s^\sigma sa^{-1}\\
&=a(\sigma s)^ea^{-1},
\end{align*}
where in the forth equality we used the fact that $s^{\sigma^e}=s$. Therefore, $a(\sigma s)^ea^{-1}\in \mathcal{H}$.

 This observation is the key remark for setting up the Shintani correspondence. Indeed, the Shintani correspondence, or Shintani descent, is the  map $f$ from the $\mathcal{G}\rtimes \mathcal{A}$-conjugacy classes in the coset $\sigma \mathcal{G}$ and the set of $\mathcal{H}$-conjugacy classes in $\mathcal{H}$. The mapping $f$ is defined by
$$f:(\sigma s)^{\mathcal{G}\rtimes \mathcal{A}}\mapsto (a(\sigma s)^ea^{-1})^{\mathcal{H}}.$$
We summarize the basic properties of this mapping in the following lemma, see for instance~\cite[Lemma~2.13]{BurGue}.
\begin{lemma}\label{shintani}
The mapping $f$ is a bijection between the the $\mathcal{G}\rtimes \mathcal{A}$-conjugacy classes in the coset $\sigma\mathcal{G}$ and the set of $H$-conjugacy classes $\mathcal{H}$-conjugacy classes. Moreover, ${\bf C}_{\mathcal{G}}(\sigma s)=a^{-1}{\bf C}_{\mathcal{H}}(f(\sigma s))a={\bf C}_{a^{-1}\mathcal{H}a}((\sigma s)^e)$.% \textcolor{blue}{In particular $|{\bf C}_{\mathcal{G}}(\sigma s)|=|{\bf C}_{\mathcal{H}}(f(\sigma s))|$.}
\end{lemma}

\begin{theorem}\label{theorem:silly2}
Let $A$ be an almost simple group with socle $G$ and let $\{H,K\}$ be a normal $2$-covering of $A$ with $G\le H$ and $G\nleq K$. Let $H_M$ be a maximal subgroup of $A$ with $H\le H_M$ and $K_M$ a maximal subgroup of $A$ with $K\le K_M.$ 
Then one of the following holds: 
\begin{enumerate}
\item\label{parTT1} $A/G$ is not nilpotent and $G$ is isomorphic to one of the groups in~\eqref{parT1}--~\eqref{parT5} of Lemma~$\ref{lemma:silly2}$;
\item\label{parTT2} $A/G$ is nilpotent and one of the following holds:
\begin{enumerate}
\item\label{parTTT2}$G$ is a group of Lie type and $K_M\cap G$ is the centralizer in $G$ of a field automorphism of odd prime order $r$. Moreover, $r$ is not the characteristic
of $G$, unless $G = \mathrm{PSL}_2(q)$;
\item\label{parTT3} $G = \mathrm{PSL}_2(2^f)$, $f>1$ is odd, $A=\mathrm{Aut}(\mathrm{PSL}_2(2^f))=\mathrm{PSL}_2(2^f)\rtimes\langle\phi\rangle$, $H=H_M=\mathrm{PSL}_2(2^f)\rtimes\langle \phi^p\rangle$ for some prime divisor $p$ of $f$ is normal in $A$ and $K_M\cong D_{2(2^f+1)}.\langle \phi\rangle$, where $\phi$ is a field automorphism of order $f$;
\item\label{parTT4}$G=\mathrm{PSL}_2(p^f)$, $p$ is an odd prime number, $f$ is an even natural number, $A=G\langle\iota \phi^i\rangle$, $H=H_M=G\rtimes\langle\phi^{2i}\rangle$ has index $2$ in $A$, $\iota\in\mathrm{PGL}_2(p^f)\setminus\mathrm{PSL}_2(p^f)$, $\phi$ is a field automorphism of $\mathrm{PSL}_2(p^f)$ of order $f$, $i$ is a divisor of $f$ with $f/i$ even and $K_M\cap G=D_{p^f-1}$;
\item\label{parTT6} $G = \mathrm{Sz}(2^f)$,  $A=\mathrm{Aut}(\mathrm{Sz}(2^f))=\mathrm{Sz}(2^f)\rtimes\langle\phi\rangle$, $f\equiv 3,5\pmod 8$, $H=\mathrm{Sz}(2^f)\rtimes\langle \phi^p\rangle$ for some prime divisor $p$ of $f$ and $K_M={\bf N}_A(P)\cong (q-\sqrt{2q}+1):4:f$, where $P$ is a Sylow $5$-subgroup of $G$ and $\phi$ is a field automorphism of order $f$;
\item\label{parTT7} $G = \mathrm{PSU}_3(2^f )$ with $f > 1$ odd, $f$ divisible by $3$, $A=\mathrm{PSU}_3(2^f)\langle\phi^2\iota\rangle$, where $\iota\in\mathrm{PGU}_3(q)\setminus\mathrm{PSU}_3(q)$ and $\phi$ is a field automorphism of order $2f$, $H=\mathrm{PSU}_3(q)\rtimes\langle \phi^6\rangle$  and $K_M\cap G$ is the stabilizer in $G$ of a
decomposition of the $3$-dimensional space underlying $G$ into the direct sum of three
orthogonal non-singular $1$-subspaces. 
\end{enumerate}
\end{enumerate}
\end{theorem}

\begin{proof}
The hypothesis that $\{H,K\}$ is a normal $2$-covering of $A$ implies that $H\ne A$. In particular, since we are assuming $G\le H$, we have $$G\ne A$$
and $A/G$ is a non-identity group.

If $A/G$ is not nilpotent, then the outer automorphism group of $G$ is not nilpotent and hence, using Lemma~\ref{lemma:silly2}, part~\eqref{parTT1} holds. 

Suppose next that $A/G$ is nilpotent.
As  $G\le H\le H_M<A$, we have that $H_M\unlhd A$ and $A/H_M$ is cyclic of prime order.
Let $\bar A:=A/G$ and let us adopt the ``bar'' notation for the projection of $A$ onto $\bar{A}$. As $\{H,K\}$ are components of a normal $2$-covering of $A$, we have
\begin{equation}\label{bar-equation}
\bar{A}=\bigcup_{x\in \bar A}\bar H^x\cup\bigcup_{x\in \bar A}\bar K^x.
\end{equation}
We claim that
\begin{equation}\label{eq:elephant}
A=KG.
\end{equation}
Note that since $H$ is a proper subgroup of $A$ containing $G$, we have that $\bar H$ is a proper subgroup of $\bar{A}$. If $\bar{A}$ is cyclic, then  \eqref{bar-equation} implies $\bar{A}=\bar H\cup\bar K$ and hence we deduce $\bar{A}=\bar K$, that is,~\eqref{eq:elephant} is satisfied. If  $\bar{A}$ is not cyclic, then by Lemma~\ref{lemma:nilpotent} $\bar{A}$ admits no normal $2$-covering and thus \eqref{bar-equation} implies $\bar{A}=\bar K$, that is,~\eqref{eq:elephant} is satisfied.

 If $G\le K_M$, then~\eqref{eq:elephant} gives $A=KG\le K_MG=K_M$, contradicting the maximality of $K_M$. Therefore $G\nleq K_M$. Since $G$ is the unique minimal normal subgroup of $A$, we deduce that $K_M$ is a core-free subgroup of $A$.

Let $\Omega$ be the set of right cosets of $K_M$ in $A$. As $K_M$ is core-free in $A$, we may identify $A$ with the primitive permutation group induced by the faithful action of $A$ on $\Omega$.
Recall that an element $g\in A$ is said to be a \textbf{\textit{derangement}} if $g$ fixes no point of $\Omega$. Following the notation in~\cite{BCGR}, we denote by  $\mathcal{D}(A)$ the subset of $A$ consisting of the derangements of $A$ and by $D(A)$ the subgroup of $A$ generated by the derangements of $A$, that is, $$D(A):=\langle g\in A\mid g\textrm{  derangement}\rangle.$$
Clearly $D(A)\unlhd A$. Observe that $A\setminus \mathcal{D}(A)$ consists of all the permutations of $A$ fixing some point of $\Omega$ and hence 
$$A\setminus \mathcal{D}(A)=\bigcup_{g\in A}K_M^g.$$ Therefore, 
\begin{equation}\label{endless}
\mathcal{D}(A)=A\setminus\bigcup_{g\in A}K_M^g\subseteq\bigcup_{g\in A}H^g\subseteq H_M.
\end{equation}
As a consequence, $D(A)\le H_M<A$. 

Let $\Omega_*^2:=\{(\alpha,\beta)\in \Omega\times \Omega\mid \alpha\ne \beta\}$  and consider the action of $A$ and $H_M$ on $\Omega_*^2$. We claim that no orbit of $A$ on  $\Omega_*^2$  coincides with an orbit of $H_M$ on  $\Omega_*^2$. Assume, by contradiction, that there exists $(\alpha, \beta)\in \Omega_*^2$ such that $(\alpha, \beta)^{H_M}=(\alpha, \beta)^{A}$. Let $x\in A\setminus H_M$. Then there exists $y\in H_M$ such that $(\alpha^y, \beta^y)=(\alpha^x, \beta^x)$. Thus we have $\alpha^{xy^{-1}}=\alpha$ and $\beta^{xy^{-1}}=\beta$ so that $xy^{-1}\in A$ fixes at least  two points of $\Omega$. Now by the main result ~\cite[Theorem~1.1~(b)]{BCGR} of Bailey, Cameron, Giudici and Royle, all the permutations of $A$ fixing at least two points of $\Omega$ belong to $D(A)$. It follows that $xy^{-1}\in D(A)\leq H_M$. Since $y\in H_M$ we deduce the contradiction $x\in H_M.$ Recall that $A$ acts primitively on $\Omega$ so that its normal subgroup $H_M$ acts transitively on $\Omega$.
Hence, within the terminology of Guralnick, M\"{u}ller and Saxl~\cite{GMS}, we have that the pair $(A, H_M)$ is \textbf{\textit{exceptional}}  with cyclic quotient $A/H_M$. Those pairs are classified in ~\cite[Theorem~1.5]{GMS}, which implies\footnote{Here we are actually slightly weakening the full strength of the theorem of Guralnick, M\"{u}ller and Saxl. For example, we are including the groups appearing in~\cite[Theorem~1.5 part~(g)]{GMS} in our part~\eqref{parTTTT2}.} that one of the following holds:

\begin{enumerate}
\item\label{parTTTT2}$G$ is a group of Lie type and $K_M\cap G$ is  the centralizer in $G$ of a field automorphism of odd prime order $r$. Moreover, $r$ is not the characteristic
of $G$, unless $G = \mathrm{PSL}_2(q)$;
\item\label{parTTT3} $G = \mathrm{PSL}_2(2^f)$ with $f>1$ odd and $K_M\cap H_M=D_{2(2^f +1)}$;
\item\label{parTTT4} $G = \mathrm{PSL}_2(p^f )$, $H_M\in\{G, \mathrm{PGL}_2(p^f )\}$ and $K_M\cap G=D_{p^f -1}$ with $p$ odd and $f$ even;
\item\label{parTTT5} $G = \mathrm{PSL}_2(3^f )$ and $K_M \cap G=D_{3^f +1}$, with $f \ge 3$  odd;
\item\label{parTTT6} $G = \mathrm{Sz}(2^f )$ and $K_M\cap G$ is the normalizer of a Sylow $5$-subgroup of $G$;
\item\label{parTTT7} $G = \mathrm{PSU}_3(2^f )$ with $f > 1$ odd and $K_M\cap G$ is the stabilizer in $G$ of a
decomposition of the $3$-dimensional space underlying $G$  into the direct sum of three
orthogonal non-singular 1-subspaces.
\end{enumerate}

We now deal with each of these six possibilities in turn, starting with the easiest cases. In particular, in part~\eqref{parTTTT2}, we obtain that part~\eqref{parTTT2} holds.

\medskip

Suppose that $G=\mathrm{PSL}_2(2^f)$ and that we are in part~\eqref{parTTT3}. Observe that $\mathrm{PSL}_2(2^f)=\mathrm{SL}_2(2^f)$.
We aim to prove that part~\eqref{parTT3} holds. Now, $\mathrm{Aut}(\mathrm{SL}_2(2^f))=\mathrm{SL}_2(2^f)\rtimes\langle\phi\rangle$, where $\phi$ is a field automorphism of odd order $f>1$. 
 We have $A=\mathrm{SL}_2(2^f)\rtimes\langle\phi^e\rangle$, for some divisor $e$ of $f$ with $e\geq 1$. The nature of $K_M$ is easily found. Indeed, we know that $K_M\not \geq G$ and thus $A=K_M G$, so that $K_M\cong (K_M\cap G).A/G=(K_M\cap G). f/e$. Since $K_M\cap G\leq D_{2(2^f+1)}$, we have that the proper subgroup of $A$ given by $D_{2(2^f+1)}.\langle\phi^e\rangle$ contains $K_M$ and, by the maximality of $K_M$, we deduce that $K_M\cong D_{2(2^f+1)}.\langle\phi^e\rangle\cong D_{2(2^f+1)}.f/e.$
 We claim that $A=\mathrm{Aut}(\mathrm{SL}_2(2^f))$.
Suppose, by contradiction, that $A<\mathrm{Aut}(\mathrm{SL}_2(2^f))$ so that $e>1$.
As the outer automorphism group of $\mathrm{SL}_2(2^f)$ is cyclic and $H_M$ is a maximal subgroup of $A$ containing $\mathrm{SL}_2(2^f)$, we have $H_M=\mathrm{SL}_2(2^f)\rtimes\langle\phi^{pe}\rangle$, for some prime divisor $p$ of $f/e$.  Let $x\in {\bf C}_{\mathrm{SL}_2(2^f)}(\phi^{e})=\mathrm{SL}_2(2^e)$ having order $2^e-1$ and consider $g:=x\phi^e\in A$. Now, $g\notin \mathrm{SL}_2(2^f)\rtimes\langle\phi^{pe}\rangle=H_M\unlhd A$ and hence $g$ has a conjugate in $K_M$. 
 From $x\phi^e\in D_{2(2^f+1)}.\langle\phi^{e}\rangle$, it follows now that $x\in D_{2(2^f+1)}$, so that ${\bf o}(x)=2^e-1$ divides $2^f+1$, which is impossible because $e$ is odd and $e>1$. 
Therefore, $A=\mathrm{SL}_2(2^f)\rtimes\langle\phi\rangle=\mathrm{Aut}(\mathrm{SL}_2(2^f))$.

Now, $H=\mathrm{SL}_2(2^f)\rtimes\langle\phi^e\rangle$, where $e$ is a divisor of $f$ with $e>1$. Suppose that $e$ is not a prime number and let $p$ be a prime dividing $e$.  Let $x\in {\bf C}_{\mathrm{SL}_2(2^f)}(\phi^{p})=\mathrm{SL}_2(2^p)$ having order $2^p-1$ and consider $g:=x\phi^p\in A$. Now, $g\notin H=\mathrm{SL}_2(2^f)\rtimes\langle\phi^e\rangle$, because $p<e$. Hence $g=x\phi^p$ has a conjugate in $K_M\cong D_{2(2^f+1)}.\langle\phi\rangle$.  It follows that $x\in D_{2(2^f+1)}$, so that ${\bf o}(x)=2^p-1$ divides $2^f+1$, which is impossible because $p>1$ is odd.
Therefore, $e=|A:H|$ is prime and thus $H=H_M$ is normal in $A$.

\medskip

Suppose that $G=\mathrm{Sz}(2^f)$ and that we are in part~\eqref{parTTT6}. Here the argument follows verbatim the proof for the case~\eqref{parTTT3}.  Set $q:=2^f$.
We aim to prove that part~\eqref{parTT6} holds. Now, $\mathrm{Aut}(\mathrm{Sz}(q))=\mathrm{Sz}(q)\rtimes\langle\phi\rangle$, where $\phi$ is a field automorphism of odd order $f>1$. 
 We have $A=\mathrm{Sz}(q)\rtimes\langle\phi^e\rangle$, for some divisor $e$ of $f$ with $e> 1$ because $A\ne G$. The nature of $K_M$ is easily found. We know that $K_M\cap G={\bf N}_G(P)$, where $P$ is a Sylow $5$-subgroup of $G$. In particular, $P\ne 1$, because $K_M$ is a proper subgroup of $A$. As 
$$|G|=q^2(q^2-1)(q-1)=q^2(q+\sqrt{2q}+1)(q-\sqrt{2q}+1)(q-1)$$
and $5$ divides $|G|$,
 it is easily seen that $f\equiv 3,5\pmod 8$. With this information, it is easily seen that $5$ divides $q-\sqrt{2q}+1$ and $5$ is relatively prime to $(q+\sqrt{2q}+1)(q-1)$. Now, using~\cite[Table~8.16]{bhr}, we deduce
$$K_M={\bf N}_A(P)\cong (q-\sqrt{2q}+1):4: \frac{f}{e}.$$ 

 We claim that $A=\mathrm{Aut}(\mathrm{Sz}(q))=\mathrm{Sz}(q)\rtimes\langle\phi\rangle$. 
Suppose, by contradiction, that $A<\mathrm{Aut}(\mathrm{Sz}(q))$ so that $e>1$.
As the outer automorphism group of $\mathrm{Sz}(q)$ is cyclic and $H_M$ is a maximal subgroup of $A$ containing $\mathrm{Sz}(q)$, we have $H_M=\mathrm{Sz}(q)\rtimes\langle\phi^{pe}\rangle$, for some prime divisor $p$ of $f/e$.  Let $x\in {\bf C}_{\mathrm{Sz}(q)}(\phi^{e})=\mathrm{Sz}(2^e)$ having order $2^e-1$ and consider $g:=x\phi^e\in A$. Now, $g\notin \mathrm{Sz}(q)\rtimes\langle\phi^{pe}\rangle=H_M\unlhd A$ and hence $g$ has a conjugate in $K_M$. 
 From $x\phi^e\in {\bf N}_A(P)=((q-\sqrt{2q}+1):4):\langle\phi^{e}\rangle$, it follows now that $x\in (q-\sqrt{2q}+1):4$, so that ${\bf o}(x)=2^e-1$ divides $q-\sqrt{2q}+1=2^{f}-2^{(f+1)/2}+1$. A computation shows that this is possible only when $e=1$, which is a contradiction.  Therefore, $A=\mathrm{Sz}(q)\rtimes\langle\phi\rangle=\mathrm{Aut}(\mathrm{Sz}(q))$.

Now, $H=\mathrm{Sz}(q)\rtimes\langle\phi^e\rangle$, where $e$ is a divisor of $f$ with $e>1$. Suppose that $e$ is not a prime number and let $p$ be a prime dividing $e$.  Let $x\in {\bf C}_{\mathrm{Sz}(q)}(\phi^{p})=\mathrm{Sz}(2^p)$ having order $2^p-1$ and consider $g:=x\phi^p\in A$. Now, $g\notin H=\mathrm{SL}_2(q)\rtimes\langle\phi^e\rangle$, because $p<e$. Hence $g=x\phi^p$ has a conjugate in $K_M={\bf N}_A(P)\cong (q-\sqrt{2q}+1):4:\langle\phi\rangle$.  It follows that $x\in (q-\sqrt{2q}+1):4$, so that ${\bf o}(x)=2^p-1$ divides $q-\sqrt{2q}+1$. Arguing as above, we see that this  is impossible because $p>1$ is odd.
Therefore, $e=|A:H|$ is prime. 
\medskip

Suppose that $G=\mathrm{PSL}_2(3^f)$ and that we are in part~\eqref{parTTT5}. Here we prove that no normal $2$-covering arises. Let $q:=3^f$. Now, let $\iota$ be the projective image of 
\[
\begin{pmatrix}
-1&0\\
0&1
\end{pmatrix}\in \mathrm{GL}_2(q)
\]
in $\mathrm{PGL}_2(q)$ and let $\phi$ be a field automorphism of $\mathrm{PSL}_2(q)$ of order $f$. Thus 
$\mathrm{Aut}(G)=\mathrm{Aut}(\mathrm{PSL}_2(q))=G\rtimes\langle \iota\phi\rangle$, where $\langle\iota \phi\rangle$ is a cyclic group of order $2f$, because $f$ is odd and $\iota$ and $\phi$ commute. 

In particular, $A=G\rtimes\langle (\iota \phi)^{e_1}\rangle$, for some divisor $e_1$ of $2f$, and $H=G\rtimes\langle(\iota \phi)^{e_1e_2}\rangle$, for some divisor $e_2$ of $2f/e_1$ with $e_2>1$. 
Assume first that $e_1$ is even, that is, $e_1:=2e_1'$, for some divisor $e_1'$ of $f$. Thus $$A=G\rtimes\langle (\iota \phi)^{e_1}\rangle=G\rtimes \langle \phi^{e_1'}\rangle,\, H=G\rtimes\langle\phi^{e_1'e_2}\rangle.$$ 
Observe that, for every $s\in G$, $\phi^{e_1'}s\in A\setminus H$ and hence $\phi^{e_1'}s$ has a conjugate in $K_M$. We now apply the Shintani descent with $\sigma:=\phi^{e_1'}$, $e:=f/e_1'$ and  $X:=\mathrm{PSL}_2(\bar{\mathbb{F}}_3)$, where $\bar{\mathbb{F}}_3$ is the algebraic closure of $\mathbb{F}_3$. Thus, $\mathcal{G}=\mathrm{PSL}_2(q)$ and $\mathcal{H}=\mathrm{PSL}_2(3^{e_1'}).$
 Using Lemma~\ref{shintani} and the surjectivity of the Shintani descent, we choose $s\in G$ such that $(\sigma s)^e$ has order $3$. Since $\sigma s=\phi^{e_1'}s$ has a conjugate in $K_M$, we have that  $(\sigma s)^e\in K_M\cap G=D_{3^f+1}$  against the fact that $D_{3^f+1}$ has no element of order $3.$\footnote{
Strictly speaking, in this particular case, the Shintani descent is not necessary and we could argue by simply using the norm with respect to a field extension. Consider $q':=q^{e_1'}$ and let $$N:\mathbb{F}_q\to \mathbb{F}_{q'}$$
be the norm map of the Galois extension $\mathbb{F}_q/\mathbb{F}_{q'}$. In particular, for every $\alpha\in \mathbb{F}_q$, 
$$N(\alpha)=\prod_{\theta\in \mathrm{Gal}(\mathbb{F}_q/\mathbb{F}_{q'})}\alpha^\theta=
\prod_{\theta\in \langle\sigma\rangle}\alpha^\theta
=\alpha^{\sigma^{f/e_1'-1}}\cdots \alpha^\sigma \alpha.$$
From the surjectivity of the norm, there exists $\alpha\in \mathbb{F}_q$ with $N(\alpha)=1$. Consider now
\[
s:=
\left[
\begin{array}{cc}
1&\alpha\\
0&1
\end{array}
\right]\in G=\mathrm{PSL}_2(q).
\]
Then $g:=\phi^{e_1'}s=\sigma s\in A\setminus H$. In particular, $g$ has a conjugate in $K_M$. Therefore $g^{f/e_1'}$ has also a conjugate in $K_M$. Thus $g^{f/e_1'}\in K_M\cap G  =D_{q+1}$. Now, by \eqref{orderelt}, we have
$$g^{f/e_1'}=(\sigma s)^{f/e_1'}=s^{\sigma^{f/e_1'-1}}\cdots s^\sigma s=
\left[
\begin{array}{cc}
1&\alpha^{\sigma^{f/e_1'-1}}\cdots \alpha^\sigma \alpha\\
0&1
\end{array}
\right]
=
\left[
\begin{array}{cc}
1&1\\
0&1
\end{array}
\right].
$$ 
Hence $K_M\cap G= D_{3^f+1}$ contains an element having order $3$, however this is a contradiction.

Assume now that $e_1$ is odd. Thus $$A=\mathrm{PSL}_2(q)\rtimes \langle(\iota\phi)^{e_1}\rangle
=\mathrm{PSL}_2(q)\rtimes\langle \iota,\phi^{e_1}\rangle=\mathrm{PGL}_2(q)\rtimes\langle \phi^{e_1}\rangle.$$ Now, we may argue exactly  as above applying the Shintani descent to   $\sigma:=\phi^{e_1}$, $e:=f/e_1$ and $X:=\mathrm{PGL}_2(\bar{\mathbb{F}}_3)$ and using the fact that $D_{q+1}.2$ does not contain elements of order $3$.}
In particular, in this case no example arises for our main result.

\medskip

Suppose that $G=\mathrm{PSL}_2(p^f)$ and that we are in part~\eqref{parTTT4}. Let $$\iota\in \mathrm{PGL}_2(q)\setminus \mathrm{PSL}_2(q).$$ 
Let $\phi$ be a field automorphism of $\mathrm{PSL}_2(q)$ of order $f$. Thus 
$\mathrm{Aut}(G)=\mathrm{Aut}(\mathrm{PSL}_2(q))=G\langle \iota,\phi\rangle$.  Note that the outer automorphism group of $\mathrm{PSL}_2(q)$ is isomorphic to $\mathbb{Z}_2\times \mathbb{Z}_f$, because $f$ is even. Since $A$ is a subgroup of $\mathrm{Aut}(G)$ containing $G$, we obtain one of the following possibilities
\begin{enumerate}
\item\label{bestio1} $A=G\langle \iota,\phi^i\rangle$, where $i$ is a divisor of $f$,
\item\label{bestio2} $A=G\rtimes\langle \phi^i\rangle$, where $i\neq f$ is a divisor of $f$,
\item\label{bestio3} $A=G\langle \iota\phi^i\rangle$, where $i$ is a divisor of $f$ with $f/i$ even.
\end{enumerate}
Assume first that we are in case \eqref{bestio1} or in case \eqref{bestio2}. 
We claim that $\phi^i\in H.$ This is obvious when we are dealing with case~\eqref{bestio1} with $i=f$; therefore, in the proof of this claim, when discussing case~\eqref{bestio1}, we may assume that $i\ne f$. Observe that in case~\eqref{bestio2}, we always have $i\ne f$.

For proving the previous claim, we use the Shintani descent with $e:=f/i$ and $\sigma:= \phi ^i$ establishing a bijection among the conjugacy classes of $G\rtimes\langle \phi^i\rangle\leq A$ contained in $\phi ^iG\subseteq A$ and the conjugacy classes in $\mathrm{PSL}_2(p^i).$ Since obviously $\mathrm{PSL}_2(p^i)$ contains an element of order $p$, there exists $s\in G$ such that $(\phi^is)^e$ has order $p$ and moreover, by \eqref{orderelt}, we have $(\phi^is)^e\in G$. Consider now $\phi^is\in A$. Assume, by contradiction, that $\phi^is$ belongs to an $A$-conjugate of $K_M.$ Since $G\unlhd A$ we also have that $(\phi^is)^e$ belongs to an $A$-conjugate of $K_M\cap G=D_{q-1}$, against the fact that $D_{q-1}$ contains no element of order $p$. It follows that $\phi^is\in H$ and, since $H\geq G$, we deduce that $\phi^i\in H.$

If  we are in case \eqref{bestio2} this leads to the contradiction $H=A$. Thus case \eqref{bestio2} cannot arise.
If we are in case \eqref{bestio1}, this leads to $H\geq G\rtimes\langle \phi^i\rangle$ and, since $G\rtimes\langle \phi^i\rangle$ has index $2$ in $A$, we deduce that $H=H_M=G\rtimes\langle \phi^i\rangle$. We show that this does not give rise to a normal $2$-covering of $A$. We argue by contradiction and we suppose that this is a normal $2$-covering. Let $g\in \mathrm{PGL}_2(q)=\langle G,\iota\rangle\le A$ be a Singer cycle. Clearly, $g\notin H\unlhd A$, because $H$ has no elements having order $q+1$. Therefore, $g$ has an $A$-conjugate in $K_M$. Now, $K_M$ is a maximal subgroup of $A$ with $K_M\cap \mathrm{PSL}_2(q)=D_{q-1}$. From~\cite[Table~8.1]{bhr}, when $q\ne 9$, we have $K_M={\bf N}_A(K_M\cap \mathrm{PSL}_2(q))=D_{2(q-1)}.(f/i)$. Moreover, as $g\in \mathrm{PGL}_2(q)\unlhd A$, we deduce that $g$ has an $A$-conjugate in $K_M\cap\mathrm{PGL}_2(q)=D_{2(q-1)}$. However, $D_{2(q-1)}$ has no elements of order $q+1$. The case $q=9$ also does not give rise to a normal $2$-covering; this has been verified with a computer computation.

Finally assume that we are in case \eqref{bestio3}. Then $A=G\langle \phi^i\iota\rangle $ contains the coset $(\phi^i\iota)^2G=\phi^{2i}G$. Note that, since $i$ is odd dividing $f$ even we also have that $2i$ divides $f.$ We claim that  $\phi^{2i}\in H.$ 

For that purpose, we use the Shintani descent with $e:=f/2i$ and $\sigma:= \phi ^{2i}$, establishing a bijection among the conjugacy classes of $G\rtimes\langle \phi^{2i}\rangle\leq A$ contained in $\phi ^{2i}G\subseteq A$ and the conjugacy classes in $\mathrm{PSL}_2(p^{2i}).$ Since obviously $\mathrm{PSL}_2(p^i)$ contains an element of order $p$, there exists $s\in G$ such that $(\phi^{2i}s)^e$ has order $p$ and moreover, by \eqref{orderelt}, we have $(\phi^{2i}s)^e\in G$. Consider now $\phi^{2i}s\in A$. Assume, by contradiction, that $\phi^{2i}s$ belongs to an $A$-conjugate in $K_M.$ Since $G\unlhd A$ we also have that $(\phi^{2i}s)^e$ belongs to an $A$-conjugate of $K_M\cap G=D_{q-1}$, against the fact that $D_{q-1}$ contains no element of order $p$. It follows that $\phi^{2i}s\in H$ and, since $H\geq G$, we deduce that $\phi^{2i}\in H.$ 

As a consequence, we have that $H\geq G\langle \phi^{2i}\rangle$ and, since $G\rtimes\langle \phi^{2i}\rangle$ has index $2$ in $A=G\langle \iota\phi^i\rangle$, we deduce that $H=H_M=G\rtimes\langle \phi^{2i}\rangle$.

\medskip

Suppose finally that $G=\mathrm{PSU}_3(2^f)$ and that we are in part~\eqref{parTTT7}.  We consider the Hermitian form having matrix:
\[
\begin{pmatrix}
0&0&1\\
0&1&0\\
1&0&0
\end{pmatrix}.
\]Set $q:=2^f$ and  let $\lambda$ be a generator of the subgroup of $\mathbb{F}_{q^2}^\ast$ having order $q+1$. Now, let $\iota$ be the projective image of 
\[
\begin{pmatrix}
1&0&0\\
0&\lambda&0\\
0&0&1
\end{pmatrix}\in \mathrm{GU}_3(q)
\]
in $\mathrm{PGU}_3(q)$ and let $\phi$ be a graph-field automorphism of $\mathrm{PSU}_3(q)$ of order $2f$. Thus 
$\mathrm{Aut}(G)=\mathrm{Aut}(\mathrm{PSU}_3(q))=G\langle \iota,\phi\rangle$, where the outer automorphism group of $\mathrm{PSU}_3(q)$ is isomorphic to $(\mathbb{Z}_3\times \mathbb{Z}_f).\mathbb{Z}_2\cong\mathbb{Z}_3\rtimes\mathbb{Z}_{2f}$, because $f$ is odd. In particular, we have $\iota^\phi=\iota^{-1}\pmod G$ and $\iota^3=1\pmod G$. Since $A$ is a subgroup of $\mathrm{Aut}(G)=G\langle \iota,\phi\rangle$ containing $G$, we obtain one of the following possibilities
\begin{enumerate}
\item\label{Bestio1} $A=G\langle\iota,\phi^i\rangle$, for some divisor $i$ of $2f$,
\item\label{Bestio2} $A=G\rtimes\langle \phi^i\rangle$, for some divisor $i$ of $2f$, with $i\neq 2f$,
\item\label{Bestio3} $A=G\langle \iota^\varepsilon\phi^{i}\rangle$, for some 
$\varepsilon\in \{-1,1\}$ and for some divisor $i$ of $f$,
\item\label{Bestio4} $A=G\langle \iota^\varepsilon\phi^{2i}\rangle$, for some 
$\varepsilon\in \{-1,1\}$ and for some divisor $i$ of $f$ with $f/i$ divisible by $3$.
\end{enumerate}
Observe that in~\eqref{Bestio3}, we have $$(\iota^\varepsilon \phi^i)^2=\phi^{2i},$$ and that in~\eqref{Bestio4} we have 
\begin{equation}\label{strum}
(\iota^\varepsilon \phi^{2i})^3=\iota^{3\iota}\phi^{6i}\pmod G=\phi^{6i}\pmod G.
\end{equation}
 Moreover, $\iota^{-1}\phi\iota\equiv \phi\iota^2\pmod G$ and hence the groups appearing in~\eqref{Bestio2} with $i$ odd and those in~\eqref{Bestio3} are conjugate under an element in $\mathrm{Aut}(G)$.

We claim that 
\begin{center}
$(\dag)\ $ $\forall j\mid f$,  if $\phi^{2j}\in A$, then $\phi^{2j}\in H$.
\end{center}
Let $j$ be a divisor of $f$ and assume that $\phi^{2j}\in A$. Let $\sigma:=\phi^{2j}$ so that $\order \sigma=f/j=:e$. We have $G\rtimes\langle \sigma\rangle\leq A$
and we apply the Shintani descent to the conjugacy classes of this group contained in the coset $\phi^{2j}G=\sigma G$. Note that the codomain of the descent is given here by $\mathrm{PSU}_3(2^j)\geq \mathrm{PSU}_3(2)\geq Q_8$. Then there exists $s\in G$ such that $(\phi^{2j}s)^e$ has order $4.$ Consider then $\phi^{2j}s\in A$. If it belongs to $K_M$ up to $A$-conjugacy, then we have $(\phi^{2j}s)^e$ in $K_M\cap G\cong (2^f+1)^2.S_3$ up to $A$-conjugacy, against the fact that the group $(2^f+1)^2.S_3$ contains no element of order $4$.

We next claim that 

\begin{center}
$(\dag\dag)\ $ if $\iota\in A$, then $\iota\in H$.
\end{center}

Let $\iota\in A$. Then $A\supseteq \mathrm{PGU}_3(q)$. Let $g\in \mathrm{PGU}_3(q)$ be a Singer cycle. %Then  $g\notin \mathrm{PSU}_3(q)$ so that $g\in G\iota\cup G\iota ^{-1}$. 
Assume that $g\in K_M$ up to $A$-conjugacy. Then $$g\in K_M\cap \mathrm{PGU}_3(q)\leq (K_M\cap \mathrm{PSU}_3(q)).3=(q+1)^2.S_3.3$$ up to $A$-conjugacy.
Since $f>1$, there exists $r\in P_{6}(q)$, that is $r$ is a primitive prime divisor of $q^{6}-1=2^{6f}-1$, and $r\mid \order g$. Then $r\nmid q+1$ and $r\geq 7$ so that it cannot divide $(q+1)^2.S_3.3$ and we get a contradiction. It follows that $\iota\in H$ up to $A$-conjugacy. It remains to justify that $\iota\in H$. Now, for some $a\in A$, we have $\iota^a\in H$ and hence $$H\ge\langle G,\iota^a\rangle=\langle G,\iota\rangle^a=\mathrm{PGU}_3(q)^a=\mathrm{PGU}_3(q)=\langle G,\iota\rangle.$$

\smallskip

We consider now the various possibilities for $A$.

Let $A$ be as in~\eqref{Bestio1}. If $i$ is even, then we have $i=2j$, for some $j\mid f$, and by $(\dag)$ and $(\dag\dag)$, we deduce $H=A$, which is a contradiction. If $i$ is odd, then $i\mid f$ and we have $\phi^{2i}\in A$. So, by $(\dag)$, we deduce $\phi^{2i}\in H.$ Moreover, by $(\dag\dag)$, we have $\iota \in H$. Thus $H\supseteq G\langle \iota, \phi^{2i}\rangle$. Since this last subgroup of $A$ has index $2$, we deduce $H=H_M=G\langle \iota, \phi^{2i}\rangle$. Now, let $s\in \mathrm{PGU}_3(q)\le A$ be a Singer cycle and let $g:=s\phi^i$. Observe that $g\notin H$ and hence $g\in K_M=(K_M\cap \mathrm{PGU}_3(q)).\langle \phi^i\rangle$ up to $\mathrm{Aut}(A)$-conjugacy.  Thus $s\in K_M\cap \mathrm{PGU}_3(q)$. However, arguing as above we reach a contradiction.

\smallskip

Let $A$ be as in~\eqref{Bestio2}. If $i$ is even, then we have $i=2j$ for some $j\mid f$ and, by $(\dag)$, we get $\phi^i\in H$. Then  $H=A$, which is a contradiction. If $i$ is odd, by  $(\dag)$, we get $\phi^{2i}\in H$ so that  $H\supseteq G\rtimes\langle \phi^{2i}\rangle$. Since this last subgroup of $A$ has index $2$, we deduce $H=H_M=G\rtimes\langle \phi^{2i}\rangle$.  Now, let $s\in \mathrm{PSU}_3(q)\le A$ be a Singer cycle and let $g:=s\phi^i$. Observe that $g\notin H$ and hence $g\in K_M=(K_M\cap \mathrm{PSU}_3(q)).\langle \phi^i\rangle$ up to $\mathrm{Aut}(A)$-conjugacy.  Thus $s\in K_M\cap \mathrm{PSU}_3(q)$ and, arguing as above, we reach a contradiction.

\smallskip

Let $A$ be as in~\eqref{Bestio3}. We have shown above that in this case $A$ is $\mathrm{Aut}(G)$-conjugate to a group as in part~\eqref{Bestio2}. Therefore, there exists $\tilde A$ as in~\eqref{Bestio2} and $\nu\in\mathrm{Aut}(G)$ such that $A^{\nu}=\tilde A$. From  $A=\bigcup_{a\in A}H^a\cup \bigcup_{a\in A}K^a,$ it follows that 
\begin{equation*}\label{info1}
\tilde A=A^{\nu}=\bigcup_{a\in A}H^{a\nu}\cup \bigcup_{a\in A}K^{a\nu}.
\end{equation*}
Now, inside $\mathrm{Aut}(G)$, we have $A\nu=\nu \tilde A$ and $H^{\nu}, K^{\nu}\leq \tilde A$. Thus we get 
$$\tilde A=\bigcup_{a\in \tilde A}H^{\nu a}\cup \bigcup_{a\in \tilde A}K^{\nu a},$$
so that $\tilde A$ admits a $2$-normal covering. However we have shown that no group  of type ~\eqref{Bestio2} admits a $2$-normal covering.
\smallskip

Let finally $A$ be as in~\eqref{Bestio4}. Therefore, $A=G\langle \iota^\varepsilon\phi^{2i}\rangle$, where $\varepsilon\in \{-1,1\}$ and $i$ is a divisor of $f$ with $f/i$ divisible by $3$. By~\eqref{strum}, we have that $\phi^{6i}\in A$ so that, by $(\dag)$, $\phi^{6i}\in H.$ It follows that $H=H_M=G\rtimes \langle \phi^{6i}\rangle$ has index $3$ in $A$. In order to prove~\eqref{parTT7}, it suffices to show that $i=1$.

For each $a\in\mathbb{F}_{q}$ with $a\ne 0$, the projective image $x_a$ of the matrix
$$\begin{pmatrix}
a&0&0\\
0&1&0\\
0&0&a^{-1}
\end{pmatrix}$$
lies in $\mathrm{PSU}_3(q)$ and hence $\iota^\varepsilon\phi^{2i} x_a\in A$. Clearly, $\iota^\varepsilon\phi^{2i} x_a\notin G\rtimes\langle \phi^{2i}\rangle=H$ and hence $\iota^\varepsilon\phi^{2i} x_a$ has a conjugate in $K_M$. Now, as $\phi^{2i}$ commutes with $\iota$ and as $\phi^{2f}=1$, we obtain
$$(\iota^{\varepsilon}\phi^{2i}x_a)^{f/i}=(\phi^{2i}\iota^\varepsilon x_a)^{f/i}\in \mathrm{PSU}_3(q),$$
 which is $A$-conjugate to an element of $K_M\cap \mathrm{PSU}_3(q)\cong (q+1)^2.S_3$. Set $\sigma:=\phi^{2i}$. Now a direct computation using~\eqref{orderelt} shows that the $f/i$ power of $\phi^{2i}\iota^\varepsilon x_a$ is
$$\begin{pmatrix}
N(a^\varepsilon)&0&0\\
0&1&0\\
0&0&N(a^{-\varepsilon})
\end{pmatrix},$$
where $$N:\mathbb{F}_q\to\mathbb{F}_{q^{\frac{i}{f}}}$$ is the Galois norm of the field extension $\mathbb{F}_q/\mathbb{F}_{q^{i/f}}$. Since the norm map is surjective, there exists $a\in \mathbb{F}_q$ such that $N(a)$ has order $q^{i/f}-1=2^i-1.$ This proves that $K_M\cap \mathrm{PSU}_3(q)\cong (q+1)^2.S_3$ contains an element having order $2^i-1$. However, recalling that $i$ is odd, we have  that $2^i-1$ is relatively prime to $(q+1)^2.|S_3|$. Thus $i=1$, which is what we wanted.
\end{proof}

All good things in life at some point come to an end. Therefore, also in our work,
we need to wrap up and finish. We close this section observing that Theorems~\ref{theorem:silly} and~\ref{theorem:silly2} are far to give a complete classification of the degenerate normal 2-coverings
of almost simple groups. However, we believe that they could be the beginning of
interesting new research. 

Actually, we cannot resist from making one last observation concerning Theorem~\ref{theorem:silly2}. Indeed, in Theorem~\ref{theorem:silly2} part~\eqref{parTTT2}, for a suitable prime power $q$, every group of Lie type can arise. Rather than giving full details we just give one example using projective special linear groups.

Let $n$ be a positive integer with $n\ge 2$ and let $p$ be a prime number. Now, let $f$ be a prime number with $f$ relatively prime to $p^i-1$, for every $i\in\{1,\ldots,n\}$. For instance, we may take $f$ to be any prime number greater than $p^n-1$. Next, let $q:=p^f$, $G:=\mathrm{PSL}_n(q)$, let $\phi\in\mathrm{Aut}(G)$ be a field automorphism of order $f$ and let $A:=G\rtimes \langle\phi\rangle $.

Observe that $f$ is relatively prime to $|G|$: indeed, from Fermat's little theorem, for every $i\in \{1,\ldots,n\}$, we have $q^i=p^{if}\equiv p^i\pmod f$. Since $f$ is relatively prime to $p^i-1$, we deduce that also $f$ is relatively prime to $q^i-1$.

Now, let $H:=G=\mathrm{PSL}_n(q)$ and let $K:={\bf C}_A(\phi)=\mathrm{PSL}_n(p)\times\langle\phi\rangle$. We claim that $\{H,K\}$ is a normal $2$-covering of $A$. Let $g\in A$. If $g\in G$, then $g\in H$, because $H=G$. If $g\notin G$, then $g=s\phi^i$, where $i\in \{1,\ldots,f-1\}$ and $s\in G$. Since $g\pmod G$ has order $f$ in $A/G$, ${\bf o}(g)$ is divisible by $f$. Set $x:=g^{{\bf o }(g)/f}$ and observe that $x$ has order $f$. By Sylow's theorem, $\langle x\rangle$ and $\langle\phi\rangle$ are $A$-conjugate, and hence there exists $a\in A$ with $\langle x\rangle^a=\langle\phi\rangle$. Now,
$$g^a\in ({\bf C}_A(g^{{\bf o}(g)/f}))^a=({\bf C}_A(x))^a={\bf C}_A(x^a)={\bf C}_A(\phi)=K.$$

\end{document}